\documentclass[12pt]{amsart}
\usepackage{fullpage}
\usepackage{amsmath, amsthm, amsfonts}
\usepackage{amssymb}
\usepackage{latexsym}
\usepackage{verbatim}
\usepackage{url}
\usepackage[all,cmtip]{xy}
\usepackage{color}
\usepackage{hyperref}
\definecolor{mylinkcolor}{rgb}{0.6,0.0,0.1}
\definecolor{mycitecolor}{rgb}{0.0,0.0,0.65}
\definecolor{myurlcolor}{rgb}{0.0,0.37,0.0}
\hypersetup{colorlinks=true,urlcolor=myurlcolor,citecolor=mycitecolor,linkcolor=mylinkcolor,linktoc=page,breaklinks=true}

\usepackage{booktabs}
\usepackage{colonequals}

\usepackage{array}
\usepackage{nicefrac}
\usepackage[section]{placeins}

\usepackage[inline,shortlabels]{enumitem}
\setitemize{itemsep=4pt}
\setenumerate{itemsep=3pt}

\newtheorem{theorem}{Theorem}[section]

\newtheorem{proposition}[theorem]{Proposition}
\newtheorem{corollary}[theorem]{Corollary}
\newtheorem{lemma}[theorem]{Lemma}

\theoremstyle{definition}
\newtheorem{definition}[theorem]{Definition}
\newtheorem{example}[theorem]{Example}
\newtheorem{remark}[theorem]{Remark}

\numberwithin{equation}{section}

\newcommand{\Aut}{\operatorname{Aut}}
\newcommand{\Diag}{\operatorname{Diag}}
\newcommand{\End}{\operatorname{End}}
\newcommand{\et}{\mathrm{et}}
\newcommand{\Exp}{\mathrm{E}}

\newcommand{\Frob}{\operatorname{Frob}}
\newcommand{\Gal}{\mathrm{Gal}}
\newcommand{\Hom}{\operatorname{Hom}}
\newcommand{\Jac}{\operatorname{Jac}}

\newcommand{\Lie}{\operatorname{Lie}}
\newcommand{\Norm}{\operatorname{Norm}}

\newcommand{\Prym}{\operatorname{Prym}}
\newcommand{\Res}{\operatorname{Res}}
\newcommand{\Sym}{\operatorname{Sym}}
\newcommand{\topo}{\operatorname{top}}
\newcommand{\Trace}{\operatorname{Trace}}
\newcommand{\Zar}{\mathrm{Zar}}

\newcommand{\spl}{\operatorname{spl}}
\newcommand{\twist}{\operatorname{twist}}
\newcommand{\disc}{\operatorname{disc}}

\newcommand{\GL}{\mathrm{GL}}
\newcommand{\GSp}{\mathrm{GSp}}
\newcommand{\GU}{\mathrm{GU}}
\newcommand{\Hg}{\mathrm{Hg}}
\newcommand{\Lef}{\operatorname{L}}
\newcommand{\MT}{\mathrm{MT}}
\newcommand{\PGL}{\mathrm{PGL}}
\newcommand{\PSL}{\mathrm{PSL}}
\newcommand{\PSU}{\mathrm{PSU}}
\newcommand{\SL}{\mathrm{SL}}
\newcommand{\SO}{\mathrm{SO}}
\newcommand{\Sp}{\mathrm{Sp}}
\newcommand{\SU}{\mathrm{SU}}
\newcommand{\ST}{\mathrm{ST}}
\newcommand{\TL}{\operatorname{TL}}
\newcommand{\Unitary}{\mathrm{U}}
\newcommand{\USp}{\mathrm{USp}}

\newcommand{\cyc}[1]{{\mathrm{C}_{#1}}}
\newcommand{\cycsup}[2]{{\mathrm{C}_{#1}^{#2}}}
\newcommand{\dih}[1]{{\mathrm{D}_{#1}}}
\newcommand{\alt}[1]{{\mathrm{A}_{#1}}}
\newcommand{\sym}[1]{{\mathrm{S}_{#1}}}

\newcommand{\bA}{{\mathbf{A}}}
\newcommand{\bB}{{\mathbf{B}}}
\newcommand{\bC}{{\mathbf{C}}}
\newcommand{\bD}{{\mathbf{D}}}
\newcommand{\bE}{{\mathbf{E}}}
\newcommand{\bF}{{\mathbf{F}}}
\newcommand{\bG}{{\mathbf{G}}}
\newcommand{\bH}{{\mathbf{H}}}
\newcommand{\bI}{{\mathbf{I}}}
\newcommand{\bJ}{{\mathbf{J}}}
\newcommand{\bK}{{\mathbf{K}}}
\newcommand{\bL}{{\mathbf{L}}}
\newcommand{\bM}{{\mathbf{M}}}
\newcommand{\bN}{{\mathbf{N}}}

\newcommand{\C}{\mathbb C}
\newcommand{\F}{\mathbb F}
\newcommand{\Fp}{\F_p}
\newcommand{\Fq}{\F_q}

\newcommand{\M}{\mathrm{M }}

\newcommand{\Q}{\mathbb Q}
\newcommand{\Qbar}{{\overline{\mathbb Q}}}
\newcommand{\R}{\mathbb R}
\newcommand{\Z}{\mathbb Z}
\newcommand{\p}{\mathfrak{p}}
\newcommand{\q}{\mathfrak{q}}

\newcommand{\bu}{\mathbf{u}}
\newcommand{\bt}{\mathbf{t}}
\newcommand{\fm}{\mathfrak{m}}
\newcommand{\sfM}{\mathsf{M}}
\newcommand{\cO}{\mathcal{O}}

\newcommand{\Gap}{\textsc{GAP}}
\newcommand{\GitHub}{GitHub}
\newcommand{\LMFDB}{LMFDB}
\newcommand{\Magma}{\textsc{Magma}}
\newcommand{\Pari}{\textsc{PARI}}
\newcommand{\Sage}{\textsc{SageMath}}
\newcommand{\smalljac}{\textsc{smalljac}}

\newcommand{\ttimes}{{\hspace{-0.5pt}\times\hspace{0pt}}}

\newcommand{\arXiv}[2]{\href{https://arxiv.org/abs/#1}{arXiv:#1v#2}}

\newcommand{\mr}[1]{(\href{https://mathscinet.ams.org/mathscinet-getitem?mr=#1}{MR#1})}

\newcommand{\doi}[2]{\href{https://doi.org/#1}{#2}}
\newcommand{\jstor}[2]{\href{https://jstor.org/stable/#1}{#2}}

\newcommand*{\nfield}[2][]{\href{https://www.lmfdb.org/NumberField/#2}{{\ifx&#1& #2 \else #1 \fi}}}
\newcommand*{\stgroup}[2][]{\href{https://www.lmfdb.org/SatoTateGroup/#2}{{\ifx&#1& #2 \else #1 \fi}}}
\newcommand*{\group}[3][]{\href{http://www.lmfdb.org/Groups/Abstract/#2.#3}{{\ifx&#1& \langle #2, #3 \rangle \else #1 \fi}}}

\title{Sato--Tate groups of abelian threefolds}
\author{Francesc Fit\'e}
\address{Departament de matem\`atiques i inform\`atica,
Universitat de Barce\-lona,
Gran via de les Corts Catalanes 585, 08007 Barcelona, Catalonia, Spain}
\email{ffite@ub.edu}
\urladdr{\url{http://www.ub.edu/nt/ffite/}}
\author{Kiran S. Kedlaya}
\address{Department of Mathematics, University of California San Diego, 9500 Gilman Drive \#0112, La Jolla, CA 92093, United States}
\email{kedlaya@ucsd.edu}
\urladdr{\url{https://kskedlaya.org/}}
\author{Andrew V. Sutherland}
\address{Department of Mathematics,
Massachusetts Institute of Technology,
77 Massachusetts Ave., Cambridge, MA 02139, United States}
\email{drew@math.mit.edu}
\urladdr{\url{https://math.mit.edu/~drew}}


\begin{document}

\begin{abstract}
Given an abelian variety over a number field, its Sato--Tate group is a compact Lie group which conjecturally controls the distribution of Euler factors of the $L$-function of the abelian variety. It was previously shown by Fit\'e, Kedlaya, Rotger, and Sutherland that there are 52 groups (up to conjugation) that occur as Sato--Tate groups of abelian surfaces over number fields;
we show here that for abelian threefolds, there are 410 possible Sato--Tate groups, of which 33 are maximal with respect to inclusions of finite index.
We enumerate candidate groups using the Hodge-theoretic construction of Sato--Tate groups, the classification of degree-3 finite linear groups by Blichfeldt, Dickson, and Miller, and a careful analysis of Shimura's theory of CM types that rules out 23 candidate groups;
we cross-check this using extensive computations in \Gap, \Sage, and \Magma.
To show that these 410 groups all occur, we exhibit explicit examples of abelian threefolds realizing each of the 33 maximal groups; we also compute moments of the corresponding distributions and numerically confirm that they are consistent with the statistics of the associated $L$-functions.
\end{abstract}

\maketitle

\section{Introduction}\label{intro}

\subsection{Sato--Tate groups of abelian varieties}

Let $A$ be an abelian variety of dimension $g$ defined over a number field $k$. Associated to $A$ is the $L$-function
\[
L(A, s) = \prod_{\p} L_{\p}(\mathrm{Norm}(\p)^{-s})^{-1}
\]
where for each prime ideal $\p$ of the ring of integers $\cO_k$ of $k$ at which $A$ has good reduction,
$L_{\p}(T)$ is the reverse characteristic polynomial of Frobenius acting on the $\ell$-adic Tate module of $A$ for any prime $\ell$ coprime to $\p$. The polynomial $L_{\p}(T)$ has integer coefficients and its roots in $\C$ all lie on the circle $|T| = \Norm(\p)^{\nicefrac{-1}{2}}$; if one rescales these roots to have norm~1, it is then meaningful to ask whether the resulting polynomials are equidistributed with respect to some measure.

When $A$ is an elliptic curve, the answer to this question is well-known (although partially conjectural).
For $A$ with complex multiplication, one can show using Hecke's theory of Gr\"o\ss encharacters that one of exactly two distributions occurs, depending on whether or not the field of complex multiplication is contained in $k$.
By contrast, for $A$ without complex multiplication, the Sato--Tate conjecture predicts that exactly one distribution occurs, and this is known when $k$ is a totally real field \cite{HSBT10}, \cite{BLGG11} or a CM field \cite{AAC22}.

These equidistribution assertions can be made in a somewhat more robust fashion, which generalizes to $A$ of arbitrary dimension: one can define a certain compact Lie subgroup of $\USp(2g)$ associated to $A$ and a sequence of conjugacy classes therein whose characteristic polynomials are the normalized $L$-polynomials, which are equidistributed for the image of Haar measure. The only dependence on $A$ is then in the value of this group: for $g=1$, it is $\USp(2) = \SU(2)$ for $A$ without complex multiplication, and otherwise is either $\SO(2)$ or its normalizer in $\SU(2)$. Following a recipe of Serre \cite{Ser12}
(as expanded upon in \cite{BK15}), one can give a uniform definition of the group (and the sequence of conjugacy classes therein) applicable in cases where the Mumford--Tate conjecture is known, including all cases with $g \leq 3$; we call this group the \emph{Sato--Tate group} of $A$. One then conjectures a corresponding equidistribution statement,
which can be shown (as in \cite[\S I.A.2]{Ser68}) to follow from the analytic continuation of the $L$-functions associated to the various representations of the group; this is the approach taken in the results for elliptic curves
cited above. The equidistribution statement can be thought of as an analogue of the Chebotaryov density theorem, with the Sato--Tate group playing the role of the Galois group; this analogy can be built out by defining a Galois group associated to an arbitrary motive, as in
\cite{Ser94}.

Since the Sato--Tate group of a motive is a compact Lie group, one can retain much (but not all) of the same information by keeping track of the identity component and of the group of connected components. The identity component turns out to be invariant under base extension, and so can be interpreted over $\Qbar$ or even $\C$;
it is closely related to the \emph{Mumford--Tate group} of $A$.
By contrast, the component group is not invariant under base extension: it is canonically isomorphic to the Galois group of a certain extension of the base field, which in many cases (including all of dimension at most 3) is the \emph{endomorphism field} of $A$, defined as the minimum field of definition of all endomorphisms of the base extension
of $A$ to $\Qbar$.

\subsection{Abelian surfaces}

While \cite{Ser94} includes a general form of the Sato--Tate conjecture, it sheds very little light on the set of possible groups or measures that can occur for a particular class of motives.
Building on some initial investigations by Kedlaya and Sutherland \cite{KS08, KS09},
a comprehensive analysis of Sato--Tate groups of abelian surfaces was undertaken by Fit\'e, Kedlaya, Rotger, and Sutherland \cite{FKRS12}. The Sato--Tate group of an abelian surface is a closed subgroup of $\USp(4)$, with equality in the generic case; the classification of \cite{FKRS12} shows that there are 52 conjugacy classes of groups that can occur in general. If one restricts to abelian surfaces over $\Q$, then by analogy with the case of elliptic curves (where one cannot have complex multiplication defined over the base field) one expects to find somewhat fewer groups, and indeed only 34 of the 52 groups occur over $\Q$. (It is possible to find a single number field over which all 52 groups occur; see \cite{FG18}.)

We note in passing that potential modularity for abelian surfaces over totally real fields is now known
\cite{BCGP18}, but this is not enough to prove equidistribution for a generic abelian surface (one with trivial endomorphism ring over $\Qbar$); one needs potential modularity also for symmetric powers and other functorial transfers, which seems to be out of reach at the moment. By contrast, for nongeneric abelian surfaces equidistribution results are known in many cases due to work of Fit\'e and Sutherland \cite{FS14}, Johansson \cite{Joh17}, and Taylor \cite{Tay18}, which together cover all nongeneric abelian surfaces over $\Q$; see \cite[Remark~3.5]{Tay18}.

\subsection{Abelian threefolds}

We turn now from surfaces to threefolds and state our main result, which was previously announced in \cite{FKS19}.
The proof of this theorem occupies most of the present work, which has no logical dependence on \cite{FKS19}.

\begin{theorem} \label{T:ST result}
The following statements hold:
\begin{enumerate}[{\rm(a)}]
\item
There are exactly $410$ conjugacy classes of closed subgroups of $\USp(6)$ that occur as Sato--Tate groups
of abelian threefolds over number fields. This includes $14$ distinct options for the identity component.\label{Ta}
\item
Of the groups in \ref{Ta}, $33$ are maximal for inclusions of finite index.
\end{enumerate}
\end{theorem}

A corollary of the analysis is that the component group orders which are maximal for divisibility are $192 = 2^6 \times 3$, $432 = 2^4 \times 3^3$,
and $336 = 2^4 \times 3 \times 7$. As noted above, these translate into bounds on the degree of the minimal field of definition of all endomorphisms defined over $\Qbar$; these are consistent with the upper bounds
on this degree given by Silverberg \cite{Sil92} and Guralnick and Kedlaya \cite{GK17},
the latter being a sharp bound on the least common multiple of the endomorphism field degrees in a given dimension.

Compared to the case of abelian surfaces, Theorem~\ref{T:ST result} exhibits two notable deficiencies whose resolution we leave to future work. One is that we do not perform an analysis of Galois types as in \cite{FKRS12}, and so we
make no statements about realizability of particular Sato--Tate groups over particular number fields.
The other is that we do not determine which Sato--Tate groups can be realized by Jacobians of genus 3 curves over number fields, or even by principally polarized abelian threefolds; there is some reason to believe that there is a genuine distinction to be made here, in contrast with what happens for abelian surfaces.

\subsection{Applying the classification}

We insert here a few words about how to apply the classification to determine the Sato--Tate group of a given abelian threefold $A$. For this discussion, we conflate the Sato--Tate group with its associated distribution; this is harmless in dimension 2, but in dimension 3 this loses a small amount of information (see the discussion of new phenomena below).

As noted above, the Sato--Tate group controls the distribution of normalized $L$-polynomials of $A$, so one can make an empirical determination of the Sato--Tate group by compiling data about normalized $L$-polynomials.
For the Jacobians of hyperelliptic curves of genus $\leq 3$, this is done very effectively by Sutherland's \smalljac{}
package \cite{KS08}; this incorporates the technique described in \cite{HS14, HS16} to compute the $L$-polynomials
in average polynomial time. Similar approaches apply in other cases;
see \cite{HMS16} for double covers of pointless conics, \cite{Sut20} for other cyclic covers of $\mathbf{P}^1$,
and \cite{CKR20} for hypergeometric motives.

In order to discretize the problem of comparing the observed Sato--Tate distribution of~$A$ to the known candidates,
we employ the usual device of working with \emph{moment statistics}.
For the random matrix distribution associated to a linear group $G$, these statistics carry an additional representation-theoretic meaning: computing a moment statistic amounts to computing the average value of some symmetric polynomial of the eigenvalues of a random matrix from $G$, and for a polynomial with integer coefficients this corresponds to taking the multiplicity of the trivial character in a certain virtual character of $G$. In particular, this quantity is necessarily an integer; for Sato--Tate groups, even averaging over any one component of $G$ yields an integer value (see Definition~\ref{Sato--Tate axioms}).

While in principle moments are sufficient to completely distinguish distributions, for Sato--Tate distributions it is more useful in practice to supplement this data with some additional measurements of the $L$-polynomials. In particular, for disconnected groups it is quite common that certain functions of the eigenvalues are constant on certain nonidentity components. A familiar example is that of an elliptic curve with complex multiplication: every prime which is inert in the CM field gives a Frobenius trace of 0, corresponding to the fact that the trace is identically 0 on the nonidentity component of the normalizer of $\SO(2)$ in $\SU(2)$. By tracking these statistics in addition to moments, one can make a much more rapid putative determination of the Sato--Tate group than with moments alone.

With the empirical problem well understood, we now turn to the question of making a rigorous determination of the Sato--Tate group of $A$. One cannot hope to do this purely from computed $L$-polynomials except conditionally on the assumption that various motivic $L$-functions admit analytic continuation, functional equation, and the Riemann hypothesis, as these are needed to make the error term effective in the convergence to the Sato--Tate distribution. (see \cite{BFK20, BuK16}). Instead, we propose to compute the endomorphism algebra of $A$ over 
$\Qbar$, which together with its Galois action completely determines the Sato--Tate group. (Without the Galois action, one only recovers the identity component of the Sato--Tate group.)
A good practical approach to this problem has been introduced by Costa, Mascot, Sijsling, and Voight
\cite{CMSV19}, building on earlier work of van Wamelen; see \cite{CS19} for an implementation of this approach.

Finally, note that any \emph{a priori} knowledge about the structure of $A$ can potentially be fed into the previous calculations.
In particular, if $A$ is the Jacobian of a curve, then the automorphisms of the curve generate a subalgebra of the endomorphism algebra of $A$ which restricts the Sato--Tate group to a certain subgroup of $\USp(6)$; see Proposition~\ref{P:automorphism groups}.
For example, if $A$ is the Jacobian of a Picard curve, then its Sato--Tate group is contained in the normalizer of $\Unitary(3)$ in $\USp(6)$.

\subsection{New phenomena}

We mention some examples of phenomena that appear for abelian threefolds but not for abelian surfaces.

\begin{itemize}
\item
In dimension 3, one finds component groups that are not solvable, namely the simple group of order 168 and its double cover. These appeared previously in the analysis of Sato--Tate groups of twists of the Klein quartic \cite{FLS18}.

\item
One also finds for the first time component groups which are too big to be explained by twisting curves,
notably the Hessian group of order 216 (which exceeds the Hurwitz bound of 168 for the order of the automorphism group of a genus 3 curve) and its double cover.

\item
While in dimension 2 it is nearly automatic that every Sato--Tate group that can be realized occurs for the Jacobian of a curve, this is less clear in dimension 3. In some cases, there appears to be an obstruction even to constructing a realization that admits a principal polarization, although we do not attempt to prove this here.

\item
On a further related note, for abelian threefolds, we expect (but have not yet confirmed) that certain Sato--Tate groups can only occur 
in cases where the abelian threefold is isogenous to the cube of a \emph{particular} CM elliptic curve;
in particular, the CM field of this curve must occur within the endomorphism field of the threefold. The corresponding phenomenon for abelian surfaces is precluded by the work of Fit\'e and Guitart \cite[Theorem~1.6]{FG18}, who exhibit a single number field over which all~52 possible Sato--Tate groups occur.

\item
Whereas the 52 Sato--Tate groups of abelian surfaces correspond to 52 distinct distributions of normalized $L$-polynomials, there is one pair of Sato--Tate groups of abelian threefolds which give rise to identical distributions of $L$-polynomials, but are not isomorphic as abstract groups; in fact, their component groups are not isomorphic
(see Proposition~\ref{proposition: same measures}).
One might compare this to the phenomenon of nonisomorphic number fields with the same zeta function \cite{Per77}; see Remark~\ref{remark: Gassmann equivalence}.

\item
In dimension 3, the joint coefficient distribution of the normalized $L$-polynomials is not determined by the individual coefficient distributions of the normalized $L$-polynomials,
as is the case in dimension 2. This was already reported in \cite{FLS18}.
\end{itemize}

\subsection{Ingredients of the classification: upper bound}

The proof of Theorem~\ref{T:ST result} involves an upper bound, to show that every Sato--Tate group occurs among our 410 candidates, and a lower bound, to show that each of these candidates can occur.
We summarize here the key ingredients in the upper bound.

\begin{itemize}

\item
The classification of identity components follows by identifying the connected closed subgroups of $\USp(6)$, then comparing these to the classification of endomorphism algebras of abelian threefolds over $\C$
(\S\ref{section: endomorphism rings}).

\item
For the classification of component groups, we use the ``Sato--Tate axioms'' from \cite{FKRS12} (\S\ref{subsec:axioms}) to compile a list of candidate extensions of each of the 14 possible identity components (see \S\ref{section:STgroups} for all cases except $\Unitary(1)_3$). In the cases where the identity component has dimension at least 2, this is mostly straightforward given the results of \cite{FKRS12}. One key exception is that when the identity component is $\Unitary(1) \times \Unitary(1) \times \Unitary(1)$, a direct enumeration of groups yields 33 candidates, but a more refined analysis of CM types in the sense of Shimura \cite{Shi71} reduces this list to 13 groups. This is comparable to a feature of \cite{FKRS12}: there the direct group-theoretic classification yields 55 groups, of which 8 have identity component $\Unitary(1) \times \Unitary(1) \times \Unitary(1)$, but only~5 of the latter survive the matching of Galois types. 

\item
As one might guess from the case of abelian surfaces, the classification of component groups is by far the most complicated when the identity component is the one-dimensional diagonal torus $\Unitary(1)_3$ (\S\ref{section: unitary}).
To make this classification, we use the description of (and to some extent the notation for)
finite subgroups of $\PGL_3(\C)$ given by Blichfeldt, Dickson, and Miller \cite[Chapter~XII]{MBD61}.
Since $\Unitary(1)_3$ has index 2 in its normalizer in $\USp(6)$,
we must also figure out what double covers arise from a given subgroup of $\PGL_3(\C)$;
this is a rather delicate calculation involving the normalizers of these groups in $\PGL_3(\C)$.

\item
In light of the complexity of the argument, we make a number of computer calculations to verify some key steps. It would be desirable to give a computer-based verification of the entire classification, but we are not aware of any software that can handle compact Lie groups at the necessary level of generality.

\end{itemize}

\subsection{Ingredients of the classification: lower bound}

We now summarize the ingredients of the lower bound aspect of Theorem~\ref{T:ST result}. A complete proof of this direction was already included in the announcement in \cite{FKS19}, but our approach here contains some new features.

\begin{itemize}

\item
To show that the candidate Sato--Tate groups all occur, we restrict attention to the 33 maximal groups in the classification. Of these, about half arise from products of elliptic curves with abelian surfaces,
for which we may rely on the analysis from \cite{FKRS12} to obtain examples. Most of the others have identity component $\Unitary(1)_3$ (the diagonal torus in $\Unitary(1) \times \Unitary(1) \times \Unitary(1)$) and can be obtained by twisting the cube of a CM elliptic curve, using a similar framework as has been used previously for curves
\cite{FS14} and more recently for abelian varieties \cite{GK17}.

\item
We compute moments for the random matrix distribution corresponding to each possible Sato--Tate group,
using the Weyl character formula (\S\ref{sec:statistics}).
As noted above, there is one pair of groups (neither maximal) 
that give rise to identical distributions, so there are only 409 possible distributions for normalized $L$-polynomials of an abelian threefold over a number field.

\item
For each of the 33 maximal groups, we select an explicit example realizing this Sato--Tate group (Table~\ref{table: maximal ST groups})
and compare the distribution arising from its $L$-function to the predicted one (\S\ref{subsec:explicit L-functions}). 
This provides an empirical confirmation of the lower bound.
In some cases we provide multiple examples of distinct shapes, such as a product construction and a genus 3 Jacobian.

\item
We also confirm the matching of examples with Sato--Tate groups rigorously (\S\ref{section: realizability}).
Here we build upon existing analyses of twists of genus $3$ curves with extremal automorphism groups
\cite{ACLLM18, FLS18, FS16} (\S\ref{subsec:curves with automorphisms}); we also find some novel examples of genus $2$ curves whose Jacobians
have quaternionic multiplication (Proposition~\ref{P:automorphism groups} and Example~\ref{exa:type I JE6}).

\end{itemize}

\subsection{Tables, code, and data}

Given the complexity of the classification of Sato--Tate groups of abelian threefolds, it would be neither practical
nor useful to include tables at the same level of detail as in \cite[\S 6]{FKRS12}. We have elected instead to present the relevant data in several alternate forms.

\begin{itemize}

\item
Table~\ref{table:connected ST groups} lists the possible identity components of Sato--Tate groups of abelian threefolds,
as they correspond to real geometric endomorphism algebras, and indicates how many Sato--Tate groups appear in the classification for a given identity component. This table also introduces a shorthand notation by which each identity component is labeled by a letter called its \emph{absolute type}; for instance, the generic case $\USp(6)$ corresponds to absolute type $\bA$. 

\item
Table~\ref{table:3-diagonal} lists, for each possible connected Sato--Tate group of an abelian threefold, a certain numerical statistic
called the \emph{$3$-diagonal} (see Definition~\ref{definition: m-diagonal}). For disconnected groups, see instead the \LMFDB{} (see below).

\item
Table~\ref{table: types of generic automorphism groups} lists the possible automorphism groups of genus 3 curves and, for each of these groups,
the identity component of the Sato--Tate group of a generic curve with this automorphism group.  These assignments are verified in Proposition~\ref{P:automorphism groups}.

\item
Table~\ref{table: maximal ST groups} lists, for each of the 33 maximal groups in Theorem~\ref{T:ST result}(ii),
the identity component and component group, and one or more examples of abelian threefolds realizing this group as a Sato--Tate group.
These assignments have been verified both empirically by comparison of moment statistics, and rigorously by analysis of endomorphism algebras;
see \S\ref{subsec:explicit L-functions} for the latter. For endomorphism fields, see Table~\ref{table: maximal endomorphism fields}.

\item
We have gathered code referenced in this paper in a \GitHub{} repository \cite{FKS21}.
This code depends on \Gap{} \cite{GAP20}, \Magma{} \cite{Magma}, and \Sage{} \cite{Sage};
the latter also uses functionality derived from \Pari{} \cite{Pari}.
Components include:
presentations of the groups with identity component $\Unitary(1)_3$
(\S\ref{section: unitary});
consistency checks for the classification of groups with identity component $\Unitary(1)_3$
(\S\ref{subsec:consistency checks}); 
computation of moment statistics and other numerical invariants
(\S\ref{sec:statistics});
verification of the explicit realizations needed for the lower bound aspect of Theorem~\ref{T:ST result}(a)
(\S\ref{section: realizability}).

\item
We have created a \href{https://math.mit.edu/~drew/st3/g3SatoTateDistributions.html}{web site} with animated histograms that allow one to observe the convergence of various moment statistics for an explicit realization of each of the 410 groups in Theorem~\ref{T:ST result}(a).

\item
Each of the 410 groups in Theorem~\ref{T:ST result}(a) has its own home page in the \href{https://www.lmfdb.org/SatoTateGroup}{L-Functions and Modular Forms Database (\LMFDB)}. These pages include explicit generators, subgroups and supergroups, moment statistics (including the 3-diagonals), and point densities for each group.
\end{itemize}

\subsection{Desiderata in dimension 3}

While the classification of Sato--Tate groups of abelian threefolds can be considered complete in one sense,
there are a few things left to be done in order to achieve a result that is directly comparable to the classification for abelian surfaces.

\begin{itemize}
\item
Develop closed-form expressions for the moments of Sato--Tate groups of abelian threefolds (Remark~\ref{R:closed formulas for moments}, Remark~\ref{R:Kostant character formula}).
\item
Determine which Sato--Tate groups occur for abelian threefolds over $\Q$, or over totally real number fields; this may require an analysis of Galois types as in \cite{FKRS12}. In particular, there is a \emph{maximal} group which we were able to realize over  $\Q(\sqrt{3})$ but not over $\Q$ (Example~\ref{exa:type N 48 41}).

\item
Make a finer analysis of polarizations to determine which Sato--Tate groups occur for principally polarizable abelian threefolds.
This may involve combining our existing analysis of polarizations on twists (Remark~\ref{remark: twisting construction induced polarization}) with some arguments in the style of \cite{FG18}.

\item
Determine which Sato--Tate groups occur for Jacobians of genus 3 curves, or for products of elliptic curves with Jacobians of genus 2 curves, or for products of elliptic curves with Prym varieties of dimension 2. In particular, one of our examples involves an abelian surface admitting a principal polarization which we have not realized as the Jacobian of a genus 2 curve (Remark~\ref{R:identify Jacobian}). 

\end{itemize}

\subsection{Abelian varieties of higher dimension}

With the classification of Sato--Tate groups complete for abelian varieties of dimension at most 3, it is natural to ask about additional cases, and in particular abelian fourfolds. Going from threefolds to higher-dimensional
abelian varieties brings several new difficulties.
\begin{itemize}
\item
In order to get started one needs a classification of identity components, or equivalently of Mumford--Tate groups.
This has been done for dimension up to 5 by Moonen and Zarhin \cite{MZ99}.

\item
For fourfolds,
the number of groups resulting in the classification will run into the thousands, so some additional automation of the
classification will be required; in particular, when the identity component is a one-dimensional torus, 
computing double covers by hand is quite laborious and prone to errors.
Such automation should be possible, but will require new ideas.

\item
Starting in dimension 4, the possible components of the real endomorphism algebra include matrix algebras not only
over $\R$ and $\C$, but also over the quaternions.

\item
Starting in dimension 4, one encounters cases where the Mumford--Tate group is not determined solely by endomorphisms;
this phenomenon was first observed by Mumford \cite{Mu69} for simple abelian fourfolds and by Shioda
\cite[Example~6.1]{Shi81} for nonsimple abelian fourfolds.
A related issue is that the Mumford--Tate conjecture is not known for arbitrary fourfolds (see
\cite{Zha14} for discussion of the absolutely simple case). Consequently, one must be careful when passing between the Hodge-theoretic and Galois-theoretic definitions of the Sato--Tate group of an abelian fourfold.

\item
In order to determine the Sato--Tate groups whose identity components are tori, one must extend the analysis of
CM types to higher dimensions, where new combinatorial possibilities emerge; for instance, in dimension 4
it is possible to have a CM abelian variety over $\Q$ whose CM field is not a cyclic Galois extension of $\Q$
(it can also be a $\dih 4$-octic). We give some initial analysis here, but some additional work is needed to extract
a precise answer in any given dimension.

\item
It is possible that the rationality condition is not sufficient due to phenomena involving Schur indices (i.e., the distinction between a representation having traces in $\Q$ versus being defined over $\Q$), which have not appeared in dimensions up to 3.

\item
In sufficiently large dimensions, certain product groups may be obstructed by the failure of the corresponding endomorphism fields to be linearly disjoint. For instance, this may happen because both factors require complex multiplication by a particular imaginary quadratic field. (That said, the result of \cite{FG18} suggests that this may not occur in dimension 4.)

\item
In dimension 4, there is no reason to expect that all Sato--Tate groups can be realized using Jacobians of curves. (As a reminder, it remains unclear whether this holds even in dimension 3.)
Consequently, some new work would be needed in order to compute the $L$-functions
of sufficiently many abelian fourfolds to have a reasonable chance of finding examples of all possible groups;
reasonably good algorithms and implementations now exist for general curves
\cite{Tui16, Tui18, CT18}, but applying these to an abelian variety would require relating that variety to a Jacobian in some fashion that does not inflate the dimension too much (e.g., using the Prym construction).

\item
Expanding on the previous point, it is not even clear whether all Sato--Tate groups of abelian fourfolds can be realized using principally polarized abelian fourfolds. (Again, this is already unclear in dimension 3.)

\end{itemize}

It should be noted that it may be feasible to identify the Sato--Tate groups of individual abelian varieties of a given dimension without a full classification of all Sato--Tate groups that occur in that dimension, thanks to work of Zywina \cite{Zyw22}. 

In another direction, one can consider motives of weight greater than 1, although this requires some extra care
with the foundations (see \cite{BaK16} for odd weight and \cite{BK21} for even weight).
For example, \cite{FKS18} gives a classification of Sato--Tate groups associated to 
motives of weight 3 with Hodge vector $(1,1,1,1)$ (e.g., the symmetric cube of an elliptic curve).
A natural next step would be to consider motives of weight 2 associated to K3 surfaces, starting with those of high Picard number; this can be viewed as including the case of abelian surfaces via the Kummer construction of K3 surfaces.
See \cite{EJ23} for some work in this direction.

One can also ask similar questions for abelian varieties (or other motives) over function fields. In this case,
the equidistribution problem can be settled using the known analytic properties of $L$-functions, as originally
observed by Deligne \cite[Th\'eor\`eme~3.5.3]{Del80}; the combinatorial problem of classifying images of monodromy representations remains to be settled, but the exact constraints are slightly different due to the absence of archimedean places. This means that Hodge theory no longer plays a role, although $p$-adic Hodge theory continues to do so.

\subsection{Acknowledgments}

Thanks to Armand Brumer, Daniel Bump, Edgar Costa, John Cremona, Heidi Goodson, Ever\-ett Howe, Alexander Hulpke, P\i nar K\i l\i \c{c}er, Jeroen Hanselman, Elisa Lorenzo Garc\'\i a, Jennifer Paulhus, Jean-Pierre Serre, Jeroen Sijsling, Liang Xiao, and David Zywina for helpful discussions.
Kedlaya also thanks the organizers of AGCCT 2019 for the invitation to present \cite{FKS19}.

Financial support was received as follows:
Fit\'e: AEI (grant RYC-2019-027378-I), DGICYT (grants MTM2009-13060-C02-01, MTM2015-63829-P), German Research Council (SFB 701, SFB/TR 45), IAS (NSF grant DMS-1638352 plus additional funds),
ERC (grant 682152);
Kedlaya: NSF (grants DMS-0545904,
DMS-1101343, DMS-1501214, DMS-1802161, DMS-2053473), DARPA (grant HR0011-09-1-0048), MIT (NEC Fund,
Cecil and Ida Green professorship), UC San Diego (Stefan E. Warschawski professorship),
Guggenheim Fellowship (fall 2015),
IAS (visiting professorship 2018--2019); Sutherland: NSF
(grants DMS-1115455, DMS-1522526), Google Cloud Platform (GCP Research Credit Program).  All three authors were supported by the Simons Collaboration on Arithmetic Geometry, Number Theory, and Computation (Simons Foundation grant 550033).

Additional hospitality was received as follows:
Fit\'e: UC San Diego (February--March 2012);
Kedlaya: MSRI (spring 2011, fall 2014, spring 2019, spring 2023), IPAM (spring 2014);
Kedlaya and Sutherland: ICERM (fall 2015); all authors: CIRM (February--March 2023).

\section{Conventions}

We start by establishing some conventions that we follow throughout the text.

Throughout, we write $A$ to denote an abelian variety of dimension $g \geq 1$ over a number field $k$.
We consistently consider abelian varieties \emph{without} a preferred choice of polarization, as the Sato--Tate group does not depend on the choice;
but see \S\ref{subsec:polarizations of twists} for some discussion of polarizations in the context of twisting powers of elliptic curves.

For $\ell$ a field extension of $k$ (not necessarily a number field), write $A_\ell$ for the base extension of $A$ from $k$ to $\ell$.
We commonly use this notation with $\ell$ being the minimal field of definition of all endomorphisms of $A_{\Qbar}$; we call this the \emph{endomorphism field} of $A$ and denote it by $K$.

We write $\End(A)$ for the endomorphism ring of $A$ \emph{over $k$}; to denote endomorphisms of a base extension we indicate this explicitly (so $\End(A_{\Qbar})$ is the geometric endomorphism ring).
We write $\End(A)_{\Q}$ and $\End(A)_{\R}$ for the tensor products $\End(A) \otimes_{\Z} \Q$ and $\End(A) \otimes_{\Z} \R$.

For $C$ a curve over a field, we write $\Jac(C)$ rather than $J(C)$ to denote the Jacobian of~$C$. This makes the letter $J$ available for other uses, such as in the notation for Sato--Tate groups.
We use $\Aut(C)$ to denote the automorphism group of $C$.
We define the \emph{reduced automorphism group} $\Aut'(C)$ as the quotient of $\Aut(C)$
by the hyperelliptic involution if one exists, and as $\Aut(C)$ otherwise.

We write $\cyc n, \dih n, \alt n, \sym n$ for the cyclic, dihedral, alternating, and symmetric groups of respective orders $n, 2n, n!/2, n!$ (that is, $n$ is the degree of the standard permutation representation). 
In some cases, we will follow the practice of \Gap{} and \Magma{} and 
use the unique (but opaque) compact labels for isomorphism classes of finite groups
from the Small Groups library \cite{BEO01}; the label for a group $G$ always has the form $\langle n,i \rangle$ where $n = \#G$.

We present hyperelliptic curves (of any genus) as affine curves in lowercase variables (typically $y^2 = P(x)$) and plane quartics as projective curves in uppercase variables
(i.e., $P(X,Y,Z) = 0$). We will exploit this typographic choice in order to commingle equations of the two types without risk of confusion.
In rare cases, we extend the hyperelliptic convention to other cyclic covers of $\mathbb{P}^1$, such as Picard curves.
In curve equations, we sometimes write $P_d(\ldots)$ to indicate an unspecified homogeneous polynomial of degree $d$ over $k$ in the
indicated variables.

We write $\spl(f)$ for the splitting field of $f\in \Q[x]$. Note that $f$ need not be irreducible; if $f = f_1 f_2$, we also denote this field by $\spl(f_1,f_2)$.

We will make frequent references into the L-Functions and Modular Forms Database (\LMFDB) \cite{LMFDB}.
We remind the reader to consult the \LMFDB{} for more details about the computation of data computed therein; see in particular the following pages for attributions related to the corresponding object types.
\begin{itemize}
\item
Finite groups: \url{https://beta.lmfdb.org/Groups/Abstract/Source}.
\item Number fields: \url{https://www.lmfdb.org/NumberField/Source}.
\item
Elliptic curves over $\Q$: \url{http://www.lmfdb.org/EllipticCurve/Q/Source}.
\item
Elliptic curves over other number fields: \url{https://www.lmfdb.org/EllipticCurve/Source}.
\item
Genus $2$ curves over $\Q$: \url{http://www.lmfdb.org/Genus2Curve/Source}.
\item
Automorphism groups of curves over $\C$: \url{http://www.lmfdb.org/HigherGenus/C/Aut/Source}.
\end{itemize}

In addition, the \LMFDB{} includes a database of Sato--Tate groups combining the results of \cite{FKRS12} with the classification for abelian threefolds made in this paper:
\begin{center} 
\url{https://www.lmfdb.org/SatoTateGroup/}.
\end{center}
In most cases, we have formatted labels of Sato--Tate groups of abelian varieties as hyperlinks to the home pages of these groups.
We also format labels of other objects found in the \LMFDB{}, such as finite groups, elliptic curves over number fields, and genus 2 curves over~$\Q$, as hyperlinks. For number fields themselves, we omit the hyperlink when the field is represented by adjoining a single algebraic number (e.g, $\Q(i)$); instead, in the notation  $\spl(f)$ for $f \in \Q[x]$ irreducible, if the number field defined by~$f$ (not the splitting field) is found in the \LMFDB{}, then we format $f$ as a hyperlink to this~field.

\section{Background on Sato--Tate groups}

In this section, we recall some of the theoretical results from \cite{FKRS12}
which therein
provide the basis for the classification of Sato--Tate groups of abelian surfaces, and herein
do likewise for abelian threefolds. 
We then classify the possible identity components of Sato--Tate groups of abelian threefolds.
See also \cite{Sut16} for an overview of this circle of ideas.

\subsection{The Mumford--Tate group}

We begin by recalling the definition of the Mumford--Tate group of an abelian variety over the number field $k$.
This construction carries the same information as the identity component of the Sato--Tate group.

\begin{definition}
Fix an embedding of $k$ into $\C$ and a polarization of $A$
and set $V \colonequals H^1(A_{\C}^{\topo}, \Q)$; it carries an alternating pairing $\psi$ induced by the polarization.
Note that while the rational symplectic space $(V, \psi)$ depends on the choice of the polarization, the real symplectic space $(V \otimes_{\Q} \R, \psi)$ does not.

Via the description of $A_{\C}^{\topo}$ as a complex torus, we obtain an $\R$-linear identification of $V \otimes_{\Q} \R$ with the tangent space of $A_{\C}$ at the origin (or equivalently, the dual of the space of holomorphic differentials on $A_{\C}$). The \emph{Mumford--Tate group} of $A$, denoted $\MT(A)$, is the smallest $\Q$-algebraic subgroup of $\GL(V)$ whose base extension to $\R$ contains the scalar action of $\C^\times$ on
$V \otimes_{\Q} \R$; this group is contained in the symplectic group $\GSp(V, \psi)$.
It follows from Deligne's theorem on absolute Hodge cycles (see \cite{DM82})
that the definition of $\MT(A)$ does not depend on the choice of the embedding of $k$ into $\C$; it is also clearly invariant under enlarging $k$.
\end{definition}

\begin{definition}
The \emph{Hodge group} $\Hg(A)$ is the intersection $\MT(A) \cap \Sp(V, \psi)$.
The \emph{Lefschetz group} $\Lef(A)$ is the identity 
connected component of the centralizer of $\End(A_{\overline{k}})_{\Q}$ in $\Sp(V, \psi)$.
There is an obvious inclusion $\Hg(A) \subseteq \Lef(A)$.
\end{definition}

\begin{proposition} \label{Hodge equals Lefschetz}
For $g \leq 3$, we have $\Hg(A) = \Lef(A)$. (That is, the Mumford--Tate group is completely controlled by 
endomorphisms.)
\end{proposition}
\begin{proof}
See \cite[Theorem~2.16]{FKRS12}
or \cite[Theorem~9.4]{BK15}.
\end{proof}

\subsection{Definition of the Sato--Tate group}

We continue with the definition of the algebraic Sato--Tate group and Sato--Tate group of an abelian variety over 
a number field in
terms of $\ell$-adic monodromy, and the statement of the refined Sato--Tate conjecture.
This material is taken from \cite[\S 2.1]{FKRS12}; see also \cite[\S 2]{BK15}, \cite{BaK16}.

\begin{definition}
For each $\tau \in G_k$, define
\[
\Lef(A, \tau) \colonequals \{\gamma \in \Sp(V, \psi): \gamma^{-1} \alpha \gamma = \tau(\alpha) \mbox{ for all } \alpha \in \End(A_{\overline{k}})_{\Q}\}.
\]
The \emph{twisted Lefschetz group} $\TL(A)$ is the union of $\Lef(A, \tau)$ over all $\tau$.
\end{definition}

\begin{proposition} \label{MT conjecture in dim 3}
For $g \leq 3$, for any prime $\ell$, $\TL(A) \otimes_{\Q} \Q_\ell$ is the kernel of the symplectic character 
on the Zariski closure of the image of $\rho_\ell: G_k \to \GSp(H^1_{\et}(A_{\overline{k}}, \Q_\ell), \psi)$.
\end{proposition}
\begin{proof}
This amounts to the statement that the Mumford--Tate conjecture holds for $A$ whenever $g \leq 3$. See for example 
\cite[Theorem~6.11]{BK15}.
\end{proof}

\begin{definition}
In light of Proposition~\ref{Hodge equals Lefschetz} and Proposition~\ref{MT conjecture in dim 3},
for $g \leq 3$, we define the \emph{Sato--Tate group} $\ST(A)$ to be a maximal compact subgroup of $\TL(A)(\C)$;
we recall from \cite[\S VII.2]{Kna02} that such subgroups exist, any two are conjugate, and one example is the intersection $\TL(A)(\C) \cap \Unitary(2g)$ within $\GL(H^1_{\et}(A_{\overline{k}}, \Q_\ell))(\C)$.
For $C$ a curve, we define the Sato--Tate group of $C$ to be that of its Jacobian $\Jac(C)$.
\end{definition}

\begin{lemma}\label{lemma: Galois action of component group}
For $g \leq 3$, there is a canonical isomorphism $\ST(A)/\ST(A)^0 \to \Gal(K/k)$ where $K$ is the endomorphism field of $A$. In particular, this isomorphism is compatible with base change: for any finite extension $k'$ of $k$, 
$\ST(A_{k'})$ is the inverse image of the subgroup $\Gal(k'K/k') \subseteq \Gal(K/k)$ in $\ST(A)$.
\end{lemma}
\begin{proof}
This is again a consequence of \cite[Theorem~6.11]{BK15}.
\end{proof}

\begin{remark} \label{remark: general ST group}
In the definition of the Sato--Tate group, we are implicitly using the fact that for $g \leq 3$, 
the Mumford--Tate group of $A$ is determined by endomorphisms. For general $g$, it is expected that the role of
$\TL(A)$ in Proposition~\ref{MT conjecture in dim 3} can be filled by a certain algebraic group over $\Q$,
the \emph{motivic Sato--Tate group}, whose definition involves algebraic cycles on $A \times A$ of all codimensions, not just endomorphisms \cite{BaK16}. There is still a canonical surjection $\ST(A)/\ST(A)^0 \to \Gal(K/k)$,
but it is not in general an isomorphism.
\end{remark}

\subsection{Twisting and the Sato--Tate group}

\begin{definition} \label{D:twist}
For $L/k$ a finite Galois extension and $\xi$ a $1$-cocycle of $\Gal(L/k)$ valued in $\Aut(A_L)$, 
there exists a unique (up to unique isomorphism) abelian variety $A_\xi$ over $k$ and an isomorphism
$\theta: A_{\xi,L} \to A_L$ such that
\[
\xi_\sigma = \theta^\sigma \circ \theta^{-1} \qquad (\sigma \in \Gal(L/k)).
\]
Moreover, $A_\xi$ depends only on the class of $\xi$ in the pointed set $H^1(\Gal(L/k), \Aut(A_L))$.
\end{definition}

\begin{definition}
Set notation as in Definition~\ref{D:twist} with $L = K$ (where $K$ is again the endomorphism field of $A$).
As described in \cite[Lemma~2.3]{FS14}, there is an inclusion
\[
\TL(A_\xi) \subseteq \TL(A) \cdot \Aut(A_L)
\]
in $\Sp(V, \psi)$. 
(Recall that $\Aut(A_L) \to \Sp(V, \psi)$ is injective; e.g., over $\C$ this is evident from the complex uniformization.)
Taking a maximal compact subgroup of $(\TL(A) \cdot \Aut(A_L))(\C)$ gives the 
\emph{twisting Sato--Tate group} of $A$ \cite[Definition~2.2]{FS14}.
Achieving this group as a Sato--Tate group of $A$ amounts to solving a certain Galois embedding problem for
the component group of the twisting Sato--Tate group.
\end{definition}

\begin{remark}
If $A$ is equipped with a polarization, then we may similarly consider twisting by a cocycle $\xi$
valued in automorphisms of $A_L$ as a polarized abelian variety, and then $A_\xi$ will acquire a polarization with respect to which $\theta$ is an isomorphism of polarized abelian varieties. 
In particular, if $A$ is the Jacobian of the curve $C$ equipped with the corresponding principal polarization,
then by the Torelli theorem, twisting $A$ as a polarized abelian variety corresponds to twisting $C$ as a curve.
\end{remark}

\subsection{Axioms for Sato--Tate groups}
\label{subsec:axioms}

A key tool used to classify Sato--Tate groups in \cite{FKRS12} is a list of necessary conditions
called the \emph{Sato--Tate axioms}. The formulation in \cite[Definition~3.1]{FKRS12} is applicable to
arbitrary motives; we state here a restricted form of the three original Sato--Tate axioms applicable only
to the 1-motives associated to abelian threefolds. We also add a fourth axiom coming from the 
fact that the Sato--Tate group of an abelian threefold is determined by endomorphisms
(Proposition~\ref{Hodge equals Lefschetz}).

\begin{definition} \label{Sato--Tate axioms}
For a group $G$ with identity connected component $G^0$, the \emph{Sato--Tate axioms (for abelian threefolds)} are as follows.
\begin{enumerate}
\item[(ST1)]
The group $G$ is a closed subgroup of $\USp(6)$. For definiteness, we take the latter to be defined with respect to the symplectic form given by the block matrix
\[
J = \begin{pmatrix} 0 & I_3 \\
-I_3 & 0
\end{pmatrix}
\]
unless otherwise specified.
\item[(ST2)] (Hodge condition)
There exists a homomorphism $\theta: \Unitary(1) \to G^0$
such that $\theta(u)$ has eigenvalues $u, u^{-1}$ each with multiplicity $3$;
the image of $\theta$ is called a \emph{Hodge circle}.
Moreover, the Hodge circles generate a dense subgroup of $G^0$.
\item[(ST3)] (Rationality condition)
For each component
$H$ of $G$ and irreducible character $\chi$ of~$\GL(\C^6)$,
the expected value (under the Haar measure) of $\chi(\gamma)$ over $\gamma\in H$ is an integer.
In particular, $\Exp[\Trace(\gamma, \wedge^m \C^{6})^n: \gamma \in H]\in\Z$ for all integers $m,n>0$.
\item[(ST4)] (Lefschetz condition)
The subgroup of $\USp(6)$ fixing $\End(\C^6)^{G^0}$ is equal to $G^0$.
\end{enumerate}
\end{definition}

\begin{proposition} \label{necessity of Sato--Tate}
Let $A$ be an abelian threefold over $k$.
Then $G = \ST_A$ satisfies the Sato--Tate axioms.
\end{proposition}
\begin{proof}
For (ST1), (ST2), (ST3), this is \cite[Proposition~3.2]{FKRS12}
except that the last condition in (ST2) is not stated therein; that statement is a consequence of
the definition of the Mumford--Tate group \cite[Definition~2.11]{FKRS12}.
For (ST4), apply Proposition~\ref{Hodge equals Lefschetz}.
\end{proof}

\subsection{Endomorphism rings of abelian threefolds}\label{section: endomorphism rings}

We next recall the classification of geometric endomorphism rings of abelian threefolds.
For $A$ simple of dimension $g \leq 3$, the work of Albert and Shimura leaves the following possibilities for the $\Q$-algebra of endomorphisms $\End(A_K)_\Q$.
(See for example \cite{Oo87}.) 
\begin{itemize}
\item If $g=1$: $\Q$ (non-CM type), or a quadratic imaginary field (CM type).
\item If $g=2$: $\Q$ (type I(1)), a real quadratic field (type I(2)), an indefinite division quaternion algebra over $\Q$ (type II(1)), or a quartic CM field (type IV(2,1)).
\item If $g=3$: $\Q$ (type I(1)), a totally real cubic field (type I(3)), an imaginary quadratic field (type IV(1,1)), or a CM field of degree $6$ over $\Q$ (type IV(3,1)). 
\end{itemize}
It easily follows that for $A$ of dimension 3, not necessarily simple, there are $14$ possibilities for the real endomorphism algebra $\End(A_K)_\R$, as listed in \cite[\S 4, Table~2]{Sut16}:
\begin{enumerate}
\setlength{\itemsep}{3pt}
\item[($\bA$)] $\R$, when $A_K$ is simple of type I(1);
\item[($\bB$)] $\C$, when $A_K$ is simple of type IV(1,1);
\item[($\bC$)] $\R\times\R$, when $A_K$ is isogenous to the product of an elliptic curve without CM and a simple abelian surface of type I(1);
\item[($\bD$)] $\C\times \R$, when $A_K$ is isogenous to the product of an elliptic curve with CM and a simple abelian surface of type I(1);
\item[($\bE$)] $\R\times\R\times\R$, when either
\begin{itemize}
\setlength{\itemsep}{0pt}
\item $A_K$ is simple of type I(3), or
\item $A_K$ is isogenous to the product of an elliptic curve without CM and a simple abelian surface of type I(2), or
\item $A_K$ is isogenous to the product of three pairwise nonisogenous elliptic curves without CM;
\end{itemize}
\item[($\bF$)] $\C\times\R\times\R$, when either
\begin{itemize}
\setlength{\itemsep}{0pt}
\item $A_K$ is isogenous to the product of an elliptic curve with CM and two nonisogenous elliptic curves without CM, or
\item $A_K$ is isogenous to the product of an elliptic curve with CM and a simple abelian surface of type I(2);
\end{itemize}
\item[($\bG$)] $\C\times\C\times\R$, when either
\begin{itemize}
\setlength{\itemsep}{0pt}
\item $A_K$ is isogenous to the product of an elliptic curve without CM and two nonisogenous elliptic curves with CM, or
\item $A_K$ is isogenous to the product of an elliptic curve without CM and a simple abelian surface of type IV(2,1);
\end{itemize}
\item[($\bH$)] $\C\times\C\times\C$, when either
\begin{itemize}
\setlength{\itemsep}{0pt}
\item $A_K$ is simple of type IV(3,1), or
\item $A_K$ is isogenous to the product of an elliptic curve with CM and a simple abelian surface of type IV(2,1), or
\item $A_K$ is isogenous to the product of three pairwise nonisogenous elliptic curves with CM; 
\end{itemize}
\item[($\bI$)] $\R\times\M_2(\R)$, when either
\begin{itemize}
\setlength{\itemsep}{0pt}
\item $A_K$ is isogenous to the product of an elliptic curve without CM and a simple abelian surface of type II(1), or
\item $A_K$ is isogenous to the product of an elliptic curve $E$ without CM and the square of an elliptic curve without CM nonisogenous to $E$;
\end{itemize}
\item[($\bJ$)] $\C\times\M_2(\R)$, when either
\begin{itemize}
\setlength{\itemsep}{0pt}
\item $A_K$ is isogenous to the product of an elliptic curve with CM and the square of an elliptic curve without CM, or
\item $A_K$ is isogenous to the product of an elliptic curve with CM and a simple abelian surface of type II(1);
\end{itemize}
\item[($\bK$)] $\R\times\M_2(\C)$, when $A_K$ is isogenous to the product of an elliptic curve without CM and the square of an elliptic curve with CM;
\item[($\bL$)] $\C\times\M_2(\C)$, when $A_K$ is isogenous to the product of an elliptic curve $E$ with CM and the square of an elliptic curve with CM nonisogenous to $E$;
\item[($\bM$)] $\M_3(\R)$, when $A_K$ is isogenous to the cube of an elliptic curve without CM;
\item[($\bN$)] $\M_3(\C)$, when $A_K$ is isogenous to the cube of an elliptic curve with CM.
\end{enumerate}

The letter used above to refer to $\End(A_K)_\R$ is called the \emph{absolute type} of $A$.

\subsection{Connected Sato-Tate groups of abelian threefolds}\label{section: connected groups}

Let $G^0$ be a connected group satisfying the Sato-Tate axioms. Its Lie algebra $\mathfrak h$ is of rank at most 3. The set of Lie algebras of rank at most $3$ can easily be determined from the classification of Dynkin diagrams.

$$
\begin{array}{|c|c|}\hline
r & \mbox{Lie algebras of rank }r \\\hline
1 & \mathfrak{t}_1,\ \mathfrak{sl}_2 = \mathfrak{so}_3 \\\hline
2 & \mathfrak{t}_2,\ \mathfrak{t}_1 \ttimes \mathfrak{sl}_2,\ \mathfrak{sl}_2 \ttimes \mathfrak{sl}_2 = \mathfrak{so}_4,\ \mathfrak{sl}_3,\  \mathfrak{sp}_4 = \mathfrak{so}_5,\ \mathfrak{g}_2 \\\hline	
3 & 
\begin{array}{c}
\mathfrak{t}_3, \mathfrak{t}_2 \ttimes \mathfrak{sl}_2,\ 
\mathfrak{t}_1 \ttimes \mathfrak{sl}_2 \ttimes \mathfrak{sl}_2,\ 
\mathfrak{t}_1 \ttimes \mathfrak{sl}_3, \mathfrak{t}_1 \ttimes \mathfrak{sp}_4, \mathfrak{t}_1 \ttimes \mathfrak{g}_2,\ \mathfrak{sl}_2 \ttimes \mathfrak{sl}_2 \ttimes \mathfrak{sl}_2, \\
\mathfrak{sl}_2 \ttimes \mathfrak{sl}_3,\ \mathfrak{sl}_2 \ttimes \mathfrak{sp}_4, \mathfrak{sl}_2 \ttimes \mathfrak{g}_2,\ \mathfrak{sl}_4 = \mathfrak{so}_6,\  \mathfrak{so}_7,\ \mathfrak{sp}_6
\end{array}\\\hline
\end{array}
$$
\smallskip

The Lie algebra $\mathfrak h$ comes equipped with a 6-dimensional
unitary symplectic self-dual representation without trivial factors. We leave as an exercise for the reader to see that this eliminates the Lie algebras $ \mathfrak{sp}_4$, $\mathfrak{g}_2$, $\mathfrak{t}_1 \ttimes \mathfrak{g}_2$, $\mathfrak{sl}_2 \ttimes \mathfrak{g}_2$, $\mathfrak{sl}_2 \ttimes \mathfrak{sl}_3$, $\mathfrak{so}_6$, $\mathfrak{so}_7$
from the above list (see \cite[\S3.1]{FKS19} for further details). Together with an additional argument to rule out the
group $\Unitary(2)$ (see \cite[Rem. 2.3]{FKS18}), this leaves the $14$ possibilities for the group $G^0$ listed in the second column of Table~\ref{table:connected ST groups}. 
The group $G^0$ comes equipped with a 6-dimensional representation,
which may be inferred from the table using the following considerations.
\begin{itemize}
\item
The groups $\SU(2), \USp(4), \USp(6)$ carry their standard representations.
\item
Products of groups correspond to direct sums of representations, in which we interlace coordinates to match up the symplectic forms. That is, a product $G_1 \times G_2$ is represented using block matrices
\[
\begin{pmatrix} * & 0 & * & 0 
\\
0 & \star & 0 & \star 
\\
* & 0 & * & 0 
\\
0 & \star & 0 & \star
\end{pmatrix}
\]
where the blocks labeled $*$ represent an element of $G_1$ and the blocks labeled $\star$ represent an element of $G_2$,
and similarly for a threefold product. (Beware that other orderings of coordinates can be found in some presentations constructed in our code, in the computation of moments in \S\ref{sec:statistics}, and in the \LMFDB{}.)
\item
The groups $\Unitary(1), \Unitary(3)$ are embedded into $\SU(2), \USp(6)$ respectively by the formula
\[
A \mapsto \begin{pmatrix} A & 0 \\ 0 & \overline{A} 
\end{pmatrix}.
\]
\item
For $d=2,3$, $\Unitary(1)_d, \SU(2)_d$ refer to the $d$-fold diagonal representations of $\Unitary(1), \SU(2)$; in terms of block matrices, this means that
\[
\begin{pmatrix} 
A_{11} & A_{12} \\
A_{21} & A_{22}
\end{pmatrix}
\mapsto
\begin{pmatrix}
A_{11}I_d & A_{12}I_d \\
A_{21}I_d & A_{22}I_d
\end{pmatrix}.
\]
\end{itemize}

We note in passing that $\SO(3)$ and $\SU(3)$ (embedded via $\Unitary(3)$) satisfy (ST1), (ST2), (ST3) but not (ST4); this explains their absence above. Note also that $\Unitary(1)_2 \times \SU(2)$ and $\Unitary(1) \times \SU(2)_2$ are isomorphic but nonconjugate in $\USp(6)$.

By \cite[Prop. 2.19]{FKRS12}, the group $\ST(A)^0$ uniquely determines $\End(A_K)_\R$. As there are~14 possibilities for each, each option for $\End(A_K)_\R$
corresponds to a unique option for $\ST(A)^0$, as described in Table~\ref{table:connected ST groups}. 

\section{The classification upper bound, part 1}
\label{section:STgroups}

In this section and the next, we establish the ``upper bound'' assertion of Theorem~\ref{T:ST result},
by showing that there are at most 410 conjugacy classes of subgroups of $\USp(6)$ that can occur as the Sato--Tate groups of abelian threefolds over number fields. 
We organize the discussion in terms of the identity component, making an analogue in each case of 
the group-theoretic classification from \cite[\S 3]{FKRS12};
the case where this equals $\Unitary(1)_3$ requires more analysis than all of the others combined,
and so is treated separately in \S\ref{section: unitary}. In addition, when the identity component contains
$\Unitary(1) \times \Unitary(1)$ as a factor, we refine the group-theoretic argument by making a further analysis of CM types; this adapts a key component of the analysis of Galois types from \cite[\S 4]{FKRS12}, which we otherwise do
not reproduce.

Along the way, we identify those conjugacy classes of subgroups that are maximal with respect to inclusions of finite index,
which is to say inclusions that preserve the identity component. We say for short that such conjugacy classes are \emph{maximal}.

Throughout \S\ref{section:STgroups}, 
let $G$ denote a closed subgroup of $\USp(6)$ satisfying the Sato--Tate axioms
(Definition~\ref{Sato--Tate axioms}); let $G^0$ denote the identity connected component of $G$;
and let $Z$ and $N$ denote the centralizer and normalizer, respectively, of $G^0$ within $\USp(6)$.

\subsection{Basic cases}
\label{sec:basic cases}

We start with two of the least complicated options for $G^0$.
In the case $G^0 = \USp(6)$ (type $\bA$), it is clear that $G = G^0$. In the case $G^0 \simeq \Unitary(3)$
(type $\bB$), we have
$Z = \Unitary(1)_3 \subset G^0$ and $N = G^0 \cup J G^0$.
Consequently, $N/G^0 \simeq \cyc 2$, so $G$ is equal to either $G^0$ or~$N$.
In both cases, $N$ is the unique maximal subgroup.

\subsection{Split products}
\label{section: split products}

We next analyze the cases in which $G^0$ splits as a product $H_1 \times H_2$
in which $H_1$ and $H_2$ have no common factor; this rules out the cases 
$G^0 = \Unitary(1) \times \Unitary(1) \times \Unitary(1)$ and 
$G^0 = \SU(2) \times \SU(2) \times \SU(2)$, for which see below.
In the remaining cases, $G$ is a subgroup of $G_1 \times G_2$ where $G_i$ is some group with connected part $H_i$
and the projections $G \to G_i$ are surjective.
Moreover, this splitting corresponds to an isogeny factorization of $A$, so we may rely on the classification of Sato--Tate groups of abelian surfaces made in \cite{FKRS12}.

As an aid to the reader, we have reproduced in Table~\ref{table: genus 2 recall part 1} and Table~\ref{table: genus 2 recall part 2} some key data from \cite[Table~8]{FKRS12} about the Sato--Tate groups of abelian surfaces, plus some additional information from \LMFDB{} (maximal groups and index-2 subgroups). We omit from Table~\ref{table: genus 2 recall part 1} the 3 groups listed in 
\cite[Table~8]{FKRS12} which appear in the group-theoretic classification of \cite{FKRS12} but cannot occur for abelian surfaces.

\begin{table}
\begin{center}
\footnotesize
\setlength{\extrarowheight}{0.5pt}

\begin{tabular}{|c|c|c|c|l|}
\hline
$G^0$ & $[G:G^0]$  & $G$          & $G/G^0$ & Index-2 subgroups \\
\hline
$\SU(2)_2$ & $1$  & $\stgroup{E_1}$        & $\cyc{1}$              & ---     \\
& $2$  & $\stgroup{E_2}$        & $\cyc{2}$              & $\stgroup{E_1}$           \\
& $3$  & $\stgroup{E_3}$        & $\cyc{3}$              & ---           \\
& $4$  & $\stgroup{E_4}$        & $\cyc{4}$              & $\stgroup{E_2}$           \\
& $6$  & $\stgroup{E_6}$        & $\cyc{6}$              & $\stgroup{E_3}$           \\
& $2$  & $\stgroup{J(E_1)}$     & $\cyc{2}$              & $\stgroup{E_1}$ \\
& $4$  & $\stgroup{J(E_2)}$     & $\dih{2}$              & $\stgroup{E_2}, \stgroup{J(E_1)}^{\times 2}$           \\
& $6$  & $\stgroup{J(E_3)}$     & $\dih{3}$              & $\stgroup{E_3}$           \\
& $8$  & $\stgroup{J(E_4)}\star$     & $\dih{4}$              & $\stgroup{E_4}, \stgroup{J(E_2)}^{\times 2}$           \\
& $12$ & $\stgroup{J(E_6)}\star$     & $\dih{6}$              & $\stgroup{E_6}, \stgroup{J(E_3)}^{\times 2}$           \\
\hline
$\Unitary(1) \times \Unitary(1)$ & $1$  & $\stgroup{F}$ & $\cyc{1}$              & --- \\
& $2$  & $\stgroup{F_a}$        & $\cyc{2}$              & $\stgroup{F}$ \\
& $2$  & $\stgroup{F_{ab}}$     & $\cyc{2}$              & $\stgroup{F}$ \\
& $4$  & $\stgroup{F_{ac}}\star$     & $\cyc{4}$              & $\stgroup{F_{ab}}$           \\
& $4$  & $\stgroup{F_{a,b}}\star$    & $\dih{2}$              & $\stgroup{F_a}^{\times 2}, \stgroup{F_{ab}}$ \\
\hline
$\Unitary(1) \times \SU(2)$ & $1$  & $\stgroup{G_{1,3}}$    & $\cyc{1}$              & --- \\
& $2$  & $\stgroup{N(G_{1,3})}\star$ & $\cyc{2}$              & $\stgroup{G_{1,3}}$ \\
\hline
$\SU(2) \times \SU(2)$ & $1$  & $\stgroup{G_{3,3}}$    & $\cyc{1}$              & --- \\
& $2$  & $\stgroup{N(G_{3,3})}\star$ & $\cyc{2}$              & $\stgroup{G_{3,3}}$           \\
\hline
$\USp(4)$ & $1$  & $\stgroup{\USp(4)}\star$    & $\cyc{1}$              & ---        \\\hline
\end{tabular}
\end{center}
\medskip

\caption{Sato--Tate groups $G$ of abelian surfaces with $G^0 \neq \Unitary(1)_2$. The symbol $\star$ indicates a maximal conjugacy class of subgroups of $\USp(4)$.
For index-2 subgroups, the notation $H^{\times n}$ indicates that $n$ subgroups of~$G$ (up to conjugacy in $G$) are conjugate to $H$ in $\USp(4)$.}
\label{table: genus 2 recall part 1}
\end{table}

\begin{table}
\begin{center}
\footnotesize
\setlength{\extrarowheight}{0.5pt}
\begin{tabular}{|c|c|c|l|}
\hline
$[G:G^0]$  & $G$          & $G/G^0$ & Index-2 subgroups \\ \hline
$1$  & $\stgroup{C_1}$        & $\cyc{1}$              & --- \\
$2$  & $\stgroup{C_2}$        & $\cyc{2}$              & $\stgroup{C_1}$ \\
$3$  & $\stgroup{C_3}$        & $\cyc{3}$              & --- \\
$4$  & $\stgroup{C_4}$        & $\cyc{4}$              & $\stgroup{C_2}$ \\
$6$  & $\stgroup{C_6}$        & $\cyc{6}$              & $\stgroup{C_3}$ \\
$4$  & $\stgroup{D_2}$        & $\dih{2}$              & $\stgroup{C_2}^{\times 3}$           \\
$6$  & $\stgroup{D_3}$        & $\dih{3}$              & $\stgroup{C_3}$           \\
$8$  & $\stgroup{D_4}$        & $\dih{4}$              & $\stgroup{C_4}, \stgroup{D_2}^{\times 2}$           \\
$12$ & $\stgroup{D_6}$        & $\dih{6}$              & $\stgroup{C_6}, \stgroup{D_3}^{\times 2}$           \\
$12$ & $\stgroup{T}$          & $\alt{4}$              & ---           \\
$24$ & $\stgroup{O}$          & $\sym{4}$              & $\stgroup{T}$           \\
$2$  & $\stgroup{J(C_1)}$     & $\cyc{2}$              & $\stgroup{C_1}$          \\
$4$  & $\stgroup{J(C_2)}$     & $\dih{2}$              & $\stgroup{C_2}, \stgroup{J(C_1)}, \stgroup{C_{2,1}}$           \\
$6$  & $\stgroup{J(C_3)}$     & $\cyc{6}$              & $\stgroup{C_3}$           \\
$8$  & $\stgroup{J(C_4)}$     & $\cyc{4}\times\cyc{2}$ & $\stgroup{C_4}, \stgroup{J(C_2)}, \stgroup{C_{4,1}}$           \\
$12$ & $\stgroup{J(C_6)}$     & $\cyc{6}\times\cyc{2}$ & $\stgroup{C_6}, \stgroup{J(C_3)}, \stgroup{C_{6,1}}$           \\
$8$  & $\stgroup{J(D_2)}$     & $\dih{2}\times\cyc{2}$ & $\stgroup{D_2}, \stgroup{J(C_2)}^{\times 3}, \stgroup{D_{2,1}}^{\times 3}$           \\
$12$ & $\stgroup{J(D_3)}$     & $\dih{6}$              & $\stgroup{D_3}, \stgroup{J(C_3)}, \stgroup{D_{3,2}}$           \\
$16$ & $\stgroup{J(D_4)}$     & $\dih{4}\times\cyc{2}$ & $\stgroup{D_4}, \stgroup{J(C_4)}, \stgroup{J(D_2)}^{\times 2}, \stgroup{D_{4,1}}^{\times 2}, \stgroup{D_{4,2}}$           \\
$24$ & $\stgroup{J(D_6)}\star$     & $\dih{6}\times\cyc{2}$ & $\stgroup{D_6}, \stgroup{J(C_6)}, \stgroup{J(D_3)}^{\times 2}, \stgroup{D_{6,1}}^{\times 2}, \stgroup{D_{6,2}}$           \\
$24$ & $\stgroup{J(T)}$       & $\alt{4}\times\cyc{2}$ & $\stgroup{T}$           \\
$48$ & $\stgroup{J(O)}\star$       & $\sym{4}\times\cyc{2}$ & $\stgroup{O}, \stgroup{J(T)}, \stgroup{O_1}$           \\
$2$  & $\stgroup{C_{2,1}}$    & $\cyc{2}$              & $\stgroup{C_1}$     \\
$4$  & $\stgroup{C_{4,1}}$    & $\cyc{4}$              & $\stgroup{C_2}$           \\
$6$  & $\stgroup{C_{6,1}}$    & $\cyc{6}$              & $\stgroup{C_3}$           \\
$4$  & $\stgroup{D_{2,1}}$    & $\dih{2}$              & $\stgroup{C_2}, \stgroup{C_{2,1}}^{\times 2}$ \\
$8$  & $\stgroup{D_{4,1}}$    & $\dih{4}$              & $\stgroup{D_2}, \stgroup{C_{4,1}}, \stgroup{D_{2,1}}$           \\
$12$ & $\stgroup{D_{6,1}}$    & $\dih{6}$              & $\stgroup{D_3}, \stgroup{C_{6,1}}, \stgroup{D_{3,2}}$           \\
$6$  & $\stgroup{D_{3,2}}$    & $\dih{3}$              & $\stgroup{C_3}$ \\
$8$  & $\stgroup{D_{4,2}}$    & $\dih{4}$              & $\stgroup{C_4}, \stgroup{D_{2,1}}^{\times 2}$ \\
$12$ & $\stgroup{D_{6,2}}$    & $\dih{6}$              & $\stgroup{C_6}, \stgroup{D_{3,2}}^{\times 2}$ \\
$24$ & $\stgroup{O_1}$        & $\sym{4}$              & $\stgroup{T}$           \\
\hline
\end{tabular}
\end{center}
\medskip

\caption{Sato--Tate groups $G$ of abelian surfaces with $G^0 = \Unitary(1)_2$. Notation is as in Table~\ref{table: genus 2 recall part 1}.}
\label{table: genus 2 recall part 2}
\end{table}

We start with the cases where we must have $G_i = H_i$ for one $i$, in which case the options come directly from the other factor.

\begin{itemize}
\item[Type $\bC$:]
In the case $G^0 = \SU(2) \times \USp(4)$, we must have $G_i = H_i$ for $i=1,2$, so $G = G^0$.

\item[Type $\bD$:]
In the case $G^0 = \Unitary(1) \times \USp(4)$, we must have $G_2 = H_2$, so $G = G_1 \times \USp(4)$
for $G_1 \in \{ \Unitary(1), N(\Unitary(1))\}$.

\item[Type $\bI$:]
In the case $G^0 = \SU(2) \times \SU(2)_2$, we must have $G_1 = H_1$, so $G = \SU(2) \times G_2$ for some
$G_2$ in Table~\ref{table: genus 2 recall part 1} with $G^0 = \SU(2)_2$.

\item[Type $\bK$:]
In the case $G^0 = \SU(2) \times \Unitary(1)_2$, we must have $G_1 = H_1$, so $G = \SU(2) \times G_2$ for some $G_2$ in 
Table~\ref{table: genus 2 recall part 2}.

\item[Type $\bG$:]
In the case $G^0 = \Unitary(1) \times \Unitary(1) \times \SU(2)$, take $H_1=\Unitary(1)\times\Unitary(1)$ and $H_2=\SU(2)$; we must have $G_2 = H_2$,
so $G = G_1 \times \SU(2)$ for some $G_1$ in Table~\ref{table: genus 2 recall part 1} with $G^0 = \Unitary(1) \times \Unitary(1)$.

\end{itemize}

We next consider cases where $H_1 = \Unitary(1)$. In such cases, $G_1 \in \{\Unitary(1), N(\Unitary(1))\}$ (taking the normalizer in $\SU(2)$),
so for any given $G$ there exists a group $G_2$ with connected part $H_2$ such that either
$G = \Unitary(1) \times G_2$, $G = N(\Unitary(1)) \times G_2$, or $G$ is the fiber product
$N(\Unitary(1)) \times_{\cyc 2} G_2$ with respect to the unique nontrivial homomorphism $N(\Unitary(1)) \to \cyc 2$
and some nontrivial homomorphism $G_2 \to \cyc 2$.
The latter may be characterized by the group $G_2$ and the index-2 subgroup $G_2'$ which is the kernel of the homomorphism.
It is clear that choices of $G_2'$ which are not conjugate in $\USp(4)$ lead to fiber products which are not conjugate in $\USp(6)$; however, in cases where there is a single conjugacy class $H$ that represents $n \geq 2$ different index-2 subgroups of $G_2$
(notated $H^{\times n}$ in Table~\ref{table: genus 2 recall part 1} and Table~\ref{table: genus 2 recall part 2}), some extra analysis will be needed to ensure that we only obtain one conjugacy class of fiber products. It will suffice to verify that in such cases, there is a group of outer automorphisms induced by the normalizer in $\USp(4)$ acting transitively on each collection of subgroups sharing a common label.

\begin{itemize}

\item[Type $\bF$:]
In the case $G^0 = \Unitary(1) \times \SU(2) \times \SU(2)$, we may take $H_1=\Unitary(1)$ and $H_2=\SU(2)\times\SU(2)$; then
$G_2 \in \{\SU(2) \times \SU(2), N(\SU(2) \times \SU(2))\}$
(taking the normalizer in $\mathrm{USp}(4)$).

\item[Type $\bJ$:]
In the case $G^0 = \Unitary(1) \times \SU(2)_2$, 
$G_2$ must be one of the groups in 
Table~\ref{table: genus 2 recall part 1} with $G^0 = \SU(2)_2$.
Potentially ambiguous fiber products occur when $G_2 = J(E_{2n})$ for $n=1,2,3$ and $G_2'$ is one of the two copies of $J(E_n)$ in $G_2$; the latter are conjugate to each other via the normalizer of $J(E_{2n})$.

\item[Type $\bL$:]
In the case $G^0 = \Unitary(1) \times \Unitary(1)_2$,
$G_2$ must be one of the groups in Table~\ref{table: genus 2 recall part 2}.
Potentially ambiguous fiber products occur for
\[
G_2 = \stgroup{D_2}, \stgroup{D_4}, \stgroup{D_6}, \stgroup{J(D_2)}, \stgroup{J(D_4)}, \stgroup{J(D_6)}, \stgroup{D_{2,1}}, \stgroup{D_{4,2}}, \stgroup{D_{6,2}};
\]
in these cases, we must check that the normalizer of $G_2$ in $\USp(4)$ acts transitively on the subgroups with a given label.
\begin{itemize}
\item
For $\stgroup{D_2}$ and $\stgroup{J(D_2)}$, this normalizer is $\sym 3$. The image of $\stgroup{D_2}$ in $\SO(3)$ is the group generated by the half-turns about three orthogonal axes; $\sym 3$ acts permuting these.
\item
For $\stgroup{D_4},\stgroup{D_6}, \stgroup{J(D_4)}, \stgroup{J(D_6)}, \stgroup{D_{2,1}}, \stgroup{D_{4,2}}, \stgroup{D_{6,2}}$, this normalizer is $\cyc 2$ interchanging the two conjugacy classes of reflections.
(Note that in contrast with $\stgroup{D_2}$, in $\stgroup{D_{2,1}}$ the subgroup $\stgroup{C_2}$ is distinguished from the others.)
\end{itemize}

\end{itemize}

\begin{remark}
We checked in \Magma{} that for each choice of $G_2$, running over all possible index-$2$ subgroups $G_2'$ of $G_2$ gives no new distributions (as measured by moments; see \S\ref{sec:statistics}) other than the ones listed. This corroborates (but does not independently confirm) the claim that subgroups with the same label correspond to conjugate fiber products.
\end{remark}

In this analysis, the group $G$ is maximal if and only if $G_1$ and $G_2$ are both maximal. To tabulate maximal subgroups, it thus suffices to recall the maximal subgroups in genus 2;
these are indicated with stars in Table~\ref{table: genus 2 recall part 1} and Table~\ref{table: genus 2 recall part 2}.

\subsection{Triple products}
\label{subsec: triple products classification}

We next classify the options for $G$ when $G^0 \simeq \SU(2) \times \SU(2) \times \SU(2)$ (type $\bE$).
In this case, $N/G^0 \simeq \sym 3$; more precisely, $N$ is generated by $G^0$ plus the
permutation matrices. The options for $G$ in this case thus correspond to the subgroups of $\sym 3$ up to conjugacy,
which we may identify with $\stgroup[\sym 1]{1.6.E.1.1a}, \stgroup[\sym 2]{1.6.E.2.1a}, \stgroup[\alt 3]{1.6.E.3.1a}, \stgroup[\sym 3]{1.6.E.6.1a}$. The unique maximal case is $\stgroup[\sym 3]{1.6.E.6.1a}$.

We now classify the options for $G$ in the case $G^0 \simeq \Unitary(1) \times \Unitary(1) \times \Unitary(1)$
(type $\bH$).
In this case, $Z = G^0$ and $N/G^0$ is isomorphic to the wreath product
$\cyc 2 \wr \sym 3$. 
Let $a,b,c$ be representatives of the nontrivial cosets in the three factors of $N(\Unitary(1))$,
and set $t,s$ correspond to the permutations $(1\,2), (1\,2\,3)$ in $\sym 3$.

If $A$ is not simple, then it splits as a product of a CM elliptic curve and a CM abelian surface;
this means that up to conjugation, $\ST(A)$ is contained in the product of a subgroup of 
$\langle a,b,t \rangle$ and a subgroup of $\langle c \rangle$. 
By the analysis from \cite[\S 4.3]{FKRS12}, the former must be a subgroup of either $\langle a,b \rangle$
or $\langle at \rangle$.

To handle the case where $A$ is simple, 
we recall some facts about abelian varieties with complex multiplication,
following \cite[\S 1]{Dod84} and \cite[\S 1]{Kil16} (and ultimately, \cite[\S 5.5]{Shi71}).
For future reference, we temporarily allow $A$ to have arbitrary dimension.

\begin{definition}
A \emph{CM field} is a totally imaginary quadratic extension of a totally real number field.
A \emph{CM type} is a pair $(M, \Phi)$ consisting of a CM field $M$ and a section $\Phi$ of the map grouping the complex embeddings of $M$ into conjugate pairs.

Given a CM type $(M, \Phi)$, for any $\Z$-lattice $\fm$ in $M$, we may use $\Phi$ to specify a complex structure on $M \otimes_\Q \R$, and thus view $(M \otimes_\Q \R)/\fm$ as a complex torus. Changing the choice of~$\fm$ gives rise to an isogenous torus \cite[Proposition~5.13]{Shi71}; consequently, we may associate a polarization to this torus using the trace pairing on $M$, and thus obtain an abelian variety over $\C$ with complex multiplication by some order in 
$M$.
\end{definition}

\begin{definition}
Given a CM type $(M, \Phi)$, let $L$ be the Galois closure of $M$ in $\C$, which is again a CM field.
The action of $\Gal(L/\Q)$ on the complex embeddings of $M$ by postcomposition acts on the CM types for $M$.
The \emph{reflex field} of $(M, \Phi)$ is the fixed field $M^*$ of the stabilizer of $\Phi$ in $\Gal(L/\Q)$; note that $M^*$ is defined as a subfield of $L$ (and hence of $\C$), 
whereas $M$ does not come with a distinguished embedding into $L$.
(There is also a CM type associated to $M^*$, the \emph{reflex type}, which we will not discuss here.)
\end{definition}

\begin{lemma}
Fix an embedding of $k$ into $\C$. 
Let $A$ be a simple polarized abelian variety of dimension $n$ over $k$ with CM type $(M, \Phi)$.
\begin{enumerate}[{\rm(a)}]
\item
The field $M^*$ is Galois over $k \cap M^*$.
\item
The field $kM^*$ equals the endomorphism field of $A$.
\item
The group $\Gal(kM^*/k)$ is canonically isomorphic to a subgroup of $\Aut(M/\Q)$.
\end{enumerate}
\end{lemma}
\begin{proof}
For (a), (b), see \cite[Theorem~5.15, Proposition~5.17]{Shi71}. 
For (c), we obtain the identification by interpreting (b) as giving a \emph{faithful} action of $\Gal(kM^*/k)$ on $\End(A_\C)_{\Q} = M$. (For $L$ the endomorphism field, it is tautological that $\Gal(L/k)$ acts faithfully on $\End(A_L)_{\Q} = \End(A_\C)_{\Q}$.)
\end{proof}

\begin{remark}
To trace the effect on Sato--Tate groups of the previous discussion, we identify $\ST(A)/\ST(A)^0$
with a subgroup of $\cyc 2 \wr {\sym n}$ (for $n = \dim(A)$) using the embedding of $M$ into $\End(A_K)_{\R}$
to fix the identification of $\Unitary(1) \times \cdots \times \Unitary(1)$ with $\prod_{\tau \in \Phi} \C^\times$.
Then the subgroup $\ST(A)/\ST(A)^0$ is contained in $\Gal(L/(k \cap M^*))$
and the composition
\[
\ST(A)/\ST(A)^0 \to \Gal(L/(k \cap M^*)) \to \Gal(M^*/(k \cap M^*)) \cong \Gal(kM^*/k)
\]
is the canonical map (see Remark~\ref{remark: general ST group}).
\end{remark}

We now specialize back to the case where $A$ is 
a simple CM abelian threefold. In this case, by \cite[\S 5.1.1]{Dod84},
the sextic field $M$ must be one of the following.
\begin{itemize}
\item
A cyclic Galois extension of $\Q$. In this case,
we must have $G/G^0 \subseteq \langle abc,s \rangle$ as this is the unique copy of $\cyc 6$ in $\cyc 2 \wr \sym 3$.

\item
A compositum of an imaginary quadratic field $M_0$ and a non-Galois totally real cubic field $M_1$.
In this case, $\Aut(M/\Q) = \Gal(M_0/\Q)$ and so $G/G^0 \subseteq \langle abc \rangle$.

\item
A non-Galois extension of $\Q$ whose Galois closure has group $\cyc 2 \wr \alt 3$ or $\cyc 2 \wr \sym 3$.
In this case, $\Aut(M/\Q)$ is trivial, as then is $G/G^0$.

\end{itemize}
 
To summarize, we list the 33 conjugacy classes of subgroups of $N/G^0$ (as may be verified using \Gap{} or \Magma) in Table~\ref{table:subgroups SU2xSU2xSU2}. In the table, an underline denotes a group that is permitted by the previous analysis to occur as a Sato--Tate group of an abelian threefold. Among the latter, the maximal cases are $\stgroup[\langle abc,s \rangle]{H_{abc,s}}$, $\stgroup[\langle c,at \rangle]{H_{c,at}}$,
$\stgroup[\langle a,b,c \rangle]{H_{a,b,c}}$.

\begin{table}[ht]
\renewcommand{\arraystretch}{1.1}
\small
\begin{tabular}{|c|c|c|c|c|}
\hline
Order & Isom type & Groups & Total & Realizable \\
\hline &&&& \\[-13pt]
1 & $\cyc 1$ & $\underline{\stgroup[\langle e \rangle]{H}}$ & 1 & 1\\[2pt]
2 & $\cyc 2$ & $\underline{\stgroup[\langle a \rangle]{H_a}}$, $\underline{\stgroup[\langle ab \rangle]{H_{ab}}}$, $\underline{\stgroup[\langle abc \rangle]{H_{abc}}}$, $\langle t \rangle$, $\langle ct \rangle$& 5 & 3\\[2pt]
3 & $\cyc 3$ & $\underline{\stgroup[\langle s \rangle]{H_s}}$ & 1 & 1\\[2pt]
4 & $\cyc 4$ & $\underline{\stgroup[\langle at \rangle]{H_{at}}}$, $\underline{\stgroup[\langle act \rangle]{H_{act}}}$ &2 & 2\\[2pt]
4 & $\cycsup{2}{2}$ &  $\underline{\stgroup[\langle a,b \rangle]{H_{a,b}}}$, $\underline{\stgroup[\langle a,bc \rangle]{H_{a,bc}}}$, $\underline{\stgroup[\langle ab, bc \rangle]{H_{ab,bc}}}$, $\langle c,t\rangle$,
$\langle ab,t \rangle$, $\langle ab, ct \rangle$, $\langle abc, t \rangle$ &7 & 3\\[2pt]
6 & $\cyc 6$  & $\underline{\stgroup[\langle abc,s \rangle]{H_{abc,s}}}$&1 & 1\\[2pt]
6 & $\sym 3 $  & $\langle s,t \rangle$, $\langle abct, s \rangle$ &2 & 0\\
8 & $\cyc{2} \times \cyc{4}$  & $\underline{\stgroup[\langle c,at \rangle]{H_{c,at}}}$ &1 & 1\\[2pt]
8 & $\cycsup{2}{3}$  & $\underline{\stgroup[\langle a,b,c \rangle]{H_{a,b,c}}}$,$\langle ab,c,t \rangle$&2 & 1\\[2pt]
8 & $\dih 4$  & $\langle a,b,t \rangle$, $\langle ab,bc,t \rangle$,
$\langle a,b,ct \rangle$, $\langle ab,bc,ct \rangle$ &4 & 0\\
12 & $\dih 6$ & $\langle abc,s,t \rangle$ &1 & 0\\
12 & $\alt 4$  & $\langle ab,bc,s \rangle$ &1& 0\\
16 & $\cyc{2} \times \dih 4$ & $\langle a,b,c,t \rangle$&1& 0\\
24 & $\cyc 2 \wr \alt 3$ &$\langle a,b,c,s \rangle$&1&0 \\
24 & $\sym 4$ & $\langle ab,bc,s,t \rangle$, 
$\langle ab,bc,at,s \rangle$&2&0\\
48 & $\cyc 2 \wr \sym 3$ & $\langle a,b,c,s,t \rangle$ &1&0\\[1pt]
\hline
Total& & & $33$ & 13 \\
\hline
\end{tabular}
\medskip

\caption{Conjugacy classes of subgroups of $N/G^0$ for $G^0 \simeq \Unitary(1) \times \Unitary(1) \times \Unitary(1)$. The underlined groups can be realized by Sato--Tate groups.}
\label{table:subgroups SU2xSU2xSU2}
\end{table}

\subsection{The case \texorpdfstring{$G^0 = \SU(2)_3$}{G0=SU(2)}}
\label{subsec:su2}

We next classify the options for $G$ assuming that $G^0 \simeq \SU(2)_3$ (type $\bM$).
In this case, $Z$ equals the group $\mathrm{O}(3)$ realized as a block matrix with scalar entries,
$Z \cap G^0=\{\pm 1\}$, 
and $N = ZG^0$, so $N/G^0 \cong \SO(3)$.

At this point, we must make the rationality condition explicit for the first time in this discussion (though it intervened implicitly in the discussion of products).
Any cyclic subgroup of $\SO(3)$  of order $n$ is conjugate to the group generated by
\[
A = \begin{pmatrix} 1 & 0 & 0 \\
0 & \cos 2\pi/n & \sin 2\pi/n \\
0 & -\sin 2\pi/n & \cos 2\pi/n
\end{pmatrix}.
\]
For the representation $\wedge^2 \C^6$, the average of the trace of the coset of $G^0$ containing $A$ equals the trace of $A$ on the $G^0$-fixed subspace of $\wedge^2 \C^6$. The latter equals $(1 + 2\cos 2 \pi/n)^2$;
this is an integer if and only if $n \in \{1,2,3,4,6\}$.
Since every subgroup of $\SO(3)$ is either cyclic, dihedral, or one of the three exceptional groups
(tetrahedral, octahedral, icosahedral), we obtain component
groups isomorphic to 
\[
\stgroup[\cyc 1]{M(C_1)}, \stgroup[\cyc 2]{M(C_2)}, \stgroup[\cyc 3]{M(C_3)}, \stgroup[\cyc 4]{M(C_4)}, \stgroup[\cyc 6]{M(C_6)}, \stgroup[\dih 2]{M(D_2)}, \stgroup[\dih 3]{M(D_3)}, \stgroup[\dih 4]{M(D_4)}, \stgroup[\dih 6]{M(D_6)}, \stgroup[\alt 4]{M(A_4)}, \stgroup[\sym 4]{M(S_4)}.
\]
By realizing these groups over $\Z$, one may see that they satisfy the rationality condition for all representations of $\GL(\C^6)$. The maximal cases are $\stgroup[\dih 6]{M(D_6)}, \stgroup[\sym 4]{M(S_4)}$.

\section{The classification upper bound, part 2: the case \texorpdfstring{$G^0 = \Unitary(1)_3$}{G0=U(1)}}
\label{section: unitary}

Retaining notation as in \S\ref{section:STgroups}, we classify (up to conjugation) the subgroups $G$ of $\USp(6)$ that satisfy the Sato--Tate axioms in the case where $G^0 \simeq \Unitary(1)_3$.
The reader familiar with \cite[\S 3]{FKRS12} may have anticipated that this is the richest and most difficult case;
indeed, the analysis is very loosely parallel with that of \cite[\S 3]{FKRS12}, with the role of the classification of finite subgroups of $\SU(2)$ played instead by the classification of finite subgroups of $\SU(3)$ containing~$\mu_3$ by Blichfeldt, Dickson, and Miller \cite[Chapter~XII]{MBD61}. The main technical difficulty is in filtering candidate groups for the rationality condition. 

It will be helpful to record explicit presentations for the groups in question, both to aid with some aspects of the classification and for the computation of moments. These presentations are implemented in
accompanying \Sage{} and \Magma{} scripts; note that the \Sage{} computations depend heavily on \Gap, which ships
as a standard component of \Sage.

\subsection{Overview}
\label{subsec:preliminaries}

We start with a high-level description of the classification for $G^0=U(1)_3$, whose results are summarized in
Table~\ref{table:subgroups by case}.

For $G^0 \simeq \Unitary(1)_3$, the centralizer and normalizer of $G^0$ in $\USp(6)$
are respectively 
$Z \simeq \Unitary(3)$ and $N = Z \cup JZ$;
conjugation by $J$ acts on $Z$ via complex conjugation.
We thus have an exact sequence
\begin{equation} \label{eq:exact sequence for normalizer U1}
1 \to \PSU(3) \to N/G^0 \to \cyc{2} \to 1
\end{equation}
which splits because $J^2 = -1 \in G^0$;
we may thus identify $N/G^0$ with the semidirect product $\PSU(3) \rtimes \cyc{2}$ via the outer\footnote{A comparison to \cite[(3.2)]{FKRS12} may be misleading: for $n \geq 3$, complex conjugation acts on $\PSU(n)$ as an outer automorphism corresponding to the nontrivial automorphism of the Dynkin diagram; however, for $n=2$, the Dynkin diagram has only one node, and so this automorphism becomes inner.} action of $\cyc{2}$ on $\PSU(3)$ via complex conjugation.

Let $J(\SU(3))$ be the subgroup of $\USp(6)$ generated by $\SU(3)$ and $J$.
There is now a bijection between conjugacy classes of subgroups $G$ of $\USp(6)$ with $G^0 \simeq \Unitary(1)_3$ satisfying the Sato--Tate axioms and finite subgroups $\tilde{H}$ of $J(\SU(3))$ 
containing~$\mu_6$ and satisfying a corresponding rationality condition.

Our methodology will be to proceed through the Blichfeldt--Dickson--Miller (BDM) classification of finite subgroups of $\SU(3)$. At each step, we will identify subgroups satisfying a restricted form of the rationality condition  (see \S\ref{section: rationality condition}); this will yield a finite list of conjugacy classes of subgroups of $\SU(3)$. We will then compute the normalizer of a representative of each of these classes, and use this to count extensions from $\SU(3)$ to $J(\SU(3))$ 
(see \S\ref{section: standard and nonstandard extensions}). 

In principle, this yields only an upper bound on the set of conjugacy classes of groups $G$ satisfying the Sato--Tate axioms with $G^0 \cong \Unitary(1)_3$, rather than the exact value; this is because we do not check that our candidate groups satisfy the rationality condition (ST3) at its full strength. In practice, we will see that all of these groups
can be realized as Sato--Tate groups (\S\ref{section: realizability}), which will imply \emph{a posteriori} that they satisfy the full rationality condition. For the same reason, we do not need to check that the groups we write down satisfy the restricted rationality condition (although we did verify this using the \Sage{} and \Magma{} scripts); rather, we only need to check that the groups we have excluded do not satisfy (ST3).

\begin{table}[ht]
\small
\begin{tabular}{|l|l|c|c|c|c|c|}
\hline&&&&&&\\[-11pt]
Case & \S &  $H$ & $J$ & $J_s$ & $J_n$ & Total\\
\hline&&&&&&\\[-11pt]
Abelian groups &\ref{section: BDM type A} &  22 & 22 & 15 & 9 & 68\\
Extensions of $\cyc 2$ & \ref{section: BDM type B} & 18 & 18 & 12 & 0 & 48 \\
Exceptional groups from $\SU(2)$ & \ref{section:SU2 exceptional} & 6 & 5 & 4 & 0 & 15\\
Extensions of $\alt 3$ & \ref{subsection: BDM type C} &7 & 7 & 5 & 0 & 19 \\
Extensions of $\sym 3$ & \ref{section: BDM type D} &6 & 6 & 0 & 0 & 12\\
Solvable exceptional groups & \ref{section:solvable exceptional} & 3 & 3 & 0 & 1 & 7\\
Simple exceptional groups & \ref{section:simple exceptional} & 1 & 1 & 0 & 0 & 2\\[1pt]
\hline&&&&&&\\[-11pt]
Total & & 63 & 62 & 36 & 10 & 171 \\
\hline
\end{tabular}
\medskip

\caption{Numbers of finite subgroups of $\SU(3) \rtimes \cyc 2$ accounted for at the various stages of the classification. The columns $H,J,J_s,J_n$ count subgroups of $\SU(3)$, standard extensions, split nonstandard extensions, and nonsplit nonstandard extensions; see \S\ref{section: standard and nonstandard extensions}
for explanation of how extensions are classified.} \label{table:subgroups by case}
\end{table}

To conclude this overview, we record some running notations and general remarks.

\begin{definition}
To lighten notation, we use $e(x)$ to denote $e^{2 \pi i x}$
and $D(u,v,w)$ to denote the diagonal matrix $\Diag(e(u), e(v), e(w))$.
\end{definition}

\begin{remark}
For a group $G$ occurring as the Sato--Tate group of an abelian threefold, the main theorem of  \cite{GK17}
implies that $[G:G^0]$ divides $2^6  \times 3^3 \times 7$.
We will not use this bound explicitly (as it does not directly apply to groups satisfying the Sato--Tate axioms but not yet known to be realizable), but we do wish to point out its consistency with our results.
\end{remark}

\subsection{The rationality condition}
\label{section: rationality condition}

We next specialize the Sato--Tate axiom (ST3) to the representation $\wedge^2 \C^6$ of $\GL(\C^6)$
to obtain a restriction on subgroups of $\SU(3) \rtimes \cyc 2$ containing~$\mu_3$. As noted above, this is \emph{a priori} weaker than the full statement of (ST3), which is quantified over arbitrary representations of $\GL(\C^6)$; however,
it will suffice for the identification of Sato--Tate groups. 

The formulation of (ST3) refers to individual connected components of $G$, so it translates into a condition on individual elements of $\SU(3) \rtimes \cyc 2$, which depends only on the conjugacy classes of these elements.
For elements in $\SU(3)$, we may represent any conjugacy class by a $3 \times 3$ diagonal matrix $A$; 
we then want to ensure that the average over $r \in [0,1]$
of the trace of the block matrix
\[
\begin{pmatrix}
e(2r) {\wedge^2 A} & 0 & 0 \\ 0 & A \otimes \overline{A} & 0 \\
0 & 0 & e(-2r) {\wedge^2 \overline{A}}
\end{pmatrix}
\]
is an integer. This average equals $\Trace(A \otimes \overline{A}) = \left|\Trace(A)\right|^2$
(the first and third blocks each average to 0), yielding the following definition.

\begin{definition} We say that a subgroup $H$ of $\SU(3)$ satisfies the \emph{restricted rationality condition} if $\left|\Trace(A)\right|^2 \in \Z$ for all $A \in H$.
We impose the restricted rationality condition on a subgroup of $\SU(3) \rtimes \cyc 2$ by imposing it on the intersection with $\SU(3)$.
\end{definition}

To make the rationality condition explicit, we identify the conjugacy classes of elements of $\SU(3)$ which can appear in a subgroup satisfying the condition; that is, we find the diagonal matrices $A = \Diag(z_1, z_2, z_3)$ such that $z_1,z_2,z_3$ are roots of unity, $z_1z_2 z_3 = 1$, and $|z_1 + z_2 + z_3|^2 \in \Z$. To do this, we first extract from the literature the solution of a closely related problem.

\begin{lemma} \label{L:conway-jones}
Suppose that $x,y,z$ are roots of unity in $\C$ such that
\[
\left( x + x^{-1} \right) + \left( y + y^{-1} \right) + \bigl( z + z^{-1} \bigr) \in \Z.
\]
Then at least one of the following holds. (In the following discussion, the equivalence relation
is generated by permutations, inversion of any one term, or negation of all terms at once.)
\begin{enumerate}[{\rm(i)}]
\item
All three summands are in $\Z$. In this case, up to equivalence, we have
\[
x,y,z \in \bigl\{\pm 1, \pm i, \pm e(\tfrac{1}{3}) \bigr\}.
\]
\item
Two summands sum to zero but are not in $\Z$. In this case, up to equivalence, we have
\[
x \in  \bigl\{\pm 1, \pm i, \pm e(\tfrac{1}{3})\bigr\}, \qquad z = -y.
\]
\item
Exactly one summand is in $\Z$ and the other two do not sum to zero. In this case, up to equivalence, we have
\[
x \in  \bigl\{\pm 1, \pm i, \pm e(\tfrac{1}{3})\bigr\}, \qquad (y,z) = (e(\tfrac{1}{5}), e(\tfrac{2}{5})).
\]
\item
No summand is in $\Z$ and the sum of all three is zero. In this case, up to equivalence, we have
\[
(y,z) = (e(\tfrac{1}{3}) x, e(\tfrac{2}{3}) x).
\]
\item
No summand is in $\Z$ and the sum of all three is nonzero. In this case, up to equivalence, we have
\[
(x,y,z) \in \bigl\{(e(\tfrac{1}{7}),e(\tfrac{2}{7}), e(\tfrac{4}{7})), (e(\tfrac{1}{15}), e(\tfrac{4}{15}), e(\tfrac{3}{10})), (e(\tfrac{1}{10}), e(\tfrac{2}{15}),  e(\tfrac{7}{15}))\bigr\}.
\]
\end{enumerate}
\end{lemma}
\begin{proof}
This can be read off from \cite[Theorem~7]{CJ76}. Alternatively, by appending copies of $\pm 1$ we obtain a sum of at most 12 roots of unity equal to 0, and such sums are classified by \cite[Theorem~3.1]{PR98} (see also \cite[Theorem~4.1.2]{CMS11} or \cite[Theorem~5.1]{KKPR20}).
\end{proof}

\begin{definition}
Define an equivalence relation on triples of rational numbers by setting
$(u_1, u_2, u_3) \sim (v_1, v_2, v_3)$ if the multisets 
$\{e(u_j)\}_{j=1}^3$
and $\{e(v_j)\}_{j=1}^3$ are Galois conjugate over~$\Q$.
In other words, there must exist
a permutation $\sigma$ of $\{1,2,3\}$ and a positive integer~$n$ coprime to the least common denominator of $u_1, u_2, u_3$
such that
\[
v_j - n u_{\sigma(j)} \in \Z \qquad (j=1,2,3).
\]
\end{definition}

\begin{proposition} \label{roots of unity problem}
Let $a,b,c$ be roots of unity in $\C$ such that $abc = 1$.
Then $\left|a+b+c\right|^2 \in \Z$ if and only if either $(a+b)(b+c)(c+a) = 0$ or $(a,b,c) = (e(u), e(v), e(w))$ for some triple $(u,v,w)$ equivalent to one of the following:
\begin{gather}
\nonumber
(0,0,0), \left(0, \frac{1}{3}, \frac{2}{3}\right), \left(\frac{1}{3}, \frac{1}{3}, \frac{1}{3} \right),
\left( \frac{1}{4}, \frac{1}{4}, \frac{1}{2} \right),
\left(0, \frac{1}{6}, \frac{5}{6}\right),
\left(\frac{1}{6}, \frac{1}{3}, \frac{1}{2} \right), \\
\nonumber
\left(\frac{1}{7}, \frac{2}{7}, \frac{4}{7}\right),
\left(\frac{1}{8}, \frac{3}{8}, \frac{1}{2}\right),
\left(\frac{1}{9}, \frac{1}{9}, \frac{7}{9}\right),
\left(\frac{1}{12}, \frac{1}{12}, \frac{5}{6}\right),
\left(\frac{1}{12}, \frac{1}{6}, \frac{3}{4}\right),
\left(\frac{1}{12}, \frac{5}{12}, \frac{1}{2}\right), \\
\left(\frac{1}{18}, \frac{1}{18}, \frac{8}{9}\right),
\left(\frac{1}{21}, \frac{4}{21}, \frac{16}{21}\right),
\left(\frac{1}{24}, \frac{1}{6}, \frac{19}{24}\right),
\left(\frac{1}{90}, \frac{19}{90}, \frac{7}{9}\right).
\label{roots of unity list1}
\end{gather}
\end{proposition}
\begin{proof}
For
\[
x = \frac{a}{b}, \qquad y = \frac{b}{c}, \qquad z = \frac{c}{a},
\]
we have
\[
\left|a+b+c\right|^2  = x + x^{-1} + y + y^{-1} + z + z^{-1} + 3;
\]
consequently, $(x,y,z)$ must be as described in Lemma~\ref{L:conway-jones}.
Conversely, given $x,y,z$ as above, we may solve for $a,b,c$ by writing 
\[
a = \left(\frac{y}{x}\right)^{\nicefrac{1}{3}}, \qquad b = \left(\frac{z}{y}\right)^{\nicefrac{1}{3}}, \qquad c = \left(\frac{x}{z}\right)^{\nicefrac{1}{3}}
\]
and choosing cube roots consistently.
Enumerating cases of Lemma~\ref{L:conway-jones} yields the following.
\begin{enumerate}[{\rm(i)}]
\item
Each of $x,y,z$ has order dividing 4 or 6, so each of $a,b,c$ has order dividing 12 or 18.
Excluding cases where $(a+b)(b+c)(c+a) = 0$, we obtain the possibilities
\begin{gather*}
(u,v,w) = (0,0,0), 
\left( 0, \frac{1}{3}, \frac{2}{3} \right), 
\left( \frac{1}{3}, \frac{1}{3}, \frac{1}{3} \right),
\left( \frac{1}{4}, \frac{1}{4}, \frac{1}{2} \right),
\left( 0, \frac{1}{6}, \frac{5}{6} \right), \\
\left( \frac{1}{6}, \frac{1}{3}, \frac{1}{2} \right),
\left( \frac{1}{9}, \frac{1}{9}, \frac{7}{9} \right),
\left( \frac{1}{12}, \frac{1}{12}, \frac{5}{6} \right),
\left( \frac{1}{18}, \frac{1}{18}, \frac{8}{9} \right).
\end{gather*}
\item
Up to equivalence, $x = \frac{a}{b}$ has order dividing 4 or 6 and $z = -y$; the latter forces 
$b \in \{e(\tfrac{1}{6}), e(\tfrac{1}{2}), e(\tfrac{5}{6})\}$.
This yields the additional cases
\[
(u,v,w) = 
\left( \frac{1}{8}, \frac{3}{8}, \frac{1}{2} \right),
\left( \frac{1}{12}, \frac{1}{6}, \frac{3}{4} \right),
\left( \frac{1}{12}, \frac{5}{12}, \frac{1}{2} \right),
\left( \frac{1}{24}, \frac{1}{6}, \frac{19}{24} \right).
\]
\item
In this situation, there is no way to have $xyz=1$.
\item
In this situation, we have $x^3 = y^3 = z^3 = 1$ and so no new cases occur.
\item
From the listed options for $(x,y,z)$, we obtain the additional cases
\[
(u,v,w) = \left( \frac{1}{7}, \frac{2}{7}, \frac{4}{7} \right),
\left( \frac{1}{21}, \frac{4}{21}, \frac{16}{21} \right),
\left( \frac{1}{90}, \frac{19}{90}, \frac{7}{9} \right).
\] 
\end{enumerate}
The proposition follows.
\end{proof}

\begin{remark} \label{remark: multiplicative manin-mumford}
Proposition~\ref{roots of unity problem} is a special case of the general problem of finding solutions in roots of unity to a system of polynomial equations (i.e., the \emph{multiplicative Manin-Mumford problem}).
The approach used in \cite{CJ76} yields an algorithm for the more general problem, as described by Leroux \cite{Ler12}. 

Another algorithm has been given by Beukers and Smyth \cite{BS02}
in the case of a single polynomial of two variables, and generalized by Aliev and Smyth \cite{AS08} to a system.
See Remark~\ref{beukers-smyth} for an explanation of how the Beukers--Smyth algorithm
can be used to confirm Proposition~\ref{roots of unity problem}.
\end{remark}

It might appear that the result of Proposition~\ref{roots of unity problem} violates \cite[Remark~3.3]{FKRS12}
by failing to limit $(a,b,c)$ to a finite set. This discrepancy disappears when we impose the rationality
condition not on individual group elements, but on entire cyclic subgroups.

\begin{proposition} \label{roots of unity problem2}
Let $a,b,c$ be roots of unity in $\C$ with $abc = 1$.
Then $\left|a^n+b^n+c^n\right|^2 \in \Z$ for all positive integers $n$ if and only if $(a,b,c) = (e(u), e(v), e(w))$ for some triple $(u,v,w)$ equivalent to one of the following:
\begin{gather}
\nonumber
(0,0,0), \left(0, \frac{1}{2}, \frac{1}{2}\right), 
 \left(0, \frac{1}{3}, \frac{2}{3}\right), \left(\frac{1}{3}, \frac{1}{3}, \frac{1}{3} \right),
  \left(0, \frac{1}{4}, \frac{3}{4}\right), 
\left( \frac{1}{4}, \frac{1}{4}, \frac{1}{2} \right), \\
\nonumber
\left(0, \frac{1}{6}, \frac{5}{6}\right),
\left(\frac{1}{6}, \frac{1}{6}, \frac{2}{3} \right),
\left(\frac{1}{6}, \frac{1}{3}, \frac{1}{2} \right), 
\left(\frac{1}{7}, \frac{2}{7}, \frac{4}{7}\right),
\left(\frac{1}{8}, \frac{1}{4}, \frac{5}{8} \right), 
\left(\frac{1}{8}, \frac{3}{8}, \frac{1}{2}\right),
\left(\frac{1}{9}, \frac{1}{9}, \frac{7}{9}\right),\\
\nonumber
\left(\frac{1}{12}, \frac{1}{12}, \frac{5}{6}\right),
\left(\frac{1}{12}, \frac{1}{6}, \frac{3}{4}\right),
\left(\frac{1}{12}, \frac{1}{3}, \frac{7}{12}\right),
\left(\frac{1}{12}, \frac{5}{12}, \frac{1}{2}\right), 
\left(\frac{1}{18}, \frac{1}{18}, \frac{8}{9}\right),
 \left(\frac{1}{18}, \frac{7}{18}, \frac{5}{9}\right),\\
\left(\frac{1}{21}, \frac{4}{21}, \frac{16}{21}\right),
\left(\frac{1}{24}, \frac{1}{6}, \frac{19}{24}\right),
\left(\frac{1}{24}, \frac{5}{12}, \frac{13}{24}\right), 
\left(\frac{1}{36}, \frac{4}{9}, \frac{19}{36}\right).
\label{roots of unity list2}
\end{gather}
(Each listed triple is lexicographically first in its equivalence class; the list is sorted first by multiplicative order, and then lexicographically.)
\end{proposition} 
\begin{proof}
We first check that the complete list of triples $(u,v,w)$ for which some two of $u,v,w$ are equal is
\begin{gather*}
(0,0,0), \left(0, \frac{1}{2}, \frac{1}{2}\right),
\left(\frac{1}{3}, \frac{1}{3}, \frac{1}{3} \right),
\left( \frac{1}{4}, \frac{1}{4}, \frac{1}{2} \right),\\
\left(\frac{1}{6}, \frac{1}{6}, \frac{2}{3} \right),
\left(\frac{1}{9}, \frac{1}{9}, \frac{7}{9}\right),
\left(\frac{1}{12}, \frac{1}{12}, \frac{5}{6}\right),
\left(\frac{1}{18}, \frac{1}{18}, \frac{8}{9}\right).
\end{gather*}
Up to equivalence we have $(u,v,w) = \bigl(\tfrac{1}{n}, \tfrac{1}{n}, \tfrac{n-2}{n}\bigr)$ for some $n$. Since $b=a, c=a^{-2}$ we have
\[
\left| a + b + c \right| = 2(a^3 + a^{-3}) + 5 \in \Z;
\]
this holds if and only if $a^3$ has order dividing 4 or 6. 
This yields the options
\[
n \in \{1,2,3,4,6,9,12,18\}
\]
and moreover confirms that they all work (because this list is closed under taking divisors).

We next check that the complete list of triples $(u,v,w)$ for which some two of $u,v,w$ differ by $\tfrac{1}{2}$ modulo $\Z$ is 
\begin{gather*}
\left(0, \frac{1}{2}, \frac{1}{2}\right), 
\left(0, \frac{1}{4}, \frac{3}{4}\right), 
\left(\frac{1}{6}, \frac{1}{6}, \frac{2}{3} \right),
\left(\frac{1}{8}, \frac{1}{4}, \frac{5}{8} \right), \\
\left(\frac{1}{12}, \frac{1}{3}, \frac{7}{12}\right), 
\left(\frac{1}{18}, \frac{7}{18}, \frac{5}{9}\right),
\left(\frac{1}{24}, \frac{5}{12}, \frac{13}{24}\right), 
\left(\frac{1}{36}, \frac{4}{9}, \frac{19}{36}\right).
\end{gather*}
We may assume without loss of generality that $u$ and $v$ differ by $\tfrac{1}{2}$ modulo $\Z$ and that $a$ has even order; then $a$ generates the group $\langle a,b,c \rangle$, and therefore up to equivalence we have $(u,v,w) = (\tfrac{1}{2n},\tfrac{n-2}{2n}, \tfrac{n+1}{2n})$ for some $n$. Since $b=-a, c = -a^{-2}$, we have
\[
\left| a^2 + b^2 + c^2 \right| = 2(a^6 + a^{-6}) + 5 \in \Z;
\]
this holds if and only if $a^6$ has order dividing 4 or 6. 
This yields the options
\[
n \in \{2,4,6,8,12,18,24,36\}
\]
and also confirms that they all work (because this list is closed under taking even divisors).

We finally check that the complete list of triples $(u,v,w)$ for which no two of $u,v,w$ differ by $0$ or $\tfrac{1}{2}$ modulo $\Z$ is 
\begin{gather*}
 \left(0, \frac{1}{3}, \frac{2}{3}\right), 
\left(0, \frac{1}{6}, \frac{5}{6}\right),
\left(\frac{1}{6}, \frac{1}{3}, \frac{1}{2} \right), 
\left(\frac{1}{7}, \frac{2}{7}, \frac{4}{7}\right), \\
\left(\frac{1}{8}, \frac{3}{8}, \frac{1}{2}\right),
\left(\frac{1}{12}, \frac{1}{6}, \frac{3}{4}\right),
\left(\frac{1}{12}, \frac{5}{12}, \frac{1}{2}\right), 
\left(\frac{1}{21}, \frac{4}{21}, \frac{16}{21}\right),
\left(\frac{1}{24}, \frac{1}{6}, \frac{19}{24}\right).
\end{gather*}
Note that for each positive integer $n$, the triple $(nu,nv,nw)$ must either be equivalent to an entry in \eqref{roots of unity list1} or must have two terms which differ by $\tfrac{1}{2}$ modulo $\Z$.
Already for $n=1$, this criterion eliminates all cases except the listed ones plus
$(\tfrac{1}{90}, \tfrac{19}{90}, \tfrac{7}{9})$; the latter case is eliminated by taking $n=2$.
The remaining cases persist for all $n$ (see Remark~\ref{remark: check roots of unity list}
for an easy way to check this).
\end{proof}

\begin{corollary}
Let $n$ be a positive integer, and let $H$ be the subgroup of $\SU(3)$ consisting of those diagonal matrices of order dividing $n$. Then $H$ satisfies the restricted rationality condition
if and only if $n \in \{1,2,3,4,6\}$.
\end{corollary}

\begin{remark} \label{remark: check roots of unity list}
One way to cross-check \eqref{roots of unity list2} is to verify the following statements
about each triple $(u,v,w)$ appearing therein.
\begin{enumerate}[{\rm(a)}]
\item
Either two of $u,v,w$ differ by $\tfrac{1}{2}$, or $(u,v,w)$ appears in \eqref{roots of unity list1}.
\item
Let $n$ denote the least common denominator of $u,v,w$. Then for each prime divisor $p$ of $n$,
the triple $(pu, pv, pw)$ is equivalent to an earlier triple in \eqref{roots of unity list2}.
\end{enumerate}
This confirms that each triple $(u,v,w)$ in the list defines a triple $(a,b,c)$ with the desired properties.
\end{remark}

\subsection{Abelian groups}
\label{section: BDM type A}

For the remainder of \S\ref{section: unitary}, let $H$ denote a finite subgroup of $\SU(3)$ containing $\mu_3$ that satisfies the restricted rationality condition; we identify the possibilities for $H$ of various types according to the BDM classification. We start with abelian groups, which occur in type (A) of the BDM classification.
We represent these groups as groups of diagonal matrices,
which are then stable under complex conjugation.

\begin{definition} \label{phi convention}
Following \cite[Chapter~XII]{MBD61}, we use the symbol $\phi$ as a synonym for the number $3$; specifically, we write
$n\phi$ to refer to the order of a finite subgroup of $\SU(3)$ containing~$\mu_3$ whose image in $\PSU(3)$ has order $n$. When including $\phi$ in a label for such a subgroup, we remove $\phi$ to indicate the image of this group in $\PSU(3)$ and (by abuse of notation) the resulting Sato--Tate group.
\end{definition}


\begin{lemma}\label{lemma: generators mu3} For $H$ abelian, there exist diagonal matrices $g_1$ and $g_2$ of orders $m$ and $n\phi$, with $m|n\phi$, such that $H$ is conjugate to 
$$
\langle g_1,g_2\rangle\simeq \Z/m\Z\times \Z/n\phi\Z
$$ 
and $\mu_3$ is contained in $\langle g_2\rangle$.
\end{lemma}

\begin{proof}
Since $H$ is abelian, its elements can be diagonalized simultaneously. Since it injects into $\Unitary(1)\times\Unitary(1)$, it is generated by just two elements; we can thus find matrices $g_1$ and $g_2$ of respective orders $m$ and $n\phi$, with $m|n\phi$, such that $H$ is equal to
$$
\langle g_1,g_2\rangle\simeq \Z/m\Z\times \Z/n\phi\Z.
$$
It remains to check that these can be chosen so that $\mu_3 \subseteq \langle g_2 \rangle$.
For this, we may ignore the prime-to-3 parts of $m$ and $n$, and thus reduce to the case where $m$ and $n$ are powers of 3. 
If $m = 1$, then $\mu_3 \subseteq H = \langle g_2 \rangle$ and we are done.
If $n>1$, then by Proposition \ref{roots of unity problem2} we must have~$g_2^{\nicefrac{n}{9}}$ conjugate to $D(\tfrac{1}{9}, \tfrac{1}{9}, \tfrac{7}{9})$, and therefore $\mu_3\subseteq \langle g_2\rangle$.
This leaves only the case $(m,n) = (3,1)$, in which case $H$ is a maximal 3-torsion subgroup of $\SU(3)$; in this case, we may take $g_2$ to be a generator of $\mu_3$ and $g_1$ to be a complementary generator.
\end{proof}

\begin{proposition}\label{proposition: abelian subgroups}
Suppose that $H$ is an abelian group.
Then $H$ is conjugate to exactly one of the $22$ groups listed in Table \ref{table: abelian subgroups}.
\end{proposition}

\begin{table}[ht]
\small
\begin{center}
\begin{tabular}{|l|l|l|l|}
\hline &&&\\[-11pt]
$H$ & Presentation & Symmetries & $N'_H$ \\[2pt]
\hline &&&\\[-10pt]
$\stgroup[A(1,1\phi)]{A(1,1)}$ &    $\langle D(\tfrac{1}{3},\tfrac{1}{3},\tfrac{1}{3}) \rangle$	& $\sym 3$ & Trivial\\[4pt]

$\stgroup[A(1,2\phi)]{A(1,2)}$ &    $\langle D(\tfrac{1}{3},\tfrac{5}{6},\tfrac{5}{6}) \rangle$ & $\cyc 2$ & Trivial \\[4pt]

$\stgroup[A(1,3\phi)]{A(1,3)}$ &    $\langle D(\tfrac{1}{9},\tfrac{4}{9},\tfrac{4}{9}) \rangle$	 & $\cyc 2$ & Trivial\\[4pt] 

$\stgroup[A(1,4\phi)_1]{A(1,4)_1}$ & $\langle D(\tfrac{1}{6},\tfrac{5}{12},\tfrac{5}{12}) \rangle$	& $\cyc 2$ & Trivial \\[4pt]

$\stgroup[A(1,4\phi)_2]{A(1,4)_2}$ & $\langle D(\tfrac{1}{3}, \tfrac{1}{12}, \tfrac{7}{12}) \rangle$	& $\cyc 2$ & $\cyc 2$ \\[4pt]

$\stgroup[A(1,6\phi)_1]{A(1,6)_1}$ & $\langle D(\tfrac{1}{9}, \tfrac{17}{18}, \tfrac{17}{18}) \rangle$& $\cyc 2$ & Trivial\\[4pt]

$\stgroup[A(1,6\phi)_2]{A(1,6)_2}$ & $\langle D(\tfrac{1}{18}, \tfrac{2}{9},  \tfrac{13}{18}) \rangle$& Trivial & Trivial	\\[4pt]

$\stgroup[A(1,7\phi)]{A(1,7)}$ &   $\langle D(\tfrac{1}{21}, \tfrac{16}{21}, \tfrac{4}{21})\rangle$ & $\alt 3$ & $\alt 3$ \\[4pt]

$\stgroup[A(1,8\phi)_1]{A(1,8)_1}$ & $\langle D(\tfrac{1}{6}, \tfrac{1}{24}, \tfrac{19}{24}) \rangle$& $\cyc 2$& $\cyc 2$	\\[4pt]

$\stgroup[A(1,8\phi)_2]{A(1,8)_2}$ & $\langle D(\tfrac{1}{12},\tfrac{5}{24},\tfrac{17}{24}) \rangle$ & $\cyc 2$ & $\cyc 2$ \\[4pt]

$\stgroup[A(1,12\phi)]{A(1,12)}$ &   $\langle D(\tfrac{1}{9}, \tfrac{7}{36}, \tfrac{25}{36}) \rangle$ & $\cyc 2$ & $\cyc 2$ \\[4pt]

 \hline & & & \\[-8pt]

$\stgroup[A(2,2\phi)]{A(2,2)}$ & $\langle D(0,\tfrac{1}{2},\tfrac{1}{2}), D(\tfrac{1}{6}, \tfrac{1}{6}, \tfrac{2}{3})\rangle$ & $\sym 3$ & $\sym 3$ \\[4pt]

$\stgroup[A(2,4\phi)]{A(2,4)}$ & $\langle D(\tfrac{1}{2},0,\tfrac{1}{2}), D(\tfrac{1}{6}, \tfrac{5}{12}, \tfrac{5}{12})\rangle$ & $\cyc 2$ & $\cyc 2$ \\[4pt]

$\stgroup[A(2,6\phi)]{A(2,6)}$ &$\langle D(\tfrac{1}{2},0,\tfrac{1}{2}), D(\tfrac{1}{9},\tfrac{17}{18},\tfrac{17}{18})\rangle$& $\cyc 2$ & $\cyc 2$ \\[4pt]

$\stgroup[A(3,1\phi)]{A(3,1)}$ & $\langle D(\tfrac{1}{3},0,\tfrac{2}{3}), D(\tfrac{1}{3},\tfrac{1}{3},\tfrac{1}{3}) \rangle$& $\sym 3$ & $\sym 3$ \\[4pt]

$\stgroup[A(3,2\phi)]{A(3,2)}$ &$\langle D(\tfrac{1}{3},0,\tfrac{2}{3}), D(\tfrac{2}{3},\tfrac{1}{6},\tfrac{1}{6})\rangle$& $\cyc 2$ & $\cyc 2$ \\[4pt]

$\stgroup[A(3,3\phi)]{A(3,3)}$ & $\langle D(0,\tfrac{1}{3},\tfrac{2}{3}), D(\tfrac{1}{9},\tfrac{1}{9},\tfrac{7}{9}))\rangle$& $\sym 3$ & $\sym 3$\\[4pt]

$\stgroup[A(3,4\phi)]{A(3,4)}$ & $\langle D(0,\tfrac{1}{3},\tfrac{2}{3}),D(\tfrac{1}{6},\tfrac{5}{12},\tfrac{5}{12})\rangle$& $\cyc 2$ & $\cyc 2$ \\[4pt]

$\stgroup[A(3,6\phi)]{A(3,6)}$ & $\langle D(0,\tfrac{1}{3},\tfrac{2}{3}),D(\tfrac{1}{9},\tfrac{5}{18},\tfrac{11}{18})\rangle$ & $\cyc 2$ & $\cyc 2$ \\[4pt]

$\stgroup[A(4,4\phi)]{A(4,4)}$ & $\langle D(0,\tfrac{1}{4},\tfrac{3}{4}), D(\tfrac{1}{12},\tfrac{1}{12},\tfrac{5}{6})\rangle$ & $\sym 3$& $\sym 3$\\[4pt]

$\stgroup[A(6,2\phi)]{A(6,2)}$ & $\langle D(\tfrac{1}{6},0,\tfrac{5}{6}), D(\tfrac{2}{3},\tfrac{1}{6},\tfrac{1}{6}) \rangle$ & $\sym 3$& $\sym 3$\\[4pt]

$\stgroup[A(6,6\phi)]{A(6,6)}$ & $\langle D(0,\tfrac{1}{6},\tfrac{5}{6}), D(\tfrac{1}{18},\tfrac{1}{18},\tfrac{8}{9}) \rangle$ & $\sym 3$& $\sym 3$\\[4pt] \hline
\end{tabular}
\end{center}
\medskip

\caption{Finite abelian subgroups of $\SU(3)$ containing $\mu_3$ and satisfying the restricted rationality condition. Following the notation in \cite{MBD61}, we define $\phi\colonequals 3$. The group $A(m,n\phi )_{*}=\langle g_1,g_2\rangle$ is isomorphic to $\Z/m\Z\times \Z/n\phi\Z$, the generators $g_1$ and $g_2$ have orders $m$ and $n\phi$ and $\mu_3\subseteq \langle g_2\rangle$.}\label{table: abelian subgroups}
\end{table}

\begin{proof}
In Table~\ref{table: abelian subgroups}, the cyclic groups are precisely the ones arising from
Proposition~\ref{roots of unity problem2} which contain $\mu_3$.
(The labels $A(m,n\phi)_*$ are chosen to preserve the order from \eqref{roots of unity list2}, but the generators themselves are not preserved; see Remark~\ref{R:swap diagonal entries}.)
In particular, this implies that any group that appears
has order divisible by only the primes 2,3,7; moreover, since $\stgroup[A(1,7\phi)]{A(1,7)}$ is the only cyclic group that occurs with order divisible by 7, no noncyclic group with order divisible by 7 can occur.

Any noncyclic group that occurs must have the form $\Z/m\Z \times \Z/n'\phi \Z$ where $m,n'$ are divisible only by the primes 2 and 3. Since there is no entry $A(1, 9\phi)$, $A(1, 16\phi)$, or $A(1, 24\phi)$ in Table~\ref{table: abelian subgroups}, neither $m$ nor $n'$ can be divisible by 9, 16, or 24. 

We next note that for
\[
(m,n') \in \{(2,2), (3,1), (4,4), (6,6)\}
\]
the group $A(m,n')$ can occur in only one way: it must be generated by $\mu_3$ and the maximal elementary abelian subgroup of the diagonal torus of exponent $m$. In particular, this group is stable under the action of $\sym 3$ on diagonal matrices (see Remark~\ref{R:swap diagonal entries} for more on this point).

We next note that for
\[
(m,n') \in \{(3,3), (6,6)\},
\]
the group $A(m,n')$ can again occur in only one way: it must be generated by $A(m,(n'/3)\phi)$ and $\stgroup[A(1, 3\phi)]{A(1,3)}$.
This group is again stable under the action of $\sym 3$.

We next rule out the cases
\[
(m,n') \in \{(2,8), (2,12), (3,8), (3,12), (6,4)\}.
\]
In each case, we do so by identifying the candidate groups, then identify a ``witness'' in each group which
does not appear in Proposition~\ref{roots of unity problem2}.
\begin{itemize}
\item
To achieve $(m,n') = (2,8)$, we would have to take the group generated by $\stgroup[A(2,2\phi)]{A(2,2)}$ and 
one of $\stgroup[A(1,8\phi)_1]{A(1,8)_1}$ or $\stgroup[A(1,8\phi)_2]{A(1,8)_2}$. We take as witnesses
\[
D(\tfrac{1}{2},0,\tfrac{1}{2}) \cdot D(\tfrac{1}{6},\tfrac{1}{24},\tfrac{19}{24}) = D(\tfrac{2}{3},\tfrac{1}{24},\tfrac{7}{24}), \quad
D(\tfrac{1}{2},0,\tfrac{1}{2}) \cdot D(\tfrac{1}{12},\tfrac{5}{24},\tfrac{17}{24}) = D(\tfrac{7}{12},\tfrac{5}{24},\tfrac{5}{24}).
\]
\item
To achieve $(m,n') = (2,12)$, we would have to take the group generated by $\stgroup[A(2, 2\phi)]{A(2,2)}$
and $\stgroup[A(1,12\phi)]{A(1,12)}$. We take as witness
\[
D(\tfrac{1}{2},0,\tfrac{1}{2}) \cdot D(\tfrac{1}{9},\tfrac{7}{36},
\tfrac{25}{36}) = D(\tfrac{11}{18}, \tfrac{7}{36}, \tfrac{7}{36}).
\]
\item
To achieve $(m,n') = (3,8)$, we would have to take the group generated by
$\stgroup[A(3,1\phi)]{A(3,1)}$ and 
one of $\stgroup[A(1,8\phi)_1]{A(1,8)_1}$ or $\stgroup[A(1,8\phi)_2]{A(1,8)_2}$. We take as witnesses
\[
D(\tfrac{2}{3},0,\tfrac{1}{3}) \cdot D(\tfrac{1}{3},\tfrac{1}{12},\tfrac{7}{12}) = D(0,\tfrac{1}{12},\tfrac{11}{12}), 
\quad
D(\tfrac{2}{3},0,\tfrac{1}{3}) \cdot D(\tfrac{1}{12},\tfrac{5}{24},\tfrac{17}{24}) = D(\tfrac{3}{4},\tfrac{5}{24},\tfrac{1}{24});
\]
the first witness also rules out the group generated by $A(3,1\phi)$ and $A(1,4\phi)_2$.
\item
To achieve $(m,n') = (3,12)$, we would have to take the group
generated by $\stgroup[A(3,1\phi)]{A(3,1)}$ and $\stgroup[A(1,12\phi)]{A(1,12)}$. We take as witness
\[
D(\tfrac{1}{3},0,\tfrac{2}{3}) \cdot D(\tfrac{1}{9},\tfrac{7}{36},\tfrac{25}{36}) = D(\tfrac{4}{9}, \tfrac{7}{36}, \tfrac{13}{36}).
\]
\item
To achieve $(m,n') = (6,4)$, we would have to take the group generated by $\stgroup[A(6,2\phi)]{A(6,2)}$
and one of $\stgroup[A(1,4\phi)_1]{A(1,4)_1}$ or $\stgroup[A(1,4\phi)_2]{A(1,4)_2}$. The latter case is ruled out above; for the former,
we take as witness
\[
D(\tfrac{1}{6},0,\tfrac{5}{6}) \cdot D(\tfrac{1}{6},\tfrac{5}{12},\tfrac{5}{12}) = D(\tfrac{1}{3},\tfrac{5}{12},\tfrac{1}{4}).
\]
\end{itemize}

We finally show that in the cases
\[
(m,n') \in \{(2,4), (2,6), (3,2), (3,4), (3,6)\},
\]
at most one group occurs.
\begin{itemize}
\item
For $(m,n') = (2,4)$, we would have to take the group generated by $\stgroup[A(2,2\phi)]{A(2,2)}$ and 
one of $\stgroup[A(1,4\phi)_1]{A(1,4)_1}$ or $\stgroup[A(1,4\phi)_2]{A(1,4)_2}$.
The resulting groups are conjugate because the group generated by $\stgroup[A(2,2\phi)]{A(2,2)}$ and $\stgroup[A(1,4\phi)_1]{A(1,4)_1}$ contains a conjugate of $\stgroup[A(1,4\phi)_2]{A(1,4)_2}$.
\item
For $(m,n') = (2,6)$, we must take the group generated by $\stgroup[A(2,2\phi)]{A(2,2)}$ and $\stgroup[A(1,3\phi)]{A(1,3)}$.
\item
For $(m,n') = (3,2)$, we must take the group generated by $\stgroup[A(3,1\phi)]{A(3,1)}$ and $\stgroup[A(1,2\phi)]{A(1,2)}$.
\item
For $(m,n') = (3,4)$, we would have to take the group generated by $\stgroup[A(3,1\phi)]{A(3,1)}$ and 
one of $\stgroup[A(1,4\phi)_1]{A(1,4)_1}$ or $\stgroup[A(1,4\phi)_2]{A(1,4)_2}$; the latter case is ruled out above. 
\item
For $(m,n') = (3,6)$, we must take the group generated by $\stgroup[A(3,3\phi)]{A(3,3)}$ and $\stgroup[A(1,2\phi)]{A(1,2)}$.
\end{itemize}
This leaves the options enumerated in Table~\ref{table: abelian subgroups}.
\end{proof}

\begin{remark}
Recall that we do not need to check that the groups listed in Table~\ref{table: abelian subgroups}
satisfy the rationality condition, as this will be confirmed \emph{a posteriori} by the realizability
of these groups as Sato--Tate groups. If one did want to do
this by hand for the noncyclic groups, it would suffice to do for the maximal noncyclic groups in the table, namely
\[
\stgroup[A(3,4\phi)]{A(3,4)}, \stgroup[A(4,4\phi)]{A(4,4)}, \stgroup[A(6,6\phi)]{A(6,6)}.
\]
\end{remark}

\begin{remark} \label{R:swap diagonal entries}
In Table~\ref{table: abelian subgroups}, the ``Symmetries'' column indicates the subgroup of permutation matrices
which normalize the image of the group in $\PSU(3)$; this data is needed for the analysis of the remaining types
in the BDM classification. When the group is given as $\cyc 2$, the presentation has been chosen so that the normalizing transposition is $(2\,3)$.
\end{remark}

\begin{remark}
Note that certain pairs of groups in Table~\ref{table: abelian subgroups} give rise to isomorphic but non-conjugate pairs or triples of subgroups of $\PSU(3)$:
\begin{gather*}
\stgroup{A(1,3)} \simeq \stgroup{A(3,1)},
\qquad
\stgroup{A(1,4)_1} \simeq \stgroup{A(1,4)_2},\\
\stgroup{A(1,6)_1} \simeq \stgroup{A(1,6)_2} \simeq \stgroup{A(3,2)},
\qquad
\stgroup{A(1,8)_1} \simeq \stgroup{A(1,8)_2},
\qquad
\stgroup{A(1,12)} \simeq \stgroup{A(3,4)}.
\end{gather*}
Some differences between these groups will appear as we continue the classification.
\end{remark}

\subsection{Extensions of \texorpdfstring{$\cyc 2$}{C2}}
\label{section: BDM type B}

Type (B) in the BDM classification consists of nonabelian groups contained in 
$\SU(3) \cap (\Unitary(1) \times \Unitary(2))$.
Note that $\Unitary(1) \times \SU(2)$ surjects onto $\SU(3) \cap (\Unitary(1) \times \Unitary(2))$
via the map 
$$
\pi\colon \Unitary(1) \times \SU(2)\rightarrow\SU(3) \cap (\Unitary(1) \times \Unitary(2)),\qquad (u, A) \mapsto (u^2, u^{-1} A)
$$ 
with kernel of order 2 generated by $\{(-1, -1)\}$;
the groups in question thus correspond to finite subgroups of $\Unitary(1) \times \SU(2)$ containing
both $(-1,-1)$ and $(e(\tfrac{1}{3}), 1)$. 

For the moment, we focus on groups which are (nonabelian) extensions of $\cyc 2$ by abelian groups;
this amounts to requiring that they have dihedral image in $\PSU(2) \simeq \SO(3)$.
This includes products of cyclic groups with dihedral groups, but also some nonsplit extensions such as the quaternion group of order 8. The remaining cases are groups which project to one of the exceptional subgroups of $\SU(2)$;
we classify those in \S\ref{section:SU2 exceptional}.

For $\xi$, $\alpha\in \C^\times$ of norm $1$, let
\begin{equation} \label{eq:type D conjugation matrix}
R_{\xi,\alpha}\colonequals
\begin{pmatrix}
-\overline{\xi}  & 0 & 0\\
0 & 0 & -\overline{\alpha}\\
0 & -\xi\alpha & 0 
\end{pmatrix} \in \SU(3);
\end{equation}
note that
\[
R_{\xi,\alpha}^2=\Diag(\xi^{-2}, \xi, \xi),
\qquad
R_{\xi, \alpha}^{-1} = R_{\xi^{-1}, \xi \alpha}.
\]
In particular, if $\xi \alpha = \overline{\alpha}$, then $\overline{R}_{\xi,\alpha} = R_{\xi,\alpha}^{-1}$; this will ensure that any presentation using this matrix is invariant under complex conjugation. With this in mind, we set
\begin{equation} \label{eq:type B conjugation matrices}
T_1 = R_{1, 1}, \qquad T_2 = R_{-1,i}, \qquad T_4 = R_{i, e(\nicefrac{3}{8})}.
\end{equation}

\begin{lemma}\label{lemma: dihedral groups}
Let $H$ be an extension of $\cyc 2$ by an abelian group. Then there is a group $H'$ of order $\tfrac{1}{2}|H|$ in Table \ref{table: abelian subgroups} with at least $\cyc 2$ symmetries such that $H$ is conjugate to one of:
\begin{enumerate}[{\rm(i)}]
\setlength{\itemsep}{2pt}
\item $\langle H', T_1\rangle$; or
\item $\langle H', T_2\rangle$ and then $H'$ contains $T_2^2=D(0,\tfrac{1}{2},\tfrac{1}{2})$; or
\item $\langle H', T_4\rangle$ and then $H'$ contains $T_4^2=D(\tfrac{1}{2}, \tfrac{1}{4},\tfrac{1}{4})$.
\end{enumerate}
\end{lemma}

\begin{proof} 
Write $H$ as an extension of $\cyc 2$ by $H'$; by Proposition~\ref{proposition: abelian subgroups}, we may conjugate $H'$ into one of the forms listed in Table~\ref{table: abelian subgroups}, in such a way that the action of $\cyc 2$ is given by interchanging the second and third coordinates. Consequently, the nontrivial class in $\cyc 2$ is represented by~$R_{\xi,\alpha}$ for some $\xi, \alpha$. The element $R_{\xi, \alpha}^2 = \Diag(\xi^{-2}, \xi, \xi)$ must belong to $H'$,
so $\xi$ must be a root of unity; by replacing $R_{\xi,\alpha}$ with $R_{\xi,\alpha}^s =
R_{\xi^s, \xi^{\nicefrac{(1-s)}{2}}\alpha}$ for a suitable odd integer~$s$, we may ensure that $\xi = e(\nicefrac{1}{2^h})$ for some $h$.
Since $R_{\xi,\alpha}^2$ must satisfy the restricted rationality condition, we see from Proposition \ref{proposition: abelian subgroups} that $\xi$ has order at most 4; hence $\xi \in \{1, -1, i\}$. We may then rescale $\alpha$ 
(by conjugating by a matrix of the form $\Diag(1, \beta, \beta^{-1})$ for suitable $\beta$)
to ensure that $R_{\xi,\alpha} \in \{T_1, T_2, T_4\}$, from which the claim follows.
\end{proof}

\begin{remark} \label{B redundancy}
In case (ii) of Lemma~\ref{lemma: dihedral groups}, if $H'$ contains an element $g$ of the form $D(\tfrac{1}{2},*,*)$, then it also contains 
\[
g T_2 = \begin{pmatrix} -1 & 0 &0 \\0 & 0 & * \\
0 & * & 0
\end{pmatrix};
\]
by conjugating by a suitable matrix in the diagonal torus, we see that $\langle H', gT_2 \rangle$ is conjugate to $\langle H', T_1 \rangle$. By the same token, in case (iii) of Lemma~\ref{lemma: dihedral groups}, if $H'$ contains an element $g$ of the form $D(\tfrac{1}{4},*,*)$, then $\langle H', gT_4 \rangle$ is conjugate to $\langle H', T_1 \rangle$.
\end{remark}

\begin{proposition}  \label{proposition: dihedral subgroups}
Suppose that $H$ is an extension of $\cyc 2$ by an abelian group and is not abelian.
Then $H$ is conjugate  to exactly one of the $18$ groups listed in Table \ref{table: dihedral subgroups}.
\end{proposition}

\begin{proof} 
Write $B(m,n\phi;t)_*$ for the group $\langle A(m,n\phi)_*,T_t\rangle$, omitting $t$ when it equals 1.
By Lemma \ref{lemma: dihedral groups}, it suffices to include the following entries.
\begin{enumerate}[{\rm(i)}]
\item 
The group $B(m,n\phi)_*$ for each group $A(m,n\phi)_*$ of Table \ref{table: abelian subgroups} with at least $\cyc 2$ symmetries, excluding $\stgroup[A(1, 1\phi)]{A(1,1)}, \stgroup[A(1, 2\phi)]{A(1,2)}, \stgroup[A(1, 3\phi)]{A(1,3)}, \stgroup[A(1, 4\phi)_1]{A(1,4)_1}, \stgroup[A(1, 6\phi)_1]{A(1,6)}$ which give rise to abelian groups; this exhausts all cases where $H$ is a semidirect product of an abelian group by $\cyc 2$.
Of these, $\stgroup[B(2,2\phi)]{B(1,4)_2}$ is conjugate to $\stgroup[B(1, 4\phi)_2]{B(1,4)_2}$
(by writing it as the semidirect product of a cyclic subgroup by $\cyc 2$) and has thus been omitted;
similarly, $\stgroup[B(2,6\phi)]{B(1,12)}$ is conjugate to $\stgroup[B(1,12\phi)]{B(1,12)}$ and has thus been omitted.
We also omit $\stgroup[B(1,8\phi)_2]{B(2,4;4)}$ because it will be redundant with (iii).
This yields $12$ entries.

\item 
The group $B(m,n\phi;2)_*$ for each group $A(m,n\phi)_*$ in (i) containing $D(0,\tfrac{1}{2},\tfrac{1}{2})$ but not containing any element of the form $D(\tfrac{1}{2}, *, *)$ (see Remark~\ref{B redundancy}).
This yields~$4$ entries.

\item
The group $B(m,n\phi;4)_*$ for each group $A(m,n\phi)_*$ listed in (i) containing $D(\tfrac{1}{2},\tfrac{1}{4},\tfrac{1}{4})$ but not containing any element of the form $D(\tfrac{1}{4}, *, *)$ (see Remark~\ref{B redundancy}).
This yields~$2$ entries;
the group $\stgroup[B(2,4\phi;4)]{B(2,4;4)}$ is conjugate to $\stgroup[B(1,8\phi)_2]{B(2,4;4)}$ (with the cyclic subgroup of order 8 being generated by $T_4$), justifying the omission of the latter.
\end{enumerate}
The results of this tabulation are shown in Table \ref{table: dihedral subgroups} below.
\end{proof}

\begin{table}[h]
\small
\begin{center}
\begin{tabular}{|lll|l|ll|}
\hline&&&&&\\[-10pt]
$\stgroup[B(1,4\phi)_2]{B(1,4)_2}$ &$(\simeq \stgroup[B(2,2\phi)]{B(1,4)_2})$ &  $\stgroup[B(3,3\phi)]{B(3,3)}$ &      $\stgroup[B(1,4\phi;2)_2]{B(1,4;2)_2}$ & $\stgroup[B(2,4\phi;4)]{B(2,4;4)}$ &$(\simeq \stgroup[B(1,8\phi)_2]{B(2,4;4)})$ \\[4pt]
$\stgroup[B(1,8\phi)_1]{B(1,8)_1}$  && $\stgroup[B(3,4\phi)]{B(3,4)}$ &  $\stgroup[B(1,12\phi;2)]{B(1,12;2)}$    &   $\stgroup[B(3,4\phi;4)]{B(3,4;4)}$ & \\[4pt] 
$\stgroup[B(1,12\phi)]{B(1,12)}$ & $(\simeq \stgroup[B(2,6\phi)]{B(1,12)})$  & $\stgroup[B(3,6\phi)]{B(3,6)}$ & $\stgroup[B(3, 2\phi; 2)]{B(3,2;2)}$  &    &\\ [4pt]
$\stgroup[B(2,4\phi)]{B(2,4)}$ &       & $\stgroup[B(4,4\phi)]{B(4,4)}$  & $\stgroup[B(3,6\phi;2)]{B(3,6;2)}$    & & \\[4pt]
$\stgroup[B(3,1\phi)]{B(3,1)}$ && $\stgroup[B(6,2\phi)]{B(6,2)}$  &&  &   \\ [4pt]
$\stgroup[B(3,2\phi)]{B(3,2)}$ && $\stgroup[B(6,6\phi)]{B(6,6)}$ &&& \\[1pt]
\hline
\end{tabular}
\medskip

\caption{Finite subgroups of $\SU(3)$ containing $\mu_3$ that satisfy the restricted rationality condition which are extensions of $\cyc 2$ by an abelian group. The group $B(m,n\phi;t)_*$ denotes the group $\langle A(m,n\phi)_*,T_t\rangle$, with $t=1$ omitted.}
\label{table: dihedral subgroups}
\end{center}
\end{table}

\subsection{Exceptional groups from \texorpdfstring{$\SU(2)$}{SU(2)}}
\label{section:SU2 exceptional}

We now finish off type (B) of the BDM classification, which consists of groups contained in $\SU(3) \cap (\Unitary(1) \times \Unitary(2))$. In what follows, it will be convenient to identify $\SU(2)$
with the group of unit quaternions via the isomorphism
\[
a + b \mathbf{i} + c \mathbf{j} + d \mathbf{k} \mapsto \begin{pmatrix}
a + bi & c +di \\ -c+di & a- bi
\end{pmatrix}.
\]
As in \S\ref{section: BDM type B}, we work with the surjection
$$
\pi\colon \Unitary(1) \times \SU(2)\rightarrow\SU(3) \cap (\Unitary(1) \times \Unitary(2)),\qquad (u, A) \mapsto (u^2, u^{-1} A)
$$ 
and look for finite subgroups $\tilde{H}$ of $\Unitary(1) \times \SU(2)$ containing
both $(-1,-1)$ and $(e(\tfrac{1}{3}), 1)$. 
Any such group whose image in $\SU(2)$ is cyclic or dihedral gives rise to a subgroup $H$ of $\SU(3)$ which is either abelian, or an extension of $\cyc 2$ by an abelian group; these cases have already been accounted for by Proposition~\ref{proposition: abelian subgroups} 
and Proposition~\ref{proposition: dihedral subgroups}. It thus remains to consider cases where the image in $\SU(2)$ is one of the three exceptional finite subgroups: the binary tetrahedral, octahedral, and icosahedral groups.
We may immediately rule out the case where~$\tilde{H}$ projects to the binary icosahedral group: in any such case, the component group of $H$ has order divisible by 5, and so $H$ cannot satisfy the restricted rationality condition.

We next consider the case where $\tilde{H}$ projects onto the binary tetrahedral group, presented~as
\[
2T=\{ \pm 1, \pm \mathbf{i}, \pm \mathbf{j}, \pm \mathbf{k}, \tfrac{1}{2}(\pm 1 \pm \mathbf{i} \pm \mathbf{j} \pm \mathbf{k})\}= \langle \mathbf{j}, \tfrac{1}{2}(1+\mathbf{i} + \mathbf{j} + \mathbf{k}) \rangle;
\]
this presentation is invariant under complex conjugation, as are all of the presentations we will derive from it,
with one exception; see Proposition~\ref{tetrahedral extensions}. The only nontrivial cyclic quotient of $2T$ is
\[
2T/Q \simeq \cyc{3}, \quad Q = \{ \pm 1, \pm \mathbf{i}, \pm \mathbf{j}, \pm \mathbf{k}\}.
\]

\begin{proposition} \label{binary tetrahedral}
Suppose that $H$ is the image of a finite subgroup $\tilde{H}$ of $\Unitary(1) \times \SU(2)$ containing $(-1,-1)$ and $(e(\tfrac{1}{3}), 1)$ which projects to the subgroup $2T$ of $\SU(2)$. Then $H$ is conjugate to one of
\[
\stgroup[B(T, 1\phi)]{B(T,1)},\quad \stgroup[B(T,2\phi)]{B(T,2)},\quad \stgroup[B(T,3\phi)]{B(T,3)},\quad \stgroup[B(T,1\phi;1)]{B(T,1;1)}.
\]
Here $B(T, n\phi)$ denotes $\pi(\cyc{2n\phi} \times 2T)$,
while $\stgroup[B(T,1\phi;1)]{B(T,1;1)}$ denotes $\pi(\cyc{6\phi} \times_{\cyc{3}} 2T)$ where the fiber product is via 
the nontrivial homomorphism $\cyc{6} \to \cyc{3}$
and one of the two nontrivial homomorphisms $2T/Q \to \cyc{3}$.
\end{proposition}
\begin{proof}
Let $C$ be the kernel of the map $\tilde{H} \to 2T$, identified with a subgroup of $\Unitary(1)$.
Then~$\tilde{H}$ in $\Unitary(1)/C \times 2T$ is the graph of a homomorphism $2T \to \Unitary(1)/C$.
If this homomorphism is trivial, then $\tilde{H} = C \times 2T$.
Because $(-1,-1)$ and $(e(\tfrac{1}{3}), 1)$ belong to $\tilde{H}$, the order of~$C$ can be written as $2n\phi$;
the group $C$ projects to the subgroup of $\SU(3)$
generated by $D(\tfrac{1}{n\phi},-\tfrac{1}{2n\phi},-\tfrac{1}{2n\phi})$. By Proposition~\ref{roots of unity problem2}, 
the latter satisfies the restricted rationality condition if and only if 
$n \in \{1,2,3 \}$. 
For all such $n$, the image of $C \times 2T$ in $\SU(3)$ satisfies the restricted rationality condition.

Suppose now that $2T \to \Unitary(1)/C$ is nontrivial;
it must then have image $\cyc 3$.
Since $-1 \in Q$, we must have $-1 \in C$; hence the order of $C$ must again be of the form $2n\phi$.
As above, we must have $n \in \{1,2,3\}$. 
When $n=1$, we see from above that $H$ satisfies the restricted rationality condition
(the two possible groups of this form are conjugate to each other).
In the other two cases, $\tilde{H}$ must contain an element of the form
$(u, \tfrac{1}{2}( 1 + \mathbf{i} + \mathbf{j} + \mathbf{k}))$
where $u \in \Unitary(1)$ has order $12\phi$ (if $n=2$) or $18\phi$ (if $n=3$).
The latter case is impossible, as it would imply that $H$ contains an element of order $9\phi$ and this
is inconsistent with Table~\ref{table: abelian subgroups}.
As for the former case, projecting the element in question to $H$ yields a matrix conjugate in $\SU(3)$ to
$\Diag(u^2, u^{-1} e(\tfrac{1}{6}), u^{-1} e(-\tfrac{1}{6}))$;
the three components of this diagonal matrix
have orders $6\phi, 12\phi, 12\phi$, so this matrix does not generate a group conjugate to $\stgroup[A(1,12\phi)]{A(1,12)}$ (in which the component orders of a generator are $3\phi,12\phi,12\phi$) and therefore cannot satisfy the restricted rationality condition.
\end{proof}

We finally consider the case where $\tilde{H}$ projects to the binary octahedral group, presented~as
\[
2O = 2T \cup 2T \cdot \left(\tfrac{1}{\sqrt{2}} (\pm 1 \pm \mathbf{i}) \right);
\]
this presentation is invariant under complex conjugation, as are all of the presentations we will derive from it.
The only nontrivial cyclic quotient of $2O$ is $2O/2T \simeq \cyc{2}$.

\begin{proposition} \label{binary octahedral}
Suppose that $H$ is the image of a finite subgroup $\tilde{H}$ of $\Unitary(1) \times \SU(2)$ containing $(-1,-1)$ and $(e(\tfrac{1}{3}), 1)$ which projects to the subgroup $2O$ of $\SU(2)$. Then $H$ is conjugate to one of
\[
\stgroup[B(O, 1\phi)]{B(O,1)},\quad \stgroup[B(O, 2\phi)]{B(O,2)},
\]
where $B(O, n\phi)$ denotes the nontrivial fiber product $\pi(C_{4n\phi} \times_{\cyc 2} 2O)$.
(Beware that $\stgroup[B(O,1\phi)]{B(O,1)}$ is not conjugate to a subgroup of $\stgroup[B(O, 2\phi)]{B(O,2)}$.)
\end{proposition}
\begin{proof}
Let $C$ be the kernel of the map $\tilde{H} \to 2O$, identified with a subgroup of $\Unitary(1)$.
Then~$\tilde{H}$ in $\Unitary(1)/C \times 2O$ is the graph of a homomorphism $2O \to \Unitary(1)/C$.
If this homomorphism is trivial, then $\tilde{H} = C \times 2O$;
however, already the image of $\cyc{\phi} \times 2O$ fails to satisfy the restricted rationality condition
due to the presence of the matrix $D(0,\tfrac{1}{8},\tfrac{7}{8})$.

If the homomorphism is nontrivial, then it must map $2O$ to its only nontrivial cyclic quotient, namely
$2O/2T \simeq \Z/2\Z$; in particular, we have $C \times 2T \subset \tilde{H}$, so the order of $n$ of $C$ must be in
$\{2\phi, 4\phi, 6\phi\}$ by Proposition~\ref{binary tetrahedral}.
Then $\tilde{H}$ must contain an element of the form
$(u, \tfrac{1}{\sqrt{2}} ( 1 + \mathbf{i}))$
where $u \in \Unitary(1)$ has order $4\phi$ (if $n=2\phi$), $8\phi$ (if $n=4\phi$), or $12\phi$ (if $n=6\phi$).
If $n = 6\phi$, then projecting the element in question to $H$ yields 
$\Diag(u^2, u^{-1} e(\tfrac{1}{8}), u^{-1} e(-\tfrac{1}{8}))$,
whose diagonal components have orders $6\phi, 24\phi, 24\phi$; however, there is no cyclic group of order $24\phi$
satisfying the restricted rationality condition.
In the other two cases, one does get groups satisfying the restricted rationality condition.
\end{proof}

\subsection{Extensions of \texorpdfstring{$\alt 3$}{A3}}
\label{subsection: BDM type C}

Type (C) in the BDM classification consists of semidirect products of abelian groups,
represented as diagonal matrices, with the cyclic group $\alt 3$, represented as cyclic permutation matrices. (In particular, we need not consider nonsplit extensions by $\alt 3$; this will be apparent from the proof of Lemma~\ref{lemma: type (C) groups}.) Let
\begin{equation} \label{eq:permutation matrix S}
S = \begin{pmatrix}
0 & 0 & 1 \\ 1 & 0 & 0 \\ 0 & 1 & 0
\end{pmatrix};
\end{equation}
since this matrix is real, the following presentations are again invariant under complex conjugation.

\begin{lemma}\label{lemma: type (C) groups} Suppose that $H$ is an extension of $\alt 3$ by an abelian group. Then there exists an abelian group $H'$ of order $\tfrac{1}{3}|H|$ in Table \ref{table: abelian subgroups} with at least $\alt 3$ symmetries  such that $H$ is conjugate  to $\langle H', S\rangle$; in particular, $H \simeq H' \rtimes \alt 3$.
\end{lemma}

\begin{proof}
By conjugating suitably, we may ensure that $H$ is generated by a diagonal abelian subgroup $H'$ and a matrix of the form
\begin{equation} \label{matrix S-zeta-mu}
S_{\zeta,\mu}=
\begin{pmatrix}
0 & 0 & \zeta^ {-1}\mu^{-1}\\
\zeta & 0 & 0\\
0 & \mu& 0 
\end{pmatrix}
\end{equation}
for some $\zeta, \mu \in \C^ *$. Since 
\begin{equation}\label{equation: conjugation diagonal elements} 
S_{\zeta,\mu}\cdot \Diag(d_1,d_2,d_3)\cdot S_{\zeta,\mu}^{-1}=\Diag(d_3,d_1,d_2),
\end{equation}
the group $H'$ is invariant under a cyclic permutation of the diagonal elements of its constituting matrices;
we can thus conjugate it into a form appearing in Table \ref{table: abelian subgroups}.

Write $\tilde S=S_{\zeta^{\nicefrac{2}{3}}\mu^{\nicefrac{1}{3}},\zeta^{\nicefrac{-1}{3}}\mu^{\nicefrac{1}{3}}}$. Since 
$$
\tilde S \cdot S_{\zeta,\mu}\cdot \tilde S^{-1}=S,
$$
and conjugation by $\tilde S$ preserves $H'$ (as in \ref{equation: conjugation diagonal elements}), we see that $H$ is conjugate to $\langle H',S\rangle$.
\end{proof}

\begin{proposition}  \label{A3 extensions}
Suppose that $H$ is the semidirect product of an abelian group by $\alt 3$ and is not abelian. Then $H$ is conjugate  to exactly one of the $7$ following groups:
$$
 \stgroup[C(1,7\phi)]{C(1,7)},\ \stgroup[C(2,2\phi)]{C(2,2)},\ \stgroup[C(3,1\phi)]{C(3,1)},\ \stgroup[C(3,3\phi)]{C(3,3)},\ \stgroup[C(4,4\phi)]{C(4,4)},\ \stgroup[C(6,2\phi)]{C(6,2)},\ \stgroup[C(6,6\phi)]{C(6,6)}.
$$
Here $C(m,n\phi)$ denotes the group $\langle A(m,n\phi),S\rangle\simeq (\Z/m\Z\times \Z/n\phi\Z) \rtimes\alt 3$.
\end{proposition}

\begin{proof} 
By Lemma \ref{lemma: type (C) groups}, the given list includes all possibilities; note that the group $C(1,1\phi)$ is abelian and are hence omitted. (By contrast, $\stgroup[C(3,1\phi)]{C(3,1)}$ is not abelian, but its image in $\PSU(3)$ is abelian.)
It is automatic that these groups satisfy the restricted rationality condition because the product of $S$ with any matrix in $A(m,n\phi)$ has eigenvalues $1,e(\tfrac{1}{3}),e(\tfrac{2}{3})$.
\end{proof}

\subsection{Extensions of \texorpdfstring{$\sym 3$}{S3}}
\label{section: BDM type D}

Type (D) in the BDM classification consists of extensions of the symmetric group $\sym 3$, represented as scaled permutation matrices, by abelian groups,
represented as diagonal matrices.

Define the matrices $S$ and $T_1$ as in \eqref{eq:permutation matrix S} and \eqref{eq:type B conjugation matrices}, respectively.
Since these matrices have real entries, the following presentations are again invariant under complex conjugation.

\begin{lemma}\label{lemma: type (D) groups} 
Suppose that $H$ is an extension of $\sym 3$ by an abelian group. Then there exists an abelian group $H'$ of order $\tfrac{1}{6}|H|$ in Table \ref{table: abelian subgroups} with $\sym 3$ symmetries such that $H$ is conjugate to 
$\langle H', S, T_1\rangle$. In particular, $H \simeq H' \rtimes \sym 3$.
\end{lemma}
\begin{proof} 
As in Lemma \ref{lemma: type (C) groups}, we may assume that $H$ is conjugate to  $\langle H',S,R_{\xi,\alpha}\rangle$, where $H'$ is one of the abelian groups of order $\tfrac{1}{6}|H|$ of Table \ref{table: abelian subgroups},
and $R_{\xi,\alpha}$ is defined as in \eqref{eq:type D conjugation matrix}
for some $\xi,\alpha\in\C^\times$. Since
$$
R_{\xi,\alpha}^2= \Diag(\xi^{-2}, \xi, \xi), \qquad
(S\cdot R_{\xi,\alpha})^2= \Diag(\alpha,\alpha,\alpha^{-2}),
$$ 
$\xi$ and $\alpha$ are roots of unity. 
Moreover, for $n$ the least common multiple of the orders of $\xi$ and~$\alpha$,
the $\sym 3$-symmetry of $H'$ implies that $H'$ contains the full $n$-torsion subgroup of $\SU(3)$.
We can thus rescale to force $\xi = \alpha = 1$, yielding the desired result.
\end{proof}

\begin{proposition}  \label{sym3 groups}
Suppose that $H$ is an extension of $\sym 3$ by an abelian group,
but not an extension of $\cyc 2$ by an abelian group.
Then $H$  is conjugate to one of the $6$ following groups:
\begin{gather*}
\stgroup[D(2,2\phi)]{D(2,2)},\ \stgroup[D(3,1\phi)]{D(3,1)},\ \stgroup[D(3, 3\phi)]{D(3,3)},\ \stgroup[D(4,4\phi)]{D(4,4)},\ \stgroup[D(6,2\phi)]{D(6,2)},\ \stgroup[D(6,6\phi)]{D(6,6)}.
\end{gather*}
Here $D(m,n\phi)$ denotes the group $\langle A(m,n\phi),S,T_1\rangle\simeq (\Z/m\Z\times \Z/n\phi\Z) \rtimes\sym 3$.
\end{proposition}

\begin{proof}
By Lemma \ref{lemma: type (D) groups}, the given list includes all possibilities.
As in the proof of Proposition~\ref{A3 extensions}, the restricted rationality condition is automatic in these cases.
\end{proof}

\subsection{Solvable exceptional subgroups}
\label{section:solvable exceptional}

Types (E), (F), (G) in the BDM classification are exceptional groups whose images in
$\SU(3)/\mu_3$ have orders 36, 72, 216, respectively; we denote these by $\stgroup[E(36\phi)]{E(36)}, \stgroup[E(72\phi)]{E(72)}, \stgroup[E(216\phi)]{E(216)}$. 
The group $\stgroup{E(216)}$ is the \emph{Hessian group}.

Define the matrices
\begin{equation} \label{eq:Hessian matrices}
g_1=D(0,\tfrac{1}{3},\tfrac{2}{3}),
\quad
g_2=\frac{1}{e(\tfrac{1}{3})-e(\tfrac{2}{3})}
\begin{pmatrix}
1 & 1 & 1\\[0.7mm]
1 & e(\tfrac{1}{3}) & e(\tfrac{2}{3})\\[0.7mm]
1 & e(\tfrac{2}{3}) & e(\tfrac{1}{3})
\end{pmatrix},
\quad
g_3= D(\tfrac{2}{9},\tfrac{2}{9},\tfrac{5}{9}).
\end{equation}
Then $\stgroup[E(36\phi)]{E(36)}, \stgroup[E(72\phi)]{E(72)}, \stgroup[E(216\phi)]{E(216)}$ are the respective groups
$$
\langle g_1,g_2\rangle,
\quad
\langle g_1,g_2,g_3g_2g_3^{-1}\rangle,
\quad
\langle g_1,g_2,g_3\rangle.
$$
(The presentations given in \cite[Chapter~XII]{MBD61} include the matrix $S$ of \eqref{eq:permutation matrix S},
but this is unnecessary because $S = g_2^{-1} g_1 g_2$.)

Note that $\overline{g}_2 = g_2^3$ and that 
the complex conjugate of $g_3 g_2 g_3^{-1}$ belongs to $\langle g_1, g_2, g_3 g_2 g_3^{-1} \rangle$:
since $S \in \langle g_1, g_2 \rangle$, we can see the latter by writing the conjugate as
\[
g_2 (g_3 g_2 g_3^{-1}) D(\tfrac{2}{3},\tfrac{2}{3},\tfrac{2}{3})
\]
and observing that $D(\tfrac{1}{3},\tfrac{1}{3},\tfrac{1}{3}) = g_1 S g_1^{-1} S^{-1}$. This means that all three of these presentations are invariant under complex conjugation; these groups also satisfy the restricted rationality condition.

\begin{remark}
As an aside, we note that for $S_{\zeta,\mu}$ as in \eqref{matrix S-zeta-mu}, the group
$\stgroup[E(216\phi)]{E(216)}$ contains the subgroup $\langle \stgroup[A(1,1\phi)]{A(1,1)},
S_{e(\nicefrac{2}{9}),e(\nicefrac{2}{9})} \rangle$
which is conjugate to $\stgroup[C(3,1\phi)]{C(3,1)}$ in $\SU(3)$ (by Lemma~\ref{lemma: type (C) groups})
but not in $\stgroup[E(216\phi)]{E(216)}$.
\end{remark}

\subsection{Simple exceptional groups}
\label{section:simple exceptional}

The final types (H), (I), (J) in the BDM classification are exceptional groups whose images in
$\SU(3)/\mu_3$ have orders 60, 360, 168, respectively. These images are the unique simple groups of these orders:
$\alt 5$, $\alt 6$, and $\PSL_3(\F_2) \simeq \PSL_2(\F_7)$.
We may immediately rule out the groups of order 60 and 360 because any groups made from them contain elements of order 5 and thus cannot satisfy the restricted rationality condition. We denote the group of order $168\phi$ by $\stgroup[E(168\phi)]{E(168)}$; it admits the presentation
\begin{equation} \label{eq:168 matrix}
\left\langle
D(\tfrac{1}{3}, \tfrac{1}{3}, \tfrac{1}{3}), S, D(\tfrac{1}{7}, \tfrac{2}{7}, \tfrac{4}{7}),
\tfrac{1}{\sqrt{-7}}
\begin{pmatrix}
e(\tfrac{4}{7}) - e(\tfrac{3}{7}) & e(\tfrac{2}{7}) - e(\tfrac{5}{7}) & e(\tfrac{1}{7}) - e(\tfrac{6}{7})\\[0.7mm]
e(\tfrac{2}{7}) - e(\tfrac{5}{7}) & e(\tfrac{1}{7}) - e(\tfrac{6}{7}) & e(\tfrac{4}{7}) - e(\tfrac{3}{7}) \\[0.7mm]
e(\tfrac{1}{7}) - e(\tfrac{6}{7}) & e(\tfrac{4}{7}) - e(\tfrac{3}{7}) & e(\tfrac{2}{7}) - e(\tfrac{5}{7})
\end{pmatrix}
\right\rangle.
\end{equation}
Since the latter matrix has real entries, the presentation is invariant under complex conjugation;
it also satisfies the restricted rationality condition.
(Note that omitting $D(\tfrac{1}{3}, \tfrac{1}{3}, \tfrac{1}{3})$ yields a group which projects isomorphically onto $\stgroup{E(168)}$.)

\subsection{Standard and nonstandard extensions: overview}
\label{section: standard and nonstandard extensions}

At this point, we have completed the classification of finite subgroups of $\SU(3)$ containing $\mu_3$ and satisfying the restricted rationality condition, and it remains to classify the enlargements of these to subgroups to~$J(\SU(3))$.
Before embarking on the specific calculations, we read off some general features of the situation
and set some terminology.

\begin{definition}
Let $H$ be a subgroup of $\SU(3)$ which is stable under complex conjugation. (See Remark~\ref{change of conjugation action} for the changes that need to be made to handle more general $H$.)
Write ~$\pm H$ for the group $\langle H, -1\rangle$.
For any subgroup $\tilde{H}$ of $J(\SU(3))$ with $H = \tilde{H} \cap \SU(3)$,
we have an exact sequence
\begin{equation} \label{eq:exact sequence for subgroup U1}
1 \to \pm H \to \tilde{H} \to \cyc{2} \to 1.
\end{equation}
One candidate for $\tilde{H}$ is the group $\langle H, J\rangle$;
we denote it by $J(H)$ and call it the \emph{standard extension} by $H$, and characterize any other $\tilde{H}$ as a
\emph{nonstandard extension}.

Let $N_H$ denote the normalizer of $H$ in $\SU(3)$ (which is itself stable under complex conjugation); 
then any overgroup of $H$ must have the form
$\langle H, Jg \rangle$ for some $g \in \Unitary(3)$ whose image in $\SU(3)$ belongs to $N_H$
and satisfies $JgJg \in \pm H$ (so that $[\tilde{H}:H] \leq 2$). This last condition is equivalent to requiring that the image of $g$ in $N_H/H$ maps to its inverse under complex conjugation; let $I_H$ be the set of elements of $N_H/H$ with this property.

Conjugation by $h \in N_H$ takes $Jg$ to 
\[
h^{-1}(Jg)h = J \overline{h}^{-1} g h,
\]
and in particular $N_H/H$ acts on $I_H$. 
We thus have a bijection between the set of extensions of $H$ and the quotient
set of $I_H$ by the action of $\langle N_H/H \rangle$ where $h$ acts as $g \mapsto \overline{h}^{-1} g h$.
(It is not necessary to record the conjugation action of $J$: if $g \in N_H$ represents a class in $I_H$,
then the actions of $J$ and $g^{-1}$ on this class coincide.)
\end{definition}

\begin{definition}
Consider the nonstandard extension defined by the class in $I_H$ of some $g \in N_H$. We further classify the extension according to the following criteria.

\begin{itemize}
\item
A \emph{split} nonstandard extension, denoted $J_s(H)$, is one where
$\overline{g} g = \overline{h} h$ for some $h \in H$.
Such an extension is still a semidirect product by $\cyc 2$, but via a different action.

\item
A \emph{nonsplit} nonstandard extension, denoted $J_n(H)$, is one corresponding to the class in $I_H$
of some $g \in N_H$ such that $\overline{g} g \neq \overline{h} h$ for any $h \in H$.
Such an extension is not a semidirect product by $\cyc 2$.
\end{itemize}

While this is not clear \emph{a priori}, it will turn out that every group $H$ we consider has at
most one nonstandard extension of each type. Consequently, the notations $J_s(H)$ and $J_n(H)$ will be completely unambiguous.
\end{definition}

\begin{lemma}
Let $N_H^0$ denote the identity component of $N_H$, which centralizes $H$ since $H$ is finite.
Let $N'_H \colonequals N_H/\langle N_H^0, H \rangle$ be the component group of $N_H/H$.
Then the action of $N_H/H$ on $I_H$ (where $h$ acts as $g \mapsto \overline{h}^{-1} g h$) induces an action of 
$N'_H$ on the set of connected components of $I_H$.
\end{lemma}
\begin{proof}
Note that the action in question has only finitely many orbits: given $H$ there are finitely many isomorphism classes of groups $\tilde{H}$ as in \eqref{eq:exact sequence for subgroup U1}, and each occurs for finitely many conjugacy classes of subgroups of $J(\SU(3))$, by \cite[Remark~3.3]{FKRS12}.
If we restrict the action from $N_H/H$ to its identity component $\langle N_H^0, H \rangle/H$,
we still have only finitely many orbits (because this subgroup is of finite index), but now the orbits are connected and closed (because they are images of a continuous map between compact topological spaces); these orbits must thus be entire connected components of $I_H$. This yields the desired result.
\end{proof}

\begin{remark} \label{extension analysis}
Suppose that $N_H^0$ is contained in the diagonal torus and centralizes not only~$H$, but all of 
$N_H$ (this will only fail for some abelian groups).
Then $I_H$ is the image of a set of connected components of $N_H$: if $g \in N_H$, $h \in N_H^0$,
then $\overline{g} g = \overline{gh} gh$.
As a consequence, the set of connected components of $I_H$ is in bijection with the set $I'_H$ of elements of $N'_H$
which map to their inverses under complex conjugation.

The computation of the orbits of $I'_H$ under the action of $N'_H$ will be further facilitated by the
following general observations. (In this discussion, we do not differentiate between split and nonsplit nonstandard extensions; we make this distinction by inspection in each individual case.)
\begin{itemize}

\item
If $N'_H$ is trivial, then $H$ admits no nonstandard extension.

\item
If $N'_H \simeq \cyc 2$, then $I'_H= N'_H$, the action of $N'_H$ on $I'_H$ is trivial, and so~$H$ admits a nonstandard extension.

\item
Suppose that $N'_H$ has a normal subgroup $N''_H$ stable under complex conjugation,
such that $N'_H/N''_H$ is cyclic of odd prime order.
Let $I''_H$ be the set of elements of $N''_H$ which map to their inverses under complex conjugation;
then we have a canonical bijection $I'_H/N'_H \to I''_H/N''_H$
(regardless of the action of complex conjugation on $N'_H$).
In particular, if $N'_H$ is of odd order (e.g., $\cyc 3$), then $I'_H/N'_H$ is a singleton set
and so~$H$ admits no nonstandard extension.

\item
The set of orbits of $N'_H$ acting on $I'_H$ is invariant under twisting the action of complex conjugation by an element of $N'_H$ (compare Remark~\ref{change of conjugation action} below).
In particular, if $N'_H \simeq \sym 3$, then $N'_H$ has only inner automorphisms, so 
by Remark~\ref{change of conjugation action} the number of extensions does not depend on the action of complex conjugation on $N'_H$. For the trivial action, we see that $I'_H$ is the set of transpositions with the action of $N'_H$ by conjugation, so $H$ admits one nonstandard extension.

\end{itemize}
\end{remark}

\begin{remark} \label{change of conjugation action}
Without the assumption that $H$ is stable under complex conjugation
(but still assuming that $N_H^0$ is contained in the diagonal and centralizes all of $N_H$),
the analysis of Remark~\ref{extension analysis}
remains valid if we replace $J$ with $Jg$ for some $g \in \SU(3)$
for which $Jg$ normalizes $H$, and define the action of $N_H/H$ on $I_H$ by replacing complex conjugation with
the automorphism induced by $Jg$ (that is, complex conjugation composed with the ordinary conjugation by $g$). In particular, the number of extensions of $H$ depends only on
the group $N'_H$ and the element represented by complex conjugation in the outer automorphism group of $N'_H$;
this observation is useful even if $H$ is stable under complex conjugation, as it saves us the trouble of having to completely pin down the action of complex conjugation on $N'_H$.
\end{remark}

\subsection{Standard and nonstandard extensions: computations}

We now recapitulate the classification of finite subgroups $H$ of $\SU(3)$ containing $\mu_3$ and satisfying the restricted rationality condition, determining the standard and nonstandard extensions in each case.

\begin{proposition}
Of the $23$ groups listed in Table~\ref{table: abelian subgroups}, each admits a standard extension; the $15$ groups
\begin{gather*}
\stgroup[A(1, 4\phi)_2]{A(1,4)_2},\ \stgroup[A(1, 8\phi)_1]{A(1,8)_1},\ \stgroup[A(1, 8\phi)_2]{A(1,8)_2},\ \stgroup[A(1, 12\phi)]{A(1,12)},\ 
\stgroup[A(2, 2\phi)]{A(2,2)},\ \stgroup[A(2, 4\phi)]{A(2,4)},\ \stgroup[A(2, 6\phi)]{A(2,6)}, \\
\stgroup[A(3, 1\phi)]{A(3,1)},\ \stgroup[A(3, 2\phi)]{A(3,2)},\ 
\stgroup[A(3, 3\phi)]{A(3,3)},\ \stgroup[A(3, 4\phi)]{A(3,4)},\ \stgroup[A(3, 6\phi)]{A(3,6)},\ 
\stgroup[A(4, 4\phi)]{A(4,4)},\ \stgroup[A(6, 2\phi)]{A(6,2)},\ \stgroup[A(6, 6\phi)]{A(6,6)}
\end{gather*}
admit split nonstandard extensions; and the $9$ groups
\[
\stgroup[A(1, 2\phi)]{A(1,2)},\, \stgroup[A(1, 4\phi)_1]{A(1,4)_1},\, \stgroup[A(1, 4\phi)_2]{A(1,4)_2)},\, \stgroup[A(1, 6\phi)_1]{A(1,6)_1},\, \stgroup[A(1, 12\phi)]{A(1,12)},\, \stgroup[A(2, 4\phi)]{A(2,4)},\, \stgroup[A(3, 2\phi)]{A(3,2)},\, \stgroup[A(3, 4\phi)]{A(3,4)},\, \stgroup[A(3, 6\phi)]{A(3,6)}
\]
admit nonsplit nonstandard extensions.
\end{proposition}
\begin{proof}
These groups are all invariant under complex conjugation, and so admit standard extensions.
To identify nonstandard extensions, we cannot apply Remark~\ref{extension analysis} because $N_H^0$ does not centralize $N_H$ in most cases; we thus make a direct analysis. To begin with,
for $H= \stgroup[A(1,1\phi)]{A(1,1)}$, $N_H = N_H^0 = \SU(3)$ and $I_H$ consists of all symmetric matrices in $\SU(3)$; thus~$I_H$ is connected and $H$ admits no nonstandard extension.

For the groups
\[
H = \stgroup[A(1, 2\phi)]{A(1,2)},\, \stgroup[A(1, 3\phi)]{A(1,3)},\, \stgroup[A(1,4\phi)_1]{A(1,4)_1},\, \stgroup[A(1,6\phi)_1]{A(1,6)_1},
\]
we have $N_H = N_H^0 = \SU(3) \cap (\Unitary(1) \times \Unitary(2))$. 
For any $g \in N_H$, the top-left entry of  $\overline{g} g$ is~$1$;
consequently, if $\overline{g}g \in H$ then
$\overline{g} g \in \{D(0,0,0), D(0,\tfrac{1}{2}, \tfrac{1}{2})\} \subseteq \stgroup[A(1,2\phi)]{A(1,2)}$ and each option accounts for a single connected component of $I_H$.
For $H = \stgroup[A(1,3\phi)]{A(1,3)}$, the second component does not occur because $D(0, \tfrac{1}{2}, \tfrac{1}{2}) \notin H$;
consequently, $H$ admits no nonstandard extension.
In the other three cases, $I_H$ contains a second connected component, so $H$ admits a nonsplit nonstandard extension corresponding to the class of 
\begin{equation} \label{eq:nonsplit extension generator1}
g = \begin{pmatrix}
1 & 0 & 0 \\
0 & 0 & 1 \\
0 & -1 & 0
\end{pmatrix}.
\end{equation}
In the remaining cases, $N_H^0$ is the diagonal torus, and $N_H$ is generated by $N_H^0$ plus a group of symmetries contained in $\sym 3$ and isomorphic to $N_H/H$. 
Note that the entire image of $N_H$ in~$I_H$ is a single orbit, corresponding to the standard extension;
moreover, the image in $I_H$ of an element of $N_H$ in the coset of an order-$3$ symmetry in $\sym 3$
cannot map to its inverse under complex conjugation.
Consequently, if the group of symmetries is either trivial or equal to~$\alt 3$, then $H$ admits no nonstandard extensions; this covers the cases $H = \stgroup[A(1,6\phi)_2]{A(1,6)_2}, \stgroup[A(1, 7\phi)]{A(1,7)}$.
In all other cases, every nonstandard extension arises from a matrix of the form
\[
g = \begin{pmatrix}
-a & 0 & 0 \\
0 & 0 & -c \\
0 & -b & 0
\end{pmatrix},
\]
where $abc=1$; this matrix satisfies
\[
\overline{g} g = \begin{pmatrix}
1 & 0 & 0 \\
0 & b\overline{c}& 0 \\
0 & 0 & c\overline{b}
\end{pmatrix}.
\]
It follows that the connected components of $I_H$ 
are in bijection with the quotient of the group $H \cap \{D(0,u,-u): u \in \Q\}$
by the image of $H$ under the map $D(u,v,w) \mapsto D(0,v-w,w-v)$.
In particular, we always get a split nonstandard extension by taking $g$ to be the matrix $T_1$ of 
\eqref{eq:type B conjugation matrices};
in addition, we obtain nonsplit nonstandard extensions for
\[
H = \stgroup[A(3, 2\phi)]{A(3,2)},\, \stgroup[A(3, 4\phi)]{A(3,4)},\, \stgroup[A(3, 6\phi)]{A(3,6)}. 
\]
by taking $g$ as in \eqref{eq:nonsplit extension generator1}, so that $u = \tfrac{1}{2}$, and for 
\[
H = \stgroup[A(1,4\phi)_2]{A(1,4)_2},\, \stgroup[A(1, 12\phi)]{A(1,12)},\, \stgroup[A(2, 4\phi)]{A(2,4)},
\]
by taking
\[
g = \begin{pmatrix}
i & 0 & 0 \\
0 & 0 & i \\
0 & 1 & 0
\end{pmatrix},
\]
so that $u = \tfrac{1}{4}$.
\end{proof}

\begin{proposition}\label{extensions of type B}
Of the $18$ groups appearing in Table~\ref{table: dihedral subgroups}, each admits a standard extension;
the $12$ groups 
\begin{gather*}
\stgroup[B(1, 4\phi)_2]{B(1,4)_2},\ \stgroup[B(1, 12\phi)]{B(1,12)},\
\stgroup[B(2, 4\phi)]{B(2,4)},\ \stgroup[B(3, 2\phi)]{B(3,2)},\ \stgroup[B(3, 4\phi)]{B(3,4)},\ \stgroup[B(3, 6\phi)]{B(3,6)}, \\
\stgroup[B(1, 4\phi; 2)_2]{B(1,4;2)_2},\ \stgroup[B(1, 12\phi; 2)]{B(1,12;2)},\
\stgroup[B(3, 2\phi; 2)]{B(3,2;2)},\
\stgroup[B(3, 6\phi; 2)]{B(3,6;2)},\
\stgroup[B(2, 4\phi; 4)]{B(2,4;4)},\
\stgroup[B(3, 4\phi; 4)]{B(3,4;4)}
\end{gather*}
admit split nonstandard extensions; and none admits a nonsplit nonstandard extension.
\end{proposition}
\begin{proof} 
These groups are invariant under complex conjugation and admit standard extensions.
For each group $H = B(m,n\phi; k)_*$ in the list,
$N_H^0$ is the torus $\{D(2u,-u,-u): u \in \Q\}$; this torus is the image of $\Unitary(1) \times \{1\}$ under $\pi$.
Moreover, $N_H$ is contained in $\SU(3) \cap (\Unitary(1) \times \Unitary(2))$ and thus centralized by $N_H^0$;
we may therefore apply Remark~\ref{extension analysis}.

We first treat the exceptional cases 
\[
H = \stgroup[B(2,4\phi)]{B(2,4)},\, \stgroup[B(1,4\phi;2)_2]{B(1,4;2)_2},\, \stgroup[B(1,12\phi;2)]{B(1,12;2)}.
\]
These groups all project to the subgroup $\langle \mathbf{i}, \mathbf{k} \rangle$ of $\SU(2)$, which has 
normalizer $2O$; it follows that in all three cases, $N_H = \langle N_H^0, 2O \rangle$ and $N'_H \simeq \sym 3$.
By Remark~\ref{extension analysis}, this implies the existence of a split nonstandard extension
corresponding to the class in $I_H$ of $D(\tfrac{1}{4}, \tfrac{1}{4}, \tfrac{1}{2})$.

In the other cases, every element of $N_H$ normalizes $A(m,n\phi)_*$.
To see this, note first that when $m$ is divisible by $3$,
the Sylow $3$-subgroup of $H$ is unique and not central,
so $N_H$ normalizes both the Sylow $3$-subgroup of $H$ and its centralizer in $H$
(namely $A(m,n\phi)_*$).
In the remaining cases 
\[
H =  \stgroup[B(1,4\phi)_2]{B(1,4)_2},\, \stgroup[B(1,8\phi)_1]{B(1,8)_1},\, \stgroup[B(1, 12\phi)]{B(1,12)},\, \stgroup[B(4,4\phi)]{B(4,4)},\, \stgroup[B(2,4\phi;4)]{B(2,4;4)},
\]
one may check by direct calculation that $A(m,n\phi)_*$ is a characteristic subgroup of $H$.
(Note that this does not apply to the alternate presentations $\stgroup[B(2,2\phi)]{B(1,4)_2},\stgroup[B(2,6\phi)]{B(1,12)}, \stgroup[B(1,8\phi)_2]{B(2,4;4)}$
of the groups $\stgroup[B(1,4\phi)_2]{B(1,4)_2}, \stgroup[B(1,12\phi)]{B(1,12)}, \stgroup[B(2,4\phi;4)]{B(2,4;4)}$; this explains our choice of labels in these cases.) 

By the previous paragraph, $N_H$ is contained in the subgroup of $\SU(3)$
generated by the diagonal torus and the symmetries listed in Table~\ref{table: abelian subgroups}.
(Note that all cases where the connected part of the normalizer of $A(m,n\phi)_*$ in $\SU(3)$ is larger than
the diagonal torus fail to contribute to Table~\ref{table: dihedral subgroups}, because the resulting groups are all abelian.) Moreover, symmetries in $\sym 3$ but not in $\cyc 2$ fail to normalize $T_j$ up to scalars. 
It follows that $N'_H$ is generated by the intersection of $N_H$ with the diagonal torus.

An element $g = D(u,v,w) \in \SU(3)$ belongs to $N_H$ if and only if $H$ contains
\[
T_j^{-1} g^{-1} T_j g = D(-u, -w, -v) D(u,v,w) = D(0, v-w, w-v);
\]
moreover, this element is trivial in $H$ if and only if $g \in N_H^0$.
Consequently, $N'_H$ is equal to the intersection of $H$ with the torus $\{D(0,u,-u): u \in \Q\}$
modulo the image of $A(m,n\phi)_*$ under $D(u,v,w) \mapsto D(0,v-w,w-v)$.

If $A(m,n\phi)_*$ has symmetries by $\sym 3$, then $N'_H$ is trivial because $D(0,u,-u)$ is the image of $D(-u,u,0)$;
consequently, the groups
\[
H = \stgroup[B(3,1\phi)]{B(3,1)},\, \stgroup[B(3,3\phi)]{B(3,3)},\, \stgroup[B(4,4\phi)]{B(4,4)},\, \stgroup[B(6,2\phi)]{B(6,2)},\, \stgroup[B(6,6\phi)]{B(6,6)},
\]
admit no nonstandard extension.
For $H = \stgroup[B(1,8\phi)_1]{B(1,8)_1}$, $N'_H$ is trivial because the image of $D(\tfrac{1}{2},\tfrac{1}{8}, \tfrac{3}{8}) \in H$ is $D(0, \tfrac{3}{4}, \tfrac{1}{4})$, which generates the intersection of $\stgroup[A(1,8\phi)_1]{A(1,8)_1}$ with the torus $\{D(0,u,-u): u \in \Q\}$. Hence $H$ admits no nonstandard extension.

The groups
\[
H = \stgroup[B(1,4\phi)_2]{B(1,4)_2},\, \stgroup[B(1,12\phi)]{B(1,12)},\, \stgroup[B(2,4\phi;4)]{B(2,4;4)}
\]
all contain $D(0,\tfrac{1}{4},\tfrac{3}{4})$ (whose image is $D(0,\tfrac{1}{2}, \tfrac{1}{2})$) but not $D(0, \tfrac{1}{8}, \tfrac{7}{8})$,
and $N'_H \simeq \cyc 2$ generated by the image of $D(\tfrac{1}{4},\tfrac{1}{4}, \tfrac{1}{2})$; hence $H$ admits a split nonstandard extension.
The groups
\[
H = \stgroup[B(3,2\phi)]{B(3,2)},\, \stgroup[B(3,4\phi)]{B(3,4)},\, \stgroup[B(3,6\phi)]{B(3,6)},\, \stgroup[B(3,2\phi;2)]{B(3,2;2)},\, \stgroup[B(3,6\phi;2)]{B(3,6;2)},\, \stgroup[B(3,4\phi;4)]{B(3,4;4)},
\]
all contain $D(0,\tfrac{1}{2},\tfrac{1}{2})$ but not $D(0, \tfrac{1}{4}, \tfrac{3}{4})$, so 
$N'_H \simeq \cyc 2$ generated by the image of $D(\tfrac{1}{2},0, \tfrac{1}{2})$; hence $H$ admits a split nonstandard extension.
\end{proof}

\begin{proposition}\label{tetrahedral extensions}
Of the $6$ groups listed in Proposition~\ref{binary tetrahedral}
and Proposition~\ref{binary octahedral}, all except $\stgroup[B(T,1\phi;1)]{B(T,1;1)}$ admit standard extensions; the $4$ groups
\[
\stgroup[B(T,1\phi)]{B(T,1)},\quad \stgroup[B(T,2\phi)]{B(T,2)},\quad \stgroup[B(T,3\phi)]{B(T,3)},\quad \stgroup[B(T,1\phi;1)]{B(T,1;1)}
\]
admit split nonstandard extensions; and none admits a nonsplit nonstandard extension.
\end{proposition}
\begin{proof} 
As in Proposition~\ref{extensions of type B},
$N_H^0$ is the image of $\Unitary(1) \times \{1\}$ under $\Unitary(1) \times \SU(2) \to \SU(3)$
and $N_H$ is contained in $\SU(3) \cap (\Unitary(1) \times \Unitary(2))$;
we may thus apply Remark~\ref{extension analysis}.

For each of the groups 
\[
H=\stgroup[B(T,1\phi)]{B(T,1)},\, \stgroup[B(T,2\phi)]{B(T,2)},\, \stgroup[B(T,3\phi)]{B(T,3)}
\]
listed in Proposition~\ref{binary tetrahedral}, 
$H$ maps to the subgroup $2T$ of $\SU(2)$, so $N_H = \langle N_H^0, 2O \rangle$ and $N'_H \simeq \cyc{2}$.
Consequently, each of these $3$ groups admits a split nonstandard extension
corresponding to the class of $D(\tfrac{1}{4}, \tfrac{1}{4}, \tfrac{1}{2})$ in $I_H$.
By the same token, for each of the groups 
\[
H = \stgroup[B(O,1\phi)]{B(O,1)},\, \stgroup[B(O,2\phi)]{B(O,2)}
\]
listed in Proposition~\ref{binary octahedral}, 
$H$ maps to the subgroup $2O$ of $\SU(2)$, so $N_H = \langle N_H^0, 2O \rangle$ and $N'_H$ is trivial.
Consequently, neither of these $2$ groups admits a nonstandard extension.

The case $H = \stgroup[B(T,1\phi;1)]{B(T,1;1)}$ is anomalous; see Remark~\ref{remark: BT11} below.
\end{proof}

\begin{remark} \label{remark: BT11}
In Proposition~\ref{tetrahedral extensions}, the case $H = \stgroup[B(T,1\phi;1)]{B(T,1;1)}$ requires special dispensation because the
given presentation of $H$ is not stable under complex conjugation, so there is no standard extension.
(The action of $J$ is nontrivial on $\cyc 3$ but trivial on $2T/Q$, so the two possible fiber products are interchanged by conjugation.)

That aside, $N_H^0$ is the image of
$\Unitary(1) \times \{1\}$ under the map $\Unitary(1) \times \SU(2) \to \SU(3)$
and~$N_H$ is contained in $\SU(3) \cap (\Unitary(1) \times \Unitary(2))$, so
Remark~\ref{change of conjugation action} applies with $g = D(\tfrac{1}{4}, \tfrac{1}{4}, \tfrac{1}{2})$.
However, in this case, the projection of $N_H$ to $\SU(2)$ does not hit the nontrivial coset of $T$ in~$O$,
so~$N'_H$ is trivial. Consequently, $H$ admits a unique extension; since $(Jg)^2$ is the identity,
we categorize this as a split nonstandard extension of $H$.
\end{remark}

\begin{proposition}\label{type C extensions}
Of the $7$ groups appearing in Proposition~\ref{A3 extensions},
each admits a standard extension;
the $5$ groups
\[
\stgroup[C(2,2\phi)]{C(2,2)},\quad \stgroup[C(3,3\phi)]{C(3,3)},\quad \stgroup[C(4,4\phi)]{C(4,4)},\quad \stgroup[C(6,2\phi)]{C(6,2)},\quad \stgroup[C(6,6\phi)]{C(6,6)}
\]
admit split nonstandard extensions; and none admits a nonsplit nonstandard extension.
\end{proposition}
\begin{proof} 
For each group $H = C(m,n\phi)$ in the list,
$A(m,n\phi)$ is centralized only by the diagonal torus, so the centralizer of $H$ in $\SU(3)$
is $\mu_3$. In particular, $N_H^0$ is trivial,
so Remark~\ref{extension analysis} applies with $N'_H = N_H/H$.

To analyze the case $H = \stgroup[C(3,1\phi)]{C(3,1)}$, we first note that $H$ coincides with the subgroup $\langle g_1, g_3^3, S = g_2^{-1} g_1 g_2 \rangle$ of $\stgroup[E(216\phi)]{E(216)}$. By checking that $g_2^{-1} S g_2 = g_1$ and $g_1^{-1} S g_1 = g_3^3 S$, we see that 
$N_H = \stgroup[E(216\phi)]{E(216)}$
and $N'_H \simeq 2T$; since $2T/Q \simeq \cyc 3$,
we may reduce consideration from $N'_H$ to $N''_H \simeq Q$
via Remark~\ref{extension analysis}. If we normalize the latter identification
so that $g_2$ corresponds to $\mathbf{i}$ and $g_3 g_2 g_3^{-1}$ maps to $\mathbf{j}$, then complex conjugation acts on $N''_H$ by
\[
\mathbf{i} \mapsto -\mathbf{i}, \qquad \mathbf{j} \mapsto \mathbf{k}, \qquad \mathbf{k} \mapsto \mathbf{j};
\]
hence $I''_H = \{\pm 1, \pm \mathbf{i}\}$ and the action is transitive.
(More precisely, $\mathbf{j}$ interchanges $\pm 1$ with $\pm \mathbf{i}$
while $\mathbf{k}$ interchanges $\pm 1$ with $\mp \mathbf{i}$.)
We conclude that $H$ admits no nonstandard extension.

In the remaining cases, $A(m,n\phi)$ is normalized by $N_H$. To see this, note that in the case
$H = \stgroup[C(1,7\phi)]{C(1,7)}$, the group $\stgroup[A(1,7\phi)]{A(1,7)}$ is the unique Sylow $7$-subgroup; for $H = \stgroup[C(3,3\phi)]{C(3,3)}$, we see that $\stgroup[A(3,1\phi)]{A(3,1)}$ is the unique index-$9$ normal subgroup (it being the inverse image of the center of $\stgroup{A(3,1)}$) and not central,
so its centralizer in $H$ equals $A(m,n\phi)$; while for 
\[
H = \stgroup[C(2,2\phi)]{C(2,2)},\, \stgroup[C(4,4\phi)]{C(4,4)},\, \stgroup[C(6,2\phi)]{C(6,2)},\, \stgroup[C(6,6\phi)]{C(6,6)},
\]
the Sylow $2$-subgroup of $H$ is unique and not central, so its centralizer in $H$ equals $A(m,n\phi)$.
It follows that $N_H$ is contained in the subgroup of $\SU(3)$ generated by the diagonal torus,~$S$, and~$T_1$.

An element $g = D(u,v,w) \in \SU(3)$ belongs to $N_H$ if and only if $H$ contains
\[
S^{-1} g^{-1} S g = D(-v, -w, -u) D(u,v,w) = D(u-v, v-w, w-u);
\]
note that any element of the diagonal torus can be represented in this fashion.
Moreover, $D(u,v,w) \mapsto D(u-v, v-w, w-u)$ defines a homomorphism from $A(m,n\phi)$ to itself
with kernel $\langle D(\tfrac{1}{3}, \tfrac{1}{3}, \tfrac{1}{3} \rangle \rangle$; hence
$N'_H$ consists of a copy of $\cyc 3$, times a copy of $\cyc 2$ generated by~$T_1$ in all cases
except $H = \stgroup[C(1,7\phi)]{C(1,7)}$.
Per Remark~\ref{extension analysis}, the factor of $\cyc 3$ may be ignored;
consequently, $H$ admits a standard extension in all of the cases except $H = \stgroup[C(1,7\phi)]{C(1,7)}$.
These extensions are all represented by the class of $T_1$ in $I_H$, and so are split.
\end{proof}

\begin{proposition} \label{type D extensions}
Of the $6$ groups appearing in Proposition~\ref{sym3 groups}, each admits a standard extension,
and none admits a (split or nonsplit) nonstandard extension.
\end{proposition}
\begin{proof} 
For each group $H = D(m,n\phi)$ in the list,
$A(m,n\phi)$ is centralized only by the diagonal torus, so the centralizer of $H$ in $\SU(3)$
is $\mu_3$. In particular, $N_H^0$ is trivial,
so Remark~\ref{extension analysis} applies with $N'_H = N_H/H$.

For $H = \stgroup[D(3,1\phi)]{D(3,1)}$, we again use the presentation of the group $\stgroup[E(216\phi)]{E(216)}$ described in
\S\ref{section:solvable exceptional}. In this case, 
$N_H = \stgroup[E(216\phi)]{E(216)}$, $N'_H \simeq BT/\{\pm 1\} \simeq \alt 4$, and 
$\alt 4/(Q/\{\pm 1\}) \simeq \cyc 3$, so we 
may reduce consideration from $N'_H$ to $N''_H \simeq Q/\{\pm 1\} = \{1,\overline{\mathbf{i}},\overline{\mathbf{j}},\overline{\mathbf{k}}\}$ via Remark~\ref{extension analysis}.
With notation as in Proposition~\ref{type C extensions},
the action of complex conjugation on $N''_H$ fixes $\overline{\mathbf{i}}$ 
and interchanges $\overline{\mathbf{j}}$ with~$\overline{\mathbf{k}}$;
consequently, $I''_H = \{1,\overline{\mathbf{i}}\}$ and the action on this set is again transitive.
It follows that~$H$ admits no nonstandard extension.

For the groups
\[
H = \stgroup[D(2,2\phi)]{D(2,2)},\, \stgroup[D(3,3\phi)]{D(3,3)},\, \stgroup[D(4,4\phi)]{D(4,4)},\, \stgroup[D(6,2\phi)]{D(6,2)},\, \stgroup[D(6,6\phi)]{D(6,6)},
\]
the unique index-$2$ subgroup of $H$ is $C(m,n\phi)$, so from Proposition~\ref{type C extensions}
we see that $N_H$ normalizes both $A(m,n\phi)$ and $C(m,n\phi)$,
and that $N'_H \simeq \cyc 3$. By Remark~\ref{extension analysis}, $H$ admits no nonstandard extension.
\end{proof}

\begin{proposition}
Of the $4$ groups
\[
\stgroup[E(36\phi)]{E(36)},\quad \stgroup[E(72\phi)]{E(72)},\quad \stgroup[E(216\phi)]{E(216)},\quad \stgroup[E(168\phi)]{E(168)},
\]
each admits a standard extension;
none admits a split standard extension;
and the group $\stgroup[E(36\phi)]{E(36)}$ admits a nonsplit standard extension.
\end{proposition}
\begin{proof}
For $H = \stgroup[E(36\phi)]{E(36)}, \stgroup[E(72\phi)]{E(72)}, \stgroup[E(216\phi)]{E(216)}$, $N_H^0$ is trivial and so Remark~\ref{extension analysis} applies with $N'_H = N_H/H$.
For $H = \stgroup[E(36\phi)]{E(36)}$, $N_H = \stgroup[E(72\phi)]{E(72)}$ and so $N'_H \simeq \cyc{2}$; thus $H$ admits a nonstandard extension
corresponding to $g_3 g_2 g_3^{-1}$,
which turns out to be nonsplit.
For $H = \stgroup[E(72\phi)]{E(72)}$, $N_H = \stgroup[E(216\phi)]{E(216)}$ and so $N'_H \simeq \cyc{3}$; consequently, $H$ admits no nonstandard extension.
For $H = \stgroup[E(216\phi)]{E(216)}$, $N'_H$ is trivial; thus $H$ admits no nonstandard extension.

For $H = \stgroup[E(168\phi)]{E(168)}$, $N_H^0$ is trivial (so Remark~\ref{extension analysis} applies with $N'_H = N_H/H$)
and $N_H = H$, so $H$ admits only its standard extension.
\end{proof}

\subsection{Maximal subgroups}

We now go back through the classification to identify maximal subgroups. Recall that dropping $\phi$ from a group label corresponds to passing from $\SU(3)$ to $\PSU(3)$.
(Note that in the \LMFDB{}, the entry for each Sato--Tate group includes the maximal subgroups and minimal supergroups of that group with respect to inclusions of finite index.)

\begin{proposition} \label{maximal subgroups}
Among the possible component groups of a group satisfying the Sato--Tate axioms with connected part
$\Unitary(1)_3$, the maximal subgroups that occur are 
\begin{gather*}
\stgroup{J(B(3,4;4))},\ \stgroup{J_s(B(1,12))},\ \stgroup{J(B(3,4))},\ \stgroup{J_s(B(3,4))},\ \stgroup{J(B(T,3))},\ \stgroup{J_s(B(T,3))}, \\
\stgroup{J(B(O,1))},\ \stgroup{J(B(O,2))},\ \stgroup{J(D(4,4))},\ \stgroup{J(D(6,6))},\ \stgroup{J(E(216))},\ \stgroup{J(E(168))}.
\end{gather*}
\end{proposition}
\begin{proof}
As may be confirmed by a simple \Sage{} calculation (see \cite{FKS21}), in most cases our presentations have been chosen so that 
there is an inclusion into one of the listed subgroups, without any conjugation required. The remaining cases are handled as follows.
\begin{itemize}
\item
The group $\stgroup{B(1,8)_1}$ is conjugate to a subgroup of $\stgroup{B(O,1)}$ containing $\stgroup{A(1,8)_1}$.
This induces embeddings of $\stgroup{B(1,8)_1}$, $\stgroup{J(B(1,8)_1)}$, $\stgroup{J_s(A(1,8)_1)}$ into $\stgroup{J(B(O,1))}$.

\item
The group $\stgroup{B(3,4;4)}$ is conjugate to a subgroup of $\stgroup{B(O, 2)}$.
This induces an embedding of $\stgroup{J_s(B(3,4;4))}$ into $\stgroup{J(B(O, 2))}$. (However, $\stgroup{J(B(3,4;4))}$ does not embed into $\stgroup{J(B(O,2))}$; it stands as a maximal subgroup.)

\item
The group $\stgroup{B(2,4;4)}$ is conjugate to a subgroup of $\stgroup{B(4,4)}$ by Remark~\ref{B redundancy}.
This induces embeddings of $\stgroup{B(2,4;4)}$, $\stgroup{J(B(2,4;4))}$, $\stgroup{J_s(B(2,4;4))}$ into $\stgroup{J(D(4,4))}$.

\item
Combining the previous point with the isomorphism
$B(1,8)_2 \simeq \stgroup{B(2,4;4)}$ from the proof of Proposition~\ref{proposition: dihedral subgroups},
we obtain embeddings of $\stgroup{A(1,8)_2}$, $\stgroup{J(A(1,8)_2)}$, $\stgroup{J_s(A(1, 8)_2)}$ into $\stgroup{J(D(4,4))}$.
(These groups also embed into $\stgroup{J(B(O,2))}$.)

\item
The group $\stgroup{B(3,6;2)}$ is conjugate to a subgroup of $\stgroup{B(6,6)}$ by Remark~\ref{B redundancy}.
This induces embeddings of $\stgroup{B(3,6;2)}$, $\stgroup{J(B(3,6;2))}$, $\stgroup{J_s(B(3, 6; 2))}$ into $\stgroup{J(D(6,6))}$.

\item
The group $\stgroup{B(1,12)}$ is conjugate to a subgroup of $\stgroup{B(6,6)}$; this follows from the isomorphism
$\stgroup{B(1,12)} \simeq \stgroup[B(2,6)]{B(1,12)}$ given in the proof of Proposition~\ref{proposition: dihedral subgroups}.
This induces embeddings of $\stgroup{J_s(A(1, 12))}$, $\stgroup{J(B(1, 12))}$ into $\stgroup{J(D(6,6))}$.
(However, $\stgroup{J_s(B(1,12))}$ does not embed into $\stgroup{J(D(6,6))}$; it stands as a maximal subgroup.)

\end{itemize}
It remains to check that the list of maximal subgroups is irredundant; that is, no group in the list is conjugate to a subgroup of any other group in the list. This may be checked by using \Gap{} to identify conjugacy classes of subgroups of each listed group, then checking that
there is not even an abstract isomorphism between a representative of each class and any other listed group.
(This must be done at the level of subgroups of $J(\SU(3))$, as there are some spurious
isomorphisms among the resulting subgroups of $J(\PSU(3))$; for example, $\stgroup{J(B(3,4;4))}$ and $\stgroup{J_s(B(1,12))}$ 
are isomorphic as abstract groups, but $\stgroup[J(B(3,4\phi;4))]{J(B(3,4;4))}$ and $\stgroup[J_s(B(1,12\phi))]{J_s(B(1,12))}$ are not.)
\end{proof}

\subsection{Consistency checks}
\label{subsec:consistency checks}

Given the intricacies of the preceding classification, it is reasonable to ask for ways to corroborate
the result. We conclude this section with some discussion of these; see \cite{FKS21} for relevant code.

\begin{remark} \label{beukers-smyth}
Picking up on Remark~\ref{remark: multiplicative manin-mumford},
we describe an implementation in \Sage{} of the algorithm of Beukers and Smyth \cite{BS02} that verifies
Proposition~\ref{roots of unity problem}. (See also \cite[Algorithm 7.5, Theorem 7.6]{KKPR20} for a robust generalization of this algorithm.)

Suppose that $a,b,c$ are roots of unity such that $abc=1$ and $\left|a+b+c\right|^2 = n$ for some integer $n$.
We must then have $n \in \{0,\dots,9\}$ and $f_n(a,b) = 0$ for 
\[
f_n(x,y) = (x^2y + xy^2 + 1)(x+y+x^2y^2) - nx^2y^2.
\]
Note that $f_1$ factors as $(x+y)(x^2y+1)(xy^2+1)$, so its zeroes correspond precisely to the cases where $(a+b)(b+c)(c+a) = 0$; we may thus assume hereafter that
$n \neq 1$. 

Define the set
\[
S_n = \{f_n(s_1 x^e, s_2 y^e): s_1, s_2 \in \{\pm 1\}; e \in \{1,2\}; (s_1, s_2, e) \neq (1,1,1)\}.
\]
By \cite[Lemma~1]{BS02},
there exist $n \in \{0\} \cup \{2,\dots,9\}$
and $g \in S_n$ such that $g(a,b) = 0$.
For each choice of $n$ and $g$, we performed a resultant computation to eliminate $y$;
made a similar computation to pick out the roots of the resulting polynomial in $x$ which are roots of unity; then
for each option for $x$, substituted back into $f_n(x,y)$ and found the roots
of the resulting polynomial in $y$ which are roots of unity.
\end{remark}

\begin{remark}
For each listed subgroup $H$ of $\SU(3)$, it is easy to verify the restricted rationality condition directly;
we checked this in \Sage.
Conversely, we also checked in \Sage{} that there are no counterexamples $(a,b,c)$ to 
Proposition~\ref{roots of unity problem} for which $a,b,c$ generate a group of order dividing $2^4$,
$3^3$, $7^2$, or $2^3 \times 3^2 \times 7$.
\end{remark}

\begin{remark}
Proposition~\ref{roots of unity problem2} implies that every finite cyclic subgroup of $\SU(3)$ containing~$\mu_3$ and satisfying the restricted rationality condition has order dividing one of $7\phi, 8\phi, 12\phi$.
Given only this fact, one can check in \Sage{} that Table~\ref{table: abelian subgroups} is complete,
and that the lists in Proposition~\ref{A3 extensions} and Proposition~\ref{sym3 groups} are complete.
\end{remark}

\begin{remark}
To verify that there are no redundancies in the classification,
for each listed subgroup $H$ of $\SU(3) \rtimes \cyc 2$, we used \Gap{} 
to compute the isomorphism class of the image of~$H$ in $\PSU(3) \rtimes \cyc 2$;
in cases where this does not suffice to separate groups, we also computed the isomorphism class of $H$ itself.
This computation separates all groups except for the pairs
\[
\{\stgroup[A(1,4\phi)_1]{A(1,4)_1}, \stgroup[A(1,4\phi)_2]{A(1,4)_2}\},\quad 
\{\stgroup[A(1,6\phi)_1]{A(1,6)_1}, \stgroup[A(1,6\phi)_2]{A(1,6)_2}\},\quad 
\{\stgroup[A(1,8\phi)_1]{A(1,8)_1}, \stgroup[A(1,8\phi)_2]{A(1,8)_2}\},
\]
and the corresponding pairs of standard extensions.
Since these groups are all cyclic, they are easily distinguished based on their eigenvalues
(or by the computation of moments in \S\ref{sec:statistics}, which provides another proof of irredundancy).
\end{remark}

\begin{remark} \label{verify maximal subgroups}
Recall that according to Proposition~\ref{maximal subgroups}, there are $12$ maximal groups occurring in the classification.
We used \Gap{} to compute the isomorphism classes of each conjugacy class of subgroups of each
of the maximal groups. This confirms that every listed group is abstractly isomorphic to a subgroup of one of these groups, and conversely that every subgroup of one of these groups is abstractly isomorphic to one listed in the classification. 
\end{remark}

\section{Statistics on Sato--Tate distributions}
\label{sec:statistics}

Our primary purpose for computing Sato--Tate groups of abelian threefolds is to use the associated distributions to explain the observed statistical behavior of their $L$-polynomials. To this end, we describe some statistics on these distributions and how we computed them for the groups in the classification. For this discussion, we work with the longer list of $433$ groups rather than the list of $410$ specified in Theorem~\ref{T:ST result}. For short, we refer to the groups in this longer list as groups of the \emph{extended classification}.

\subsection{Group invariants}
\label{subsec:moment computations}

We first recall the definition of the two sets of invariants that we will primarily use to distinguish the groups in the extended classification. These are the \emph{simplex of moments} and the \emph{character norms}.

\begin{definition}
Let $a_1, a_2, a_3: \USp(6) \to \R$ be the functions computing the coefficients of $T^1, T^2, T^3$ in the characteristic polynomial of an element of $\USp(6)$. For $G$ a closed subgroup of $\USp(6)$ and $(e_1,e_2,e_3)$ a triple of nonnegative integers, let $M_{e_1,e_2,e_3}(G)$ denote the average value of $a_1^{e_1} a_2^{e_2} a_3^{e_3}$ on $G$, that is, if $\mu_G$ denotes the Haar measure of $G$, then
$$
M_{e_1,e_2,e_3}(G)=\int_{G}a_1(\gamma)^{e_1}a_2(\gamma)^{e_2}a_3(\gamma)^{e_3}\mu_G(\gamma)\,.
$$
This number can be interpreted as the multiplicity of the trivial representation within the representation
\[
(\C^6)^{\otimes e_1} \otimes (\wedge^2 \C^6)^{\otimes e_2} \otimes (\wedge^3 \C^6)^{\otimes e_3}
\]
of $G$, and is thus a nonnegative integer.

To convert this infinite collection of integers into a finite computable invariant, we must choose an appropriate truncation. For
$m$ a nonnegative integer, we define the \emph{$m$-simplex of moments} associated to $G$ as the collection of moments
$M_{e_1,e_2,e_3}(G)$ for all tuples $(e_1,e_2,e_3)$ with $e_1 + 2e_2 + 3e_3 \leq m$.
\end{definition}

Ordered triples $\lambda_1\geq \lambda_2\geq \lambda_3\geq 0$ of nonnegative integers (referred to as \emph{partitions} from now on) stand in bijection with the highest weights of irreducible representations of $\USp(6)$, and we will freely identify both sets. We will denote by $\chi_\lambda$ the irreducible character of $\USp(6)$ with highest weight associated to the partition $\lambda$. Following \cite{Shi16}, we make the following definition.

\begin{definition} \label{definition: m-diagonal}
For $G$ a closed subgroup of $\USp(6)$ and partitions $\lambda$ and $\mu$, let $N_{\lambda,\mu}(G)$ denote the average of $\chi_\lambda\cdot \chi_\mu$ on $G$, that is
$$
N_{\lambda,\mu}(G)=\int_G\chi_\lambda(\gamma)\chi_\mu(\gamma)\mu_G(\gamma)\,.
$$
For $m$ a nonnegative integer, we define the \emph{$m$-orthogonality matrix of characters} associated to~$G$ as the collection of averages $N_{\lambda,\mu}(G)$ for all subpartitions $\lambda,\mu$ of the rectangular partition $m \geq m \geq m$. By the \emph{$m$-diagonal of character norms} we refer to the collection of norms~$N_{\lambda,\lambda}(G)$ for all subpartitions $\lambda$ of $m \geq m \geq m $.
\end{definition}

The irreducible characters $\chi_\lambda$ can be expressed in terms of the functions $a_1, a_2, a_3$ by means of the Brauer--Klimyk formula (see \cite[\S4]{Shi16}). In Table \ref{Table: chars in terms of coef}, we have carried out this computation for the first few values of $\lambda$.

\begin{table}[ht]
\renewcommand{\arraystretch}{1.1}
\footnotesize
\begin{tabular}{|l|l|l|l|}
\hline
$\lambda$ & $\chi_\lambda$ & $\lambda$ & $\chi_\lambda$ \\
\hline&&&\\[-12pt]
$(0,0,0)$ & $1$ & $(3,0,0)$ & $-a_1^3 + 2a_1a_2 - a_3$ \\
$(1,0,0)$ & $-a_1$ & $(3,1,0)$ & $a_1^2a_2 - a_1^2 - a_1a_3 - a_2^2 + 2a_2$ \\
$(1,1,0)$ & $a_2-1$ & $(3,1,1)$ & $a_1^3 - a_1^2a_3 - 2a_1 + a_2a_3$ \\
$(1,1,1)$ & $-a_3+a_1$ & $(3,2,0)$ & $a_1^2a_3 - a_1a_2^2 + a_2a_3 - a_3$ \\
$(2,0,0)$ & $a_1^2-a_2$ & $(3,2,1)$ & $-2a_1^2a_2 + 2a_1^2 + a_1a_2a_3 + a_1a_3 - a_3^2$ \\
$(2,1,0)$ & $-a_1a_2+a_1+a_3$ &  $(3,2,2)$ & $-a_1^2a_3 - a_1a_2^2 + 4a_1a_2 + a_1a_3^2 - 2a_1 - a_2a_3$ \\
$(2,1,1)$ & $a_1a_3-a_1^2-a_2+1$ & $(3,3,0)$ & $a_1^2a_2 - 2a_1a_2a_3 + a_1a_3 + a_2^3 - 2a_2^2 + a_3^2$ \\
$(2,2,0)$ & $a_2^2-a_1a_3-a_2$ & $(3,3,1)$ & $-a_1^3 - a_1^2a_3 + 2a_1a_2^2 + a_1a_3^2 - a_2^2a_3 - a_2a_3 + a_3$ \\
$(2,2,1)$ & $-a_2a_3+2a_1a_2-a_1$ & $(3,3,2)$ & $2a_1^2a_2 - a_1^2 - 2a_1a_2a_3 - a_1a_3 - a_2^3 + 3a_2^2 + a_2a_3^2 - 2a_2$\\
$(2,2,2)$ & $a_3^2-a_2^2-a_1a_3+2a_2-1$ & $(3,3,3)$ & $a_1^2a_3 - 3a_1a_2^2 + 2a_1a_2 + a_1a_3^2 + 2a_2^2a_3 - 2a_2a_3 - a_3^3 + a_3$\\[1pt]
\hline
\end{tabular}
\medskip

\caption{Some irreducible characters of $\USp(6)$ in terms of coefficients of characteristic polynomials.}\label{Table: chars in terms of coef}
\end{table}

\begin{remark}
If $H \subseteq G$ is an inclusion of groups (not necessarily of finite index), then for any representation of $\USp(6)$,
the multiplicity of the trivial representation within the restriction to $G$ is less than or equal to the multiplicity within the restriction to $H$. This means in particular that
for any $(e_1,e_2,e_3)$ or $(\lambda, \mu)$ we have
\[
M_{e_1,e_2,e_3}(H) \geq M_{e_1,e_2,e_3}(G), \qquad N_{\lambda, \mu}(H) \geq N_{\lambda, \mu}(G).
\]
Conversely, if any one of these inequalities fails, then $H$ is not conjugate to a subgroup of $G$ within $\USp(6)$.
This can be used, for instance, to give another consistency check of the group inclusions in Proposition~\ref{maximal subgroups}.
\end{remark}

In the next sections we will explain the methods used to compute moments and character norms for the groups in the extended classification. Although the character norms can be recovered from the moments, we develop a (more efficient) method to compute them directly (except for the group $\stgroup[N(\Unitary(3))]{1.6.B.2.1a}$; see \S\ref{section: NU3}). 

\subsection{Averages over connected groups}\label{section: momentsconnected}
Let $G$ be one of the 14 connected Sato--Tate groups of the extended classification, and let $F$ be a \emph{virtual character} of $G$, that is,
a finite $\C$-linear combination of irreducible characters of $G$. 
Although we are primarily interested in the case where $F$ is the restriction to $G$ either of the character $a_1^{e_1}a_2^{e_2}a_3^{e_3}$ or of an irreducible character of $\USp(6)$, in \S\ref{section: moments disconnected groups} it will become apparent the necessity to describe the methods of this section for general virtual characters.

To compute the average of $F$ over $G$ we can either use the Weyl integration formula~~\cite{We46} or the Weyl character formula \cite[\S 24]{FH91} (see \S\ref{section: NU3} for a third approach to this problem). 

In order to apply the Weyl integration formula, we may express $F$ in terms of the eigen\-angles $\theta_1,\theta_2,\theta_3$ of an element of $G$. 
The measures $\mu$ that we need are listed below in terms of the eigenangles $\theta_i$.
\bigskip

\begin{center}
\small
\setlength{\extrarowheight}{2pt}
\begin{tabular}{|l|l|l|}
\hline&&\\[-15pt]
Group & $\mu$ & support\\[1pt]
\hline&&\\[-12pt]
$\Unitary(1)$&$ \tfrac{1}{\pi}\ d\theta_1$ & $\theta_1 \in [0,\pi)$\\[3pt]
$\SU(2)$&$\tfrac{2}{\pi}\sin^2\theta\  d\theta_1$& $\theta_1 \in [0,\pi)$ \\[3pt]
$\USp(4)$&$\tfrac{8}{\pi^2}(\cos\theta_1-\cos\theta_2)^2\sin^2\theta_1\sin^2\theta_2\ d \theta_1 d\theta_2$ & $\theta_1,\theta_2\in[0,\pi)$ \\[3pt]
$\Unitary(3)$&$\tfrac{1}{6\pi^3}\prod_{i<j}(1-\cos\theta_i\cos\theta_j-\sin\theta_i\sin\theta_j)\ d\theta_1 d\theta_2 d\theta_3$ & $\theta_1,\theta_2,\theta_3\in [0,2\pi)$\\[3pt]
$\USp(6)$ & $\tfrac{32}{3\pi^3}(\prod_{i<j}(\cos\theta_i-\cos\theta_j)^2)(\prod_i\sin^2\theta_i)\ d\theta_1 d\theta_2 d\theta_3$ & $\theta_1,\theta_2,\theta_3\in [0,\pi)$\\[4pt]
\hline
\end{tabular}
\end{center}
\bigskip

In each case we are embedding these groups (or products of them) in $\USp(6)$ as elements with three conjugate pairs of eigenvalues $e^{\pm i\theta_1},e^{\pm i\theta_2},e^{\pm i\theta_3}$.
In the case of diagonally embedded groups, some or all of the $\theta_i$ may coincide; in the case of $G=\Unitary(1)_3$, for example, we have $\theta_1=\theta_2=\theta_3$ and simply integrate $F$ against the measure for $\Unitary(1)$. In the more complicated cases, however, computing the resulting integrals becomes quite burdensome in practice.  There is however a more algorithmic approach to compute these integrals. Let $T\simeq \Unitary(1)^r$ denote a maximal torus in $G$, $\Phi$ the set of roots of the semisimple part of $\Lie(G)$, and $W$ the Weyl group of $G$. By abuse of notation, we may regard each $\alpha\in \Phi$ as a function on $\Unitary(1)^r$, and therefore there exist integers $e_i(\alpha)$ such that $\alpha(\bt)=\prod_i t_i^{e_i(\alpha)}$ for $\bt=(t_1,\dots,t_r)\in \Unitary(1)^r$. The Weyl integration formula establishes that the average of $F$ over $G$ is
$$
\tfrac{1}{|W|}\int_{\Unitary(1)^r}F(\bt)\prod_{\alpha\in \Phi}(1-\alpha(\bt))d\bt\,.
$$
Computing the above integral reduces to computing integrals of the form $\int_{\Unitary(1)} t_i^e dt_i$, which are $1$ if $e=0$ and $0$ otherwise.

An alternative way to compute the average of $F$ on $G$ is by determining, via the Weyl character formula, the (virtual) multiplicity of the trivial representation in $F$. We now explain how to do this concretely, and note that this is the method that we used in our computations. The virtual character $F$ gives rise to an element $\tilde F$ in $\C[u_1^{\pm 1}, \dots, u_r^{\pm 1}]$, where $r\leq 3$ denotes the rank of $G$: the coefficient of $\prod_i u_i^{k_i}$ in $\tilde F$ computes the multiplicity of the weight $(k_1,\dots,k_r)$ in $F$. For an integer $k$, let $[u_i^k]\tilde F$ denote the coefficient of $u_i^k$ in the Laurent polynomial~$\tilde F$ (as a Laurent polynomial in the remaining variables).

Let us start by making three elementary observations. 
\begin{itemize}
\item
Suppose that $G =H\times  \Unitary(1)_d$ for some positive integer $d$, and that the variable $u_r$ corresponds to the factor $\Unitary(1)_d$. Then it is apparent that the multiplicity of the trivial representation in $F$ is the multiplicity of the trivial representation in the virtual character of $H$ associated with $[u_r^{0}]\tilde F$.
\item
Suppose that $G =H\times  \SU(2)_d$ for some positive integer $d$, and that the variable~$u_r$ corresponds to the factor $\SU(2)_d$. Using that the irreducible representations of $\SU(2)$ are the symmetric powers $\Sym^m\C^2$ of the standard representation $\C^2$ and the weights 0 and 2 occur with the same multiplicity in $\Sym^m \C^2$ for all $m>0$, one sees that the multiplicity of the trivial representation in $F$ is the multiplicity of the trivial representation in the virtual character of $H$ associated with $[u_r^{0}]\tilde F-[u_r^{2}]\tilde F$.
\item
Suppose that $G=\Unitary(3)$. Put 
\[
\tilde E \colonequals {\tilde F}(uw, vw, u^{-1} v^{-1}w) \in \C[u^{\pm 1},v^{\pm 1},w^{\pm 1}];
\]
then the multiplicity of the trivial representation in $F$ is the multiplicity of the trivial representation of the virtual character of $\SU(3)$ associated with $[w^{0}]\tilde E$.
\end{itemize}

These observations allow us to reduce the problem to determining the multiplicity of the trivial representation in a virtual character $F$ of $G$ for some $G \in \{\USp(4), \SU(3), \USp(6)\}$. To determine this multiplicity, we apply the theory of representations of the complex Lie algebras attached to these classical groups. That is, the multiplicity of the trivial representation is recovered by successively applying the following steps:
\begin{itemize} 
\item First identify a highest weight $\lambda$ in $F$.
\item Then subtract from $F$ the character $\chi_\lambda$ attached to $\lambda$.
\end{itemize} 

The Weyl character formula provides a way to compute $\chi_\lambda$. Let $\tilde\chi_\lambda$ be the element in $\C[u_1^{\pm},\dots,u_r^{\pm}]$ associated with $\chi_\lambda$. Suppose that the variable $u_i$ corresponds to the weight~$w_i$, and let~$\lambda_i$ denote the coordinates of $\lambda$ written with respect to the basis formed by the weights~$w_i$. Let $\bu^{\lambda}$ denote $u_1^{\lambda_1}\cdots u_r^{\lambda_r}$ and write
$$
 D_{\lambda}=\sum_{w\in W}\mathrm{sign}(w)\bu^{w(\lambda)}.
$$
The Weyl character formula expresses $\tilde\chi_\lambda$ as the quotient $D_{\lambda+\rho}/D_{\rho}$, 
where $\rho$ is the half-sum of the positive roots of the Lie algebra of $G$. 

Table~\ref{table:3-diagonal} in \S\ref{sec:tables} lists the 3-diagonals of character norms for each of the possible connected Sato--Tate groups $G = G^0$ in $\USp(6)$.

\subsection{Averages over disconnected groups of central type}\label{section: moments disconnected groups} Let $F$ be a character of $\USp(2g)$, and let $G$ be a closed subgroup of $\USp(2g)$ (not necessarily connected). Computing the average of $F$ over $G$ reduces to computing the average of $F$ over every connected component $C$ of $G$. We will show that this computation can be carried out using the methods of \S\ref{section: momentsconnected} when $G$ is a group of \emph{central type} in the extended classification. 

Let $\varphi\colon \USp(2g)\rightarrow \R[T]$ denote the map that sends $\gamma$ to its reverse characteristic polynomial $\det(1-\gamma T)$. Recall that the image of $\varphi$ consists of real reciprocal polynomials and that the level sets of $\varphi$ are precisely the conjugacy classes of $\USp(2g)$.

\begin{definition}
Let $G$ be a closed subgroup of $\USp(2g)$ and let $H$ be a finite subgroup of~$G$. We will say that $G$ is of \emph{central type relative to $H$} if $G$ is generated by $G^0$ and $H$, and for every $h\in H$ the map
$$
\Phi^h\colon G^0\xrightarrow{\cdot h}G^0h\xrightarrow{\varphi|_{G^0h}} \R[T]\,,\qquad \Phi^h(\gamma)=\det(1-\gamma hT)
$$
is a class function of $G^0$. In other words, the conjugacy class of $gh$ in $\USp(2g)$ is itself a class function of $g \in G^0$.

We say that $G$ is of \emph{central type} if it is of central type relative to some finite subgroup of itself. Note that this is a property of $G$ as an embedded group, not an abstract group; equivalently, it is a property of the representation of $G$ restricted from the standard representation of $\USp(2g)$.
\end{definition}

Let $G$ be of central type relative to $H$, and let $h$ denote an element of $H$.
Let $P^h$ denote the set of characteristic polynomials of elements of $G^0h$. 
The restriction of $F$ to $G^0 h$ factors through $P^h$ as 
$$
G^0h\xrightarrow{\varphi|_{G^0h}}P^h\xrightarrow{ F^{h}} \R\,,
$$
for a certain map $F^h$. We will denote by $\Phi_F^h$ the composition $ F^h \circ \Phi^h$. It is a virtual character of $G^0$.
Moreover, the invariance of the Haar measure under translation immediately implies the following lemma.

\begin{lemma}
If $G$ is of central type relative to $H$, then for every $h$ in $H$ the average of $F$ over $G^0h$ equals the average of $\Phi^h_F$ over $G^0$. 
\end{lemma} 

If $G$ is a group of central type in the extended classification, the above lemma reduces the computation of the average of $F$ over a connected component $C$ to the problem considered in~\S\ref{section: momentsconnected}. Fortunately, most of the groups in the extended classification turn out to be of central type. Before considering this question, let us fix presentations for the groups of absolute type $\bE$ or $\bF$ (these do not follow the conventions given in \S\ref{section: connected groups}, but rather coincide with the presentations given in the \LMFDB{}). Consider the matrices
$$
I_2=\begin{pmatrix}
1 & 0\\
0 & 1
\end{pmatrix}\,,
\quad
J_2=\begin{pmatrix}
0 & 1\\
-1 & 0
\end{pmatrix}\,.
$$ 
We fix the symplectic form $\Diag(J_2,J_2,J_2)$ and consider the groups $\SU(2)\times \SU(2)\times \SU(2)$ and $\Unitary(1)\times \SU(2)\times \SU(2)$ embedded in $\USp(6)$ as $3$-diagonal block matrices. Set
$$
t=\begin{pmatrix}
I_2 & 0 & 0\\
0 & 0 & I_2\\
0 & I_2 & 0
\end{pmatrix}\,,
\quad
s=\begin{pmatrix}
0 & 0 & I_2\\
I_2 & 0 & 0\\
0 & I_2 & 0
\end{pmatrix}\,,
\quad
a=\begin{pmatrix}
J_2 & 0 & 0\\
0 & I_2 & 0\\
0 & 0 & I_2
\end{pmatrix}
\,.
$$
Denote by $E_\star$ the group generated by $\star\subseteq \{ s,t\}$ and $\SU(2)\times \SU(2)\times \SU(2)$. The groups of absolute type $\bE$ are
$$
\stgroup[E]{1.6.E.1.1a}\,,\quad \stgroup[E_t]{1.6.E.2.1a}\,,\quad \stgroup[E_s]{1.6.E.3.1a}\,,\quad \stgroup[E_{s,t}]{1.6.E.6.1a}\,.
$$
Denote by $F_{\star}$ the group generated by $\star\subseteq \{ a,t\}$ and $\Unitary(1)\times \SU(2)\times \SU(2)$. The groups of absolute type $\bF$ are
$$
\stgroup[F]{1.6.F.1.1a}\,,\quad \stgroup[F_a]{1.6.F.2.1c}\,,\quad \stgroup[F_t]{1.6.F.2.1b}\,,\quad \stgroup[F_{at}]{1.6.F.2.1a}\,,\quad \stgroup[F_{a,t}]{1.6.F.4.2a}\,.
$$

The following lemma is an immediate consequence of the additivity of characters with respect to direct sums.

\begin{lemma}\label{lemma: centrality from factors}
Let $G$ be a closed subgroup of $\USp(2g)$. Suppose that:
\begin{enumerate}[{\rm(i)}]
\item $G^0=G_1^0\times G_2^0$, where $G_i^0$ is a subgroup of $\USp(2g_i)$ and $g=g_1+g_2$. 
\item $G$ is generated by $G^0$ and a finite subgroup $H$ of $\USp(2g_1)\times \USp(2g_2)$.
\end{enumerate} 
Let $H_i$ be the image of $H$ under the projection onto the factor $\USp(2g_i)$. If $G_1$ and $G_2$ are of central type relative to $H_1$ and $H_2$, respectively, then $G$ is of central type relative to $H$.  
\end{lemma}

The following lemma is used  in the proof of Proposition~\ref{proposition: basicgroups}.

\begin{lemma} \label{lemma: central type dimension 2}
If $G$ is the Sato--Tate group of an elliptic curve or an abelian surface and $G$ is not the group $\stgroup[N(\SU(2)\times \SU(2))]{1.4.B.2.1a}$, then $G$ is of central type. 
\end{lemma}

\begin{proof}
The three groups arising for $g=1$ are either connected or have abelian identity component, so the claim follows trivially. For $g=2$, having ruled out $\stgroup[N(\SU(2)\times \SU(2))]{1.4.B.2.1a}$, the only nonabelian connected groups that arise as subgroups of disconnected groups are the groups $G^0=\Unitary(1)\times \SU(2)$ and $G^0=\SU(2)_2$. In the former case, the lemma follows from Lemma \ref{lemma: centrality from factors}. As for $G^0=\SU(2)_2$, we recall from \cite[\S3.4]{FKRS12} that there exists a subgroup
$$
H\subseteq \langle \Diag(e^{\nicefrac{i\pi}{n}}, e^{\nicefrac{i\pi}{n}},e^{-\nicefrac{i\pi}{n}},e^{-\nicefrac{i\pi}{n}}), J_4\rangle\,, \quad \text{where }J_4=\begin{pmatrix}
0 & J_2\\
-J_2 & 0
\end{pmatrix}\,, 
$$
such that $G$ is generated by $H$ and $G^0$. Any such $H$ centralizes $G^0$ (see \cite[p.\,1404]{FKRS12}), so for every $\xi,\gamma\in G^0$ and $h\in H$ we have
\begin{equation} \label{eq:central type}
\det(1-\xi^{-1}\gamma \xi hT)=\det(1-\xi^{-1}\gamma h\xi T)=\det(1-\gamma hT)\,,
\end{equation}
proving the claim.
\end{proof}

\begin{proposition}\label{proposition: basicgroups}
Let $G$ be a group of the extended classification distinct from 
\begin{equation}\label{equation: exceptionalmomentgroups}
\stgroup[N(\Unitary(3))]{1.6.B.2.1a}\,,\quad   \stgroup[E_t]{1.6.E.2.1a}\,,\quad \stgroup[E_s]{1.6.E.3.1a}\,,\quad \stgroup[E_{s,t}]{1.6.E.6.1a}\,, \quad \stgroup[F_t]{1.6.F.2.1b}\,,\quad \stgroup[F_{at}]{1.6.F.2.1a}\,,\quad \stgroup[F_{a,t}]{1.6.F.4.2a}\,.
\end{equation}
Then $G$ is of central type. 
\end{proposition}

\begin{proof}
The proposition is trivially true when $G$ is connected or $G^0$ is abelian. Taking into account the restrictions on $G$, this implies that the proposition holds when the absolute type of $G$ is one of $\bA$,  $\bB$, $\bE$, $\bH$, $\bL$, $\bN$. 
Suppose next that $G$ is a split product (under the restrictions on $G$, this covers the cases in which $G$ has absolute type $\bC$, $\bD$, $\bF$, $\bG$, $\bI$, $\bJ$, $\bK$,~$\bL$). In this case, $G$ satisfies the hypotheses of Lemma \ref{lemma: centrality from factors} for $g_1=1$ and $g_2=2$. Moreover, by the restrictions on $G$, if $G_2^0=\SU(2)\times \SU(2)$, then we may assume that $H_2$ is trivial. The proposition then immediately follows from Lemma~\ref{lemma: central type dimension 2}.

It remains only to consider the case that $G$ has absolute type $\bM$, that is, $G^0=\SU(2)_3$. In~\S\ref{subsec:su2} we have shown that in this case $G$ can be generated by $G^0$ and a finite subgroup $H$ centralizing~$G^0$. Hence \eqref{eq:central type} holds again, proving the claim. 
\end{proof}

\subsection{Averages over disconnected groups not of central type}\label{section: NU3} Let $F$ be a character of $\USp(6)$.
We now consider the problem of computing the average of $F$ over the connected components of the seven groups listed in \eqref{equation: exceptionalmomentgroups}.
For the three groups of absolute type $\bF$, those for which $G^0=\Unitary(1)\times \SU(2)\times \SU(2)$, this can be computed using data in the row $JG_{3,3}$ in \cite[Table~6]{FKRS12}.

For the three groups of absolute type $\bE$ with $G^0=\SU(2)\times \SU(2)\times \SU(2)$, the computation of the average of $F$ for the connected component $s^jts^{-j}G^0$ follows again from data in the row~$JG_{3,3}$ of \cite[Table 6]{FKRS12}. Let us now consider the connected component 
$sG^0$ (the case of the connected component $s^2G^0$ being analogous). We claim that the distribution of $\det(1-s\gamma T)$ as $\gamma$ runs over $G^0$ with respect to the Haar measure of $G^0$ is the same as the distribution of the polynomials
$$
1-\Trace(a)T^3 +T^6\,,
$$  
as $a$ runs over $\SU(2)$ with respect to the Haar measure of $\SU(2)$. Let $\varphi$ be any continuous function on the set of polynomials $\det(1-s\gamma T)$ for $\gamma\in G^0$. If $\gamma=\Diag(a,b,c)$, let us write $\mu_{G_0}(a,b,c)$ for the Haar measure of $G^0$ at $\gamma$. Then an elementary calculation gives
$$
\int_{G^0}\varphi\big(\det(1-\Diag(a,b,c)\cdot s\cdot T)\big)\mu_{G_0}(a,b,c)=\int_{G^0} \varphi\big(\det(1-acb \cdot T^3)\big)\mu_{G_0}(a,b,c)\,. 
$$
By the invariance of the Haar measure under translations, the right-hand side of the above equality is
$$
\int_{\SU(2)} \varphi\big(\det(1-a \cdot T^3)\big)\mu_{\SU(2)}(a)=\int_{\SU(2)} \varphi\big(1-\Trace(a) T^3+T^6\big)\mu_{\SU(2)}(a)\,,  
$$
and the claim follows.

We are left with the group $\stgroup[N(\Unitary(3))]{1.6.B.2.1a}$. Given a character $F$ of $\USp(6)$ and a closed subgroup $G\subseteq \USp(6)$, let us denote by $\fm_G(F)$ the multiplicity of the trivial representation in the restriction of $F$ to $G$. Since the Weyl character formula allows to express $F$ as a sum of irreducible characters, the problem of computing $\fm_G(F)$ reduces to the problem of computing $\fm_G(\chi_\lambda)$ for all irreducible characters $\chi_\lambda$ of $\USp(6)$. 
The next lemma explains how to compute these multiplicities when $G$ is $\stgroup[N(\Unitary(3))]{1.6.B.2.1a}$. 

\begin{lemma}\label{lemma: multNU3}
Let $\lambda$ be the partition $a\geq b \geq c$. Then:
\begin{enumerate}[{\rm(a)}]
\item $\fm_{\Unitary(3)}(\chi_{\lambda})$ is $1$ if $a,b,c$ are all even, and it is $0$ otherwise.\vspace{0,1cm}
\item $\fm_{N(\Unitary(3))}(\chi_{\lambda})$ is $1$ if $a,b,c$ are even and $a+b+c\equiv 0 \pmod 4$, and it is $0$ otherwise.
\end{enumerate} 
\end{lemma}

\begin{proof}
We check (a) using the branching rule for the inclusion $\GL_3 \subset \Sp_6$ of algebraic groups as formulated in \cite[(2.3.2)]{HTW05}
(see also \cite{Kin75}).
For a representation of $\GL_3$ indexed by a pair of nonnegative partitions $(\mu^+, \mu^-)$,
its multiplicity in the restriction of $\chi_{\lambda}$ is given by a sum of Littlewood--Richardson numbers:
\[
\sum_{\delta, \gamma} c^\gamma_{\mu^+,\mu^-} c^{\lambda}_{\gamma, 2\delta}
\]
where $\delta$ and $\gamma$ run over nonnegative partitions. 
To evaluate  $\fm_{\Unitary(3)}(\chi_{\lambda})$ we take $\mu^+ = \mu^- = 0$, in which case $c^{\gamma}_{\mu^+, \mu^-} = 1$ if $\gamma=0$ and $c^{\gamma}_{\mu^+, \mu^-} = 0$ otherwise.
Hence $\fm_{\Unitary(3)}(\chi_{\lambda}) = \sum_{\delta} c^{\lambda}_{0, 2\delta}$; the claim now follows by observing that $c^{\lambda}_{0, 2\delta} = 1$ if $\lambda = 2\delta$ and $c^{\lambda}_{0, 2\delta} = 0$  otherwise.

We now prove (b). We may assume that $a$, $b$, $c$ are all even, since otherwise the statement follows from (a).
Recall that the irreducible representation $V_{\lambda}$ of $\USp(6)$ with highest weight~$\lambda$ is a subrepresentation of
$$
\Sym^{a-b}(V) \otimes \Sym^{b-c}(W) \otimes \Sym^c(U)\subseteq V^{\otimes (a+b+c)}\,,
$$
where $V$ is the standard representation of $\USp(6)$, $W$ is $\wedge^2V / 1$, and $U$ is $\wedge^3V/ V$. Let~$v$ be an element of $V_{\lambda}$ spanning the line which is fixed under the action of $\Unitary(3)$. On the one hand, this implies that $v$ is an element of weight $(0,0,0)$. On the other hand, since $J$ normalizes~$\Unitary(3)$, we see that $\Unitary(3)$ fixes $Jv$ and $Jv$ is thus a scalar multiple of $v$.

Let $v_1,\dots, v_6$ denote the standard basis of $V$. Since $v\in V^{\otimes(a+b+c)}$ is of weight $(0,0,0)$, it is a nonzero linear combination of vectors of the form
$$
w=\bigotimes_{j=1}^{a+b+c} w_j\,,
$$
where $\{w_j\}_j$, as a multiset, contains $r$ copies of $v_1$ and $v_4$, $s$ copies of $v_2$ and $v_5$, and $t$ copies of $v_3$ and $v_6$, where $2(r+s+t)=a+b+c$. Let $\hat v_j$ denote $v_{j+3}$ if $j=1,2,3$, and $v_{j-3}$ if $j=4,5,6$. Note that $Jv_j=-\hat v_j$ for $j=1,2,3$ and $Jv_j=\hat v_j$ for $j=4,5,6$, and therefore
$$
Jw=(-1)^{r+s+t}\bigotimes_{j=1}^{a+b+c} \hat w_j\,.
$$
The fact that $Jv$ is a scalar multiple of $v$ implies that $v$ is a nonzero linear combination of vectors of the form 
$$
u=\bigotimes_{j=1}^{a+b+c} w_j+(-1)^{(a+b+c)/2}\bigotimes_{j=1}^{a+b+c} \hat w_j
$$
But $J$ acts trivially on $u$ if and only if $a+b+c \equiv 0 \pmod 4$. 
\end{proof}

\begin{remark}  \label{R:closed formulas for moments}
The determination of closed formulas for the moments of the groups in the extended classification could be approached following the methods recently introduced by Lee and Oh \cite{LO20}. Let $G$ be a closed subgroup of $\USp(6)$ and let $m$ be a positive integer.    
Given a subpartition $\lambda$ of the rectangular partition $3\geq\stackrel{m}{\dots}\geq 3$, let $\lambda^*$ denote the partition
$$
m-\lambda_3'\geq m-\lambda_2' \geq m-\lambda_1'\,,
$$ 
where $\lambda_1'\geq  \lambda_2' \geq \lambda_3'$ denotes the transpose partition of $\lambda$. 
Generalizing a formula of Bump and Gamburd \cite{BG06}, Lee and Oh \cite[Prop. 3.3]{LO20} establish
\begin{equation}\label{equation: LeeOhformula}
\int_G \prod_{j=1}^m\det(1-u_j \gamma)\mu_G(\gamma)=(u_1\dots u_m)^3\sum_\lambda \fm_G(\chi_{\lambda^*})\tilde\chi_{\lambda}\,,
\end{equation}
where the sum runs over subpartitions $\lambda$ of the rectangular partition $3\geq \stackrel{m}{\dots} \geq 3$. 
Note that for $m=e_1+e_2+e_3$ the moment $M_{e_1,e_2,e_3}(G)$ is the coefficient of the monomial
$$
\prod_{j=1}^{e_1} u_j \prod_{j=e_1+1}^{e_1+e_2}u_j^2 \prod_{j=e_1+e_2+1}^{m} u_j^3
$$
in the left-hand side of \eqref{equation: LeeOhformula}. In dimension 2, Lee and Oh have derived closed formulas for $\fm_H(\chi_{\lambda^*})$ for each possible Sato--Tate group $H$ of an abelian surface. The analogous problem in dimension $3$ is an interesting question which we only attempted to solve for $\stgroup[\Unitary(3)]{1.6.B.1.1a}$ and~$\stgroup[N(\Unitary(3))]{1.6.B.2.1a}$.
\end{remark}

\begin{remark} \label{R:Kostant character formula}
The Weyl character formula admits an extension to disconnected compact Lie groups due to Kostant
\cite[\S7]{Kos61}. It should be possible to use Kostant's formula to compute moments of disconnected groups, but we have not attempted this.
\end{remark}

\subsection{Point densities}

For $G$ a nontrivial connected closed subgroup of $\USp(6)$, the distributions associated to $a_1, a_2, a_3$ are continuous; in particular, no single value occurs with positive probability. By contrast, if $G$ is disconnected, then it is possible for one or more of $a_1,a_2,a_3$ to take a constant value on a connected component of $G$.

\begin{lemma}\label{lemma: a1a2a3constants}
Let $G$ be a group satisfying (ST1), (ST2), and (ST3). Let $(i,t)$ be a pair in which $i \in \{1,2,3\}$ and $t$ is a real number with the property that the function $a_i:G\rightarrow \R$ is identically equal to the constant function $t$ on some connected component $C$ of $G$. Then
\[
(i,t) \in \bigl\{(1,0)\bigr\} \cup \bigl\{(2,t): t\in \{-1,0,1,2,3\}\bigr\} \cup \bigl\{(3, 0)\bigr\}.
\]
\end{lemma}
\begin{proof}
Since $G$ satisfies~(ST2), $G^0$ contains $-1$. For $i=1,3$, the presence of $-1$ in $G^0$ implies that the only possible constant value is $0$. Now suppose $i=2$.  Since~$G$ satisfies~(ST3), $t$ must be an integer.  Choose $\gamma_0\in C$ so that $a_1(\gamma_0)=0$; such a $\gamma_0$ exists, since if $a_1$ is not identically zero on $C$, then $a_1(\gamma) < 0 < -a_1(\gamma)= a_1(-\gamma)$ for some $\gamma\in C$, and $\gamma_0$ exists by continuity (note $-C=C$).  We have $-1\leq a_2(\gamma_0)\leq 3$ by \cite[Prop. 4]{KS08}, therefore $-1\le t \le 3$.
\end{proof}

\begin{remark}
Lemma~\ref{lemma: a1a2a3constants} can be generalized to any closed subgroup of $\USp(2g)$ that satisfies (ST2) and (ST3).  The same argument shows that if $a_2$ takes the constant value $t$ on a component, then $t$ is an integer in $[2-g,g]$ when $g$ is odd, and in $[-g,g]$ when $g$ is even.
\end{remark}

The previous lemma suggests the following definition.

\begin{definition}
Let $G$ be a group of the extended classification.
If $S$ is a string of digits and~$t$ is a real number, let $z_{S}^{t}\in \Q$ denote the proportion of connected components $C$ of $G$ with all of the following properties: $a_1$ is identically $0$ on $C$ if $1$ appears in $S$, $a_2$ is identically~$t$ on~$C$ if $2$ appears in $S$, and $a_3$ is identically $0$ on~$C$ if $3$ appears in $S$.
If $2$ does not appear in~$S$ we omit the superscript $t$. If $2$ does appear in~$S$, let $z_S$ denote the sum of all~$z_S^t$ for all possible values of $t$.
We define the \emph{matrix of point densities} associated to $G$ as
$$
Z(G)=\begin{bmatrix}
1 & z_2 & z_{2}^{-1} & z_{2}^{0} & z_{2}^{1} & z_{2}^{2} & z_{2}^{3}\\[2pt]
z_1 & z_{12} & z_{12}^{-1} & z_{12}^0 & z_{12}^1 & z_{12}^2 & z_{12}^{3}\\[2pt]
z_3 & z_{23} & z_{23}^{-1} & z_{23}^0 & z_{23}^1 & z_{23}^2 & z_{23}^3\\[2pt]
z_{13} & z_{123} & z_{123}^{-1} & z_{123}^0 & z_{123}^1 & z_{123}^2 &z_{123}^3
\end{bmatrix}.
$$
\end{definition}

\subsection{A \Magma{} computation}

For each group $G$ in the extended classification, we have implemented \Magma{} code that gives matrix generators for the group of components of~$G$, and using the methods of the previous sections, computes:
\begin{itemize}
\item the $12$-simplex of moments of $G$;
\item the $3$-diagonal of character norms of $G$;
\item the matrix of point densities $Z(G)$;
\item the averages of $a_i^e$ over $G$ for $i\in \{1,2,3\}$ and $1\leq e\leq 12$;
\item the group of connected components of $G$;
\item the lattice of finite index subgroups of $G$.
\end{itemize}

This code may be found in the \GitHub{} repository \cite{FKS21}.  The output it produces is included in the \LMFDB{} home page for each of the 410 groups realized by abelian threefolds.

\begin{remark} Each group $G$ of central type in the extended classification is of central type relative to the finite group $H$ specified in \S\ref{section:STgroups} or \S\ref{section: unitary} in order to construct it. For each such $G$ and specific choice of $H$, the \Magma{} code constructs the Laurent polynomials
\begin{equation}\label{equation: magma code polynomials}
\tilde \Phi_F^h\in \C[u_1^{\pm 1},\dots, u_r^{\pm 1}]\,,
\end{equation} 
where $h$ runs over $H$, and $F$ is either $a_1^{e_1}a_2^{e_2}a_3^{e_3}$ for $e_1+2e_2+3e_3\leq 12$, $a_i^e$ for $1\leq e\leq 12$, or $\chi_{\lambda}^2$ for $\lambda$ a subpartition of $3\geq 3\geq 3$. In fact, the explicit presentations of the finite subgroups~$H$ show that the coefficients of $\tilde \Phi_F^h$ always belong to the cyclotomic field generated by a primitive root of unity of order~504. The problem of computing the average of $F$ over~$G^0h$ is then resolved as explained in \S\ref{section: momentsconnected}.   

Computing moments of the group $\stgroup[N(\Unitary(3))]{1.6.B.2.1a}$ is immediate from the Weyl character formula for $\USp(6)$ and Lemma \ref{lemma: multNU3}. The remaining groups not of central type are treated using the results of \S\ref{section: NU3}.
\end{remark}

Our \Magma{} computation yields the following proposition; see Theorem~\ref{thm: minimal set of 3-diagonals} for a sharper result regarding the 410 groups that arise for abelian threefolds.

\begin{proposition} \label{proposition: same measures}
The following statements hold:
\begin{enumerate}[{\rm(a)}]
\item The $433$ groups in the extended classification give rise to $432$ distinct $14$-simplices of moments (resp. $3$-diagonals of character norms). 
\item For the groups 
\[
G_1 = \stgroup{J(C(3,3))}, \qquad G_2 = \stgroup{J_s(C(3,3))},
\]
the pushforward measures from $G_1, G_2$ to conjugacy classes in $\USp(6)$ coincide. In particular, we have
$M_{e_1,e_2,e_3}(G_1) = M_{e_1,e_2,e_3}(G_2)$ for all nonnegative integers $e_1,e_2,e_3$.
\end{enumerate}
\end{proposition}
\begin{proof}
The statement about the $3$-diagonals is direct result of the \Magma{} computation, which also shows that the $433$ groups appearing in the classification give rise to $431$ distinct $12$-simplices of moments. The two absolute type $\bN$ groups $\stgroup{J(C(3,3))}$ and $\stgroup{J_s(C(3,3))}$ share the same $12$-simplex of moments, and so do the absolute type $\bL$ groups $G_3=\stgroup{L(J(D_6),J(C_6))}$ and $G_4=\stgroup{L(J(D_6),D_6)}$. The latter two groups are respectively the fiber product of $\stgroup[N(\Unitary(1))]{N(U(1))}$ and $\stgroup{J(D_6)}$ along $\stgroup{J(C_6)}$, and the fiber product of $\stgroup[N(\Unitary(1))]{N(U(1))}$ and $\stgroup{J(D_6)}$ along $\stgroup{D_6}$. By slightly extending the \Magma{} computation, one finds $M_{0,4,2}(G_3)=98083$ and $M_{0,4,2}(G_4)=98082$, showing that there are at least 432 distinct 14-simplices.

For $i=1,2$, let $H_i$ be the finite group described in  \S\ref{section: unitary} which generates $G_i$ together with~$G^0_i$. Let $S_i$ denote the multiset of triples
\begin{equation}\label{equation: triple Laurent polynomials}
\big(\tilde \Phi_{a_1}^{ h}, \tilde \Phi_{a_2}^{ h},\tilde \Phi_{a_3}^{h} \big)
\end{equation}
as $h$ runs in $H_i$. Since both $G_1$ and $G_2$ are of central type, the multiset $S_i$ determines the pushforward measure from $G_i$ to conjugacy classes in $\USp(6)$ as discussed in \S\ref{section: momentsconnected} and \S\ref{section: moments disconnected groups}. The \Magma{} computation verifies that the multisets $S_1$ and $S_2$ coincide, proving (b), and this completes the proof of (a) and the proposition.
\end{proof}

\begin{remark}
Let $G$ be a group in the extended classification, and let $F$ be a character of $\USp(6)$. If $h$ and $h'$ are conjugate in $G$, then the average of $F$ over $G^0h$ coincides with that over $G^0h'$. This implies that many of the Laurent polynomials $\tilde \Phi^h_F$ of \eqref{equation: magma code polynomials} computed by the \Magma{} code coincide. This naturally raises the problem of studying the $3$-diagonals of the \emph{connected components} of the groups in the extended classification. By computation, we have observed that there are only $82$ distinct 3-diagonals of connected components, while, as observed in the previous lemma, their appropriate combination gives rise up to 432 distinct 3-diagonals for entire groups. This discrepancy is particularly remarkable for the groups of absolute type $\bN$, whose connected components give rise to just $19$ distinct $3$-diagonals.
\end{remark}

\begin{remark}
Data inspection shows that, uniformly over the groups, the size of the entries of the $3$-diagonal of character norms is much smaller than that of the entries of the $14$-simplex of moments (in addition to having just 20 nonzero entries rather than 81). As already suggested in \cite{Shi16}, this makes the $3$-diagonal of character norms a more desirable invariant to identify Sato--Tate groups from a computational perspective.
To illustrate this point, we list in Table~\ref{table:3-diagonal} the 3-diagonals for the connected groups in the extended classification; these values are all at most $10^5$, whereas the $14$-simplices include values that exceed $10^8$.
\end{remark}

\begin{remark} \label{remark: Gassmann equivalence}
Despite inducing the same measures on the set of conjugacy classes of $\USp(6)$ by Proposition~\ref{proposition: same measures}, the groups $\stgroup{J(C(3,3))}$ and $\stgroup{J_s(C(3,3))}$, which are not conjugate inside $\USp(6)$, do not even have isomorphic groups of connected components. 
Although the group-theoretic mechanism at work is somewhat different, this phenomenon
is loosely analogous to the phenomenon of \emph{Gassmann equivalence}, in which nonconjugate subgroups of an ambient group give rise to isomorphic permutation representations. See \cite{Sut18,Sut21} for a modern survey of this topic and some recent results.
\end{remark}

\begin{remark}
The computation of the lattice of subgroups of finite index of a given group relies on Proposition \ref{proposition: same measures}. This ensures that any two groups in the extended classification are uniquely distinguished by the pair formed by the group of components and the $3$-diagonal of character norms. Given $G$, we compute its set of maximal subgroups of finite index up to conjugacy and identify each of them by computing the previous pair of invariants. 
\end{remark}

\begin{remark}
In principle, the point densities do not carry any additional information beyond that of the moments. (In particular,
by Proposition~\ref{proposition: same measures}, the point densities cannot be used to separate the groups
$\stgroup{J(C(3,3))}, \stgroup{J_s(C(3,3))}$.) However, in practice it is much easier to measure point densities on experimental data than very high moments, so including them in the statistical profile when matching $L$-functions to Sato--Tate groups yields better results.
\end{remark}

\begin{theorem} \label{thm: minimal set of 3-diagonals}
The following statements hold:
\begin{enumerate}[{\rm(a)}]
\item
The $2$-simplex of moments distinguishes the $14$ connected Sato--Tate groups that arise for abelian threefolds over number fields.
\item
The $3$-diagonal of character norms distinguishes the $409$ Sato--Tate distributions that arise for abelian threefolds over number fields.
In fact, the same holds if we replace the $3$-diagonal with the list of $N_{\lambda,\lambda}(G)$ for $\lambda\in \{(3,2,2),(3,3,0),(3,3,1)\}$.
\item
The isomorphism class of $G/G^0$ together with the matrix of point densities and the $2$-diagonal of character norms distinguishes the $410$ Sato--Tate groups $G$ that arise for abelian threefolds over number fields. In fact, the same holds if we replace the $2$-diagonal with the list of $N_{\lambda,\lambda}(G)$ for $\lambda\in \{(1,1,0),(1,1,1),(2,0,0)\}$.
\end{enumerate}
\end{theorem}
\begin{proof}
This is verified by computing these invariants for the 14 Sato--Tate groups that appear in Table~\ref{table:connected ST groups}
and the 410 Sato--Tate groups admitted by Theorem~\ref{T:ST result}
using the \Magma{} scripts in \cite{FKS21}.
\end{proof}

\section{Methods for explicit realizations of abelian threefolds}
\label{sec:realization techniques}

In preparation for our treatment of the lower bound aspect of Theorem~\ref{T:ST result}, we document
some strategies for constructing abelian threefolds with particular Sato--Tate groups.
These include consideration of curves with various automorphism groups
and Galois twisting of abelian varieties.
We also discuss how to compute the $L$-functions of these abelian threefolds, so as to give
an empirical corroboration of the assignment of Sato--Tate groups.  Recall that $\Aut'(C)$ denotes the reduced automorphism group of a curve $C$, as defined in \S 2.

\subsection{Curves with automorphisms}
\label{subsec:curves with automorphisms}

\begin{proposition} \label{P:automorphism groups}
Let $C$ be a curve of genus $3$ over $\Qbar$.
\begin{enumerate}[{\rm(a)}]
\item
The group $\Aut'(C)$ takes one of the values listed in
Table~\ref{table: types of generic automorphism groups}.
\item
For each option for $\Aut'(C)$, the models listed give the generic hyperelliptic or nonhyperelliptic curves realizing that option.
\item
For each option for $\Aut'(C)$, the isogeny decomposition of $\Jac(C)$ in the generic case is as listed. Here $E, E',E''$ represent pairwise nonisogenous elliptic curves; $A_i$ represents an abelian variety of dimension $i$; $*[\alpha]$ with $\alpha \in \overline{\Q}$ denotes an action on $*$ by $\alpha$; and $*[\mathcal{O}(6)]$ denotes an action on $*$ by a maximal order in the quaternion algebra of discriminant $6$ over $\Q$.
\item
For each option for $\Aut'(C)$, the absolute type of a generic curve realizing that option is as listed.
\end{enumerate}
\end{proposition}
\begin{proof}
For (a) and (b), see \cite{Sha06} or \cite[Table~1]{MP22}
in the hyperelliptic case and \cite[Theorem~3.1]{LRRS14} or \cite[Theorem~6.5.2]{Dol12}
in the nonhyperelliptic case. (See also \cite{OS20} for an extension of \cite{MP22} to superelliptic curves.)

For (c), we establish here only the lower bound on $\End(\Jac(C))_{\Q}$ implied by Table~\ref{table: types of generic automorphism groups}; the upper bound in the generic case will follow from the explicit realizations given in \S\ref{section: realizability}. Before proceeding, we note that the two entries in the table with $\Aut'(C) = \cyc 2$ correspond to two different values of $\Aut(C)$, respectively $\cyc 4$ and $\cyc 2 \times \cyc 2$.

The decomposition into isogeny factors, without regard to endomorphisms, can be derived using the method of \cite{Pau08}, as implemented in \LMFDB; alternatively, see \cite[Table~1, Table~2]{LLRS21}. This suffices to treat the cases where the isogeny factors do not have extra endomorphisms, i.e., $\Aut'(C) \in \{
\cyc 1, \cyc 2 \times \cyc 2, \sym 3, \dih 4, \sym 4\}$ or $(\Aut'(C), \Aut(C)) = (\cyc 2, \cyc 2 \times \cyc 2)$.
It also suffices in the cases where there is a single isogeny factor, i.e., $\Aut'(C) \in \{\cyc 3, \cyc 7, \cyc 9\}$ or $(\Aut'(C), \Aut(C)) = (\cyc 2, \cyc 4)$.

In cases where all factors with extra automorphisms are elliptic,
it is sufficient to look up the $j$-invariants of these factors in \cite[Table~1, Table~2]{LLRS21}. This covers the cases
$\Aut'(C) \in \{\dih 6, \dih 8, \group{16}{13}, \group{48}{33}, \group{96}{64}, \PSL_2(\F_7)\}$.

This leaves only the case $\Aut'(C) = \cyc 6$.
In this case, the elliptic factor of $\Jac(C)$ inherits an order-3 automorphism from $C$ (see \cite[\S 3]{LLRS21}). It thus remains to identify the endomorphisms of $A_2$, which may be taken to be the Prym variety for the quotient by the order-2 automorphism. 
In this case, $A_2$ admits an automorphism of order 3 and a polarization of degree $(1,2)$, so we may appeal to \cite[Proposition 1.1, Theorem 1.2]{BvG16} to see that $A_2$ carries an action of the quaternion order $\mathcal{O}(6)$.

To conclude, note that (d) follows at once from (c).
\end{proof}

We make a few cases more explicit, starting with that of a single extra involution on a hyperelliptic curve.

\begin{lemma} \label{lemma: double cover}
Let $P(x) \in k[x]$ be a polynomial of degree $4$ such that $xP(x)$ has no repeated roots. Then
the genus $3$ curve
\[
C: y^2 = P(x^2)
\]
maps to the curves
\[
C_1: y^2 = P(x), \qquad C_2: y^2 = xP(x)
\]
of respective genera $1$ and $2$, and the Prym variety of $C \to C_1$ is $2$-isogenous to $\Jac(C_2)$.
\end{lemma}
\begin{proof}
The map $C \to C_1$ is $(x,y) \mapsto (x^2,y)$. The map $C \to C_2$ is
$(x,y) \mapsto (x^2,xy)$. For the assertion about the Prym variety, see \cite[\S 1]{RR18}.
\end{proof}

\begin{remark} \label{R:Ritzenthaler-Romagny}
The general form of a plane quartic with an involution defined over $k$ is
\[
C: Y^4 - P_2(X,Z) Y^2 + P_4(X,Z) = 0.
\]
The quotient by the involution is the genus $1$ curve
\[
C_1: y^2 - P_2(x,1)y + P_4(x,1) = 0.
\]
A formula for the genus $2$ quotient has been given by Ritzenthaler and Romagny \cite{RR18} (see also \cite{LLRS21}).
Note that it depends on a factorization of $P_4(x,1)$ as a product of two quadratic polynomials that are either Galois invariant or Galois conjugates, which in general does not exist over~$k$ (when it does not one can work over the extension defined by the cubic resolvent of $P_4(x,1))$.

One can effectively reverse this construction by gluing curves of genus 1 and 2 along their 2-torsion; see \cite{HSS21}.
\end{remark}

\begin{remark} \label{R:pair of quartics}
The general form of a genus 3 hyperelliptic curve over $k$ with reduced automorphism group $\cyc 2 \times \cyc 2$ 
(where all three automorphisms are defined over $k$)
is
\[
y^2 = ax^8 + bx^6 + cx^4 + bx^2 + a
\]
with the three involutions
\[
(x,y) \mapsto (-x, y), \qquad (x,y) \mapsto \left(\tfrac{1}{x}, y\right), \qquad (x,y) \mapsto \left(-\tfrac{1}{x}, y\right).
\]
The general form of a plane quartic over $k$ with reduced automorphism group $\cyc 2 \times \cyc 2$ is 
\[
P_2(X^2, Y^2, Z^2) = 0
\]
with the three involutions
\[
(X:Y:Z) \mapsto (-X:Y:Z), 
\quad
(X:Y:Z) \mapsto (X:-Y:Z), 
\quad
(X:Y:Z) \mapsto (X:Y:-Z).
\]
A result of Howe, Lepr\'evost, and Poonen \cite[Proposition~14, Proposition~15]{HLP00}
gives a recipe for reconstructing a genus $3$ curve with reduced automorphism group containing $\cyc 2 \times \cyc 2$
from these quotients (this imposes a matching condition on $2$-torsion).
\end{remark}

The following is a variant of \cite[Example~3.4]{Lor18}.
\begin{lemma} \label{lemma: S3 quartic}
Let $C$ be a smooth plane quartic curve over $k$ of the form
\[
X^3Z + bY^3Z + cX^2Y^2 + dXYZ^2 + eZ^4 = 0.
\]
Over $L = \Q(\zeta_3, b^{\nicefrac{1}{3}})$, $C$ admits the group of automorphisms $\sym 3$ generated by
\[
s: [X:Y:Z] \mapsto [\zeta_3 X: \zeta_3^2 Y:Z], \qquad
t: [X:Y:Z] \mapsto [b^{\nicefrac{1}{3}}Y: b^{\nicefrac{-1}{3}}X:Z].
\]
The quotient $C_1 = C_L/\langle s \rangle$ has the form
\[
C_1: y^2 + b x^3 + cx^2y + dxy + ey = 0,
\]
and the projection $C \to C_1$ is a degree-$3$ map defined over $k$. The associated Prym variety is
$3$-isogenous over $L$ to $\Jac(C_2)^2$ where $C_2$ has the form
\[
C_2: x^3 - 3b^{\nicefrac{1}{3}} xy + cy^2 + dy + e = 0.
\]
\end{lemma}
\begin{proof}
Note that $C_2$ is the quotient of $C$ by $t$.
The equation for $C_2$ can be derived either by adapting the method of \cite{LLRS21} or by changing coordinates on the model given in \cite[Table~2]{LLRS21}.
\end{proof}

The following is a variant of entry V of \cite[Proposition~2.1]{Lor18}.

\begin{lemma} \label{lemma: quadruple cover}
Let $C$ be a smooth plane quartic over $k$ of the form
\[
Y^4 = X^3Z + aX^2Z^2 + bXZ^3.
\]
Over $L = k(i, b^{\nicefrac{1}{2}})$, $C$ admits the group of automorphisms $\dih 4$ given by
\[
r: [X:Y:Z] \mapsto [X:iY:Z], \qquad
s: [X:Y:Z] \mapsto [b^{\nicefrac{1}{2}} Z: Y: b^{\nicefrac{-1}{2}} X].
\]
Then the quotient $C_L/\langle r^2 \rangle$ is the genus $1$ curve
\[
C_1: y^2 = x(x^2+ax+b),
\]
and the projection $C \to C_1$ is a degree-$2$ map defined over $k$. The associated Prym variety
is $2$-isogenous over $L$ to $\Jac(C_2)^2$ where $C_2$ is the genus $1$ curve
\[
C_2: y(x^2 + (a-2b^{\nicefrac{1}{2}})y) = 1.
\]
(Note that $\Jac(C_2)$ has CM by $\Q(i)$.)
\end{lemma}
\begin{proof}
This can be treated as a special case of Remark~\ref{R:pair of quartics}: $C_2$ is the quotient of~$C$ by $s$.
The equation for $C_2$ can be derived as in Lemma~\ref{lemma: S3 quartic}.
\end{proof}

\begin{remark}
It is also useful to consider twists of curves with extremal automorphism groups. In lieu of analyzing such cases here,
we will draw on prior results from \cite{ACLLM18}, \cite{FLS18}, \cite{FS16}.
\end{remark}

\subsection{Twisting constructions}
\label{subsec:polarizations of twists}

In \cite[\S 4.2.5]{FKS19}, a number of examples of abelian threefolds with particular Sato--Tate groups were constructed by twisting abelian varieties which are $\Qbar$-isogenous to a power of an elliptic curve with complex multiplication. We make this construction more explicit so that we can confirm numerically that the resulting $L$-functions agree with the computed Sato--Tate groups. In these cases, equidistribution is known unconditionally by an application of Hecke's theory of Gr\"o\ss encharacters (see \cite[Proposition~16]{Joh17}, for example).

Throughout \S\ref{subsec:polarizations of twists}, let $E$ be an elliptic curve over $k$ with CM by the maximal order of an imaginary quadratic field $M \not\subseteq k$ of class number $1$ and discriminant $-D$.
Let $r$ be a positive integer. 

\begin{definition} \label{D: twisting construction0}
Let $\rho_M: G_{kM} \to \GL_r(\cO_M)$ be a continuous representation for the discrete topology on $\cO_M$ with kernel $G_L$.
By an \emph{isogeny twist} of $E^r$ by $\rho_M$, we will mean an abelian variety $A_\rho$ over $k$ of dimension $r$ for which
 $A_{\rho,L}$ admits an isogeny to $E_L^r$
and the representation of $G_{kM}$ on $\Hom(E_L, A_{\rho,L})_{\Q} \cong M^r$ is $M$-equivalent to $\rho_M$ (meaning isomorphic as $M$-linear representations).

From this definition, it is not clear whether isogeny twists exist or are unique; one necessary condition for existence can be described as follows.
Let $\rho_M^c$ be the representation obtained from $\rho_M$ by precomposing with the action of $\Gal(kM/k)$ on $G_{kM}$
and postcomposing with the entrywise action of $\Gal(kM/k)$ on $\GL_r(\cO_M)$.
Then for an isogeny twist to exist, it is necessary that $\rho_M^c$ be $M$-equivalent to $\rho_M$.
The stronger condition that $\rho_M^c$ be $\cO_M$-equivalent to $\rho_M$ (meaning isomorphic as $\cO_M$-linear representations) is sufficient (see Remark~\ref{remark: twisting construction1}) but not necessary (see Remark~\ref{remark: twisting construction2a} and so on).
\end{definition}

\begin{remark} \label{R:effect of isogeny twist choices}
In Definition~\ref{D: twisting construction0}, the field $L$ is not the endomorphism field of $A_\rho$ but of $E \times A_\rho$;
the endomorphism field of $A_\rho$ is rather the kernel field of the projectivization of $\rho_M$ (as in \cite[Remark 4.4]{FS14}).

On a related note, the isogeny class of an isogeny twist is not in general uniquely determined by the image of the representation $\rho_M$ and the extension $L/k$; the choice of the 
isomorphism of $\Gal(L/kM)$ with the image of $\rho_M$ also has an effect. However, $\rho_M$ alone suffices to determine the algebraic Sato--Tate group of $A_\rho$ (as the latter is generated by diagonal matrices plus the image of $\rho_M$) and hence the Sato--Tate group.
\end{remark}

In order to construct isogeny twists, it is not enough to have equivalence of $\rho_M$ with~$\rho_M^c$ over~$M$.
But when we have equivalence over $\cO_M$, the situation is relatively straightforward.
\begin{remark} \label{remark: twisting construction1}
In Definition~\ref{D: twisting construction0}, take $k = \Q$
and suppose that the representations $\rho_M$ and~$\rho_M^c$ are $\cO_M$-equivalent, not just $M$-equivalent. We can then express this isomorphism via some matrix $V \in \GL_r(\cO_M)$ (we may omit mention of $V$ when it is the identity matrix).

In this case, we can construct an isogeny twist by applying Definition~\ref{D:twist} as follows.
We first view $\rho_M$
as a 1-cocycle on $\Gal(L/M)$ valued in $\GL_r(\cO_M))$ for the trivial action; this yields an abelian variety $A_{\rho,M}$ over $M$ (the twist of $E_M^r$ by $\rho_M$). Then use $V$ to extend $\rho_M$ to a 1-cocycle on $\Gal(L/\Q)$ valued in $\GL_r(\cO_M) \rtimes \Gal(M/\Q) \cong \Aut(E_M^r)$ by choosing a lift $\tilde{c} \in \Gal(L/\Q)$ of the complex conjugation $c \in \Gal(M/\Q)$ and mapping $\tilde{c}$ to $V$; this yields an abelian variety $A_\rho$ over $\Q$. 

The choice of $\tilde{c}$ does not affect the resulting class in $H^1(\Gal(L/\Q), \Aut(E_M^r))$, and thus does not affect the isomorphism class of $A_\rho$. By contrast, the choice of $V$ does affect the isomorphism class of $A_\rho$, although not the resulting Sato--Tate group (see Remark~\ref{R:effect of isogeny twist choices}).
\end{remark}

We next consider what happens in some cases where Remark~\ref{remark: twisting construction1} does not apply.
In these cases, we can replace $E^2$ with an isogenous variety so as to replace $\rho_M^c$ with a different integral model.

\begin{remark}  \label{remark: twisting construction2a}
In Definition~\ref{D: twisting construction0}, take $r = 2$, fix a positive integer $m$,
and suppose that $\rho_M$ is valued in the subring
\[
R = \left\{ \begin{pmatrix} a & b \\ c & d \end{pmatrix} \in M_2(\cO_M): b \in m \cO_M \right\} \subseteq \End(E_{kM}^2).
\]
Let $\rho_M^{c\prime}$ be the representation obtained from $\rho_M$ by precomposing the action of $\Gal(kM/k)$ on $G_{kM}$
and postcomposing with the involution of $R$ given by
\[
u \mapsto g^{-1} \overline{u} g, \qquad g = \begin{pmatrix} 0 & m \\ 1 & 0 \end{pmatrix}.
\]
Suppose that $\rho_M$ is $\cO_M$-equivalent to $\rho_M^{c\prime}$.
Then the twist $A_{\rho,M,0}$ of $E_{kM}^2$ by $\rho_M$
(in the sense of Remark~\ref{remark: twisting construction1})
 contains a subscheme $G$ isomorphic to $E_{kM}[m]$.
If $E$ admits a cyclic $m$-isogeny, we can identify its kernel with a subgroup of $G$ and then quotient $A_{\rho,M,0}$ by this subgroup
to obtain an abelian variety $A_{\rho,M}$ isomorphic to its conjugate by $\Gal(kM/k)$. In this way, we obtain an abelian variety $A_\rho$ over $k$.
\end{remark}

\begin{remark}  \label{remark: twisting construction2}
An important case of Remark~\ref{remark: twisting construction2a} that can be carried out with $k = \Q$ is $M = \Q(\sqrt{-3})$, $m = 2$.
In this case, we may take $E$ to be any curve with $j(E) = 0$; for definiteness, we take the particular curve $y^2 = x^3 - 1$ in applications.
\end{remark}

\begin{remark}  \label{remark: twisting construction3}
An important case of Remark~\ref{remark: twisting construction2a} that cannot be carried out with $k = \Q$ is $M = \Q(i)$, $m = 3$,
as there is no elliptic curve $E$ over $\Q$ with CM by $M$ such that $E$ admits a cyclic $m$-isogeny. (This corresponds via class field theory to the fact that the quadratic order $\mathbb{Z}[3i]$ has class number 2 rather than 1.)
In this case, the best we can do is to take $k = \Q(\sqrt{3})$, in which case we may take $E$ to be the curve  \href{https://www.lmfdb.org/EllipticCurve/2.2.12.1/9.1/a/3}{9.1-a3}.

To keep notation somewhat uniform across different twisting constructions, we only apply this construction in the case where $\rho_M$
is the restriction of a representation $G_M \to \GL_r(\cO_M)$ whose kernel field is linearly disjoint from $k$,
and write $\rho_M$ also for the latter representation.
\end{remark}

\begin{remark} \label{remark: twisting construction3b}
When applying Remark~\ref{remark: twisting construction3},
we will obtain an $L$-function which ``looks like'' the $L$-function of an abelian surface over~$\Q$. In particular,
$A_\rho$ will be $(2,2)$-isogenous to its Galois conjugate (because \href{https://www.lmfdb.org/EllipticCurve/2.2.12.1/9.1/a/3}{9.1-a3} is 2-isogenous to its conjugate), but it is unclear whether it can be chosen to be isomorphic to its Galois conjugate.

It is possible that Sato--Tate groups realized in this manner do occur for K3 surfaces over~$\Q$, rather than abelian surfaces: a given K3 surface can sometimes be realized in multiple ways as the Kummer surface of an abelian surface, using the criterion of Nikulin \cite{Nik75}.
\end{remark}

Since twisting is not guaranteed to give a Jacobian or even a principally polarizable abelian variety, we include a brief discussion of polarizations of twists.

\begin{remark} \label{remark: twisting construction induced polarization}
Suppose that $\rho_M$ is absolutely irreducible;
then $A_{\rho,M}$ will admit an invariant polarization which is unique up to rescaling.
We may compute this polarization as follows. Let $L$ be the kernel field of $\rho_M$.
The matrix
\[
\sum_{g \in \Gal(L/kM)} \rho_M(g)^* \rho_M(g) \in M_r(M)
\]
(where $*$ denotes the Hermitian transpose) is positive definite and hence invertible; let $U \in M_r(\cO_M)$ be a $\Q_{>0}$-scalar multiple (which we choose to be ``as primitive as possible'' modulo the next paragraph).
Then for $\psi: E^r \to E^r$ the product polarization, $U \circ \psi$ is a polarization on $A_{\rho,M}$ whose type can be read off from the elementary divisors of $U$.

In Remark~\ref{remark: twisting construction1}, there are no extra constraints on $U$, and the polarization type is exactly equal to the sequence of elementary divisors; in particular, the degree of the polarization equals $\det(U)$.
By contrast, in Remark~\ref{remark: twisting construction2} and Remark~\ref{remark: twisting construction3}), we must ensure the top-right entry of $U$ is divisible by $m$ so that $U \circ \psi$ induces a polarization on $A_{\rho,M}$; the degree of the polarization equals $\det(U)/m$.
\end{remark}

\begin{remark}
In our applications, the polarizations given by Remark~\ref{remark: twisting construction induced polarization}
will have degrees in $\{1,2,3,6\}$ (see Table~\ref{table: maximal ST groups}). 
In particular, we obtain some examples of abelian surfaces over~$\Q$ which carry a
$(1,6)$-polarization, and which we believe (but have not checked) admit no isogeny defined over $\Q$ to an abelian surface carrying a polarization of lower degree.
\end{remark}

\subsection{Exhaustive searches}

In addition to the explicit constructions described in the next section, we conducted several large scale enumerations of genus 3 curves in an effort to realize as many Sato--Tate groups as possible using Jacobians of curves defined over $\Q$; these enumerations were part of a larger project to construct a database of genus 3 curves to be incorporated into the \LMFDB. See \cite{Sut19} for a description of the enumeration of all smooth plane quartics over $\Q$ defined by models with integer coefficients of absolute value at most~9; this amounts to more than $10^{17}$ curves, of which only those with absolute discriminant at most~$10^7$ (about 80,000 curves) were retained. Based on an analysis of this dataset we constructed a larger dataset of about a million curves using a similar enumeration that retained smooth plane quartics whose discriminants are 7-smooth (divisible only by primes~$p\le 7$), or of absolute value at most ~$10^9$.

The motivation for imposing a smoothness constraint is that it substantially increases the likelihood of finding exceptional Sato--Tate groups. 
If $S$ denotes the set of $\Q$-isomorphism classes of smooth plane quartics represented by integral models with coefficients of absolute value at most~9, the proportion of elements of $S$ with non-generic Sato--Tate group is much higher in the subset with 7-smooth discriminants than it is in the subset with absolute discriminants bounded by $10^9$.  Indeed, among the approximately 2.6 million elements of $S$ with absolute discriminant bounded by $10^9$, more than 98 percent have Sato--Tate group $\USp(6)$, while less than half of the approximately 1.0 million elements of $S$ with 7-smooth minimal discriminants have Sato--Tate group $\USp(6)$.

We also constructed a large database of hyperelliptic curves of genus 3 over $\Q$ using a generalization of the method described in \cite{BSSVY16} for genus 2 curves.  This dataset includes approximately 4.9 million genus 3 curves defined by an integral model of the form $y^2=f(x)$ with coefficients of absolute value at most 48 and 7-smooth discriminants.

To analyze the Sato--Tate groups of these curves we first filtered out those we could easily prove have generic Sato--Tate group $\USp(6)$ by computing the reductions modulo $\ell=2,3,5$ of $L$-polynomials $L_p(T)$ at small primes $p$.  The set of characteristic polynomials that arise in $\GSp_6(\F_\ell)$ is strictly larger than the sets that arise in any of its maximal subgroups, and it suffices to collect a subset of characteristic polynomials that rule out any maximal subgroup. By the Chebotaryov density theorem, this approach is guaranteed to eventually identify any abelian threefold $A/\Q$ whose mod-$\ell$ Galois image is equal to $\GSp_6(\F_\ell)$, and on average one expects this to happen quite quickly: for $\ell=2,3,5$ respectively, approximately $8,5,4$ random elements of $\GSp_6(\F_\ell)$ suffice, on average, to rule out every maximal subgroup.  In the typical case where $A$ has surjective mod-$\ell$ representation for $\ell=2,3,5$ one heuristically expects to be able to prove this for at least one of these $\ell$ by computing $L$-polynomials $L_p(T)$ for just the first three primes $p>5$ of good reduction, on average.  It follows from Proposition~\ref{prop: surjective mod-ell} below that an abelian threefold that has a surjective mod-$\ell$ image for any prime $\ell$ necessarily has large monodromy and Sato--Tate group $\USp(6)$.

\begin{proposition}\label{prop: surjective mod-ell}
Let $C/\Q$ be a smooth and projective curve of genus $g$, and let $A$ denote its Jacobian. Let $\rho_{A,\ell}\colon G_k\to \Aut(T_\ell(A))\simeq \GSp_{2g}(\Z_\ell)$ denote the $\ell$-adic representation attached to $A$, and let $\bar\rho_{A,\ell}$ denote the mod $\ell$ representation. If $\ell\ne 2g$ and $\Sp_{2g}(\F_\ell)\subseteq \bar \rho_{A,\ell}(G_k)$, then $\ST(A)=\USp(2g)$.
\end{proposition}

\begin{proof}
Let $V_\ell(A)$ denote $T_\ell(A)\otimes \Q_\ell$, and let $G_{\ell}^{\Zar}$ denote the Zariski closure of the image of $\varrho_{A,\ell}$ in $\GL(V_\ell(A))$. It will suffice to show that $G_\ell^{\Zar}=\GSp_{2g,\Q_\ell}$, in which case $G_{\ell'}^{\rm zar}=\GSp_{2g,\Q_{\ell'}}$ for every prime $\ell'$. Indeed, if $G_\ell^{\Zar}=\GSp_{2g,\Q_\ell}$, then the inclusions  $G_\ell^{\Zar}\subseteq \MT(A)_{\Q_\ell}\subseteq \GSp_{2g,\Q_\ell}$ provided by \cite[Prop.\,6.2]{Del82} become equalities. Moreover, by \cite[Thm.\,4.3]{LP95}, if these equalities hold for one prime $\ell$, then they hold for every prime $\ell'$.

Suppose that $g=1$. If $G_\ell^{\Zar}\subsetneq \GSp_{2,\Q_\ell}$, then $A$ has CM by an order $\cO$ in an imaginary quadratic field $M$, and $\bar\rho_{A,\ell}(G_M)\subseteq \Aut_\cO(A[\ell])\simeq (\cO/\ell\cO)^\times$ cannot contain any index $2$ subgroup of $\Sp_2(\F_\ell)=\SL_2(\F_\ell)$ for $\ell >2$.

Suppose that $g=2$ and $\ell=2$. Then $C$ is given by $y^2=f(x)$, for some separable polynomial $f\in \Q[x]$ of degree $6$. It is well known that the Galois group of $f$ coincides with $\bar\rho_{A,2}(G_\Q)$. Therefore, if $\bar\rho_{A,2}(G_\Q)$ contains $\Sp_4(\F_2)\simeq \sym{6}$, then the Galois group of $f$ must be~$\sym 6$. In this case, by \cite[Theorem 1.1]{Za10}, we have $\End(A_\Qbar)=\Z$ and therefore $G_\ell^{\Zar}= \GSp_{4,\Q_\ell}$.

Finally, if $g\geq 2$ and $\ell \geq 3$ then \cite[Theorem 4.1]{Va04} shows that if $\Sp_{2g}(\F_\ell)\subseteq \bar\rho_{A,\ell}(G_\Q)$, then $\rho_{A,\ell}(G_\Q)=\GSp_{2g}(\Z_\ell)$, which clearly implies $G_\ell^{\Zar}= \GSp_{2g,\Q_\ell}$.
\end{proof}

\begin{remark}
For $\ell=2,3$ there are maximal subgroups of $\GSp_4(\F_\ell)$ in which every characteristic polynomial of $\GSp_4(\F_\ell)$ arises, but one can use $\ell=5,7,11,13$.  Similarly, for $\GSp_2(\F_\ell)=\GL_2(\F_\ell)$ one should avoid $\ell=2,3,5,7$ and instead use $\ell=11,13,17,19$.

In principle, it should be possible to extend these criteria to $\GSp_{2g}(\F_\ell)$ for other pairs $(g, \ell)$ using an enumeration of the maximal subgroups based on Aschbacher's classification theorem \cite{Asc84}. See the tables in \cite{BHRD13} for extensive data.
\end{remark}

\begin{remark}\label{rem: hyperelliptic mod-2}
When $A$ is the Jacobian of a genus $g\ge 3$ hyperelliptic curve $y^2=f(x)$, the image of $\bar\rho_{A,2}$ is never all of $\GSp_{2g}(\F_2)$; it is a subgroup isomorphic to the Galois group $\Gal(f)$ of the splitting field of $f$ (the 2-torsion field of $A$).  This is a proper subgroup because $\#\sym{2g+2}<\#\GSp_{2g}(\F_2)$ for all $g\ge 3$. But \cite[Theorem 1.1]{Za10} implies that $\End(A_\Qbar)=\Z$ when $\Gal(f)$ is $\alt{n}$ or $\sym{n}$, where $n=\deg f$, a condition that is easy to check.  If $\End(A_\Qbar)=\Z$ and $g$ is odd, 2, or 6, we may apply \cite{Ser86} to deduce that $\ST(A)=\USp(2g)$.
This includes all $g\le 3$, and we applied this criterion to the hyperelliptic genus 3 curves in our data set.
\end{remark}

Having filtered out curves that we could easily prove have generic Sato--Tate group, we attempted to compute the geometric real endomorphism type of the Jacobian $A$ of each curve using three different software packages, each of which implements a different heuristic algorithm for computing $\End(A_\Qbar)_\R$ that is practical to apply to millions of curves en masse:
\begin{enumerate}[{\rm(i)}]
\item The \textsc{endomorphisms} software repository \cite{CS19} provides a \Magma{} implementation that uses a period matrix calculation to compute a numerical approximation of $\End(A_\Qbar)_\R$ to a specified precision; the results are not guaranteed to be correct, but when performed to sufficient precision almost always are.  It is possible to certify the results after the fact as described in \cite{CMSV19}, but we did not attempt this.
\item The \textsc{monodromy} software repository \cite{Zyw21} provides a \Magma{} implementation of the algorithm described in \cite{Zyw22}, which uses a given set of $L$-polynomials at good primes $p\le B$ to compute an upper bound on the geometric endomorphism algebra $\End(A_\Qbar)_\Q$ (and a lower bound on the Mumford--Tate group) that is guaranteed to be correct for sufficiently large $B$ (but without an effective bound on such a $B$).
\item The \textsc{crystalline\_obstruction} software repository \cite{C21} provides a \Sage{} implementation of the algorithm described in \cite{CS21}, which uses a $p$-adic computation at a prime of good reduction to compute upper bounds on the rank $r$ of the N\'eron--Severi group of $A_\Qbar$ and the dimension $d$ of the geometric endomorphism algebra $\End(A_\Qbar)_\Q$; by varying the choice of $p$ and taking minima one expects to eventually obtain tight bounds.  As shown in \cite{CFS19}, the pair $(r,d)$ coincides with the 2-simplex of moments of $\ST(A_\Qbar)$, which by Theorem~\ref{thm: minimal set of 3-diagonals} determines $\ST(A)^0$ and $\End(A_\Qbar)_\R$.
\end{enumerate}

We applied these three methods to each curve in stages, increasing the precision and/or number of primes used in each stage, until all three methods output the same result.   We chose bounds on the precision and number of primes so that each method took approximately the same time, with computation times roughly doubling with each stage; the number of curves surviving each stage declined at a much faster rate (roughly a factor of~8), so that the total time was dominated by the first stage, where each method spent 5--10 seconds on each curve.  No more than 7 stages were required for any curve.
The results are summarized in Table~\ref{table: search counts} below, in which the notation $\| f\|$ denotes the maximum of the absolute values of the integer coefficients of a  polynomial $f$.

\begin{remark}
We encountered a small number of genus 3 curves with Sato--Tate group $\USp(6)$ (as proved by the upper bounds on $\End(A_\Qbar)_\Q$ given by methods (ii) and (iii) above), that nevertheless appear to have non-surjective mod-$\ell$ Galois representations for $\ell=2,3,5$.  One example is the hyperelliptic curve
\[
 y^2 = 4x^8 - 24x^6 + 45x^4 - 16x^3 - 25x^2 + 16x,
\]
for which the $L$-polynomials at good primes $p\le 2^{24}$ do not imply the surjectivity of $\bar\rho_{A,\ell}$ for any of $\ell=2,3,5$.
In this example the mod-2 image is isomorphic to the group $\group[F_5]{20}{3}$ of order~20 (so Remark~\ref{rem: hyperelliptic mod-2} does not apply), the mod-3 image appears to be an index 364 subgroup of $\GSp_6(\F_3)$, and the mod-5 image appears to be an index 7812 subgroup of $\GSp_6(\F_5)$.
\end{remark}

\begin{table}[htp!]
\begin{tabular}{|c|c|r|r|r|r|}
\hline&&&&&\\[-12pt]
type & 2-simplex & hyperelliptic & quartic & quartic cover & Picard\\[1pt]
\hline&&&&&\\[-11pt]
$\bA$ & $(1,1)$ & 4289200 & 451060 &- & -\\
$\bB$ & $(1,2)$ & 62228 & 24318 & - &  4826346\\
$\bC$ & $(2,2)$ & 456806 & 366475 & 5779677 & -\\
$\bD$ & $(2,3)$ & 35510 & 26766 &  314985 & -\\
$\bE$ & $(3,3)$ & 36846 & 90508 & 2507404 & -\\
$\bF$ & $(3,4)$ & 4238 & 18551 & 474103 & -\\
$\bG$ & $(3,5)$ & 166 & 962 & 21668 & -\\
$\bH$ & $(3,6)$ & 1382 & 137 & 202 & 7678\\
$\bI$ & $(4,5)$ & 30772 & 23821 & 435686 & -\\
$\bJ$ & $(4,6)$ & 4629 & 5703 &  61813 & 578221\\
$\bK$ & $(5,9)$ & 466 & 12050 & 175296 & -\\
$\bL$ & $(5,10)$ & 1479 & 642 & 7311 & 40937 \\
$\bM$ & $(6,9)$ & 116 & 301 & 4713 & -\\
$\bN$ & $(9,18)$ & 712 & 623 &  5643 & 2598\\[1pt]
\hline&&&&&\\[-11pt]
 &  & 4922550 & 1021917 & 9788501 & 5455780 \\[1pt]
\hline
\end{tabular}
\medskip

\caption{Counts of heuristically computed absolute types of Sato--Tate groups for Jacobians of genus~3 curves over $\Q$ with 7-smooth discriminants, including hyperelliptic curves $y^2=f(x)$ with $\|f\|\le 48$, quartic curves $f(x,y,z)=0$ with $\|f\|\le 9$, quartic covers  $y^4 + h(x,z)y^2 = f(x,y,z)$ with $\|f\|,\|h\|\le 64$ and Picard curves $y^3=f(x)$ with $\|f\|\le 2187$.}\label{table: search counts}
\end{table}

\subsection{Explicit computation of \texorpdfstring{$L$}{L}-functions of abelian threefolds}
\label{subsec:explicit L-functions}

The explicit realization of abelian threefolds $A$ matching each of the Sato--Tate groups in our classification was facilitated at several stages by the computation of $L$-functions. This was done both in the process of searching for specific examples in large datasets of curves such as those described in the previous section, and also to empirically corroborate our computations of moments and the matrices of point densities for each of the 410 Sato--Tate groups.  For each of the 33 maximal Sato--Tate groups whose realizations via explicit abelian threefolds are described in the next section we computed $L$-polynomials $L_\p(T)$ for all degree $1$ primes of norm bounded by some $B\ge 2^{30}$ outside of a small excluded set of primes $S$.  The set $S$ typically consisted of the primes of bad reduction, but in some cases it was convenient to exclude a few others; none of the excluded sets contained more than ten primes, and their exclusion has no impact on the Sato--Tate statistics we are interested in.

Once computed this data can then be used to compute Sato--Tate statistics for the base change of $A$ to any number field $K$ by including each Euler factor with multiplicity equal to the number of degree 1 primes of $K$ above $\p$ (the contribution from primes of higher degree is both asymptotically and practically negligible).  For each of the 33 abelian threefolds $A$ realizing one of the maximal Sato--Tate groups, we computed Sato--Tate statistics of the base change of $A$ to every subfield of its endomorphism field, thus obtaining Sato--Tate statistics for all 410 Sato--Tate groups (with multiplicity in cases where maximal Sato--Tate groups with the same identity component have intersection larger than the identity component).  Visualizations of these statistics are available at
\begin{center}
\url{https://math.mit.edu/~drew/st3/g3SatoTateDistributions.html}
\end{center}

We used a wide range of methods to compute the $L$-polynomials of $L_\p(T)$ for an abelian variety $A$ over $\Q$ or a quadratic field at degree 1 primes $\p$ with $\Norm(\p)\le B$ outside a small set of excluded primes that includes all primes of bad reduction.  Several of our explicit realizations of abelian threefolds with maximal Sato--Tate groups involve products of abelian varieties of dimension 1 or~2, which we exploited whenever possible.

The various methods we used are summarized below.
For purposes of comparison we give asymptotic time complexity estimates in terms of~$p\colonequals\Norm(\p)$, which should be understood as averages over $p\le B$ where average polynomial-time algorithms are noted, along with a rough estimate of the total time to compute $L_p(T)$ for all $p\le B=2^{30}$ outside the excluded set using a single core of a modern CPU (we used a 3.1GHz Intel i9-9960X CPU in our computations).  In our actual computations we had up to 128 cores working in parallel and were able to extend the computations to $B=2^{32}$ in most cases.

\begin{itemize}
\item For CM elliptic curves $E$ we compute $L_p(T)=1-t_pT+pT^2$ at primes that split in the CM field by solving the norm equation $4p=t^2-v^2D$ with $D=\disc(\End(E_\Qbar))$ using Cornacchia's algorithm and determining the Frobenius trace $t_p$ via \cite{RS10}, with $t_p=0$ at inert primes.  This has time complexity $(\log p)^{2+o(1)}$ and takes less than a core-minute for $B=2^{30}$ using the \smalljac{} library \cite{KS08}.

\item For non-CM elliptic curves $E$ we compute $L_p(T)=1-t_pT+pT^2$ by computing $t_p=p+1-\#E(\Fp)$ using the optimized baby-step giant-step algorithm implemented in \smalljac.  This has expected time complexity $p^{\nicefrac{1}{4}+o(1)}$ which is faster than an $(\log p)^{4+o(1)}$ polynomial-time approach for $B=2^{30}$, taking less than ten core-minutes.

\item For Jacobians $A$ of genus 2 curves $C/\Q$ that are twists of $y^2=x^5-x$ or $y^2=x^6+1$ we compute $L_p(T)$ using the algorithm in \cite[\S5.2.1]{FS14}, which is implemented in \smalljac{}.  In the situation where we know the endomorphism field $K$ of the twist (as is the case in all the examples we construct), the ambiguities noted in \cite[\S6.2.1]{FS14} can be efficiently distinguished by computing Frobenius elements in $\Gal(L/M)$, where $M$ is the CM field of the elliptic curve $E/\Q$ whose square is $\Qbar$-isogenous to $A$, and $L$ is the minimal field over which all homomorphisms $E_\Qbar\to A_\Qbar$ are defined; the triples $(M,K,L)$ are listed in Tables 6 and 7 of \cite{FS14}. This has time complexity $(\log p)^{2+o(1)}$ and takes less than a core-minute for $B=2^{30}$ using \smalljac{}.

\item For Jacobians of other genus 2 curves $C/\Q$ we use the average polynomial-time algorithm in \cite{HS14,HS16} to compute $L_p(T)=1+a_1T+a_2T^2+a_1pT^3+p^2T^4$ modulo $p$, which for $p>64$ determines $a_1\in \Z$, leaving at most 5 possibilities for $a_2$; see \cite[Prop.~4]{KS08}.  All but one possibility can be eliminated in $(\log p)^{2+o(1)}$ expected time by sampling random points $P\in \Jac(C_p)$ and testing if $L_p(1)P= 0$ for each value of $L_p(1)$ implied by the choices for $a_2$ and doing the same on the Jacobian of the quadratic twist of $C_p$ with $a_1$ negated.  The average expected time is $(\log p)^{4+o(1)}$, and the total time for $B=2^{30}$ is about a core-day (less if the curve has rational Weierstrass points, as explained in \cite[\S 6.1]{HS16}).

\item For Jacobians of generic Picard curves $y^3=f(x)$ over $\Q$ we use the algorithm in \cite{Sut20} to compute $L_p(T)$ modulo $p$ and then apply the algorithm in \cite{AFP21} to lift it to $\Z[T]$, which takes negligible time.  The average complexity is $(\log p)^{4+o(1)}$ and the  time for $B=2^{30}$ is about 8 core-hours.

\item For Jacobians of genus 3 hyperelliptic curves $y^2=f(x)$ over $\Q$ we compute $L_p(T)$ modulo $p$ using the average polynomial time algorithm of \cite{HS14,HS16}.  This leaves $O(p^{\nicefrac{1}{2}})$ possibilities for $L_p(T)$ which we distinguish by performing a baby-step giant-step search along an arithmetic progression determined by $L_p(T)$ modulo $p$ to compute $L_p(1)$ and $L_p(-1)$ using the group computations in the Jacobian of the reduction of the curve and its quadratic twist, which for $p>1640$ uniquely determines $L_p(T)$ \cite[Lemma~4]{Sut09}; see \cite{KS08,Sut09} for details of this computation, which is implemented in \smalljac{}.  The average complexity for computing $L_p(T)$ modulo $p$ is $(\log p)^{4+o(1)}$ and the expected cost of lifting to $\Z[T]$ is $p^{\nicefrac{1}{4}+o(1)}$; for $B=2^{30}$ the time spent on lifting is less than 20 percent of the total time, which is 3 or 4 core-days.

\item For Jacobians of genus 3 curves that are degree-2 covers of pointless conics over $\Q$ (these admit a hyperelliptic model $y^2=f(x)$ over a quadratic extension but not over~$\Q$), we use the average polynomial-time algorithm described in \cite{HMS16} to compute $L_p(T)$ modulo~$p$, and then lift $L_p(T)$ to $\Z[T]$ as above (over $\Fp$ these curves necessarily admit a hyperelliptic model).  The complexity bounds are the same as in the rationally hyperelliptic case but the constant factors are worse.  For $B=2^{30}$ the total time is one or two core-weeks, of which perhaps a core-day is spent lifting $L_p(T)$.

\item For Jacobians of smooth plane quartic curves over $\Q$ we compute $L_p(T)$ modulo~$p$ using the average polynomial time algorithm of \cite{CHS23}, which has an average complexity of $(\log p)^{4+o(1)}$ and takes about two core-weeks for $B=2^{30}$.  We did not implement the lifting phase in the generic case (we can compute trace statistics without this), but in principle it can be done in $p^{\nicefrac{1}{4}+o(1)}$ time using group operations in the Jacobian of the curve and the corresponding trace zero variety (the kernel of $\pi+1$, where $\pi$ is the Frobenius automorphism) via the addition formulas in \cite{FOR08}.

\item For smooth plane quartic curves over $\Q$ that are covers of genus 1 curves, we compute $L_p(T)$ modulo $p$ for both curves, leaving at most five options for $L_p(T) \in \Z[T]$, which we can distinguish in $(\log p)^{2+o(1)}$ expected time using group operations in the Jacobian of the genus 2 curve determined by \cite[Proposition~2.1]{RR18} (which may be defined over a cubic extension of $\Q$, see Remark~\ref{R:Ritzenthaler-Romagny}).  As above, the total time for $B=2^{30}$ is one or two core-weeks.

\item For Jacobians $A$ of twists of the Fermat and Klein quartics $X^4+Y^4+Z^4=0$ and $X^3Y+Y^3Z+XZ^3=0$, both covers of genus 1 curves,  we can compute $L_p(T)$ modulo $p$ as above and apply the method described in \cite{FLS18} to lift $L_p(T)$ to $\Z[T]$, which takes $(\log p)^{2+o(1)}$ time.  Alternatively (and more efficiently), we can compute $L_p(T)$ using an isogeny twist of $E^3$ by a 3-dimensional Artin representation~$\rho_M$, where $E/\Q$ is an elliptic curve with CM by $M$ whose cube is $\Qbar$-isogenous to $A$.  This has complexity $(\log p)^{2+o(1)}$ following the precomputation of $\rho_M$.  Using a \Magma{} implementation of the approach sketched in \S\ref{subsec:explicit isogeny twists} below, it takes a few core-hours for $B=2^{30}$; a \smalljac{} implementation would likely reduce this to a few core-minutes.

\item For abelian threefolds that are the restriction of scalars of a CM elliptic curve defined over a cubic number field we can compute $L_p(T)$ using the Euler factors of the corresponding Hecke character at the primes above $p$ in $(\log p)^{2+o(1)}$ time; this takes less than a minute for $B=2^{30}$ using \smalljac{}.

\item For abelian surfaces and threefolds $A$ that are isogeny twists of squares or cubes of CM elliptic curves over $\Q$ or $\Q(\sqrt{3})$ by Artin representations $\rho_M$ of dimension 2 or 3 over the CM field $M$ we compute $L_p(T)$ using the Hecke character of the CM elliptic curve and computing Frobenius elements via the algorithm of \cite{DD13} as described below.  This has time complexity $(\log p)^{2+o(1)}$ once a suitable precomputation for~$\rho_M$ has been completed, but that precomputation may take a long time when the endomorphism field of $A$ is large; in the most extreme cases the kernel field of $\rho_M$ is a degree 1296 extension of $M$ and it takes about 36 hours, but in most cases only a few minutes (see Remark~\ref{remark: precomputation} for how this could be improved).   Using a \Magma{} implementation of the approach sketched in \S\ref{subsec:explicit isogeny twists} below, it takes a few core-hours for $B=2^{30}$; as noted above, a \smalljac{} implementation would likely reduce this to a few core-minutes.
\end{itemize}

\begin{remark}
All 33 of the maximal Sato--Tate groups can be realized using a construction that does not involve the Jacobian of a generic smooth plane quartic or a hyperelliptic cover of a pointless conic; we give at least one such example for each group in \S \ref{section: realizability}.  It was thus never necessary to apply the algorithms described above for these cases in order to compute moment statistics with $B=2^{30}$, although we did perform computations with smaller values of~$B$ as a sanity check and to obtain timing estimates.  The most time-consuming computations we actually performed were for the generic case $\bA$, which can be realized by a genus 3 hyperelliptic curve with a rational Weierstrass point;  this took less than 3 core-days.
\end{remark}

\subsection{Computing \texorpdfstring{$L$}{L}-polynomials of isogeny twists}
\label{subsec:explicit isogeny twists}

As described in the next section, the~12 maximal Sato--Tate groups of type $\bN$ can all be realized using abelian varieties that are isogeny twists of the cube of a CM elliptic curve.
In this section we sketch the method we used to compute $L$-polynomials for these realizations. 

Let $E/k$ be an elliptic curve with CM by $M$, and let $A$ be an isogeny twist of $E^r$ by an Artin representation $\rho_M\colon \Gal(L/M)\to \GL_r(\C)$ as in Definition~\ref{D: twisting construction0}.  We may compute the $L$-polynomials of $A$ using its interpretation as the base change from $kM$ to $k$ of the twist of the Artin $L$-function of $\rho_M$ by one of the Hecke characters $\psi_E$ corresponding to $E$.  For the purpose of computing Sato--Tate distributions, we can restrict to degree 1 primes of~$k$, since the contribution from other primes is asymptotically negligible.If $\p$ is a degree 1 prime of $k$ of good reduction for $E$ that splits as $\p=\mathfrak r\bar{\mathfrak r}$ in $kM$ and is unramified in $L$, then we may compute the $L$-polynomial of $A$ as
\[
L_\p(T) = \mathrm{Norm}_{M/\Q}f(\psi_E(\mathfrak r) T),
\]
where $f$ is the reverse characteristic polynomial of $\rho_M(\Frob_\q)$, with $\q\colonequals\mathfrak r \cap M$, by applying \cite[Eq.\,(3-8)]{FS14}.

At degree 1 primes $\p$ of~$k$ of norm $p$ that are inert in $M$ we have
\[
L_\p(T) = \begin{cases}
1 + a_2T^2 + p^2T^4 &\text{if }r=2,\\
1+ a_2T^2+ pa_2T^4+p^3T^6  &\text{if }r=3,
\end{cases}
\]
where $a_2=r-4+(\zeta_s+\bar\zeta_s)^2$ with $s=2,4,6,8,12$ the order of $\rho_M(\Frob_{p\cO_M})$; see \cite[Prop.\,3.3]{FS14} for $r=2$ and see \cite[Cor.\,2.4]{FLS18} for $r=3$.

To compute $\rho_M(\Frob_\q)$ at a degree 1 prime $\q$ of $M$ that is unramified in $L$, we use the algorithm of Dokchitser and Dokchitser \cite{DD13}, which allows one to determine very efficiently the conjugacy class of $\Frob_\q$ in $\Gal(L/M)$ once a suitable precomputation has been done.  Given a monic polynomial $f\in \cO_M[x]$ with splitting field $L$ and distinct roots $r_1,\ldots,r_d\in \cO_L$, one chooses a set of polynomials $h\in \cO_M[x]$ of degree less than $d$ such that for each conjugacy class $C$ of $\Gal(L/M)$, the polynomials
\[
\Gamma_{C,h}(X) \colonequals \prod_{\sigma\in C}\left(X-\sum_{j=1}^d h(r_j)\sigma(r_j)\right)\in \cO_M[X]
\]
are coprime, and for each prime $\q$ of $M$ with residue field $\Fq\colonequals\cO_M/\q$ that is unramified in~$L$ the conjugacy class $C$ of $\Frob_\q$ is the unique conjugacy class for which
\begin{equation}\label{eq:traceval}
\Gamma_{C,h}\left(\Trace_{\frac{\Fq[x]}{(f(x))}/\Fq}\bigl(h(x)x^q\bigr)\right) = 0
\end{equation}
holds for all of the polynomials $h(x)$ that have been chosen.  Any $h$ for which the $\Gamma_{C,h}(X)$ are coprime will already distinguish the conjugacy class of $\Frob_\q$ for all but finitely many $\q$, by \cite[Thm.~5.3]{DD13}, and it typically takes only a few $h$ to handle all unramified primes.

The computation of the $\Gamma_{C,h}$ involves explicitly computing the $\Gal(L/M)$-set $\{r_1,\ldots, r_d\}$ and then trying various $h\in \cO_M[x]$ until a suitable set has been found (in our implementation we use sparse $h\in \Z[x]$ with small coefficients).  Once this precomputation has been done and the $\Gamma_{C,h}$ are fixed, one can very efficiently compute the conjugacy class of $\Frob_\q$ for any prime $\q$ of $M$ using arithmetic operations that involve only evaluation of the reduction map $\cO_M\to \Fq$ and arithmetic operations in $\Fq$ (which for us is a prime field).

\begin{remark}\label{remark: precomputation}
The cost of precomputing the polynomials $\Gamma_{C,h}$ is dominated by the time to explicitly compute the roots of a polynomial $f\in \cO_M[x]$ in its splitting field $L/M$.  This can be quite expensive when $[L:M]$ is large (in our most extreme case $[L:M]=648$ and it takes about 24 hours).  This cost can be dramatically reduced by computing the roots in a suitable $p$-adic field, together with the $\Gal(L/M)$-action, which is efficiently implemented in \Magma{}.  This requires recognizing the polynomials $\Gamma_{C,h}\in\cO_M[x]$ as polynomials whose coefficients have been embedded in this $p$-adic field.  It is in principle feasible to do this efficiently using a $p$-adic LLL computation with explicit height bounds, but we did not attempt this.
\end{remark}

We mention here an additional precomputation that can substantially improve the efficiency of computing \eqref{eq:traceval} in the cases of relevance to us here, where the groups $\Gal(L/M)$ of interest are not all that much larger than the degree $d$ of $f$ and most of the conjugacy classes have fewer than $d$ elements.  In this situation, it is advantageous to avoid computing the cycle type of $\Frob_\q$ and to avoid all of the matrix operations suggested in \cite[\S 5.9]{DD13}, which lead to an asymptotic complexity dominated by the cost of $O(d\log q)$ multiplications of $d\times d$ matrices over $\Fq$; this has a bit complexity of $O(d^{\omega+1}\log q\sfM(\log q))$, where $\omega$ is the exponent of matrix multiplication and $\sfM(n)$ is the cost of multiplying $n$-bit integers.  The values of $d$ and $\log q$ of interest to us are less than $50$, and in this range the time to compute $\rho_M(\Frob_\q)$ will be proportional to $d^4(\log q)^3$.

Let $A_f\in \cO_M^{d\times d}$ be the companion matrix of $f$, so $f(A_f)=0$, and let $\pi_\q\colon \cO_M\to \cO_M/\q=\Fq$ be the reduction map.  For any $n\ge 0$ we have
\[
\Trace_{\frac{\Fq[x]}{(f(x))}/\Fq}(x^n) = \Trace(\pi_\q(A_f)^n) = \pi_\q(\Trace(A_f^n)).
\]
If we precompute the vector $v_f\colonequals[1,\Trace(A_f),\Trace(A_f^2),\ldots, \Trace(A_f^{d-1})]\in\cO_M^d$, we can then compute~\eqref{eq:traceval} by computing the coefficient vector of $h(x)x^q\bmod f(x)$ in $\Fq^d$ and plugging the dot product of this vector with the reduction of $v_f$ to $\Fq^d$ into the reduction of $\Gamma_{C,h}(x)$ to~$\Fq$.  Assuming the cost of applying the reduction map $\pi_\p$ is $O(\mathsf{M}(\log q))$, the asymptotic complexity of this approach is $O(\log q\sfM(d\log q) + e\sfM(\log q))$, where $e=[L:M]$. Using na\"ive multiplication, the time to compute $\rho_M(\Frob_\q)$ will be proportional to $d^2(\log q)^3 +e (\log q)^2$.

All of the examples considered in the next section have $e\le 4d^2$, in which case we save a factor of $d^2$ over the approach above, which is a significant improvement.  In practice it takes substantially less time to compute $\rho_M(\Frob_\q)$ in our examples than it does to compute the cycle type of $\Frob_\q$ in the usual way (by computing $\gcd(x^{q^n}-x,f(x))$ for $n\le d/2$), which is the first step of the approach in \cite[\S 5.9]{DD13}.  But we should note that this will not be true in the generic case where $e$ is exponentially larger than $d$.

A \Magma{} implementation of this algorithm can be found in our \GitHub{} repository \cite{FKS21}, along with scripts that compute $L$-polynomials at degree 1 primes for all of the examples that involve isogeny twists.

\section{The lower bound via explicit realizations}
\label{section: realizability}

We now establish the ``lower bound'' assertion of Theorem~\ref{T:ST result} by establishing that among the 410 groups described in \S\ref{section:STgroups} and \S\ref{section: unitary},
the 33 maximal groups all occur as the Sato--Tate groups of certain abelian threefolds. 
This was done in a somewhat abstract fashion in \cite{FKS19} by systematic use of twisting constructions;
here, we exhibit explicit examples which may be verified both theoretically (by a rigorous computation of the endomorphism algebra and its Galois action) and empirically (by computing the $L$-series as described in Section~\ref{subsec:explicit L-functions} and comparing the resulting distribution to the theoretical prediction given by Section~\ref{sec:statistics}).

\subsection{Overview of techniques}
\label{subsec: realization techniques}

The examples we will be using are summarized in Table~\ref{table: maximal ST groups}. 
In the notation, we use the following elliptic curves defined over $\Q$:
\begin{align*}
E_0\colon & y^2 = x^3 - x + 1  \\
E_1\colon & y^2 = x^3 - 288x - 1872 \\
E_{-3}\colon & y^2 = x^3 - 1 \\
E_{-4}\colon & y^2 = x^3 - x \quad (E_{-4,n}\colon y^2 = x^3-nx) \\
E_{-7}\colon & y^2 + x y = x^{3} -  x^{2} - 2 x - 1 \\
E_{-8}\colon & y^2 = x^{3} +x^2-3x+1.
\end{align*}
Note that $E_0$ and $E_1$ do not have CM, whereas for $D<0$,
$E_{D}$ and $E_{D,n}$ have CM by the imaginary quadratic order of discriminant $D$.
We also use the curve $E_{-4,12-8\sqrt{3}}$ over $\Q(\sqrt{3})$, which is a model of the curve
\href{https://www.lmfdb.org/EllipticCurve/2.2.12.1/9.1/a/3}{9.1-a3}
used in Remark~\ref{remark: twisting construction3}.

The techniques used are listed below, in our order of preference. Here $k$ denotes the field of definition (equal to $\Q$ unless specified) and 
$d$ denotes the minimum (known) polarization degree for the resulting abelian threefold.
\begin{itemize}
\item
The product of three elliptic curves over $\Q$.
\item
The Weil restriction of an elliptic curve $C_{1,3}$ defined over the cubic number field $L_3 = \Q[\alpha]/(\alpha^3-\alpha^2+\alpha-2)$. The extension $L_3/\Q$ is not Galois; 
in the analysis we use $L_3'/\Q$ to denote its Galois closure
and $\alpha_1,\alpha_2,\alpha_3$ for the conjugates of $\alpha$ in $L_3'$.
\item
The product of an elliptic curve and the Jacobian of a curve $C_2$ of genus 2.
\item
The Jacobian of a curve $C_3$ of genus 3 (preferably hyperelliptic). When one of the previous constructions has been used
but a Jacobian is also available, we list both.
\item
The product of an elliptic curve with an abelian surface of the form $\Prym(C_3/C_1) = \ker(\Jac(C_3) \to \Jac(C_1))$,
where $\pi: C_3 \to C_1$ is a degree-2 morphism from a genus 3 curve to a genus 1 curve. In this case, $d=2$.
\item
The product of an elliptic curve with an abelian surface $A_{D}$ 
obtained as an isogeny twist of $E_{D}^2$ as described in Remark~\ref{remark: twisting construction1}.
In this case, $d$ is given by Remark~\ref{remark: twisting construction induced polarization}.
\item
The product of an elliptic curve with an abelian surface $A_{D[m]}$
obtained as an isogeny twist of $E_{D}^2$ as described in
Remark~\ref{remark: twisting construction2}
or Remark~\ref{remark: twisting construction3} (taking $k = \Q(\sqrt{3})$ in the latter case).
In this case, $d$ is given by Remark~\ref{remark: twisting construction induced polarization}.
(We add subscripts $a$ and $b$ to distinguish cases.)
\item
An abelian threefold constructed by twisting as in Remark~\ref{remark: twisting construction1}.
In this case, $d$ is given by Remark~\ref{remark: twisting construction induced polarization}.
\end{itemize}
In all cases that do not involve isogeny twists,
one may use the methods of \cite{CMSV19}, as implemented in \cite{CS19}, to verify that the Sato--Tate group
is as claimed; however, in all cases we give more direct (if less systematic) verifications. See \cite{FKS21} for code to perform some steps in these verifications.

\begin{remark}
The twisting constructions follow the approach described in \cite{FKS19}, using group-theoretic considerations
to formulate some Galois embedding problems.
Unlike in \cite{FKS19}, where these embedding problems are (mostly) solved uniformly using the invariant rings of complex reflection groups,
here we take a more \emph{ad hoc} approach in order to produce ``small'' explicit examples.
It would be desirable to find more explicit varieties realizing these Sato--Tate groups; however, the computed values of $d$
suggest obstructions to realizing certain cases using Jacobians or Pryms of double covers. (To promote these to genuine obstructions, one would need to make a more exhaustive analysis in the style of \cite{FG18} to show that no other twisting constructions are possible.)

For the constructions involving twists of abelian surfaces, it might be preferable to work with the associated Kummer surfaces, which are singular K3 surfaces where more of the polarization-preserving automorphisms can be seen directly. These can be described as elliptic fibrations via a construction of Inose (see 
\cite[Proposition~4.1]{Sch07}, for example, or see \cite[Theorem~11]{Kum08} for a more general result covering arbitrary genus $2$ curves).
However, it may be possible to find other models that make the symmetries more easily visible.
\end{remark}

\subsection{Examples: types \texorpdfstring{$\bA$}{A} to \texorpdfstring{$\bM$}{M}}

We remind the reader that the curves $C_{1,3}, C_2, C_3$ used in the following examples are listed in Table~\ref{table: maximal ST groups}.

\begin{example} \label{exa:type A}
The unique maximal group of type $\bA$ occurs as the Sato--Tate group of $\Jac(C_3)$ 
where $C_3$ is a hyperelliptic curve whose Jacobian has trivial geometric endomorphism ring, 
by a theorem of Zarhin \cite[Theorem~2.1]{Za00}.
\end{example}

\begin{example} \label{exa:type B}
The unique maximal group of type $\bB$ occurs as the Sato--Tate group of $\Jac(C_3)$
where $C_3$ is a hyperelliptic curve whose Jacobian has geometric endomorphism ring $\Z[i]$. 
To confirm this, we check\footnote{In a previous version of this paper this analysis followed \cite{Upt09}, but the description of the Galois image there is missing some scalars and so some arguments therein need to be reworked. We thank Pip Goodman for bringing this point to our attention.} that the mod-7 Galois representation has image equal to the subgroup of $\GSp(6, \F_7)$ generated by $\GU(3, \F_{7^2})$ and its centralizer; this then implies that the 7-adic Galois representation has image containing $\Unitary(3, \Z_{7^2})$
(see \cite[Theorem~2.2.5]{Va04}, for example). For $p = 3$ we have $L_p(T) \equiv T^6 + 6T^4 + 4T^2 + 6 \pmod{7}$;
a short \Magma{} computation shows that no maximal subgroup of $G$
contains elements with this characteristic polynomial.

This group also occurs as the Sato--Tate group of the Jacobian of a Picard curve
whose Jacobian has geometric endomorphism ring $\Z[\zeta_3]$.
For example, we may take the curve $y^3 = x^4 + x +1$; since
 $\nfield[\spl(x^4+x+1)]{4.0.229.1}$ has Galois group $S_4$,
 we may apply \cite[Theorem~1.3]{Zar18} to compute the endomorphism ring. We will verify later that this curve also has large mod-2 Galois image (see Example~\ref{exa:type N 432 734}).
\end{example}

\begin{example} \label{exa:type C D}
The unique maximal group of type $\bC$ occurs as the Sato--Tate group of $E_0 \times \Jac(C_2)$
where $C_2$ occurs in 
\cite[Table~11]{FKRS12} as an example whose Jacobian has Sato--Tate group $\stgroup[\USp(4)]{USp(4)}$. 
Table~\ref{table: maximal ST groups} also includes an example of a Jacobian which can be explained using
Lemma~\ref{lemma: double cover}; the genus $2$ quotient is the curve
\href{https://www.lmfdb.org/Genus2Curve/Q/709/a/709/1}{\texttt{709.a.709.1}} in the \LMFDB{}.

The unique maximal group of type $\bD$ occurs as the Sato--Tate group of $E_{-3} \times \Jac(C_2)$.
Table~\ref{table: maximal ST groups} also includes an example of a Jacobian which can be explained using
Lemma~\ref{lemma: double cover}; the genus $2$ quotient is the curve
\href{https://www.lmfdb.org/Genus2Curve/Q/836352/a/836352/1}{\texttt{836352.a.836352.1}} in the \LMFDB{}.
\end{example}

\begin{example} \label{exa:type E}
The unique maximal group of type $\bE$ occurs as the Sato--Tate group of $\Res_{L_3/\Q} C_{1,3}$;
note that 23 splits completely in $L_3$ and the Frobenius traces of $C_{1,3}$ at the primes of $L_3$ above 23
are not all equal up to sign, so $C_{1,3}$ cannot be a $\Q$-curve. (Another way to confirm that $C_{1,3}$ is not a $\Q$-curve is to use the algorithm described in \cite[\S 5]{CN21}.)
Table~\ref{table: maximal ST groups} also includes an example of a Jacobian. The curve $C_3$ has the form
\[
y^2 = (x-1)(2x^3+1)((1+x)^3 - 2).
\]
Over $\Q(2^{\nicefrac{1}{3}})$, $C_3$ acquires the nonhyperelliptic involution
\[
(x,y) \mapsto \left( \frac{-x + 4^{\nicefrac{1}{3}}}{2^{\nicefrac{1}{3}}x + 1}, \frac{9}{(2^{\nicefrac{1}{3}} x + 1)^4} \right).
\]
Hence over $L = \Q(\zeta_3, 2^{\nicefrac{1}{3}})$, the reduced automorphism group of $C_3$ becomes $\cyc 2 \times \cyc 2$, and the quotients by the three involutions are conjugate over $L$.
We may recover the $j$-invariant of one of these quotients by computing that the map
\[
x \mapsto x + \frac{-x + 4^{\nicefrac{1}{3}}}{2^{\nicefrac{1}{3}}x + 1}
\]
carries the Weierstrass points $1, -2^{\nicefrac{-1}{3}}, -2^{\nicefrac{-1}{3}} \zeta_3, -2^{\nicefrac{1}{3}} \zeta_3^2$ to $2^{\nicefrac{1}{3}}, \infty, 2^{\nicefrac{-1}{3}} (1+\zeta_3), 2^{\nicefrac{-1}{3}} (1+\zeta_3^2)$;
we find that this quotient does not have CM.
\end{example}

\begin{example} \label{exa:type F}
The unique maximal group of type $\bF$ occurs as the Sato--Tate group of the product $E_{-8} \times \Jac(C_2)$
where $C_2$ occurs in \cite[Table~11]{FKRS12} as an example whose Jacobian has Sato--Tate group $\stgroup{N(G_{3,3})}$; the endomorphism field of $E_{-8}\times\Jac(C_2)$ is $\Q(\sqrt{-2},i)=\Q(\zeta_{8})$.
Table~\ref{table: maximal ST groups} also includes an example of a Jacobian with endomorphism field $\Q(\zeta_8)$ which can be explained using
Lemma~\ref{lemma: double cover}; the genus $1$ quotient has Jacobian with CM by $\Q(\sqrt{-2})$, while the genus $2$ quotient becomes isomorphic over $\Q(\sqrt{2})$ to
\[
y^2 = (1 + \sqrt{2})x^6 + (-1 + \sqrt{2})x^4 + (-1-\sqrt{2})x^2 + (1-\sqrt{2}).
\]
The quotient of this curve by the nonhyperelliptic involution $(x,y) \mapsto (-x,y)$ is the elliptic curve \href{https://www.lmfdb.org/EllipticCurve/2.2.8.1/512.1/h/2}{\texttt{512.1-h2}} in the \LMFDB{}, which does not
have CM and is not a $\Q$-curve.
\end{example}

\begin{example} \label{exa:type G C4}
The maximal group of type $\bG$ with component group $\cyc 4$ occurs as the Sato--Tate group of $E_0 \times \Jac(C_2)$
where $C_2\colon y^2 = x^5+1$ occurs in 
\cite[Table~11]{FKRS12} as an example whose Jacobian has Sato--Tate group $\stgroup{F_{ac}}$ and endomorphism
field $\Q(\zeta_5)$. 
Table~\ref{table: maximal ST groups} also includes an example of a Jacobian with endomorphism field $\Q(\zeta_5)$ which can be explained using
Lemma~\ref{lemma: double cover}; the genus $2$ quotient is isomorphic to $C_2$.
\end{example}

\begin{example} \label{exa:type G}
The maximal group of type $\bG$ with component group $\cyc 2 \times \cyc 2$ occurs as the Sato--Tate group of $E_0 \times E_{-4} \times E_{-8}$, which has endomorphism field $\Q(i,\sqrt{2})=\Q(\zeta_8)$.
Table~\ref{table: maximal ST groups} also includes an example of a Jacobian with endomorphism field $\Q(\zeta_8)$ which can be explained using
Lemma~\ref{lemma: double cover}; the genus $2$ quotient is isomorphic to 
\[
2y^2 = x^6 + x^4 - 3x^2 + 1.
\]
The quotient by the involution $(x, y) \mapsto (-x, y)$ has CM by $\Q(\sqrt{-2})$ and the quotient by the involution $(x,y) \mapsto (-x, -y)$ has CM by $\Q(i)$.
\end{example}

\begin{example} \label{exa:type H}
The maximal group of type $\bH$ with component group $\cyc 6$ occurs as the Sato--Tate group of $\Jac(C_3)$, 
where $C_3$ is a hyperelliptic curve with endomorphism field
$\Q(\zeta_7)$.

The maximal group of type $\bH$ with component group $\cyc 2 \times \cyc 4$ occurs as the Sato--Tate group of $E_{-3} \times \Jac(C_2)$
where $C_2$ is as in Example~\ref{exa:type G C4};
note that the endomorphism field of $\Jac(C_2)$ does not contain $\Q(\sqrt{-3})$.
Table~\ref{table: maximal ST groups} also includes an example of a Jacobian which can be explained using
Lemma~\ref{lemma: double cover}; the genus $1$ quotient has CM by $\sqrt{-2}$,
while the genus $2$ quotient appears in \cite[Table~1]{vW99} as an example
whose Jacobian has CM by the number field $\spl(\nfield[x^4+4x^2+2]{4.0.2048.2})$.

The maximal group of type $\bH$ with component group $\cyc 2 \times \cyc 2 \times \cyc 2$ occurs as the Sato--Tate group of 
$E_{-3} \times E_{-4} \times E_{-8}$.
Table~\ref{table: maximal ST groups} also includes an example of a Jacobian (shown to us by Everett Howe)
which can be explained using Remark~\ref{R:pair of quartics}; it is isogenous to a product of certain twists of $E_{-3}, E_{-4}, E_{-8}$.
\end{example}

\begin{example} \label{exa:type I JE4}
The maximal group of type $\bI$ with component group $\dih 4$ occurs as the Sato--Tate group of 
$E_0 \times \Jac(C_2)$ where $C_2$ occurs in
\cite[Table~11]{FKRS12} as an example whose Jacobian has Sato--Tate group $\stgroup{J(E_4)}$ and endomorphism
field  $\nfield[\Q(i, 2^{\nicefrac{1}{4}})]{8.0.16777216.2}$. Per \cite[Proposici\'o~8.2.1]{Car01} or \cite[\S 5.3]{GS01}, the $\Qbar$-isogeny factors of $\Jac(C_2)$
have $j$-invariants which are not in $\Q$.
Table~\ref{table: maximal ST groups} also includes an example of a Jacobian which can be explained using
Lemma~\ref{lemma: double cover}; the genus $2$ curve is again $C_2$.

Similarly, the maximal group of type $\bJ$ with component group $\cyc 2 \times \dih 4$ occurs as the Sato--Tate group of $E_{-3} \times \Jac(C_2)$.
Table~\ref{table: maximal ST groups} also includes an example of a Jacobian  which can be explained using
Lemma~\ref{lemma: double cover}: the genus $1$ quotient has CM by $\Q(i)$, while \Magma{} confirms that the genus $2$ quotient acquires an action of $\dih 4$ over $\spl(\nfield[x^4 - 6x^2 + 3]{4.4.27648.1})$.  The endomorphism field of the Jacobian is $\spl(x^4 - 6x^2 + 3)(i) = \spl(\nfield[x^8+30x^4+9]{8.0.3057647616.6})$.
\end{example}

\begin{example} \label{exa:type I JE6}
The maximal group of type $\bI$ with component group $\dih 6$ occurs as the Sato--Tate group of 
$E_0 \times \Jac(C_2)$ where $C_2$ occurs in
\cite[Table~11]{FKRS12} as an example whose Jacobian has Sato--Tate group $\stgroup{J(E_6)}$ and endomorphism
field $\nfield[\Q(\sqrt{-3}, (-2)^{\nicefrac{1}{6}})]{12.0.20061226008576.4}$. 
Table~\ref{table: maximal ST groups} also includes an example of a Jacobian which can be explained using
Lemma~\ref{lemma: double cover}; over the field $\Q[\alpha]/(\alpha^3 - 2\alpha^2 + \alpha - 16/27)$, $\Jac(C_2)$
admits a Richelot isogeny to its quadratic twist by~$\alpha$. Hence the endomorphism field contains 
$\spl(x^6 - 2x^4 + x^2 - 16/27)=\spl(\nfield[x^6-4x^3+1]{6.2.1259712.2})$; the results of \cite{FKRS12} will now force the Sato--Tate group to be $\stgroup{J(E_6)}$ once we check that $\Jac(C_2)$ cannot have CM in some imaginary quadratic field $M$. If this were the case, then $M$ would be unramified away from the primes of bad reduction of $\Jac(C_2)$,
which in this case are $2$ and $3$; this implies $M \subseteq \Q(\zeta_{24})$.
Also, from the proof of \cite[Corollary~3.12]{FS14}, the Frobenius trace of $\Jac(C_2)$ at every prime that is inert in $M$ would be 0;
by checking this condition for $p=5, 13, 17$, we rule out all possibilities except $M = \Q(i)$.
Finally, note that the $L$-polynomial for each prime that is split in $M$ would factor nontrivially over $M$; this does not occur for $p=5$.
(See also Remark~\ref{R:Shimura curve} below.)

Similarly, the maximal group of type $\bJ$ with component group $\cyc 2 \times \dih 6$ occurs as the Sato--Tate group of $E_{-4} \times \Jac(C_2)$.
Table~\ref{table: maximal ST groups} also includes an example of a Jacobian which can be partially explained using
Remark~\ref{R:Ritzenthaler-Romagny}; writing $C_3$ in the form
\[
2Y^4+Y^2(4X^2-6Z^2)+4X^4+6X^3Z+XZ^3+3Z^4 = 0, 
\]
we see that the genus $1$ quotient has CM by $\Q(i)$,
while the Prym is 2-isogenous to the Jacobian of a curve $C_2'$ defined over $\Q[\alpha]/(\alpha^3 - \alpha^2 - 2\alpha - 6)$.
Using \Magma{} to compute a Richelot isogeny, we check that $\Jac(C_2)$ is $(2,2)$-isogenous to the Jacobian of
\begin{align*}
C_2: y^2 &= 56x^6 - 144x^5 + 12x^4 + 16x^3 + 90x^2 + 60x + 9 \\
&= (2x^2 - 4x - 1)(28x^4 - 16x^3 - 12x^2 - 24x - 9).
\end{align*}
By computing Igusa invariants, we see that $C_2$ and $C_2'$ are isomorphic over $\Qbar$, but \Magma{} confirms that they
are not isomorphic over $\spl(\nfield[x^3-x^2-2x-6]{3.1.1176.1})$. As in the previous paragraph, we can now verify that $\Jac(C_2)$ has Sato--Tate group
$\stgroup{J(E_6)}$ by checking that $\Jac(C_2)$ cannot have CM in some imaginary quadratic field $M$. In this example,
$\Jac(C_2)$ has good reduction away from 2, 3, 7, and so $M \subseteq \Q(\zeta_{168})$.
By checking the nonvanishing of Frobenius traces for $p=5, 11, 31, 53$, we rule out all possibilities except $M = \Q(\sqrt{-6})$;
we rule out that case by checking that the $L$-polynomial for $p=5$ does not factor nontrivially over $\Q(\sqrt{-6})$.

To compute the endomorphism field of $\Jac(C_2)$, note that the field of definition of the Weierstrass points of $C_2$ is an $(\cyc 2 \times \sym 4)$-extension $L$ of $\Q$
over which the Weierstrass points of $C_2'$ are also defined; therefore $C_2$ and $C_2'$ must be isomorphic over $L$. The only $\dih 6$-subextension of $L$ is $\spl(x^3-x^2-2x-6)(\sqrt{-6}) = \spl(\nfield[x^6 + 7x^4 + 7x^2 - 63]{6.2.154893312.1})$, which must then be the endomorphism field of $\Jac(C_2)$. 
Hence the endomorphism field of $\Jac(C_3)$ is
\[
\spl(x^6 + 7x^4 + 7x^2 - 63)(i) = 
\spl(x^{12} + 10x^{10} + 177x^8 + 244x^6 + 172x^4 + 96x^2 + 36).
\]
\end{example}

\begin{remark} \label{R:Shimura curve}
More generally, the family of genus $2$ curves
\[
C_t: y^2 = x^5 + 8x^4 + tx^3 + 16x^2 - 4x.
\]
has the property that for generic $t$,
the endomorphism field $L_t$ of $\Jac(C_t)$ 
is the splitting field of $x^6 - 2x^4 + x^2 + 64/(t^2-432)$
and $\End(\Jac(C_t)_{L_t})$ is an index-2 order of the quaternion algebra over $\Q$ of discriminant $6$;
that is, $\Jac(C_t)$ is the universal family of quaternionic abelian surfaces over some genus $0$ Shimura curve
(compare \cite{BG08}).
Namely, as pointed out to us by Noam Elkies, this can be confirmed by applying \cite[Theorem~11]{Kum08}
to compute the associated elliptic K3 surface
with fibers of type $E_7$ and $E_8$ at $0$ and $\infty$, respectively.
\end{remark}

\begin{example} \label{exa:type J JD6}
The maximal group of type $\bK$ with component group $\cyc 2 \times \dih 6$ occurs as the Sato--Tate group of 
$E_0 \times \Jac(C_2)$ where $C_2$ occurs in
\cite[Table~11]{FKRS12} as an example whose Jacobian has Sato--Tate group $\stgroup{J(D_6)}$ and endomorphism
field 
\[
\Q(\zeta_{24})[\alpha]/(\alpha^3 + 3\alpha - 2)=\spl(\nfield[x^{12} - 4x^9 + 8x^6 - 4x^3 + 1]{12.0.8916100448256.1}).
\]
The maximal group of type $\bL$ with component group $\cyc 2 \times \cyc 2 \times \dih 6$ occurs as the Sato--Tate group of $E_{-7} \times \Jac(C_2)$ with endomorphism field $\spl(x^{24} - 660x^{18} + 12134x^{12} - 660x^6 + 1)$.

Table~\ref{table: maximal ST groups} also includes an example of a Jacobian for the maximal group of type $\bK$ with component group $\cyc 2 \times \dih 6$,
which can be explained using Remark~\ref{R:Ritzenthaler-Romagny}: 
over $\Q[\alpha]/(\alpha^3 - 3\alpha - 4)$, the Prym of $C_3 \to C_1$ becomes isogenous to $\Jac(C_2)$ for
\[
C_2: y^2 = x^5 + 8x^4 + 16x^2 - 4x
\]
(this is the degenerate specialization $t=0$ of the family from Remark~\ref{R:Shimura curve}).
The latter acquires endomorphisms by $\cyc 2 \times \cyc 2$ over $\Q(\zeta_8)$, and the isogeny factors have CM by $\Q(\sqrt{-6})$ (compare \cite[Proposition~3.7]{BG08}).
The endomorphism field of $\Jac(C_3)$ is
\[
\spl(x^3 - 3x - 4)(\zeta_8, \sqrt{-6}) = \spl(\nfield[x^{12}-4x^6+1]{12.4.6499837226778624.22}).
\]
\end{example}

\begin{example} \label{exa:type J JO}
The maximal group of type $\bK$ with component group $\cyc 2 \times \sym 4$ occurs as the Sato--Tate group of 
$E_0 \times \Jac(C_2)$ where $C_2$ occurs in
\cite[Table~11]{FKRS12} as an example whose Jacobian has Sato--Tate group $\stgroup{J(O)}$ and endomorphism
field $\spl(\nfield[x^6+4x^4-22]{6.2.2725888.2})$.
Similarly, the maximal group of type $\bL$ with component group $\cyc 2 \times \cyc 2 \times \sym 4$ occurs as the Sato--Tate group of $E_{-3} \times \Jac(C_2)$ with endomorphism field $\spl(x^{12} - 9x^8 - 38x^6 - 9x^4 + 1 )$.

Table~\ref{table: maximal ST groups} also includes an example of a Jacobian; to analyze this example,
let $\alpha_1, \alpha_2, \alpha_3$ be the roots of the polynomial $x^3 + 2x+1$ in a splitting field $L$. 
Over $L$, we may change the model to
\[
Y^4 = X^3Z + a_j X^2Z^2 + b_j XZ^3,
\qquad (a_j,\ b_j) = (-4-3\alpha_j^2,\ 2\alpha_j^2-3\alpha_j+4).
\]
Using Lemma~\ref{lemma: quadruple cover}, we deduce that the endomorphism field of $\Jac(C_3)$ is
\[
K \colonequals \Q(i, \sqrt{-59}, \alpha_1, b_1^{\nicefrac{1}{2}}, b_2^{\nicefrac{1}{2}}, b_3^{\nicefrac{1}{2}}) = \spl(\nfield[x^6-4x^4-59]{6.2.13144256.4}).
\]
\end{example}

\begin{example} \label{exa:type M D6}
The maximal group of type $\bM$ with component group $\dih 6$ occurs as the Sato--Tate group of
$E_1 \times \Jac(C_2)$ where $C_2$ occurs in \cite[Table~11]{FKRS12} as an example whose Jacobian has Sato--Tate group $\stgroup{J(E_3)}$ with endomorphism field $\nfield[\Q(\zeta_3,2^{\nicefrac{1}{3}})]{6.0.34992.1}$.
We analyze this example following \cite[\S 4.2.4]{FKS19}.
Over $\Q(2^{\nicefrac{1}{3}})$, $C_2$ becomes isomorphic to $2y^2 = 2x^6 + x^3 + 2$, 
whose quotient by the involution $(x,y) \mapsto (1/x, y/x^3)$ may be identified with the curve
$C_1: 2y^2 = (x+2)(2x^3 - 6x + 1)$ via the map $(x,y) \mapsto (x+1/x, y(x+1)/x^2)$.
Since $\Jac(C_1)$ is the quadratic twist of $E_1$ over $\Q(i)$, 
the endomorphism field of $E_1\times\Jac(C_2)$ is $\nfield[\Q(\zeta_{12},2^{\nicefrac{1}{3}})]{12.0.313456656384.1}$.

Table~\ref{table: maximal ST groups} also includes an example of a Jacobian; 
using Lemma~\ref{lemma: S3 quartic}, we see that all three quotients have $j$-invariant $-216$ and 
the endomorphism field is $\Q(\zeta_{12}, 2^{\nicefrac{1}{3}})$. 
\end{example}

\begin{example} \label{exa:type M S4}
The maximal group of type $\bM$ with component group $\sym 4$ occurs as the Sato--Tate group of
$\Res_{L_3/\Q} C_{1,3}$.
The endomorphism field is
\[
L'_3[\sqrt{\alpha_1/\alpha_2}, \sqrt{\alpha_2/\alpha_3}] = \spl(\nfield[x^4 - 4x^2 - 8x - 3]{4.2.21248.2}).
\]
Table~\ref{table: maximal ST groups} includes an example of a Jacobian with endomorphism field $\spl(\nfield[x^4 - 2x^3 - 6x + 3]{4.2.3888.1})$; it is a reduced model of the twist of $y^2 = x^8 + 14x^4 + 1$ listed in \cite[Table~4, row~11]{ACLLM18}.
\end{example}

\subsection{Examples: type \texorpdfstring{$\bN$}{N}}

In the following examples, we sometimes define a matrix $V$ to specify an isomorphism of $\rho_M$ with $\rho_M^c$ (see Remark~\ref{remark: twisting construction1}); we also write $U$ for the matrix defining an invariant polarization (see Remark~\ref{remark: twisting construction induced polarization}).

\begin{example} \label{exa:type N 48 15 3}
Let $M=\Q(\zeta_3)$.  The maximal group $\stgroup{J_s(B(1,12))}$ of type $\bN$ with component group $\group{48}{15}$ 
occurs as the Sato--Tate group of $E_{-3} \times A_{-3[2]a}$ obtained as in 
Remark~\ref{remark: twisting construction2}, for a representation $\rho_M\colon G_M\to \GL_2(\cO_M)$ with image isomorphic to $\group{24}{10}$, generated by
\[
\begin{pmatrix}
\zeta_3 & 0 \\
0 & \zeta_3
\end{pmatrix},
\begin{pmatrix}
-1 & -2 \\
0 & 1
\end{pmatrix},
\begin{pmatrix}
-1 & 0 \\
1 & 1
\end{pmatrix}\qquad \mbox{with }
U = \begin{pmatrix} 2 & 2 \\ 2 & 4 \end{pmatrix}.
\]
(Note that $U$ must have its top-right entry divisible by 2 as per Remark~\ref{remark: twisting construction induced polarization}.)
 To exhibit an explicit example, we must find a Galois extension $L/\Q$ with
\begin{align*}
\Gal(L/\Q) &\simeq \dih 4 \rtimes \sym 3 \simeq \group{48}{15}, \\
\Gal(L/M) &\simeq \dih 4 \times \cyc 3 \simeq \group{24}{10}.
\end{align*}
We may obtain such an extension by taking the compositum of 
two Galois extensions $L_1/\Q, L_2/\Q$ containing $M$ with
\begin{gather*}
\Gal(L_1/\Q) \simeq \dih 8, \quad \Gal(L_1/M) \simeq \dih 4, \\
\Gal(L_2/\Q) \simeq \sym 3, \quad \Gal(L_2/M) \simeq \cyc 3.
\end{gather*}
These constraints may be satisfied by taking
\begin{align*}
L_1 &= \spl(\nfield[x^8 - 12x^6 + 12x^4 - 52x^2 - 3]{8.2.1087616581632.6}),\\
L_2 &= \spl(\nfield[x^3-2]{3.1.108.1}),\\
L &= L_1L_2 = \spl(x^{24} + 488x^{18} + 63384x^{12} - 5872x^6 + 1372),\\
K &= L^{\cyc{6}} = \spl(\nfield[x^4 - 4x^2 - 3]{4.2.37632.1}).
\end{align*}
There are 16 representations $\rho_M$ with kernel field $L$ and image $\rho_M(G_M)$ as above, comprising two complex conjugate $M$-equivalence classes.
Of these, 8 are $\cO_M$-equivalent (in fact equal) to~$\rho_M^{c'}$ as required by Remark~\ref{remark: twisting construction2a}.  These $\rho_M$ realize both $M$-equivalence classes and give rise to $(1,2)$-polarized abelian surfaces $A_{-3[2]a}$ over $\Q$ with Sato--Tate group~$\stgroup{D_{4,1}}$ and endomorphism field $K$ with $\Gal(K/\Q)\simeq \dih{4}$; they are distinguished by their $L$-polynomials at~13.
In  both cases the product $E_{-3}\times A_{-3[2]a}$ is an abelian threefold over $\Q$ with Sato--Tate group $\stgroup{J_s(B(1,12))}$ and endomorphism field $L$.
\end{example}

\begin{remark}
In Example~\ref{exa:type N 48 15 3}, based on the polarization type,
we expect that the abelian surface can be realized (up to a quadratic twist) as the Prym of a degree $2$ morphism from a genus $3$ curve to a genus $1$ curve, all defined over $\Q$.
\end{remark}

\begin{example} \label{exa:type N 48 15 4}
Let $M=\Q(i)$ and $k = \Q(\sqrt{3})$. 
The maximal group $\stgroup{J(B(3,4;4))}$ of type~$\bN$ with component group $\group{48}{15}$ 
occurs as the Sato--Tate group of $E_{-4,12-8\sqrt{3}} \times A_{-4[3]a}$ obtained as in 
Remark~\ref{remark: twisting construction3}, for a representation $\rho_M\colon G_M\to \GL_2(\cO_M)$ with image isomorphic to~$\group{24}{1}$, generated by
\[
\begin{pmatrix}
i & 0 \\
 0 & i
\end{pmatrix},
\begin{pmatrix}
-1-i & -3 \\
i & 1+i
\end{pmatrix},
\begin{pmatrix}
1 & 3 \\
-1 & -2
\end{pmatrix}
\qquad \mbox{with }
U = \begin{pmatrix} 6 & 3-3i \\ 3+3i & 6 \end{pmatrix}.
\]
(Note that $U$ must have its top-right entry divisible by 3 as per Remark~\ref{remark: twisting construction induced polarization}.)
To exhibit an explicit example, we must find a Galois extension $L/\Q$ not containing $k$ with
\[
\Gal(L/\Q) \simeq \dih 4 \rtimes \sym 3, \qquad \Gal(L/M) \simeq \cyc 3 \rtimes \cyc 8.
\]
We obtain such an extension as in \cite[\S 4.2.5]{FKS19} by taking the compositum $L$ of 
two Galois extensions $L_1/\Q, L_2/\Q$ such that
\begin{gather*}
\Gal(L_1/\Q) \simeq \dih 8, \quad \Gal(L_1/(L_1 \cap L_2)) \simeq \dih 4, \quad \Gal(L_1/M) \simeq \cyc 8, 
\\
\Gal(L_2/\Q) \simeq \sym 3, \quad \Gal(L_2/(L_1 \cap L_2)) \simeq \cyc 3.
\end{gather*}
We take
\begin{align*}
 L_1 &= \spl(\nfield[x^{8} - 7 x^{4} + 28]{8.0.13492928512.1}), \\
L_2 &= \spl(\nfield[x^3 - x^2 + 2x - 3]{3.1.175.1}),\\
K &= L_2M = \spl(\nfield[x^6 + x^4 + 2x^2 - 7]{6.2.13720000.2}),\\
L &= L_1L_2.
\end{align*}
There are 24 representations $\rho_M$ with kernel field $L$ and image $\rho_M(G_M)$ as above, comprising two $M$-equivalence classes.
Of these, 8 are $\cO_M$-equivalent (in fact equal) to $\rho_M^{c'}$ as required by Remark~\ref{remark: twisting construction2a}. These $\rho_M$ realize both $M$-equivalence classes and give rise to $(1,6)$-polarized abelian surfaces $A_{-4[3]a}$ over~$k$ with Sato--Tate group $\stgroup{D_{6,1}}$
and endomorphism field $kK$ with $\Gal(kK/k)\simeq \dih{6}\simeq \group{12}{4}$;
they are distinguished by their $L$-polynomials at 13.
In  both cases the product $E_{-4,12-8\sqrt{3}}\times A_{-4[3]a}$ is an abelian threefold over $k$ with Sato--Tate group $\stgroup{J(B(3,4;4))}$ and endomorphism field $kL$.
\end{example}

\begin{example} \label{exa:type N JB34}
The maximal group $\stgroup{J(B(3,4))}$ of type~$\bN$ with component group $\group{48}{38}$ occurs as the Sato--Tate group of $\Jac(C_3)$
where $C_3$ is the hyperelliptic curve
\[
y^2 = (x^2 + 2x - 1)(x^3 -3x^2 -3x - 3)(x^3 + 3x + 2).
\]
Over $\Q(\sqrt{2})$, this curve becomes isomorphic to
\[
C'_3: y^2 = (4-3\sqrt{2})x^7 + (4+3\sqrt{2}) x.
\]
Applying \cite[Lemma~5.6]{FS16} to the latter, we deduce that the endomorphism field of $\Jac(C_3)$ is given by
\[
L = \Q(\sqrt{2}, i, (-1-\sqrt{2})^{\nicefrac{1}{3}}, (-3)^{\nicefrac{1}{4}}) = \spl(\nfield[x^{12} - 3x^8 + 24x^4 - 48]{12.2.962938848411648.1}).
\]
This is a Galois extension with Galois group $\group{48}{38} \simeq \sym 3 \times \dih 4$.
Moreover, the isogeny factors of $\Jac(C_3)$ all have CM by $\Q(i)$,
and the Galois group of $L$ over $\Q(i)$ is $\group{24}{5}$.
The only option for the Sato--Tate group consistent with these observations is $\stgroup{J(B(3,4))}$.

This example can also be realized as a product of the form $E_{-4,2} \times A_{-4a}$, obtained as in 
Remark~\ref{remark: twisting construction1} with $M=\Q(i)$, for a representation $\rho_M\colon G_M\to \GL_2(\cO_M)$ with image isomorphic to $\group{24}{5}$, generated by
\[
\begin{pmatrix}
i & 0 \\
 0 & i
\end{pmatrix},
\begin{pmatrix}
0 & 1 \\
1 & 0
\end{pmatrix},
\begin{pmatrix}
0 & -1 \\
1 & -1
\end{pmatrix}
\qquad \mbox{with }
U = \begin{pmatrix} 2 & -1 \\ -1 & 2 \end{pmatrix}.
\]
There are 24 representations $\rho_M$ with kernel field $L$ and image $\rho_M(G_M)$ as above, comprising two $M$-equivalence classes.
All 24 are $\cO_M$-equivalent (in fact equal) to $\rho_M^c$.
These $\rho_M$ realize both $M$-equivalence classes and give rise to $(1,3)$-polarized abelian surfaces $A_{-4a}$ over~$\Q$ with Sato--Tate group $\stgroup{D_{6,1}}$
and endomorphism field $K=\spl(\nfield[x^6-3x^2-6]{6.2.4478976.1})$ with Galois group $\dih{6}\simeq \group{12}{4}$;
they are distinguished by their $L$-polynomials at 29:
\begin{align*}
&1 - 20T - 158T^2 - 580T^3 + 841T^4,\\
&1 + 20T - 158T^2 + 580T^3 + 841T^4.
\end{align*}
In  both cases the product $E_{-4,2}\times A_{-4a}$ is an abelian threefold over $\Q$ with Sato--Tate group $\stgroup{J(B(3,4))}$
and endomorphism field $L$; for the abelian surface $A_{-4a}$ with the second $L$-polynomial listed above, the product $E_{-4,2}\times A_{-4a}$ is isogenous to $\Jac(C_3)$.
\end{example}

\begin{example} \label{exa:type N 48 41}
Let $M=\Q(i)$ and $k = \Q(\sqrt{3})$. The maximal group $\stgroup{J_s(B(3,4))}$ of type~$\bN$ with component group $\group{48}{41}$ 
occurs as the Sato--Tate group of $E_{-4,12-8\sqrt{3}} \times A_{-4[3]b}$ obtained as in 
Remark~\ref{remark: twisting construction3}, for a representation $\rho_M\colon G_M\to \GL_2(\cO_M)$ with image isomorphic to~$\group{24}{5}$, generated by

\[
\begin{pmatrix}
i & 0 \\
 0 & i
\end{pmatrix},
\begin{pmatrix}
1 & 0 \\
-1 & -1
\end{pmatrix},
\begin{pmatrix}
1 & 3 \\
-1 & -2
\end{pmatrix}
\qquad \mbox{with }
U = \begin{pmatrix} 2 & 3 \\ 3 & 6 \end{pmatrix}.
\]
(Note that $U$ must have its top-right entry divisible by 3 as per Remark~\ref{remark: twisting construction induced polarization}.)
To exhibit an explicit example, we must find a Galois extension $L/\Q$ not containing $k$ with
\[
\Gal(L/\Q) \simeq \group{48}{41}, \qquad \Gal(L/M) \simeq \group{24}{5}.
\]
We obtain such an extension by taking the compositum of two Galois extensions $L_1/\Q, L_2/\Q$
such that
\begin{gather*}
\Gal(L_1/\Q) \simeq \group{16}{13}, \quad \Gal(L_1/M) \simeq \cyc 2 \times \cyc 4, \quad \Gal(L_1/(L_1 \cap L_2)) \simeq \group{8}{4}, \\
\Gal(L_2/\Q) \simeq \sym 3, \quad \Gal(L_2/(L_1 \cap L_2)) \simeq \cyc 3.
\end{gather*}
We take
\begin{align*}
L_1 &= \spl(\nfield[x^8 - 8x^4 + 25]{8.0.419430400.3}), \\
L_2 &= \spl(\nfield[x^3 - x^2 + 5x + 15]{3.1.1960.1}),\\
K &= L_2M = \spl(\nfield[x^6 - 2x^5 + 2x^4 + 4x^3 + 9x^2 - 30x + 50]{6.0.61465600.1}),\\
L &= L_1L_2 = \spl(x^{24}\! - 208x^{20}\! + 12434x^{16}\! - 161872x^{12}\! + 766481x^8\! - 5190400x^4\! + 10240000).
\end{align*}
There are 24 representations $\rho_M$ with kernel field $L$ and image $\rho_M(G_M)$ as above, comprising two $M$-equivalence classes.
Of these, 8 are $\cO_M$-equivalent (in fact equal) to $\rho_M^{c'}$ as required by Remark~\ref{remark: twisting construction2a}.  These $\rho_M$ realize both $M$-equivalence classes and give rise to principally polarized abelian surfaces $A_{-4[3]b}$ over~$k$ with Sato--Tate group $\stgroup{J(D_3)}$
and endomorphism field $kK$ with $\Gal(kK/k)\simeq \dih{6}\simeq \group{12}{4}$;
they are distinguished by their $L$-polynomials at 13.
In  both cases the product $E_{-4,12-8\sqrt{3}}\times A_{-4[3]b}$ is an abelian threefold over $k$ with Sato--Tate group $\stgroup{J_s(B(3,4))}$
and endomorphism field $kL$.
\end{example}

\begin{remark} \label{R:identify Jacobian}
In Example~\ref{exa:type N 48 41}, since $A_{-4[3]b}$ carries an indecomposable principal polarization, it must occur as the Jacobian of a genus 2 curve over $k$.
At this time, we have not attempted to identify this curve explicitly.
\end{remark}

\begin{example} \label{exa:type N 96}
The maximal group $\stgroup{J(B(O,1))}$ of type $\bN$ with component group $\group{96}{193}$
occurs as the Sato--Tate group of $E_{-8} \times \Jac(C_2)$
where $C_2$ is as in Example~\ref{exa:type J JO}; the abelian surface $\Jac(C_2)$ has Sato--Tate group $\stgroup{J(O)}$ and endomorphism field $K=\spl(\nfield[x^6+4x^4-22]{6.2.2725888.2})$.
From \cite[Table~4, Table~6]{FS14}, we see that the endomorphism field of $E_{-8} \times \Jac(C_2)$ is
\[
L = \spl(x^{16} - 44x^{12} - 308x^{10} - 990x^8 - 1936x^6 - 2662x^4 + 9196x^2 + 20449),
\]
which has the desired Galois group $\group{96}{193}$ and contains $K$ as a subfield of index 2.

This example can also be interpreted (up to isogeny) as a product of the form $E_{-8} \times A_{-8}$ obtained as in 
Remark~\ref{remark: twisting construction1} with $M=\Q(\sqrt{-2})$, 
for a representation $\rho_M\colon G_M\to \GL_2(\cO_M)$ with image isomorphic to the Shephard--Todd group $G_{12}\simeq \group{48}{29}$, generated by
\[
\begin{pmatrix}
-1 & 0 \\
1 - \sqrt{-2} & 1 \end{pmatrix},\ 
\begin{pmatrix}
-1 & -1  \\
1 & 0 \end{pmatrix}
\]
with
\[
U = \begin{pmatrix} 2 & 1+\sqrt{-2} \\
1-\sqrt{-2} & 2 \end{pmatrix}, \qquad
V = \begin{pmatrix} 0 & 1 \\
1 & 0 
\end{pmatrix}.
\]
There are 48 representations $\rho_M$ with kernel field $L$ and image $\rho_M(G_M)$ as above, comprising two $M$-equivalence classes.
Of these, 8 are $\cO_M$-equivalent to $\rho_M^c$ (via $V$). These $\rho_M$ realize both $M$-equivalence classes and 
give rise to principally polarized abelian surfaces $A_{-8}$ over $\Q$ with Sato--Tate group $\stgroup{J(O)}$
and endomorphism field $K$ with Galois group $\sym{4}\times \cyc{2}\simeq \group{48}{48}$;
they are distinguished by their $L$-polynomials at 17:
\begin{align*}
&1 - 8T + 32T^2 - 136T^3 + 289T^4,\\
&1 + 8T + 32T^2 + 136T^3 + 289T^4.
\end{align*}
The abelian surface  $A_{-8}$ with the second $L$-polynomial listed above is isogenous to $\Jac(C_2)$.  In 
both cases the product $E_{-8}\times A_{-8}$ is an abelian threefold over $\Q$ with Sato--Tate group $\stgroup{J(B(O,1))}$
and endomorphism field $L$.
\end{example}

\begin{example} \label{exa:type N 144 125}
Let $M=\Q(\zeta_3)$.
The maximal group $\stgroup{J_s(B(T,3))}$ of type $\bN$ with component group $\group{144}{125}$ 
occurs as the Sato--Tate group of $E_{-3} \times A_{-3}$ obtained as in 
Remark~\ref{remark: twisting construction1},  for a representation $\rho_M\colon G_M\to \GL_2(\cO_M)$ with image isomorphic to $\group{72}{25}$, generated by 
\[
\begin{pmatrix}
\zeta_3 & 0 \\
0 & \zeta_3
\end{pmatrix},
\begin{pmatrix}
0 & -1 \\
1 & 0 
\end{pmatrix},
\begin{pmatrix} 
1+\zeta_3 & -1 \\
0 & -\zeta_3
\end{pmatrix}
\]
with
\[
U = \begin{pmatrix} 3 & 1+2\zeta_3 \\ -1-2\zeta_3 & 3 \end{pmatrix}, \qquad
V = \begin{pmatrix} 0 & 1 \\ 1 & 0 \end{pmatrix}.
\]
To exhibit an explicit example, we must find a Galois extension $L/\Q$ with
\[
\Gal(L/\Q) \simeq \group{144}{125}, \qquad \Gal(L/M) \simeq \group{72}{25}.
\]
We obtain such an extension by taking the compositum of two Galois extensions
$L_1/\Q, L_2/\Q$ such that 
\begin{gather*}
\Gal(L_1/\Q) \simeq \GL(2,3), \quad
\Gal(L_1/\Q(\zeta_3)) \simeq \SL(2, 3), \\
\Gal(L_2/\Q) \simeq \sym 3, \quad \Gal(L_2/M) \simeq \cyc 3.
\end{gather*}
We take 
\begin{align*}
L_1 &= \spl(\nfield[x^8 - 6x^4 + 4x^2 - 3]{8.2.181398528.1}), \\
L_2 &= \spl(\nfield[x^3-2]{3.1.108.1}), \\
L &= L_1L_2, \\
K &= L_1^{\cyc{2}} = \spl(\nfield[x^4 - 2x^3 - 6x + 3]{4.2.3888.1}).
\end{align*}
There are 144 representations $\rho_M$ with kernel field $L$ and image $\rho_M(G_M)$ as above, comprising six $M$-equivalence classes.
Of these, 24 are $\cO_M$-equivalent to $\rho_M^c$ (via $V$).  These $\rho_M$ realize all six $M$-equivalence classes and give rise to $(1,6)$-polarized abelian surfaces $A_{-3}$ over~$\Q$ with Sato--Tate group $\stgroup{O_1}$
and endomorphism field $K$ with Galois group $\PGL(2,3)\simeq \sym{4}\simeq \group{24}{12}$;
they are distinguished by their $L$-polynomials at 7 and 13.
In every case the product $E_{-3}\times A_{-3}$ is an abelian threefold over $\Q$ with Sato--Tate group $\stgroup{J_s(B(T,3))}$
and endomorphism field $L$.
\end{example}

\begin{remark}
In Example~\ref{exa:type N 144 125}, we may also take $L_1$ to be the $3$-division field of any
elliptic curve over $\Q$ with surjective mod-3 Galois representation (e.g., curve \href{https://www.lmfdb.org/EllipticCurve/Q/24/a/5}{\texttt{24.a5}} in the \LMFDB{}).
\end{remark}

\begin{remark}
The abelian threefold in Example~\ref{exa:type N 144 125} is $2$-isogenous to another threefold
admitting an indecomposable polarization of type $(1,1,3)$.
\end{remark}

\begin{example} \label{exa:type N 144 127}
Let $M=\Q(\zeta_3)$.
The maximal group $\stgroup{J(B(T,3))}$ of type $\bN$ with component group $\group{144}{127}$ 
occurs as the Sato--Tate group of $E_{-3} \times A_{-3[2]b}$ obtained as in 
Remark~\ref{remark: twisting construction2}, for a representation $\rho_M\colon G_M\to \GL_2(\cO_M)$ with image isomorphic to $\group{72}{25}$, generated by 
\[
\begin{pmatrix}
\zeta_3 & 0 \\
0 & \zeta_3
\end{pmatrix},\ 
\begin{pmatrix}
-1 & -2 \\
1 & 1
\end{pmatrix},\ 
\begin{pmatrix} 
1+\zeta_3 & 2\zeta_3 \\
0 & -\zeta_3
\end{pmatrix}
\qquad \mbox{with }
U = \begin{pmatrix} 3 & 2-2\zeta_3^2 \\ 2-2\zeta_3^2 & 6 \end{pmatrix}.
\]
(Note that $U$ must have its top-right entry divisible by 2 as per Remark~\ref{remark: twisting construction induced polarization}.)
To exhibit an explicit example, we must find a Galois extension $L/\Q$ with
\[
\Gal(L/\Q) \simeq \group{144}{127}, \qquad \Gal(L/M) \simeq \group{72}{25}.
\]
We obtain such an extension as in \cite[\S 4.2.5]{FKS19} by taking the compositum of Galois extensions
$L_1/\Q$ and $L_2/\Q$ satisfying
\begin{gather*}
\Gal(L_1/\Q) \simeq \group{48}{33}, \quad \Gal(L_1/M) \simeq \SL(2,3), \\
\Gal(L_2/\Q) \simeq \sym 3, \quad \Gal(L_2/M) \simeq \cyc 3.
\end{gather*}
We may take
\begin{align*}
L_1 &= \spl(\nfield[x^{16} - 18x^{12} + 101x^{10} + 99x^8 - 1098x^6 + 3043x^4 - 738x^2 + 81]{16.0.12630731694101401542561.1}), \\
L_2 &= \spl(\nfield[x^3-3]{3.1.243.1}), \\
K &= L_1^{\cyc{2}} = \spl(\nfield[x^6 - 3x^5 + 10x^4 - 15x^3 - 41x^2 + 48x + 27]{6.4.4162464003.1}), \\
L &= L_1L_2.
\end{align*}
There are 144 representations $\rho_M$ with kernel field $L$ and image $\rho_M(G_M)$ as above, comprising six $M$-equivalence classes.
Of these, 16 are $\cO_M$-equivalent (in fact equal) to $\rho_M^{c'}$, as required by Remark~\ref{remark: twisting construction2}.  These $\rho_M$ realize two $M$-equivalence classes and give rise to $(1,3)$-polarized abelian surfaces $A_{-3[2]b}$ over $\Q$ with Sato--Tate group $\stgroup{J(T)}$ and endomorphism field $K$
with Galois group $\alt{4}\times \cyc{2}\simeq \group{24}{13}$;
they are distinguished by their $L$-polynomials at 7.
In both cases the product $E_{-3}\times A_{-3[2]b}$ is an abelian threefold over $\Q$ with Sato--Tate group $\stgroup{J(B(T,3))}$
and endomorphism field $L$.
\end{example}

\begin{example} \label{exa:type N 192 988}
The maximal group $\stgroup{J(B(O,2))}$ of type $\bN$ with component group $\group{192}{988}$
occurs as the Sato--Tate group of $E_{-4} \times \Prym(C_3/C_1)$, where $C_3\colon Y^4 = X^4 + 2X^2Z^2 + XZ^3$ is as in Example~\ref{exa:type J JO}
and $C_1$ is its quotient by the involution $(X:Y:Z) \mapsto (X:-Y:Z)$. 
To analyze this example, define $\alpha_j, a_j, b_j$ as in Example~\ref{exa:type J JO}, so that $C_{3,\Q(\alpha_j)}$ admits the model $Y^4 = X^3Z + a_j X^2Z^2 + b_j XZ^3$.
The endomorphism field $L$ of $E_{-4} \times \Prym(C_3/C_1)$
is the minimal number field for which
$\Hom(E_{-4,L}, \Prym(C_3/C_1)_L) =\Hom(E_{-4,\C}, \Prym(C_3/C_1)_\C)$; using Lemma~\ref{lemma: quadruple cover}, we identify
$L$ with the Galois closure of $\Q(\alpha_j, \sqrt{b_j}, (a_j-2\sqrt{b_j})^{\nicefrac{1}{4}})$.
For the field $K=\spl(\nfield[x^6-4x^4-59]{6.2.13144256.4})$ with $\Gal(K/\Q)\simeq \group{48}{48}$ from Example~\ref{exa:type J JO}, we obtain
\[
L=K\bigl(\bigl(a_1-2\sqrt{b_1}\bigr)^{\nicefrac{1}{4}}\bigr)=\spl(x^{24} - 40x^{16} + 430x^{12} - 240x^8 + 344x^4 - 59),
\]
with $\Gal(L/\Q)\simeq \group{192}{988}$.

This example can also be realized as a product of the form $E_{-4} \times A_{-4b}$,where $A_{-4b}$ is an abelian surface isogenous to $\Prym(C_3/C_1)$ obtained as in 
Remark~\ref{remark: twisting construction1} with $M=\Q(i)$, for a representation $\rho_M\colon G_M\to \GL_2(\cO_M)$ with image isomorphic to the Shephard--Todd group $G_{8}\simeq \group{96}{67}$, generated by
\[
\begin{pmatrix}
i & 0 \\ 
-i & 1 
\end{pmatrix},\ 
\begin{pmatrix}
1 & 1 \\
0 & i
\end{pmatrix}
\]
with
\[
U = \begin{pmatrix} 
2 & 1+i \\
1-i & 2
\end{pmatrix}, \qquad
V = \begin{pmatrix} 0 & 1 \\
1 & 0 \end{pmatrix}.
\]
There are 96 representations $\rho_M$ with kernel field $L$ and image $\rho_M(G_M)$ as above, comprising four $M$-equivalence classes.
Of these, 16 are $\cO_M$-equivalent to $\rho_M^c$ (via $V$).  These $\rho_M$ realize all four $M$-equivalence classes and
give rise to $(1,2)$-polarized abelian surfaces $A_{-4b}$ over~$\Q$ with Sato--Tate group $\stgroup{J(O)}$ and endomorphism field $K$; they are distinguished by their $L$-polynomials at 61:
\begin{align*}
&1 -22T +242T^2 -1342T^3 +3721T^4,\\
&1 +22T +242T^2 +1324T^3 +3721T^4,\\
&1 -2T +2T^2 -122T^3 +3721T^4,\\
&1 +2T +2T^2 -122T^3 +3721T^4.
\end{align*}
The abelian surface $A_{-4b}$ with the last of these $L$-polynomials is isogenous to $\Prym(C_3/C_1)$, and in every case the product $E_{-4}\times A_{-4b}$ is an abelian threefold over $\Q$ with Sato--Tate group $\stgroup{J(B(O,2))}$ and endomorphism field $L$. 
\end{example}

\begin{remark}
Note that Example~\ref{exa:type N 192 988} contradicts \cite[Table~6]{FG18}, which asserts that~$\stgroup{J(O)}$
cannot occur as the Sato--Tate group of an abelian surface with complex multiplication by~$\Q(i)$. In fact this is an omission in the table; it corresponds to a similar omission in \cite[Proposition~3.5(iv)]{FG18}, which fails to account for the fact that the binary octahedral group has a faithful $2$-dimensional representation with traces in $\Q(i)$. None of this affects the other results of \cite{FG18}.
\end{remark}

\begin{example} \label{exa:type N 192 956}
The maximal group $\stgroup{J(D(4,4))}$ of type $\bN$ with component group $\group{192}{956}$
occurs as the Sato--Tate group of
$\Res C_{1,3}$, where $C_{1,3}\colon y^2 = x^3-\alpha x$ and $\alpha$ is a root of the cubic polynomial $f(x)=x^3-x^2+x-2$.
The endomorphism field is
\[
K=\Q(\sqrt{-1},(\alpha_1/\alpha_2)^{\nicefrac{1}{4}}, (\alpha_2/\alpha_3)^{\nicefrac{1}{4}})=\spl(x^{12}+x^8+2x^4+4),
\]
where $\alpha_1,\alpha_2,\alpha_3$ are the three roots of $f(x)$.

Table~\ref{table: maximal ST groups} also includes an example of a Jacobian;
this is a twist of the Fermat quartic listed in \cite[Table~4, row 59]{FLS18} with
endomorphism field $\spl(x^{12} + 3x^4 - 1)$.

This example can also be realized as an isogeny twist of $E_{-4}^3$, obtained as in
Remark~\ref{remark: twisting construction1} with $M=\Q(i)$, for a representation $\rho_M\colon G_M\to \GL_3(\cO_M)$ with image isomorphic to the wreath product $\mu_4 \wr \sym 3\simeq \group{384}{5557}$,
generated by
\[
\begin{pmatrix}
0 & 1 & 0\\
0 & 0 & 1\\
1 & 0 & 0
\end{pmatrix},\ 
\begin{pmatrix}
0 & 1 & 0\\
1 & 0 & 0\\
0 & 0 & 1
\end{pmatrix},\ 
\begin{pmatrix}
i & 0 & 0\\
0 & 1 & 0\\
0 & 0 & 1
\end{pmatrix}.
\]
The kernel field of $\rho_M$ is
\[
L=\Q(\sqrt{-1},\alpha_1^{\nicefrac{1}{4}}, \alpha_2^{\nicefrac{1}{4}}, \alpha_3^{\nicefrac{1}{4}})=\spl(x^{12}-x^8+x^4-2),
\]
with Galois group $\group{768}{1088009}$; the endomorphism field $K$ is the
kernel field of the projective representation with Galois group $\group{192}{956}$.

There are 768 representations $\rho_M$ with kernel field $L$ and image $\rho_M(G_M)$ as above, comprising eight $M$-equivalence classes.
Of these, 192 are $\cO_M$-equivalent (in fact equal) to $\rho_M^c$.  These $\rho_M$ realize all eight $M$-equivalence classes and give rise to abelian threefolds defined over $\Q$ that are twists of $E_{-4}^3$ with Sato--Tate group $\stgroup{J(D(4,4))}$ and endomorphism
field~$K$.  These abelian threefolds are isogenous to the abelian threefolds $\Res \tilde C_{1,3}$ that arise for the eight twists $\tilde C_{1,3} \colon y^2 = x^3-\tilde\alpha$ of $C_{1,3}$ with
\[
\tilde\alpha\in \{\alpha,\, -\alpha,\, \alpha^3,\, -\alpha^3,\, 83^2\alpha,\, -83^2\alpha,\, 83^2\alpha^3,\, -83^2\alpha^3\},
\]
which are distinguished by their $L$-polynomials at 5 and 29.
\end{example}

\begin{example} \label{exa:type N 336 208}
The maximal group $\stgroup{J(E(168))}$ of type $\bN$ with component group $\group{336}{208}$
occurs as the Sato--Tate group of $\Jac(C_3)$ where $C_3$ is a twist of the Klein quartic,
taken from \cite[Table~5, row 14]{FLS18}, with endomorphism field $K=\spl(\nfield[x^8+4x^7+21x^4+18x+9]{8.2.153692888832.1})$.

This example can also be realized as an isogeny twist of $E_{-7}^3$ obtained as in 
Remark~\ref{remark: twisting construction1} with $M=\Q(\zeta_3)$, for a representation $\rho_M\colon G_M\to \GL_3(\cO_M)$ with image isomorphic to $\group{168}{42}$, generated by the matrices
\[
\begin{pmatrix}
-1 & 0 & \tfrac{1-\sqrt{-7}}{2} \\
0 & -1 & \tfrac{-1-\sqrt{-7}}{2} \\
0 & 0 & 1
\end{pmatrix},\ 
\begin{pmatrix}
2 & \tfrac{-1+\sqrt{-7}}{2} & \tfrac{-3+\sqrt{-7}}{2} \\
1 & 0 & 0 \\
\tfrac{1+\sqrt{-7}}{2} & -1 & -1
\end{pmatrix},
\]
with
\[
U = \begin{pmatrix} 
2 & \tfrac{-1+\sqrt{-7}}{2} & \tfrac{-3+\sqrt{-7}}{2} \\
\tfrac{-1-\sqrt{-7}}{2} & 2 & \tfrac{3+\sqrt{-7}}{2} \\
\tfrac{-3-\sqrt{-7}}{2} & \tfrac{3-\sqrt{-7}}{2} & 3 
\end{pmatrix},
\qquad
V = \begin{pmatrix}
0 & 1 & 0 \\
1 & 0 & 0 \\
0 & 0 & -1
\end{pmatrix}.
\]
This group is an index-2 subgroup of the Shephard--Todd group $G_{24}$.

There are 336 representations $\rho_M$ with kernel field $K$ and image $\rho_M(G_M)$ as above, comprising two $M$-equivalence classes.
Of these, 12 are $\cO_M$-equivalent to $\rho_M^c$ (via $V$).  These $\rho_M$ realize both $M$-equivalence classes and 
give rise to principally polarized abelian threefolds over $\Q$ that are twists of $E_{-7}^3$ with Sato--Tate group $\stgroup{J(E(168))}$ and endomorphism field $K$.  These twists are distinguished by their $L$-polynomials at 11:
\begin{align*}
&1 - 5T -3T^2 +57T^3  - 33T^4 - 605T^5 + 1331T^6,\\
&1 + 9T +53T^2 +211T^3 +583T^4 + 1089T^5 + 1331T^6.
\end{align*}
The twist corresponding to the second $L$-polynomial listed above is isogenous to $\Jac(C_3)$.
\end{example}

\begin{example} \label{exa:type N 432 523}
The maximal group $\stgroup{J(D(6,6))}$ of type $\bN$ with component group $\group{432}{523}$
occurs as the Sato--Tate group of
$\Res C_{1,3}$, where $C_{1,3}\colon y^2= x^3-\alpha$ with $\alpha$ a root of the cubic polynomial $f(x)=x^3-x^2+x-2$.
The endomorphism field is
\[
K=\Q((\alpha_1/\alpha_2)^{\nicefrac{1}{6}}, (\alpha_2/\alpha_3)^{\nicefrac{1}{6}}) = \spl(x^{36} + x^{30} + 27x^{24} - 42x^{18} + 28x^{12} - 384x^6 + 4096).
\]
This example can also be realized as an isogeny twist of $E_{-3}^3$ obtained as in 
Remark~\ref{remark: twisting construction1} with $M=\Q(\zeta_3)$ for a representation $\rho_M\colon G_M\to \GL_3(\cO_M)$ with image isomorphic to the wreath product $\mu_6 \wr \sym 3\simeq \group{1296}{1827}$,
generated by
\[
\begin{pmatrix}
0 & 1 & 0\\
0 & 0 & 1\\
1 & 0 & 0
\end{pmatrix},\ 
\begin{pmatrix}
0 & 1 & 0\\
1 & 0 & 0\\
0 & 0 & 1
\end{pmatrix},\ 
\begin{pmatrix}
\zeta_6 & 0 & 0\\
0 & 1 & 0\\
0 & 0 & 1
\end{pmatrix}.
\]
The kernel field of $\rho_M$ is the degree-2592 field
\[
L= \Q(\alpha_1^{\nicefrac{1}{6}},\alpha_2^{\nicefrac{1}{6}},\alpha_3^{\nicefrac{1}{6}}) = \spl(x^{18}-x^{12}+x^6-2)
\]
whose Galois group is the transitive group \href{https://www.lmfdb.org/GaloisGroup/18T396}{\texttt{18T396}}; $K$ is the kernel field of the projective representation, with Galois group $\group{432}{523}$.

There are 2592 representations $\rho_M$ with kernel field $L$ and image $\rho_M(G_M)$ as above, comprising twelve $M$-equivalence classes.
Of these, 288 are $\cO_M$-equivalent (in fact equal) to $\rho_M^c$.  These $\rho_M$ realize all twelve $M$-equivalence classes and give rise to abelian threefolds over~$\Q$ that are twists of $E_{-3}^3$ with Sato--Tate group $\stgroup{J(D(6,6))}$ and endomorphism field $K$.  These abelian threefolds are isogenous to the abelian threefolds $\Res \tilde C_{1,3}$ that arise for the~12 twists ${\tilde C}_{1,3} \colon y^2 = x^3-\tilde\alpha$ of $C_{1,3}$ with
\[
\tilde{\alpha} \in \{\alpha,2^2\alpha,2^4\alpha,3^383^3\alpha,2^23^383^3\alpha,2^43^383^3\alpha,\alpha^5,2\alpha^5,2^4\alpha^5,3^383^3\alpha^5,2^23^383^3\alpha^5,2^43^383^3\alpha^5\},
\]
which are distinguished by their $L$-polynomials at 13.
\end{example}

\begin{remark} \label{R:indistinguishable groups}
Similarly, the group $\stgroup{J(D(3,3))}$ occurs as the Sato--Tate group of the Weil restriction of the elliptic curve
$y^2 = x^3 - \alpha^2$ from $L_3$ to $\Q$. The base extensions of this abelian threefold to the fields
$\Q(\sqrt{-83})$ and $\Q(\sqrt{249})$ have Sato--Tate groups $\stgroup{J(C(3,3))}$ and $\stgroup{J_s(C(3,3))}$
which have nonisomorphic component groups, but give rise to identical distributions
(see Proposition~\ref{proposition: same measures}).
\end{remark}

\begin{example} \label{exa:type N 432 734}
Let $M=\Q(\zeta_3)$.  The maximal group $\stgroup{J(E(216))}$ of type $\bN$ with component group $\group{432}{734}$ 
occurs as the Sato--Tate group of $E_{-3}^3$ obtained as in Remark~\ref{remark: twisting construction1}, for a representation $\rho$ whose restriction to
$G_{\Q(\zeta_3)}$ has image equal to the Shephard--Todd group $G_{25}$, generated by
\[
\begin{pmatrix}
1 & \zeta_3^2 & 0 \\
0 & \zeta_3^2 & 0 \\
0 & 0 & 1
\end{pmatrix},
\begin{pmatrix}
\zeta_3 & 0 & 0 \\
\zeta_3 & 1 & 0 \\
\zeta_3 & 0 &  1 
\end{pmatrix},
\begin{pmatrix}
1 & 0 & \zeta_3^2 \\
0 & 1 & 0 \\
0 & 0 & \zeta_3^2
\end{pmatrix},
\]
with
\[
U = \begin{pmatrix} 
3 & \zeta_3^2-1 & \zeta_3^2 - 1 \\
\zeta_3 - 1 & 2 & 1 \\
\zeta_3 - 1 & 1 & 2 \end{pmatrix}, \qquad
V = \begin{pmatrix}
1 & 0 & 0 \\
1 & -1 & 0 \\
1 & 0 & -1
\end{pmatrix}.
\]
(The invariant polarization is of type $(1,1,3)$.)
To exhibit an explicit example,
following a suggestion of Noam Elkies, we construct $\rho$ by identifying $G_{25} \rtimes \cyc 2$ with $\Unitary(3, 2) \rtimes \Gal(\F_4/\F_2)$
and considering the $2$-torsion Galois representation of a generic Picard curve over $\Q$.

We verify that the 2-adic Galois representation for the Jacobian of the Picard curve $C: y^3 = x^4 + x + 1$ used in Example~\ref{exa:type B} has maximal mod-$2$ Galois image. Note that the $64$ theta characteristics form a principal homogeneous space for $\Jac(C)[2]$. The $28$ odd theta characteristics correspond to the bitangents of $C$; for a Picard curve, the line at infinity is a distinguished bitangent, so we have a Galois action on the other $27$. 
We compute this action by following a suggestion from \cite[\S 2]{PSV11}.
The bitangents are the lines $y = ax + b$ for which we have an equality of polynomials in $x$:
\[
(ax+b)^3 - (x^4 + x + 1) = -(x^2 + \kappa_1 x + \kappa_0)^2
\]
for suitable values of $\kappa_1, \kappa_0$. By comparing the coefficients of $x^3$ and $x^2$, we see that 
\[
\kappa_1 = -\tfrac{1}{2}a^3, \qquad \kappa_0 = -\tfrac{1}{8} a^6 + \tfrac{3}{2} a^2 b.
\]
Equating the remaining coefficients yields two additional equations:
\[
a^9 + 12a^5 b + 24ab^2 -8 = a^{12} + 24a^8b + 144a^4b^2 + 64b^3 - 64 = 0.
\]
Eliminating $b$ yields an equation in $A = \tfrac{1}{2} a^3$:
\[
A^9 + 168A^6 + 1080A^5 - 636A^3 + 864A^2 - 432A + 8 = 0.
\]
This polynomial in $A$ has Galois group $\group{432}{734}$, which is the projective image we seek.
Substituting for~$a$ yields a polynomial $f(a)$ with Galois group $\group{1296}{2891} \simeq G_{25}\rtimes\cyc 2$.
The kernel field $L$ of $\rho$ is the splitting field of $f(a)$, and the splitting field $K$ of $f(A)$ is the kernel field of the projective representation.

There are 1296 representations $\rho_M$ with kernel field $L$ and image as above, comprising 6 $M$-equivalence classes.
Of these, 36 are $\cO_M$-equivalent (via $V$) to $\rho_M^c$.  These $\rho_M$ realize all six $M$-equivalence classes and give rise to $(1,3)$-polarized abelian threefolds over $\Q$ that are isogeny twists of $E_{-3}^3$ with Sato--Tate group $\stgroup{J(E(216))}$ and endomorphism field~$K$.  They are distinguished by their $L$-polynomials at 19:
\begin{align*}
&1 +T -7T^2 -26T^3 -133T^4 +361T^5 + 6859T^6,\\
&1 +T +8T^2 -11T^3 +152T^4 +361T^5 + 6859T^6,\\
&1 +7T -7T^2 -182T^3 -133T^4 +2527T^5 + 6859T^6,\\
&1 +7T +56T^2 +259T^3 +1064T^4 +2527T^5 + 6859T^6,\\
&1 -8T +8T^2 +88T^3 +152T^4 -2888T^5 + 6859T^6,\\
&1 -8T +56T^2 -296T^3 +1064T^4 -2888T^5 + 6859T^6.\
\end{align*}
The twist giving rise to the last $L$-polynomial listed above corresponds to the representation $\rho$ defined by the Galois action on the bitangents of $C$.
\end{example}

\subsection{Numerical consistency check}
For each of the 33 abelian varieties with maximal Sato--Tate groups described above we used the methods in \S\ref{subsec:explicit L-functions} to compute $L$-polynomials for primes of norm up to $B=2^{30}$ (and in most cases up to $B=2^{32}$), and computed moment statistics and point densities over every subfield of the corresponding endomorphism field that we compared to the corresponding values predicted by the Sato--Tate group derived in \S\ref{sec:statistics}.  These moment statistics and code that compares them to their predicted values is available in our \GitHub{} repository \cite{FKS21}.  Of the 745 sets of statistics computed (covering all 410 Sato--Tate groups, some multiple times), in all but one case they match the predicted values to within one percent; the sole exception matches to within 1.3 percent.
\bigskip\bigskip\bigskip\bigskip\bigskip\bigskip\bigskip\bigskip\bigskip

\pagebreak

\section{Tables}
\label{sec:tables}

\begin{table}[ht]
\small
\[
\begin{array}{|c|c|c|c|c|c|c|}
\hline&&&&&&\\[-11pt]
\mbox{Type} & \mathrm{End}(A_K)_{\R} & \ST(A)^0  &\dim \ST(A)^0 & \mbox{Moduli} & \mbox{Extensions} & \mbox{Realizable} \\
\hline&&&&&&\\[-11pt]
\bA & \R & \stgroup[\USp(6)]{1.6.A.1.1a} & 21 & 6 & 1 & 1\\
\bB & \C & \stgroup[\Unitary(3)]{1.6.B.1.1a} &9 & 1 & 2 & 2\\
\bC & \R \times \R & \stgroup[\SU(2) \times \USp(4)]{1.6.C.1.1a} &13 & 4 & 1 & 1\\
\bD & \C \times \R & \stgroup[\Unitary(1) \times \USp(4)]{1.6.D.1.1a} &11 & 3 & 2 & 2 \\
\bE & \R \times \R \times \R & \stgroup[\SU(2) \times \SU(2) \times \SU(2)]{1.6.E.1.1a} &9 & 3 & 4 & 4\\
\bF & \C \times \R \times \R & \stgroup[\Unitary(1) \times \SU(2) \times \SU(2)]{1.6.F.1.1a} &7 & 2 & 5 & 5\\
\bG & \C \times \C \times \R & \stgroup[\Unitary(1) \times \Unitary(1) \times \SU(2)]{1.6.G.1.1a} &5 & 1 & 8 & 5\\
\bH & \C \times \C \times \C & \stgroup[\Unitary(1) \times \Unitary(1) \times \Unitary(1)]{1.6.H.1.1a} &3 & 0 & 33 & 13 \\
\bI & \R \times \M_2(\R) & \stgroup[\SU(2) \times \SU(2)_2]{1.6.I.1.1a} & 6 & 2 & 10 & 10\\
\bJ & \C \times \M_2(\R) & \stgroup[\Unitary(1) \times \SU(2)_2]{1.6.J.1.1a} & 4 & 1 &31 & 31 \\
\bK & \R \times \M_2(\C) & \stgroup[\SU(2) \times \Unitary(1)_2]{1.6.K.1.1a} & 3 & 1 & 32 & 32 \\
\bL & \C \times \M_2(\C) & \stgroup[\Unitary(1) \times \Unitary(1)_2]{1.6.L.1.1a} & 2 & 0 &122 & 122 \\
\bM & \M_3(\R) & \stgroup[\SU(2)_3]{1.6.M.1.1a} & 3 & 1 & 11 & 11\\
\bN & \M_3(\C) & \stgroup[\Unitary(1)_3]{1.6.N.1.1a} & 1 & 0 & 171 & 171\\
\hline&&&&&&\\[-11pt]
\mbox{Total} & & & && 433 & 410 \\
\hline
\end{array}
\]
\bigskip

\caption{Real endomorphism algebras and connected parts of Sato--Tate groups of abelian threefolds.
The column ``Moduli'' indicates the maximum dimension of a family whose (very) generic member has this form.
Extensions satisfying the Sato--Tate axioms are enumerated in \S\ref{section:STgroups} (types $\bA$--$\bM$) and \S\ref{section: unitary} (type $\bN$);
realizable extensions are enumerated in \S\ref{section: realizability}.
}
\label{table:connected ST groups}
\vspace{14pt}
\end{table}

\begin{table}[ht]
\small
\begin{tabular}{|c|rrrrrrrrrrrrrr|}
\hline&&&&&&&&&&&&&&\\[-11pt]
$\lambda$ & $\stgroup[\bA]{1.6.A.1.1a}$ & $\stgroup[\bB]{1.6.B.1.1a}$ & $\stgroup[\bC]{1.6.C.1.1a}$ & $\stgroup[\bD]{1.6.D.1.1a}$ & $\stgroup[\bE]{1.6.E.1.1a}$ & $\stgroup[\bF]{1.6.F.1.1a}$ & $\stgroup[\bG]{1.6.G.1.1a}$ & $\stgroup[\bH]{1.6.H.1.1a}$ & $\stgroup[\bI]{1.6.I.1.1a}$ & $\stgroup[\bJ]{1.6.J.1.1a}$ & $\stgroup[\bK]{1.6.K.1.1a}$ & $\stgroup[\bL]{1.6.L.1.1a}$ & $\stgroup[\bM]{1.6.M.1.1a}$ & $\stgroup[\bN]{1.6.N.1.1a}$ \\
\hline&&&&&&&&&&&&&&\\[-11pt]
$(0,0,0)$ & 1 & 1 & 1 & 1 & 1 & 1 & 1 & 1 & 1 & 1 & 1 & 1 & 1 & 1 \\
$(1,0,0)$ & 1 & 2 & 2 & 3 & 3 & 4 & 5 & 6 & 5 & 6 & 9 & 10 & 9 & 18 \\
$(1,1,0)$ & 1 & 3 & 3 & 4 & 7 & 9 & 12 & 16 & 14 & 18 & 26 & 34 & 34 & 82 \\
$(1,1,1)$ & 1 & 4 & 2 & 3 & 4 & 6 & 9 & 14 & 9 & 14 & 19 & 30 & 26 & 74 \\ 
$(2,0,0)$ & 1 & 4 & 3 & 6 & 6 & 10 & 17 & 27 & 15 & 23 & 43 & 61 & 45 & 153 \\
$(2,1,0)$ & 1 & 8 & 6 & 13 & 22 & 44 & 84 & 152 & 76 & 140 & 240 & 416 & 320 & 1280 \\
$(2,1,1)$ & 1 & 9 & 5 & 11 & 19 & 39 & 78 & 154 & 70 & 140 & 228 & 444 & 334 & 1450 \\
$(2,2,0)$ & 1 & 9 & 6 & 10 & 27 & 47 & 96 & 204 & 96 & 186 & 304 & 636 & 486 & 2250 \\
$(2,2,1)$ & 1 & 12 & 6 & 13 & 30 & 66 & 149 & 342 & 133 & 302 & 485 & 1122 & 810 & 4194 \\
$(2,2,2)$ & 1 & 10 & 3 & 6 & 10 & 22 & 52 & 132 & 44 & 110 & 170 & 446 & 300 & 1740 \\
$(3,0,0)$ & 1 & 6 & 4 & 10 & 10 & 20 & 45 & 92 & 35 & 68 & 147 & 264 & 164 & 848 \\
$(3,1,0)$ & 1 & 17 & 9 & 27 & 45 & 121 & 313 & 741 & 235 & 571 & 1099 & 2435 & 1539 & 9027 \\
$(3,1,1)$ & 1 & 22 & 8 & 24 & 45 & 124 & 333 & 852 & 251 & 656 & 1187 & 2924 & 1836 & 11376 \\
$(3,2,0)$ & 1 & 26 & 12 & 34 & 88 & 244 & 689 & 1886 & 535 & 1450 & 2603 & 6898 & 4325 & 28550 \\
$(3,2,1)$ & 1 & 40 & 16 & 50 & 140 & 420 & 1240 & 3600 & 940 & 2752 & 4768 & 13696 & 8448 & 59136 \\
$(3,2,2)$ & 1 & 24 & 8 & 24 & 57 & 174 & 537 & 1686 & 399 & 1262 & 2123 & 6734 & 4023 & 30654 \\
$(3,3,0)$ & 1 & 19 & 10 & 20 & 77 & 179 & 537 & 1695 & 449 & 1269 & 2205 & 6859 & 4191 & 31515\\
$(3,3,1)$ & 1 & 40 & 12 & 34 & 118 & 358 & 1177 & 3956 & 887 & 2880 & 4909 & 16500 & 9680 & 78320 \\
$(3,3,2)$ & 1 & 30 & 9 & 27 & 78 & 258 & 894 & 3210 & 642 & 2288 & 3824 & 14000 & 7920 & 69696 \\ 
$(3,3,3)$ & 1 & 20 & 4 & 10 & 20 & 62 & 221 & 862 & 155 & 590 & 965 & 3906 & 2101 & 20350 \\[1pt]
\hline
\end{tabular}
\bigskip

\caption{3-diagonals of character norms for connected Sato--Tate groups $G$ in $\USp(6)$.}
\label{table:3-diagonal}
\end{table}

\begin{table}[ht]
\small
\begin{tabular}{|c|c|c|c|c|}
\hline&&&&\\[-11pt]
$\Aut'(C)$ & Hyperelliptic & Plane quartic & $\Jac(C)$ & Type \\[1pt]
\hline&&&&\\[-11pt]
$\cyc 1$ & $P_8(x)$ & $P_4(X,Y,Z)$ & $A_3$ & $\bA$ \\
$\cyc 2$ & $xP_3(x^2)$ & None & $A_3[i]$ & $\bB$ \\
$\cyc 2$ & $P_4(x^2)$ & $Y^4 + Y^2 P_2(X,Z) + P_4(X,Z)$ & $E \times A_2$ & $\bC$ \\
$\cyc 3$ & None & $Y^3 Z + P_4(X,Z)$ & $A_3[\zeta_3]$ & $\bB$ \\
$\cyc 2 \times \cyc 2$ & $x^4 P_2(x^2+x^{-2})$ & $P_2(X^2,Y^2,Z^2)$ & $E \times E' \times E''$ & $\bE$  \\
$\cyc 6$ & None & $Y^3 Z + P_2(X^2,Z^2)$ & $E[\zeta_3] \times A_2[\mathcal{O}(6)]$ & $\bJ$ \\
$\cyc 7$ & $x^7-1\star$ & None & $A_3[\zeta_7]$ & $\bH$\\
$\cyc 9$ & None & $Y^3Z + X^4 + XZ^3\star$ & $A_3[\zeta_9]$ & $\bH$ \\
$\sym 3$ & $xP_2(x^3)$ & $X^3Z + Y^3Z + aX^2Y^2 + bXYZ^2 + cZ^4$ & $E^2 \times E'$ & $\bI$ \\
$\dih 4$ & $P_2(x^4)$ & $\mathrm{Cyc}(X^4) +  aX^2Y^2 + b(X^2+Y^2)Z^2$ & $E^2 \times E'$  & $\bI$ \\
$\dih 6$ & $x^7-x\star$ & None & $E[i]^3$ & $\bN$ \\
$\dih 8$ & $x^8 - 1\star$  & None & $E[\sqrt{-2}]^2 \times E'[i]$ & $\bL$ \\
$\group{16}{13}$ & None & $Y^4 + P_2(X^2,Z^2)$ & $E[i]^2 \times E'$ & $\bK$ \\
$\sym 4$ & $x^8 - 14x^4 + 1\star$  & $\mathrm{Cyc}(X^4) + c \mathrm{Cyc}(X^2Y^2)$ & $E^3$ & $\bM$ \\
$\group{48}{33}$ & None & $Y^4 + X^3Z + Z^4\star$ & $E[i]^2 \times E'[\zeta_3]$ & $\bL$ \\
$\group{96}{64}$ & None & $\mathrm{Cyc}(X^4)\star$ & $E[i]^3$ & $\bN$ \\
$\PSL_2(\F_7)$ & None & $\mathrm{Cyc}(X^3 Y)\star$ & $E\left[ \tfrac{1 + \sqrt{-7}}{2} \right]^3$ & $\bN$ \\[1pt]
\hline
\end{tabular}
\bigskip

\caption{Reduced automorphism groups of curves of genus 3.
In the models, $a,b,c$ are generic constants, $P_n$ is a generic polynomial of total degree $n$,
and $\mathrm{Cyc}$ is the sum over the cyclic permutations of $X,Y,Z$.
The stars denote isolated points in moduli. 
The column $\Jac(C)$ represents the isogeny decomposition of the Jacobian. See Proposition~\ref{P:automorphism groups} for further discussion.}
\label{table: types of generic automorphism groups}
\end{table}

\begin{table}[ht]
\vspace{-25pt}
\tiny
\begin{tabular}{|c|c|c|c|c|c|c|}
\hline&&&&&&\\[-8pt]
Type & ID & $k$ & Realization & $d$ & Curve $(C_{1,3}, C_2 ,C_3)$ & Ex. \\[1pt]
\hline&&&&&&\\[-8pt]
$\bA$ & $\stgroup[\langle 1,1 \rangle]{1.6.A.1.1a}$ & $\Q$ & $\Jac(C_3)$ & $1$ & $y^2 = x^7 - x + 1$ & \ref{exa:type A}\\[1pt]
\hline&&&&&&\\[-8pt]
$\bB$ & $\stgroup[\langle 2,1 \rangle]{1.6.B.2.1a}$ & $\Q$ & $\Jac(C_3)$ & $1$ & $y^2 = x^7 + 3x^5 + 4x^3 + 2x$ & \ref{exa:type B} \\[1pt]
\hline&&&&&&\\[-8pt]
$\bC$ & $\stgroup[\langle 1,1 \rangle]{1.6.C.1.1a}$ & $\Q$ & $E_0 \times \Jac(C_2)$ & $1$ & $y^2 = x^5 - x + 1$ & \ref{exa:type C D}\\
& & $\Q$ & $\Jac(C_3)$ & 1 & $y^2 = 4x^8 - 7x^2 + 4$ & \\[1pt]
\hline&&&&&&\\[-8pt]
$\bD$ & $\stgroup[\langle 2,1 \rangle]{1.6.D.2.1a}$ & $\Q$ & $E_{-3} \times \Jac(C_2)$ & $1$ & $y^2 = x^5 - x + 1$ & \ref{exa:type C D}\\
& & $\Q$ & $\Jac(C_3)$ & $1$ & $y^2 = x^8 -  x^6 - 3x^4 + x^2 - 1$ & \\[1pt]
\hline&&&&&&\\[-8pt]
$\bE$ & $\stgroup[\langle 6,1 \rangle]{1.6.E.6.1a}$ & $\Q$ & $\Res C_{1,3}$ & $1$ & $y^2 = x^3 - x + \alpha$ & \ref{exa:type E} \\
& & $\Q$ & $\Jac(C_3)$ & $1$ & $y^2=2x^7 + 4x^6 - 7x^4 + 4x^3 - 4x + 1$ & \\[1pt]
\hline&&&&&&\\[-8pt]
$\bF$ & $\stgroup[\langle 4,2 \rangle]{1.6.F.4.2a}$ & $\Q$ & $E_{-8} \times \Jac(C_2)$ & $1$ & $y^2 = x^6 + x^5 + x - 1$ & \ref{exa:type F} \\
& & $\Q$ & $\Jac(C_3)$ & $1$ & $y^2 = x^8 + 2x^6 + 4x^4 + 4x^2 + 4$ & \\[1pt]
\hline&&&&&&\\[-8pt]
$\bG$ & $\stgroup[\langle 4,1 \rangle]{1.6.G.4.1a}$ & $\Q$ & $E_0 \times \Jac(C_2)$ & $1$ & $y^2 = x^5 + 1$ & \ref{exa:type G C4} \\
& & $\Q$ & $\Jac(C_3)$ & $1$ & $y^2 = x^8 - 5x^6 + 10x^4 - 10x^2 + 5$ & \\
$\bG$ & $\stgroup[\langle 4,2 \rangle]{1.6.G.4.2a}$ & $\Q$ & $E_0 \times E_{-4} \times E_{-8}$ & $1$ & none & \ref{exa:type G}\\
& & $\Q$ & $\Jac(C_3)$ & $1$ & $y^2 = x^8 + 4x^6 - 2x^4 + 4x^2 + 1$ & \\[1pt]
\hline&&&&&&\\[-8pt]
$\bH$ & $\stgroup[\langle 6,2 \rangle]{1.6.H.6.2a}$ & $\Q$ & $\Jac(C_3)$ & $1$ & $y^2 = x^7 - 1$ & \ref{exa:type H}\\
$\bH$ & $\stgroup[\langle 8,2 \rangle]{1.6.H.8.2a}$ & $\Q$ & $E_{-3} \times \Jac(C_2)$ & $1$ & $y^2 = x^5 + 1$ &\ref{exa:type H}\\
& & $\Q$ & $\Jac(C_3)$ & $1$ & $y^2 = x^8 - 8x^6 + 20x^4 - 16x^2 + 2$&  \\
$\bH$ & $\stgroup[\langle 8,5 \rangle]{1.6.H.8.5a}$ & $\Q$ & $E_{-3} \times E_{-4} \times E_{-8}$ & $1$ & none & \ref{exa:type H}\\
& & $\Q$ & $\Jac(C_3)$ & $1$ & $3X^4 + 2Y^4 + 6Z^4 - 6X^2Y^2 +6X^2Z^2 - 12Y^2Z^2$ & \\[1pt]
\hline&&&&&&\\[-8pt]
$\bI$ & $\stgroup[\langle 8,3 \rangle]{1.6.I.8.3a}$ & $\Q$ & $E_0 \times \Jac(C_2)$ & $1$ & $y^2 = x^5 + x^3 + 2x$ & \ref{exa:type I JE4}\\
& & $\Q$ & $\Jac(C_3)$ & $1$ & $y^2 = x^8 + x^4 + 2$ & \\
$\bI$ & $\stgroup[\langle 12,4 \rangle]{1.6.I.12.4a}$ & $\Q$ & $E_0 \times \Jac(C_2)$ & $1$ & $y^2 = x^6 + x^3 - 2$ & \ref{exa:type I JE6}\\
& & $\Q$ & $\Jac(C_3)$ & $1$ & $ y^2 = x^8 + 8x^6 + 18x^4 + 16x^2 - 4$ & \\[1pt]
\hline&&&&&&\\[-8pt]
$\bJ$ & $\stgroup[\langle 16,11 \rangle]{1.6.J.16.11a}$ & $\Q$ & $E_{-3} \times \Jac(C_2)$ & $1$ & $y^2 = x^5 + x^3 + 2x$ & \ref{exa:type I JE4} \\
& & $\Q$ & $\Jac(C_3)$ & $1$ & $y^2 = x^8 + 2x^6 + 4x^2 - 4$ & \\
$\bJ$ & $\stgroup[\langle 24,14 \rangle]{1.6.J.24.14a}$ & $\Q$ & $E_{-4} \times \Jac(C_2)$ & $1$ & $y^2 = x^6 + x^3 - 2$ & \ref{exa:type I JE6} \\
& & $\Q$ & $\Jac(C_3)$ & $1$ & $2Y^4+4X^2Y^2-6Y^2Z^2+6X^3Z+XZ^3+3Z^4$ & \\[1pt]
\hline&&&&&&\\[-8pt]
$\bK$ & $\stgroup[\langle 24,14 \rangle]{1.6.K.24.14a}$ & $\Q$ & $E_0 \times \Jac(C_2)$ & $1$ & $y^2 = x^6 + 3x^5 + 10x^3 - 15x^2 + 15x - 6$ &  \ref{exa:type J JD6} \\
& & $\Q$ & $\Jac(C_3)$ & $1$ & $X^4 - 8X^2Y^2 + 24Y^4 + 24XY^2Z - 12XZ^3 - 9Z^4$ & \\
$\bK$ & $\stgroup[\langle 48,48 \rangle]{1.6.K.48.48a}$ & $\Q$ & $E_0 \times \Jac(C_2)$ & $1$ & $y^2 = x^6 - 5x^4 + 10x^3 - 5x^2 + 2x - 1$ & \ref{exa:type J JO} \\
& & $\Q$ & $\Jac(C_3)$ & $1$ & $X^4  - Y^4 + 2X^2Z^2 + XZ^3$ & \\[1pt]
\hline&&&&&&\\[-8pt]
$\bL$ & $\stgroup[\langle 48,51 \rangle]{1.6.L.48.51a}$ & $\Q$ & $E_{-7} \times \Jac(C_2)$ & $1$ & $y^2 = x^6 + 3x^5 + 10x^3 - 15x^2 + 15x - 6$ & \ref{exa:type J JD6} \\
$\bL$ & $\stgroup[\langle 96,226 \rangle]{1.6.L.96.226a}$& $\Q$ & $E_{-3} \times \Jac(C_2)$ & $1$ & $y^2 = x^6 - 5x^4 + 10x^3 - 5x^2 + 2x - 1$  & 
\ref{exa:type J JO} \\[1pt]
\hline&&&&&&\\[-8pt]
$\bM$ & $\stgroup[\langle 12,4 \rangle]{1.6.M.12.4a}$ & $\Q$ & $E_1 \times \Jac(C_2)$ & $1$ & $y^2 = x^6 + x^3 + 4$ & \ref{exa:type M D6}\\
& & $\Q$ & $\Jac(C_3)$ & $1$ & $X^3 Z + 4Y^3Z  + 3X^2Y^2 - Z^4$ & \\
$\bM$ & $\stgroup[\langle 24,12 \rangle]{1.6.M.24.12a}$ & $\Q$ & $\Res C_{1,3}$ & $1$ & $\alpha y^2 = x^3 - x + 1$  & \ref{exa:type M S4} \\
& & $\Q$ & $\Jac(C_3)$ & $1$ & $y^2 = x^8 + 2x^7 - 14x^5 - 49x^4 - 42x^3 - 14x^2 + 2x + 6$ & \\[1pt]
\hline&&&&&&\\[-9pt]
$\bN$ & $\stgroup[\langle 48,15\rangle]{1.6.N.48.15b}$ & $\Q$ & $E_{-3} \times A_{-3[2]a}$ & $2$ & none & \ref{exa:type N 48 15 3} \\
$\bN$ & $\stgroup[\langle 48,15 \rangle]{1.6.N.48.15a}$ & $\Q(\sqrt{3})$ & $E_{-4,12-8\sqrt{3}} \times\! A_{-4[3]a}$ & $6$ & none & \ref{exa:type N 48 15 4} \\
$\bN$ & $\stgroup[\langle 48,38 \rangle]{1.6.N.48.38b}$ & $\Q$ & $\Jac(C_3)$ &  $1$ & $y^2 = x^8\! - x^7\! - 7x^6\! - 7x^5\! - 35x^4\! - 35x^3\! - 21x^2\! + 3x + 6$ & \ref{exa:type N JB34} \\
& & $\Q$ & $E_{-4,2} \times\! A_{-4a}$  & $3$ & none & \\
$\bN$ & $\stgroup[\langle 48,41 \rangle]{1.6.N.48.41a}$ & $\Q(\sqrt{3})$ & $E_{-4,12-8\sqrt{3}} \times\! A_{-4[3]b}$ & $1$ & none & \ref{exa:type N 48 41}\\
$\bN$ & $\stgroup[\langle 96,193 \rangle]{1.6.N.96.193a}$ & $\Q$ & $E_{-8} \times \Jac(C_2)$ & $1$ & $y^2 = x^6 - 5x^4 + 10x^3 - 5x^2 + 2x - 1$ & \ref{exa:type N 96} \\
& & $\Q$ & $E_{-8} \times A_{-8}$ & $1$ & none & \\
$\bN$ & $\stgroup[\langle 144,125 \rangle]{1.6.N.144.125a}$ & $\Q$ &  $E_{-3} \times\! A_{-3}$ & $6$ & none & \ref{exa:type N 144 125} \\
$\bN$ & $\stgroup[\langle 144,127 \rangle]{1.6.N.144.127a}$ & $\Q$ &  $E_{-3} \times\! A_{-3[2]b}$ & $3$ & none & \ref{exa:type N 144 127} \\
$\bN$ & $\stgroup[\langle 192,988 \rangle]{1.6.N.192.988a}$ & $\Q$ & $E_{-4} \times \Prym(C_3/C_1)$ & $2$ & $X^4  - Y^4 + 2X^2Z^2 + XZ^3$ & \ref{exa:type N 192 988} \\
& & $\Q$ & $E_{-4} \times A_{-4b}$ & $2$ & none &  \\
\hline&&&&&&\\[-8pt]
$\bN$ & $\stgroup[\langle 192,956 \rangle]{1.6.N.192.956a}$ & $\Q$ & $\Res C_{1,3}$ & $1$ &  $y^2 = x^3 - \alpha x$ &\ref{exa:type N 192 956} \\
& & $\Q$ & $\twist(E_{-4}^3)$ & 1 & none & \\
& & $\Q$ & $\Jac(C_3)$ & $1$ & \parbox[c]{5.6cm}{$44X^4 + 120X^3Y + 36X^3Z + 60X^2YZ + 9X^2Z^2 - 200XY^3 + XZ^3 - 150Y^4 - 15Y^2Z^2$} & \\
$\bN$ & $\stgroup[\langle 336,208 \rangle]{1.6.N.336.208a}$ & $\Q$ & $\twist(E_{-7}^3)$ & $1$ & none & \ref{exa:type N 336 208} \\
& &  $\Q$ & $\Jac(C_3)$ & $1$ & $2X^3(Z\!-\!Y) \!+\! 3(X^2\!-\!Y^2)Z^2 \!+\! 2X(Z^3\!-\!Y^3)\!+\! 4Y^3Z\!+\!YZ^3$ & \\
$\bN$ & $\stgroup[\langle 432,523 \rangle]{1.6.N.432.523a}$ & $\Q$ & $\Res C_{1,3}$ & $1$ & $y^2 = x^3 - \alpha$& \ref{exa:type N 432 523} \\
& & $\Q$ & $\twist(E_{-3}^3)$ & 1 & none &\\
$\bN$ & $\stgroup[\langle 432,734 \rangle]{1.6.N.432.734a}$ & $\Q$ & $\twist(E_{-3}^3)$ & $3$ & none & \ref{exa:type N 432 734} \\[2pt]
\hline
\end{tabular}
\vspace{5pt}
\small
\caption{Explicit realizations of maximal Sato--Tate groups of abelian threefolds.
See \S\ref{subsec: realization techniques} for notation.\vspace{-8pt}}\label{table: maximal ST groups}
\end{table}

\begin{table}[ht]
\renewcommand{\arraystretch}{1.25}
\footnotesize
\begin{tabular}{|c|c|c|c|}
\hline
\!Type\! & ID & Endomorphism field(s) & Ex. \\
\hline
$\bA$ & $\stgroup[\langle 1,1 \rangle]{1.6.A.1.1a}$ & $\Q$ & \ref{exa:type A} \rule{0pt}{10pt}\!\!\\
\hline
$\bB$ & $\stgroup[\langle 2,1 \rangle]{1.6.B.2.1a}$ & $\Q(i)$ & \ref{exa:type B}  \rule{0pt}{10pt}\!\!\\
\hline
$\bC$ & $\stgroup[\langle 1,1 \rangle]{1.6.C.1.1a}$ & $\Q$ & \ref{exa:type C D} \rule{0pt}{10pt}\!\!\\
\hline
$\bD$ & $\stgroup[\langle 2,1 \rangle]{1.6.D.2.1a}$ & $\Q(\zeta_3)$ & \ref{exa:type C D} \rule{0pt}{10pt}\!\!\\
\hline
$\bE$ & $\stgroup[\langle 6,1 \rangle]{1.6.E.6.1a}$ & $\spl(\nfield[x^3-2]{3.1.108.1})$ & \ref{exa:type E}  \rule{0pt}{10pt}\!\!\\
\hline
$\bF$ & $\stgroup[\langle 4,2 \rangle]{1.6.F.4.2a}$ & $\Q(\zeta_{8})$ & \ref{exa:type F}  \rule{0pt}{10pt}\!\!\\
\hline
$\bG$ & $\stgroup[\langle 4,1 \rangle]{1.6.G.4.1a}$ & $\Q(\zeta_5)$ & \ref{exa:type G C4} \rule{0pt}{10pt}\!\!\\
$\bG$ & $\stgroup[\langle 4,2 \rangle]{1.6.G.4.2a}$ & $\Q(\zeta_8)$ & \ref{exa:type G} \\
\hline
$\bH$ & $\stgroup[\langle 6,2 \rangle]{1.6.H.6.2a}$ & $\Q(\zeta_7)$ & \ref{exa:type H} \rule{0pt}{10pt}\!\!\\
$\bH$ & $\stgroup[\langle 8,2 \rangle]{1.6.H.8.2a}$ & $\Q(\zeta_{15})$, $\Q(\zeta_{16})$ & \ref{exa:type H}\\
$\bH$ & $\stgroup[\langle 8,5 \rangle]{1.6.H.8.5a}$ & $\Q(\zeta_{24})$& \ref{exa:type H}\\
\hline
$\bI$ & $\stgroup[\langle 8,3 \rangle]{1.6.I.8.3a}$ & $\spl(\nfield[x^4-2]{4.2.2048.1})$& \ref{exa:type I JE4} \rule{0pt}{10pt}\!\!\\
$\bI$ & $\stgroup[\langle 12,4 \rangle]{1.6.I.12.4a}$ & $\spl(\nfield[x^6+2]{6.0.1492992.1}),\ \ \spl(\nfield[x^6-4x^3+1]{6.2.1259712.2})$ & \ref{exa:type I JE6}\\
\hline
$\bJ$ & $\stgroup[\langle 16,11 \rangle]{1.6.J.16.11a}$ & $\spl(\nfield[x^8 - 2x^4 + 4]{8.0.339738624.4}),\ \ \spl(\nfield[x^8 + 30x^4 + 9]{8.0.3057647616.6})$ & \ref{exa:type I JE4}  \rule{0pt}{10pt}\!\!\\
$\bJ$ & $\stgroup[\langle 24,14 \rangle]{1.6.J.24.14a}$ & $\spl(\nfield[x^{12} + 3x^8 + 27x^4 + 9]{12.0.320979616137216.14})$, & \ref{exa:type I JE6} \\
& & $\spl(x^{12} + 10x^{10} + 177x^8 + 244x^6 + 172x^4 + 96x^2 + 36)$ & \\
\hline
$\bK$ & $\stgroup[\langle 24,14 \rangle]{1.6.K.24.14a}$ & $\spl(\nfield[x^{12} - 4x^9 + 8x^6 - 4x^3 + 1]{12.0.8916100448256.1}), \ \ \spl(\nfield[x^{12}-4x^6+1]{12.4.6499837226778624.22})$ & \ref{exa:type J JD6}  \rule{0pt}{10pt}\!\!\\
$\bK$ & $\stgroup[\langle 48,48 \rangle]{1.6.K.48.48a}$ & $\spl(\nfield[x^6 + 4x^4 - 22]{6.2.2725888.2}),\ \ \spl(\nfield[x^6-4x^4-59]{6.2.13144256.4})$  & \ref{exa:type J JO} \\
\hline
$\bL$ & $\stgroup[\langle 48,51 \rangle]{1.6.L.48.51a}$ & $\spl(x^{24} - 660x^{18} + 12134x^{12} - 660x^6 + 1)$ & \ref{exa:type J JD6} \rule{0pt}{10pt}\!\! \\
$\bL$ & $\stgroup[\langle 96,226 \rangle]{1.6.L.96.226a}$ & $\spl(x^{12} - 9x^8 - 38x^6 - 9x^4 + 1 )$ & \ref{exa:type J JO} \\
\hline
$\bM$ & $\stgroup[\langle 12,4 \rangle]{1.6.M.12.4a}$ & $\spl(\nfield[x^6-2x^3+2]{6.0.186624.1})$ & \ref{exa:type M D6}\rule{0pt}{10pt} \\
$\bM$ & $\stgroup[\langle 24,12 \rangle]{1.6.M.24.12a}$ & $\spl(\nfield[x^4 - 2x^3 - 6x + 3]{4.2.3888.1})$ & \ref{exa:type M S4} \\
\hline 
$\bN$ & $\stgroup[\langle 48,15\rangle]{1.6.N.48.15b}$ & $\spl(x^{24} + 488x^{18} + 63384x^{12} - 5872x^6 + 1372)$ & \ref{exa:type N 48 15 3} \rule{0pt}{10pt}\!\! \\
$\bN$ & $\stgroup[\langle 48,15 \rangle]{1.6.N.48.15a}$ & $\spl(x^{24}-119x^{20}+3626x^{16}+44149x^{12}+10152x^8-93296x^4+21952)$ & \ref{exa:type N 48 15 4} \\
$\bN$ & $\stgroup[\langle 48,38 \rangle]{1.6.N.48.38b}$ & $\spl(\nfield[x^{12} - 3x^8 + 24x^4 - 48]{12.2.962938848411648.1})$ & \ref{exa:type N JB34} \\
$\bN$ & $\stgroup[\langle 48,41 \rangle]{1.6.N.48.41a}$ & $\spl(x^{24}\! - 208x^{20}\! + 12434x^{16}\! - 161872x^{12}\! + 766481x^8\! - 5190400x^4\! + 10240000)$ & \ref{exa:type N 48 41}\\
$\bN$ & $\stgroup[\langle 96,193 \rangle]{1.6.N.96.193a}$ & $\spl(x^{16} - 44x^{12} - 308x^{10} - 990x^8 - 1936x^6 - 2662x^4 + 9196x^2 + 20449)$ & \ref{exa:type N 96} \\
$\bN$ & $\stgroup[\langle 144,125 \rangle]{1.6.N.144.125a}$ & $\spl(x^{24} + 12x^{12} + 28x^6 - 12)$ & \ref{exa:type N 144 125} \\
$\bN$ & $\stgroup[\langle 144,127 \rangle]{1.6.N.144.127a}$ & $\spl(x^{16} - 18x^{12} + 101x^{10} + 99x^8 - 1098x^6 + 3043x^4 - 738x^2 + 81, x^3-3)$ & \ref{exa:type N 144 127} \\
$\bN$ & $\stgroup[\langle 192,988 \rangle]{1.6.N.192.988a}$ & $\spl(x^{24} - 40x^{16} + 430x^{12} - 240x^8 + 344x^4 - 59)$ & \ref{exa:type N 192 988}\\
\hline
$\bN$ & $\stgroup[\langle 192,956 \rangle]{1.6.N.192.956a}$ & $\spl(x^{12} + x^8 + 2x^4 + 4),\ \ \spl(x^{12} + 3x^4 - 1)$ & \ref{exa:type N 192 956} \rule{0pt}{10pt}\!\! \\
$\bN$ & $\stgroup[\langle 336,208 \rangle]{1.6.N.336.208a}$ & $\spl(\nfield[x^8 +4x^7 +21x^4 +18x +9]{8.2.153692888832.1})$ & \ref{exa:type N 336 208} \\
$\bN$ & $\stgroup[\langle 432,523 \rangle]{1.6.N.432.523a}$ & $\spl(x^{36} + x^{30} + 2x^{24} - 42x^{18} + 28x^{12} - 384x^6 + 4096)$ & \ref{exa:type N 432 523} \\
$\bN$ & $\stgroup[\langle 432,734 \rangle]{1.6.N.432.734a}$ & $\spl(x^9 + 168x^6 + 1080x^5 - 636x^3 + 864x^2 - 432x + 8)$ & \ref{exa:type N 432 734} \\
\hline
\end{tabular}
\bigskip

\caption{Endomorphism fields of the abelian threefolds listed in Table~\ref{table: maximal ST groups}.
In the two cases where the threefold is defined over $k \neq \Q$ (Examples~\ref{exa:type N 48 15 4} and~\ref{exa:type N 48 41}),
we give an extension of $\Q$ linearly disjoint from $k$ whose compositum with $k$ is the endomorphism field.}
\label{table: maximal endomorphism fields}
\end{table}


\begin{thebibliography}{KLLRSS18}

\bibitem[AAC...23]{AAC22}
P.B. Allen, F. Calegari, A. Caraiani, T. Gee, D. Helm, B.V. Le Hung, J. Newton, P. Scholze, R. Taylor, and
J.A. Thorne, \doi{10.4007/annals.2023.197.3.2}{Potential automorphy over CM fields},
\textit{Ann. of Math.} \textbf{197} (2023), 897--1113.
\mr{4564261}

\bibitem[AS08]{AS08}
I. Aliev and C. Smyth, \doi{10.1515/form.2011.087}{Solving algebraic equations in roots of unity}, 
\textit{Forum Math.} \textbf{24} (2012), 641--665. \mr{2926639}

\bibitem[ACLLM18]{ACLLM18}
S. Arora, V. Cantoral-Farf\'an, A. Landesman, D. Lombardo, and J.S. Morrow,
\doi{10.1007/s00209-018-2049-6}{The twisting Sato--Tate group of the curve $y^2=x^8-14x^4+1$},
\textit{Math. Z.} \textbf{290} (2018), 991--1022. \mr{3856841}

\bibitem[Asc84]{Asc84}
M. Aschbacher, \doi{10.1007/BF01388470}{On the maximal subgroups of the finite classical groups}, \textit{Invent. Math.} \textbf{76} (1984), 469--514. \mr{0746539}

\bibitem[AFP21]{AFP21}
S. Asif, F. Fit\'e, and D. Pentland,
\doi{10.1090/mcom/3675}{Computing $L$-polynomials of Picard curves from Cartier-Manin matrices (with an appendix by A.V. Sutherland)}, \textit{Math. Comp.} \textbf{91} (2022), 943--971. \mr{4379983}

\bibitem[BG08]{BG08}
S. Baba and H. Granath, \doi{10.4153/CJM-2008-033-7}{Genus 2 curves with quaternionic multiplication},
\textit{Canad. J. Math.} \textbf{60} (2008), 734--757. \mr{2423455}

\bibitem[BaK15]{BK15}
G. Banaszak and K.S. Kedlaya,
\jstor{26315458}{An algebraic Sato--Tate group and Sato--Tate conjecture},
\textit{Indiana Univ. Math. J.} \textbf{64} (2015), 245--274. \mr{3320526}

\bibitem[BaK16]{BaK16}
G. Banaszak and K.S. Kedlaya,
\doi{10.1090/conm/663/13348}{Motivic Serre group, algebraic Sato--Tate group, and Sato--Tate conjecture}, in \textit{Frobenius Distributions: Lang-Trotter and Sato--Tate Conjectures}, Contemporary Math. 663, Amer. Math. Soc., 2016, 11--44. \mr{3502937}

\bibitem[BaK23]{BK21}
G. Banaszak and K.S. Kedlaya, Motivic Serre group and Sato--Tate conjecture, \arXiv{2302.13016}{1} (2023).

\bibitem[BLGG11]{BLGG11}
T. Barnet-Lamb, T. Gee, and D. Geraghty, \doi{10.1090/S0894-0347-2010-00689-3}{The Sato--Tate conjecture for Hilbert modular forms},
\textit{J. Amer. Math. Soc.} \textbf{24} (2011), 411--469. \mr{2748398}

\bibitem[BEO01]{BEO01}
H.U. Besche, B. Eick, and E. O'Brein, \doi{10.1142/S0218196702001115}{A millennium project: constructing small groups},
\textit{Int. J. Alg. Comp.} \textbf{12} (2001), 623--644;
\url{http://www.icm.tu-bs.de/ag_algebra/software/small/}. \mr{1935567}

\bibitem[BS02]{BS02}
F. Beukers and C. Smyth, \href{https://www.maths.ed.ac.uk/~chris/preprints/beukers_smyth.pdf}{Cyclotomic points on curves}, in \textit{Number theory for the millennium, I (Urbana, IL, 2000)},
A.K. Peters, Natick, 2002, 67--85. \mr{1956219}

\bibitem[BvG16]{BvG16}
M.A. Bonfanti and B. van Geemen, \doi{10.4153/CJM-2014-045-4}{Abelian surfaces with an automorphism and quaternionic multiplication}, \textit{Canad. J. Math.} \textbf{68} (2016), 24--43.

\bibitem[BSSVY16]{BSSVY16}
A.R. Booker, J. Sijsling, A.V. Sutherland, J. Voight, and D. Yasaki, \doi{10.1112/S146115701600019X}{A database of genus 2 curves over the rational numbers}, in \textit{Proceedings of the Twelfth Algorithmic Number Theory Symposium (ANTS XII)}, LMS Journal of Computation and Mathematics \textbf{19} (2016), 235-254. \mr{3540958}

\bibitem[BCGP18]{BCGP18}
G. Boxer, F. Calegari, T. Gee, and V. Pilloni, \doi{10.1007/s10240-021-00128-2}{Abelian surfaces over totally real fields are potentially modular}, \textit{Publ. Math. IH\'ES} \textbf{134} (2021), 153--501. \mr{4349242}

\bibitem[BHRD13]{BHRD13}
J.N. Bray, D.F. Holt, and C.M. Roney-Dougal, \doi{10.1017/CBO9781139192576.001}{\textit{The Maximal Subgroups of the Low-Dimensional Finite Classical Groups}}, London Math. Soc. Lecture Note Series 407, Cambridge Univ. Press, Cambridge, 2013. \mr{3098485}

\bibitem[BuK16]{BuK16}
A. Bucur and K.S. Kedlaya,
\doi{10.1090/conm/663/13349}{An application of the effective Sato--Tate conjecture},
in \textit{Frobenius Distributions: Lang-Trotter and Sato--Tate Conjectures}, 
Contemp. Math. 663, Amer. Math. Soc., 2016, 45--56. \mr{3502938}

\bibitem[BFK20]{BFK20}
A. Bucur, F. Fit\'e, and K.S. Kedlaya,
Effective Sato--Tate conjecture for abelian varieties and applications,
\arXiv{2002.08807}{1} (2020).

\bibitem[BG06]{BG06} D. Bump and A. Gamburd, \doi{10.1007/s00220-006-1503-1}{On the averages of characteristic polynomials from classical groups}. \textit{Comm. Math. Phys.} \textbf{265} (2006), 227--274. \mr{2217304}

\bibitem[CMS11]{CMS11}
F. Calegari, S. Morrison, and N. Snyder,
\doi{10.1007/s00220-010-1136-2}{Cyclotomic integers, fusion categories, and subfactors},
\textit{Commun. Math. Phys.} \textbf{303} (2011), 845--896. \mr{2786219}

\bibitem[Car01]{Car01}
G. Cardona Juanals, Models racionals de corbes de g\`enere 2, PhD Thesis, Universitat Polit\`ecnica de Catalunya,
2001.

\bibitem[CT18]{CT18}
W. Castryck and J. Tuitman, \doi{10.1093/qmath/hax031}{Point counting on curves using a gonality preserving lift},
\textit{Q.J. Math.} \textbf{69} (2018), 33--74. \mr{3771385}

\bibitem[Chi92]{Chi92}
W. Chi, \doi{10.2307/2374706}{$\ell$-adic and $\lambda$-adic representations associated to abelian varieties defined over number fields}, \textit{American Journal of Mathematics} \textbf{114} (1992), 315--353. \mr{1156568}

\bibitem[CJ76]{CJ76}
J.H. Conway and A.J. Jones, \doi{10.4064/aa-30-3-229-240}{Trigonometric Diophantine equations (On vanishing sums of roots of unity)},
\textit{Acta Arith.} \textbf{30} (1976), 229--240. \mr{0422149}

\bibitem[CFS19]{CFS19}
E. Costa, F. Fit\'e, A.V. Sutherland, \doi{10.1016/j.crma.2019.11.008}{Arithmetic invariants from Sato--Tate moments}, \textit{Comptes Rendus Mathematique} \textbf{357} (2019), 823--826. \mr{4038255}

\bibitem[CKR20]{CKR20}
E. Costa, K.S. Kedlaya, and D. Roe,
\doi{10.2140/obs.2020.4.143}{Hypergeometric $L$-functions in average polynomial time}, 
in \textit{Proceedings of the Fourteenth Algorithmic Number Theory Symposium (ANTS-XIV)},
Open Book Series 4, Math. Sci. Pub., Berkeley, 2020, 143--159. \mr{4235111}

\bibitem[CMSV19]{CMSV19}
E. Costa, N. Mascot, J. Sijsling, and J. Voight,
\doi{10.1090/mcom/3373}{Rigorous computation of the endomorphism ring of a Jacobian},
\textit{Math. Comp.} \textbf{88} (2019), 1303--1339. \mr{3904148}

\bibitem[Cos21]{C21}
E. Costa, \textsc{crystalline\_obstruction} library, \GitHub{} repository \url{https://github.com/edgarcosta/crystalline_obstruction} (retrieved Mar 2021).

\bibitem[CS21]{CS21}
E. Costa and E.C. Sert\"{o}z, \doi{10.1007/978-3-030-80914-0_9}{Effective obstruction to lifting Tate classes from positive characteristic}, in \textit{Arithmetic Geometry, Number Theory, and Computation}, Simons Symposia, Springer, 2021, 293-333. \mr{4427967}

\bibitem[CHS23]{CHS23}
E. Costa, D. Harvey, and A.V. Sutherland, \doi{10.1007/s40993-022-00397-8}{Counting points on smooth plane quartics}, \textit{Res. Num. Theory} \textbf{9} (2023), article number 1. \mr{4514545}

\bibitem[CS19]{CS19}
E. Costa and J. Sijsling, \textsc{Endomorphisms} library, \GitHub{} repository
\url{https://github.com/edgarcosta/endomorphisms/} (retrieved Feb 2021).

\bibitem[CN21]{CN21}
J. Cremona and F. Najman, \doi{10.1007/s40993-021-00270-0}{$\Q$-curves over odd degree number fields}, \textit{Res. Number Theory} \textbf{7} (2021), Paper No. 62, 30pp. \mr{4314224}

\bibitem[DD13]{DD13}
T. Dokchitser and V. Dokchitser, \doi{10.2140/ant.2013.7.1325}{Identifying Frobenius elements in Galois groups}, \textit{Algebra Number Theory} \textbf{7} (2013), 1325--1352. \mr{3107565}

\bibitem[Del80]{Del80}
P. Deligne, \href{http://www.numdam.org/article/PMIHES_1980__52__137_0.pdf}{La conjecture de Weil, II}, \textit{Publ. Math. IH\'ES}
\textbf{52} (1980), 137--252. \mr{0601520}

\bibitem[Del82]{Del82}
P. Deligne, 
\href{https://www.jmilne.org/math/Documents/Deligne82.pdf}{Hodge cycles on abelian varieties}, in \textit{Hodge cycles, motives, and Shimura varieties} 
Lecture Notes in Math. 900, Springer, New York, 1982, 9--100. \mr{654325}

\bibitem[DM82]{DM82}
P. Deligne and J.S. Milne, 
\href{https://www.jmilne.org/math/xnotes/tc.pdf}{Tannakian Categories}, in \textit{Hodge cycles, motives, and Shimura varieties} 
Lecture Notes in Math. 900, Springer, New York, 1982, 101--228. \mr{654325}

\bibitem[Dod84]{Dod84}
B. Dodson, \doi{10.2307/1999987}{The structure of Galois groups of CM-fields}, \textit{Trans. Amer. Math. Soc.} \textbf{283} (1984),
1--32. \mr{0735406}

\bibitem[Dol12]{Dol12}
I. Dolgachev, \doi{10.1017/CBO9781139084437}{\textit{Classical Algebraic Geometry. A Modern View}},
Cambridge Univ. Press, Cambridge, 2012. \mr{2964027}

\bibitem[EJ23]{EJ23}
A.-S. Elsenhans and J. Jahnel, \doi{10.1007/s40993-022-00378-x}{Frobenius trace distributions for K3 surfaces}, \textit{Res. Num. Theory} \textbf{8} (2022), article number 85. \mr{4537066}

\bibitem[FG18]{FG18}
F. Fit\'e and X. Guitart, \doi{10.1090/tran/7074}{Fields of definition of elliptic $k$-curves and the realizability
of all genus 2 Sato--Tate groups over a number field}, \textit{Trans. Amer. Math. Soc.} \textbf{370} (2018),
4623--4659. \mr{3812090}

\bibitem[FKRS12]{FKRS12}
F. Fit\'e, K.S. Kedlaya, V. Rotger, and A.V. Sutherland, \doi{10.1112/S0010437X12000279}{Sato--Tate distributions and Galois endomorphism modules in genus $2$}, \textit{Compos. Math.} \textbf{148} (2012), 1390--1442. \mr{2982436}

\bibitem[FKS18]{FKS18}
F. Fit\'e, K.S. Kedlaya, and A.V. Sutherland, 
\doi{10.1090/conm/663/13350}{Sato--Tate groups of some weight 3 motives},
in \textit{Frobenius distributions: Lang-Trotter and Sato--Tate conjectures},
Contemp. Math., 663, Amer. Math. Soc., Providence, 2016, 57--101. \mr{3502939}

\bibitem[FKS21a]{FKS19}
F. Fit\'e, K.S. Kedlaya, and A.V. Sutherland,
\doi{10.1090/conm/770}{Sato--Tate groups of abelian threefolds: a preview of the classification}, in 
\textit{Arithmetic, Geometry, Cryptography and Coding Theory (AGC$^2$T 2019)}, 
Contemporary Math. 770, Amer. Math. Soc., 2021. \mr{4280389}

\bibitem[FKS21b]{FKS21}
F. Fit\'e, K.S. Kedlaya, and A.V. Sutherland,
\GitHub{} repository \url{https://github.com/kedlaya/satotateg3} (updated April 2022).

\bibitem[FLS18]{FLS18}
F. Fit\'e, E. Lorenzo Garc\'ia, and A.V. Sutherland, \doi{10.1007/s40687-018-0162-0}{Sato--Tate distributions of twists of the Fermat and the Klein quartics}, \textit{Res. Math. Sci.} \textbf{5}:41 (2018). \mr{3864839}

\bibitem[FS14]{FS14}
F. Fit\'e and A.V. Sutherland, \doi{10.2140/ant.2014.8.543}{Sato--Tate distributions of twists of $y^2=x^5-x$ and $y^2=x^6+1$},
\textit{Algebra Number Theory} \textbf{8} (2014), 543--585. \mr{3218802}

\bibitem[FS16]{FS16}
F. Fit\'e and A.V. Sutherland, \doi{10.1090/conm/663/13351}{Sato--Tate groups of $y^2=x^8+c$ and $y^2=x^7-cx$},
in \textit{Frobenius distributions: Lang-Trotter and Sato--Tate conjectures},
Contemp. Math., 663, Amer. Math. Soc., Providence, 2016, 103--126. \mr{3502940}

\bibitem[FOR08]{FOR08}
S. Flon, R. Oyono, and C. Ritzenthaler, \doi{10.1142/9789812793430\_0001}{Fast addition on non-hyperelliptic genus 3 curves}, in \textit{Algebraic geometry and its applications}, Ser. Number Theory Appl. \textbf{5}, World Sci. Publ., Hackensack, NJ, 2008, 1--28. \mr{2484046}

\bibitem[FH91]{FH91}
W. Fulton and J. Harris, \doi{10.1007/978-1-4612-0979-9}{\textit{Representation Theory: A First Course}},
Graduate Texts in Math., Springer, New York, 1991. \mr{1153249}

\bibitem[GAP20]{GAP20}
The GAP Group, \textit{GAP -- Groups, Algorithms, and Programming}, version 4.11.0, 2020, \url{https://www.gap-system.org}.

\bibitem[GS01]{GS01}
P. Gaudry and \'E. Schost, \doi{10.1007/3-540-45624-4\_39}{On the invariants of the quotients of the Jacobian of a curve of genus 2},
in \textit{AAECC 2001: Applied Algebra, Algebraic Algorithms and Error-Correcting Codes},
Lecture Notes in Comp. Sci. 2227, Springer, Berlin, 2001. \mr{1913484}

\bibitem[GK17]{GK17}
R. Guralnick and K.S. Kedlaya, \doi{10.1007/s40993-017-0088-4}{Endomorphism fields of abelian varieties},
\textit{Res. Number Theory} \textbf{3}:22 (2017). \mr{3720508}

\bibitem[HSS21]{HSS21}
J. Hanselman, S. Schiavone, and J. Sijsling, \doi{10.1090/mcom/3627}{Gluing curves of genus 1 and 2 along their 2-torsion}, \textit{Math. Comp.} \textbf{90} (2021) 2333-2379. \mr{4280304}

\bibitem[HSBT10]{HSBT10}
M. Harris, N. Shepherd-Barron, R. Taylor,
\doi{10.4007/annals.2010.171.779}{A family of Calabi-Yau varieties and potential automorphy},
\textit{Ann. of Math.} \textbf{171} (2010), 779--813. \mr{2630056}

\bibitem[HMS16]{HMS16}
D. Harvey, M. Massierer, and A.V. Sutherland,
\doi{10.1112/S1461157016000383}{Computing $L$-series of geometrically hyperelliptic curves of genus three},
Twelfth Algorithmic Number Theory Symposium (ANTS XII),
\textit{LMS J. Comp. Math.} \textbf{19} (2016), 220--234. \mr{3540957}

\bibitem[HS14]{HS14}
D. Harvey and A.V. Sutherland,
\doi{10.1112/S1461157014000187}{Computing Hasse-Witt matrices of hyperelliptic curves in average polynomial time}, 
Eleventh Algorithmic Number Theory Symposium (ANTS XI), 
\textit{LMS J. Comp. Math.} \textbf{17} (2014), 257--273. \mr{3240808}

\bibitem[HS16]{HS16}
D. Harvey and A.V. Sutherland,
\doi{10.1090/conm/663/13352}{Computing Hasse-Witt matrices of hyperelliptic curves in average polynomial time II}, 
in \textit{Frobenius distributions: Lang-Trotter and Sato--Tate conjectures},
Contemp. Math., 663, Amer. Math. Soc., Providence, 2016, 127--148. \mr{3502941}

\bibitem[HLP00]{HLP00}
E.W. Howe, F. Lepr\'evost, and B. Poonen, 
\doi{10.1515/form.2000.008}{Large torsion subgroups of split Jacobians of curves of genus two or three},
\textit{Forum Math.} \textbf{12} (2000), 315--364. \mr{1748483}

\bibitem[HTW05]{HTW05}
R. Howe, E.-C. Tan, and J.F. Willenbring,
\doi{10.1090/S0002-9947-04-03722-5}{Stable branching rules for classical symmetric pairs},
\textit{Trans. Amer. Math. Soc.} \textbf{357} (2005), 1601--1626. \mr{2115378}

\bibitem[Joh17]{Joh17}
C. Johansson, \doi{10.1090/tran/6847}{On the Sato--Tate conjecture for non-generic abelian surfaces (with an appendix by F. Fit\'e)},
\textit{Trans. Amer. Math. Soc.} \textbf{369} (2017), 6303--6325. \mr{3660222}

\bibitem[KKPR20]{KKPR20}
K.S. Kedlaya, A. Kolpakov, B. Poonen, and M. Rubinstein, Space vectors forming rational angles,
\arXiv{2011.14232}{1} (2020).

\bibitem[KS08]{KS08}
K.S. Kedlaya and A.V. Sutherland, 
\doi{10.1007/978-3-540-79456-1\_21}{Computing L-series of hyperelliptic curves}, in
\textit{Algorithmic Number theory: 8th International Symposium, ANTS-VIII (Banff, Canada, May
2008)},
Lect. Notes Comp. Sci. 5011, Springer, Berlin, 2008,
312--326. \mr{2467855}

\bibitem[KS09]{KS09}
K.S. Kedlaya and A.V. Sutherland, 
\doi{10.1090/conm/487/09529}{Hyperelliptic curves, $L$-polynomials, and random matrices},
in \textit{Arithmetic, Geometry, Cryptography, and Coding Theory: International Conference, November
5–9, 2007, CIRM, Marseilles, France}, 
Contemp. Math. 487, Amer. Math. Soc., Providence, 2009, 119--162. \mr{2555991}

\bibitem[K\i l16]{Kil16}
P. K\i l\i \c{c}er, The CM class number one problem for curves,
PhD Thesis, Universit\'e de Bordeaux, 2016, \url{https://tel.archives-ouvertes.fr/tel-01383309v2}.

\bibitem[Kin75]{Kin75}
R.C. King, \doi{10.1088/0305-4470/8/4/004}{Branching rules for classical Lie groups using tensor and spinor methods},
\textit{J. Phys.} \textbf{8} (1975), 429--449. \mr{0411400}

\bibitem[Kna02]{Kna02}
A. Knapp, \textit{Lie groups beyond an introduction}, second edition, Progress in Math. 140, Birkh\"auser, Boston, 2002. \mr{1920389}

\bibitem[Kos61]{Kos61}
B. Kostant, \doi{10.2307/1970237}{Lie algebra cohomology and the generalized Borel-Weil theorem},
\textit{Ann. of Math.} \textbf{74} (1961), 329--387. \mr{0142696}

\bibitem[Kum08]{Kum08}
A. Kumar, \doi{10.1093/imrn/rnm165}{K3 surfaces associated with curves of genus two},
\textit{Int. Math. Res. Notices} (2008), article rnm165. \mr{2427457}
 
\bibitem[LO20]{LO20} 
K.-H. Lee, S.-J. Oh, \doi{10.1016/j.aim.2022.108309}{Auto-correlation functions of Sato--Tate distributions and identities of symplectic characters}, Adv. Math. \textbf{401} (2022), Paper No. 108309. \mr{4394683}

\bibitem[LRRS14]{LRRS14}
R. Lercier, C. Ritzenthaler, F. Rovetta, and J. Sijsling,
\doi{10.1112/S146115701400031X}{Parametrizing the moduli space of curves and applications to smooth plane quartics over finite fields},
\textit{LMS J. Comput. Math.} \textbf{17} (2014), 128--147. \mr{3240800}

\bibitem[Ler12]{Ler12}
L. Leroux,
\doi{10.1090/S0025-5718-2011-02548-2}{Computing the torsion points of a variety defined by lacunary polynomials},
\textit{Math. Comp.} \textbf{81} (2012), 1587--1607. \mr{2904592}

\bibitem[LMFDB]{LMFDB}
The LMFDB Collaboration, \textit{The L-functions and Modular Forms Database}, release 1.2.1, \url{https://www.lmfdb.org}, 2021
(accessed Feb 2021).

\bibitem[LLRS23]{LLRS21}
D. Lombardo, E. Lorenzo Garc\'\i{}a, C. Ritzenthaler, and J. Sijsling,
\doi{10.1080/10586458.2021.1926008}{Decomposing Jacobians via Galois covers}, \textit{Exp. Math.} \textbf{32} (2023), 218--240.
\mr{4574430}

\bibitem[Lor18]{Lor18}
E. Lorenzo Garc\'\i{}a,
\doi{10.1142/S1793042118501075}{Twists of non-hyperelliptic curves of genus 3},
\textit{Int. J. Number Theory} \textbf{14} (2018), 1785--1812. \mr{3827959}

\bibitem[LP95]{LP95}
M. Larsen and R. Pink, 
\doi{10.1007/BF01444508}{Abelian varieties, $\ell$-adic representations, and $\ell$-independence}, \textit{Math. Ann.} \textbf{302} (1995), no. 3, 561--579. \mr{1339927}

\bibitem[Magma]{Magma}
The Magma Group, \textit{Magma} version 2.25-5, 2020, \url{http://magma.maths.usyd.edu.au/}.

\bibitem[MBD61]{MBD61}
G.A. Miller, H.F. Blichfeldt, and L.E. Dickson, \textit{Theory and Applications of Finite Groups},
Dover, New York, 1961. \mr{0123600}

\bibitem[MZ99]{MZ99}
B.J.J. Moonen and Yu.G. Zarhin, \doi{10.1007/s002080050333}{Hodge classes on abelian varieties of low dimension},
\textit{Math. Ann.} \textbf{315} (1999), 711--733. \mr{1731466}

\bibitem[MP22]{MP22}
N. M\"uller and R. Pink, \doi{10.1142/S1793042122500488}{Hyperelliptic curves with many automorphisms},
\textit{Intl. J. Num. Theory} \textbf{18} (2022), 913--930. \mr{4430357}

\bibitem[Mum69]{Mu69}
D. Mumford, A note of Shimura's paper ``Discontinuous groups and abelian varieties'',
\textit{Math. Ann.} \textbf{181} (1969), 345--351. \mr{0248146}

\bibitem[Nik75]{Nik75}
V. Nikulin, \doi{10.1007/BF01350672}{Kummer surfaces}, \textit{Izv. Akad. Nauk SSSR Ser. Mat.} \textbf{39} (1975), 278--293;
English translation: \textit{Math. USSR. Izv.} \textbf{9} (1975), 261--275. \mr{0429917}

\bibitem[OS20]{OS20}
A. Obus and T. Shaska, \doi{10.1090/mcom/3639}{Superelliptic curves with many automorphisms and CM Jacobians}, \textit{Math. Comp.} \textbf{90} (2021), 2951--2975. \mr{4305376}

\bibitem[Oor87]{Oo87} 
F. Oort, 
\doi{10.1016/B978-0-12-348032-3.50007-4}{Endomorphism algebras of abelian varieties}, in \textit{Algebraic geometry and commutative algebra, Vol. II}, Kinokuniya, Tokyo, 1988, 469--502. \mr{0977774}

\bibitem[Pari]{Pari}
The PARI Group, \textit{PARI/GP version 2.11.2}, Univ. Bordeaux, 2019, \url{http://pari.math.u-bordeaux.fr/}.

\bibitem[Pau08]{Pau08}
J. Paulhus, \doi{10.4064/aa132-3-3}{Decomposing Jacobians of curves with extra automorphisms},
\textit{Acta Arith.} \textbf{132} (2008), 231--244. \mr{2403651}

\bibitem[Per77]{Per77}
R. Perlis, \doi{10.1016/0022-314X(77)90070-1}{On the equation $\zeta_K(s)=\zeta_{K'}(s)$},
\textit{J. Number Theory} \textbf{9} (1977), 342--360. \mr{0447188}

\bibitem[PSV11]{PSV11}
D. Plaumann, B. Sturmfels, and C. Vinzant, \doi{10.1016/j.jsc.2011.01.007}{Quartic curves and their bitangents},
\textit{J. Symbolic Comput.} \textbf{46} (2011), 712--733. \mr{2781949}

\bibitem[PR98]{PR98}
B. Poonen and M. Rubinstein, \doi{10.1137/S0895480195281246}{The number of intersection points made by the diagonals of a regular polygon}, \textit{SIAM J. Disc. Math.} \textbf{11} (1998), 135--156. \mr{1612877}

\bibitem[RR18]{RR18}
C. Ritzenthaler and M. Romagny, \doi{10.46298/epiga.2018.volume2.3663}{On the Prym variety of genus $3$ covers of genus $1$ curves},
\textit{\'Epijournal G\'eom. Alg.} \textbf{2} (2018). \mr{3781951}

\bibitem[RS10]{RS10}
K. Rubin and A. Silverberg, \doi{10.1090/S0025-5718-09-02266-2}{Choosing the correct elliptic curve in the CM method}, \textit{Math. Comp.} \textbf{79} (2010), 545--561. \mr{2552240}

\bibitem[Sage]{Sage}
The Sage Development Team, \emph{{S}age {M}athematics {S}oftware ({V}ersion
  9.3)}, 2021, \url{http://www.sagemath.org}.
  
\bibitem[Sch07]{Sch07}
M. Sch\"utt, \doi{10.4310/CNTP.2007.v1.n2.a2}{Fields of definition of singular K3 surfaces},
\textit{Comm. Num. Theory Phys.} \textbf{1} (2007), 307--321. \mr{2346573}
  
\bibitem[Ser68]{Ser68}
J.-P. Serre, \textit{Abelian $\ell$-adic Representations and Elliptic Curves},
W.A. Benjamin, New York, 1968. \mr{1043865}

\bibitem[Ser84]{Ser84}
J.-P. Serre, R\'esum\'e des cours au Coll\`ege de France 1984-1985, in \textit{Oeuvres -- Collected papers IV},
Springer-Verlag, Berlin, IV (2013), 27--32. \mr{3185222}

\bibitem[Ser86]{Ser86}
J.-P. Serre, Lettre a Marie-France Vigneras du 10/2/1986, in \textit{Oeuvres -- Collected papers IV},
Springer-Verlag, Berlin, IV (2013), 137. \mr{3185222}

\bibitem[Ser94]{Ser94}
J.-P. Serre, Propri\'et\'es conjecturales des groupes de Galois motiviques et des repr\'esentations
$l$-adiques, in \textit{Motives (Seattle, WA, 1991)}, Proc. Symp. Pure Math. 55,
Amer. Math. Soc., Providence, 1994, 377--400. \mr{1265537}

\bibitem[Ser12]{Ser12}
J.-P. Serre, \href{https://www.college-de-france.fr/media/jean-pierre-serre/UPL9140953003901187132_14___Lectures_on_N_X.pdf}{\textit{Lectures on $N_X(p)$}}, CRC Press, Boca Raton, FL, 2012. \mr{2920749}
  
\bibitem[Sha06]{Sha06}
T. Shaska, 
\href{http://www.math.bas.bg/serdica/2006/2006-355-374.pdf}{Subvarieties of the hyperelliptic moduli determined by group actions},
\textit{Serdica Math. J.} \textbf{32} (2006), 355--374. \mr{2287373}

\bibitem[Shi16]{Shi16}
Y.-D. Shieh, \doi{10.1112/S1461157016000279}{Character theory approach to Sato--Tate group},
\textit{LMS J. Comp. Math.} \textbf{19} (2016), 301--314. \mr{3540962}
  
\bibitem[Shi71]{Shi71}
G. Shimura, \doi{10.2307/1970768}{On the zeta-function of an abelian variety with complex multiplication}, \textit{Ann. of
Math.} \textbf{94} (1971), 504--533. \mr{0288089}

\bibitem[Shi81]{Shi81}
T. Shioda,
\doi{10.1007/BF01450347}{Algebraic cycles on abelian varieties of Fermat type},
\textit{Math. Ann.} \textbf{258} (1981), 65--80. \mr{0641669}

\bibitem[Sil92]{Sil92}
A. Silverberg, \doi{10.1016/0022-4049(92)90141-2}{Fields of definition for homomorphisms of abelian varieties}, \textit{J. Pure Appl. Algebra}
\textbf{77} (1992), 253--262. \mr{1154704}

\bibitem[Sut09]{Sut09}
A.V. Sutherland, \doi{10.1090/S0025-5718-08-02143-1}{A generic approach to searching for Jacobians}, \textit{Math. Comp.} \textbf{78} (2009) 485--507. \mr{2448717}

\bibitem[Sut16]{Sut16}
A.V. Sutherland, \doi{10.1090/conm/740/14904}{Sato--Tate distributions},
in \textit{Analytic Methods in Arithmetic Geometry},
Contemp. Math. 740, Amer. Math. Soc., 2019, 197--248. \mr{4033732}

\bibitem[Sut18]{Sut18}
A.V. Sutherland, Arithmetic equivalence and isospectrality, lecture notes available at \url{https://math.mit.edu/~drew/ArithmeticEquivalenceLectureNotes.pdf}, 2018.

\bibitem[Sut19]{Sut19}
A.V. Sutherland,
\doi{10.2140/obs.2019.2.443}{A database of nonhyperelliptic genus 3 curves over $\Q$}, 
\textit{Thirteenth Algorithmic Number Theory Symposium (ANTS XIII)},
Open Book Series \textbf{2} (2019), 443--459. \mr{3952027}

\bibitem[Sut20]{Sut20}
A.V. Sutherland,
\doi{10.2140/obs.2020.4.403}{Counting points on superelliptic curves in average polynomial time},
\textit{Fourteenth Algorithmic Number Theory Symposium (ANTS XIV)},
Open Book Series \textbf{4} (2020), 403--422. \mr{4235126}

\bibitem[Sut21]{Sut21}
A.V. Sutherland, \doi{10.19086/da.29452}{Stronger arithmetic equivalence}, \textit{Discrete Analysis} (2021), paper no. 23. \mr{4341956}

\bibitem[Tay20]{Tay18}
N. Taylor, \doi{10.1090/tran/8025}{Sato--Tate distributions on abelian surfaces}, 
\textit{Trans. Amer. Math. Soc.} \textbf{373} (2020), 3541--3559. \mr{4082247}

\bibitem[Tui16]{Tui16}
J. Tuitman,
\doi{10.1090/mcom/2996}{Counting points on curves using a map to $\mathbf{P}^1$},
\textit{Math. Comp.} \textbf{85} (2016), 961--981. \mr{3434890}

\bibitem[Tui18]{Tui18}
J. Tuitman, \doi{10.1016/j.ffa.2016.12.008}{Counting points on curves using a map to $\mathbf{P}^1$, II},
\textit{Finite Fields Appl.} \textbf{45} (2017), 301--322. \mr{3631366}

\bibitem[Upt09]{Upt09}
M. Upton, \doi{10.1016/j.jalgebra.2009.04.021}{Galois representations attached to Picard curves}, \textit{J. Alg.} \textbf{322} (2009), 1038--1059. \mr{2537671}

\bibitem[Vas04]{Va04}
A. Vasiu, \doi{10.1007/s00229-004-0465-x}{Surjectivity criteria for $p$-adic representation, Part II}, \textit{Manuscripta Math.} \textbf{114} (2004), 399--422. \mr{2081941}

\bibitem[vW99a]{vW99}
P. van Wamelen, \doi{10.1090/S0025-5718-99-01020-0}{Examples of genus two CM curves defined over the rationals},
\textit{Math. Comp.} \textbf{68} (1999), 307--320. \mr{1609658}

\bibitem[Wey46]{We46} H. Weyl, \textit{The classical groups.  Their invariants and representations}, second ed., Princeton University Press, Princeton New Jersey, 1946. \mr{1609658}

\bibitem[Zar00]{Za00}
Y.G. Zarhin, \doi{10.4310/MRL.2000.v7.n1.a11}{Hyperelliptic Jacobians without complex multiplication},
\textit{Math. Res. Lett.} \textbf{7} (2000), 123--132. \mr{1748293}

\bibitem[Zar10]{Za10}
Y.G. Zarhin, \doi{10.1112/plms/pdp020}{Families of absolutely simple hyperelliptic Jacobians}, \textit{Proc. London Math. Soc.} \textbf{100} (2010), 24--54. \mr{2578467}

\bibitem[Zar18]{Zar18}
Y.G. Zarhin, \doi{10.1007/978-3-319-97379-1_22}{Endomorphism algebras of abelian varieties with special reference to superelliptic Jacobians}, in \textit{Geometry, Algebra, Number Theory,
and their Information Technology Applications}, Springer, Cham, 2018. \mr{3880401}

\bibitem[Zha14]{Zha14}
B. Zhao, On the Mumford--Tate conjecture of abelian fourfolds, preprint available at
\url{https://sites.google.com/site/zhaobinmath/} (retrieved March 2019).

\bibitem[Zyw21]{Zyw21}
D. Zywina, \textsc{Monodromy} library, \GitHub{} repository, \url{https://github.com/davidzywina/monodromy/} (retrieved Feb 2021).

\bibitem[Zyw22]{Zyw22}
D. Zywina, \doi{10.1007/s40993-022-00391-0}{Determining monodromy groups of abelian varieties}, \textit{Res. Number Theory} \textbf{8} (2022), article number 89. \mr{4496693}

\end{thebibliography}
\end{document}